\documentclass[11pt]{amsart}
\usepackage{enumerate,verbatim}
\usepackage{euscript}
\usepackage{amssymb}
\usepackage{amsmath}
\usepackage{amscd}
\usepackage{epic}

\usepackage{psfrag}
\usepackage{epsfig}
\usepackage{graphicx}
\usepackage{pinlabel}
\usepackage{color}

\title{Homological algebra related to surfaces with boundary}
\author{Kai Cieliebak, Kenji Fukaya and Janko Latschev}

\numberwithin{equation}{section}

\newtheoremstyle{my}{1.5em}{0.5em}{\em}{}{\sc}{.} 
{0.5em}{}

\theoremstyle{my}

\theoremstyle{my}
\newtheorem{thm}{Theorem}[section]
\newtheorem{Theorem}[thm]{Theorem}
\newtheorem*{Theorem*}{Theorem}
\newtheorem{Corollary}[thm]{Corollary}
\newtheorem{corollary}[thm]{Corollary}
\newtheorem*{corollary*}{Corollary}
\newtheorem{Lemma}[thm]{Lemma}
\newtheorem{lem}[thm]{Lemma}

\newtheorem{prop}[thm]{Proposition}
\newtheorem{Proposition}[thm]{Proposition}
\newtheorem{Conjecture}[thm]{Conjecture}
\newtheorem*{conjecture*}{Conjecture}

\newtheorem*{question*}{Question}
\newtheorem{defn}[thm]{Definition}
\newtheorem{Definition}[thm]{Definition}

\newtheorem*{definitions*}{Definitions}
\newtheorem{rem}[thm]{Remark}
\newtheorem*{rem*}{Remark}
\newtheorem{Remark}[thm]{Remark}

\newtheorem*{remark*}{Remark}

\newtheorem*{remarks*}{Remarks}
\newtheorem*{example*}{Example}

\newtheorem*{examples*}{Examples}

\newtheorem*{convention*}{Convention}
\newtheorem*{conventions*}{Conventions}
\newtheorem*{Note*}{Note}
\newtheorem*{exercise*}{Exercise}
\newtheorem*{bibliographical-note*}{Bibliographical note}

\newcommand{\R}{\mathbb{R}}
\newcommand{\Z}{\mathbb{Z}}
\newcommand{\Q}{\mathbb{Q}}

\newcommand{\N}{\mathbb{N}}

\newcommand{\id}{\mathrm{id}}

\newcommand{\im}{\mathrm{im\,}}
\renewcommand{\ker}{\mathrm{ker\,}}

\newcommand{\ext}{\mathrm{ext}}
\newcommand{\inn}{\mathrm{int}}
\newcommand{\IBL}{\mathrm{IBL}}
\newcommand{\sgn}{\mathrm{sgn}}

\newcommand{\Om}{\Omega}

\newcommand{\eps}{\varepsilon}

\newcommand{\p}{\partial}

\newcommand{\fC}{\frak C}
\newcommand{\fD}{\frak D}
\newcommand{\fP}{\frak P}
\newcommand{\fa}{\frak a}
\newcommand{\fb}{\frak b}
\newcommand{\fp}{\frak p}
\newcommand{\fq}{\frak q}

\newcommand{\ff}{\frak f}
\newcommand{\fg}{\frak g}
\newcommand{\fh}{\frak h}
\newcommand{\fm}{\frak m}
\newcommand{\fn}{\frak n}
\newcommand{\fF}{\frak F}

\newcommand{\fG}{\frak G}
\newcommand{\fH}{\frak H}
\newcommand{\fI}{\frak I}

\newcommand{\one}{\mathbf 1}
\newcommand{\wt}{\widetilde}
\newcommand{\wh}{\widehat}
\newcommand{\la}{\langle}
\newcommand{\ra}{\rangle}
\newcommand{\ol}{\overline}

\newcommand{\K}{\mathbb{K}}
\renewcommand{\H}{\mathbb{H}}
\newcommand{\HH}{\widehat{\mathbb{H}}}
\newcommand{\E}{\mathbb{E}}
\newcommand{\F}{\mathbb{F}}
\newcommand{\G}{\mathbb{G}}

\newcommand{\W}{\mathcal{W}}
\newcommand{\WW}{\widehat{\mathcal{W}}}

\newcommand{\J}{\mathbb{J}}
\newcommand{\D}{\mathcal{D}}

\newcommand{\BC}{\mathbf{C}}

\newcommand{\FF}{\mathcal{F}}
\newcommand{\MM}{\mathcal{M}}
\newcommand{\PP}{\mathcal{P}}

\newcommand{\dd}{\mathbf{d }}
\newcommand{\wGamma}{{\widehat\Gamma}}

\newcommand{\Aut}{\operatorname{Aut}}

\begin{document}

\maketitle
\begin{abstract}
In this article we describe an algebraic framework which
can be used in three related but different contexts: string topology,
symplectic field theory, and Lagrangian Floer theory of higher genus.
\end{abstract}
%
\tableofcontents
\section{Introduction}
\label{intro}
The purpose of this article is to describe an algebraic framework 
which can be used for three related but different purposes: 
(equivariant) string topology \cite{ChSu99,ChSu04,Sul07},
symplectic field theory \cite{EGH00}, and
Lagrangian Floer theory \cite{Flo88IV,Oh93} of higher genus.
It turns out that the relevant algebraic structure for all
three contexts is a homotopy version of involutive bi-Lie algebras,
which we call $\IBL_\infty$-algebras. This concept has its
roots in such diverse fields as 
string field theory \cite{Wit86, Zwi91, Sta92, MueSac11},
noncommutative geometry \cite{ASKZ97, Kon941},
homotopy theory \cite{TTsy00},
and others.
To avoid confusion, let us emphasize right away that the algebraic
structure is {\em not a topological field theory} in the sense of
Atiyah and Segal (\cite{Ati89,Seg90}, see also \cite{CoGo04,CoSc09}). 
 
The structure we discuss encodes the combinatorial
structure of certain  compactifications of moduli spaces of Riemann surfaces
with punctures and/or with boundary. Informally, this relation can be
described as follows. To each topological type of a compact connected
oriented 
surface, characterized by its genus $g \geq 0$, the number $k \geq 1$
of ``incoming'' boundary components and the number $\ell \geq 1$ of
``outgoing'' boundary components, one associates a linear map $C^{\otimes k}
\to C^{\otimes \ell}$ between tensor powers of a given vector space
(satisfying certain symmetry properties). Compositions of such maps
correspond to (partial) gluing of incoming boundaries of the second to
outgoing boundary components of the first map. However, in contrast to
a topological field theory, we do not require compositions
corresponding to the same end result of gluing to agree. 
Instead, the case $(k,\ell,g)=(1,1,0)$ gives rise to
a boundary map $\fp_{1,1,0}: C \to C$ squaring to zero, and in general
the operation associated to the surface of signature $(k,\ell,g)$ is a
chain homotopy (with respect to this boundary operator) from the zero
map to the sum of all possible compositions resulting in genus $g$ with
$k$ incoming and $\ell$ outgoing boundaries.

\par
{\bf $\IBL_\infty$-algebras. }
To be more precise, we first introduce some notations (full
details appear in \S \ref{sec:def}).
Let $R$ be a commutative ring which contains $\mathbb Q$
(think of $R = \mathbb Q$ or $\mathbb R$).
Let $C$ be a free graded module over $R$. Its degree shifted version $C[1]$
has graded pieces $C[1]^k = C^{k+1}$. We consider the symmetric product
$$
   E_k C :=
   \left(C[1] \,\,\otimes_R \cdots \,\,\otimes_R C[1]\right)/\sim,
$$
that is the quotient of the tensor product by the action of the
symmetric group permuting the factors with signs.
Let
$
EC
$
be the direct sum $\bigoplus_{k\ge 1} E_kC$.
\par
Next consider a series of $R$-module homomorphisms 
$
\frak p_{k,\ell,g} : E_k C \to E_{\ell} C
$
indexed by triples of integers $k,\ell \ge 1$, $g \geq 0$. 
(They will also have specific degrees, which we ignore in this introduction).  
They canonically extend to $R$-linear maps
$
\hat{\frak p}_{k,\ell,g} :  EC \to  EC
$
(see \S \ref{sec:def}). We introduce the formal sum
$$
\hat \fp := \sum_{k, \ell \geq 1, g \geq 0} \hat\fp_{k,\ell,g}
\hbar^{k+g-1}\tau^{k+\ell+2g-2}: EC\{\hbar,\tau\} \to EC\{\hbar,\tau\}, 
$$
where $EC\{\hbar,\tau\}$ denotes the space of power series in
the formal variables $\hbar,\tau$ with coefficients in $EC$.  
\begin{defn}
We say that $(C,\{\frak p_{k,\ell,g}\}_ 
{k,\ell\geq 1, g \geq 0})$ is an {\em $\IBL_\infty$-algebra}
if 
$$
   \hat \fp \circ \hat \fp = 0.
$$
\label{DefIBL}\end{defn}
Note that $EC$ has both an algebra and a coalgebra structure. 
However, due to the fact that both $k$ and $\ell$ are
allowed to be greater than 1, $\hat \fp$ is neither a derivation nor a
coderivation.   
In fact, one can show that $\hat{\frak p}_{k,\ell,g}$ is a
differential operator of degree $k$ in the sense of graded 
commutative rings (see \cite{CL2} and \S \ref{sec:Weyl}).
The symmetric bar complex $EC$ then inherits a structure
which is a special case of what was called a $BV_{\infty}$
structure in~\cite{CL2} (see also~\cite{ASKZ97,TTsy00}). This
structure can also be  described in terms of a Weyl 
algebra formalism (see \S \ref{sec:Weyl}). Such a formulation has its
origin in the physics literature (see for example \cite{MSS02}) and in
\cite{ASKZ97, Kon941}. The description of symplectic field theory in
\cite{EGH00} is of this form. 
The formulation in Definition \ref{DefIBL} is closer to one
whose origin is in algebraic topology (\cite{Laz55, Sta63} etc.).
It is also close to the algebraic formulation of
Lagrangian Floer theory in \cite{FOOO06}.

\par
In the applications in the context of
holomorphic curves that we have in mind, one is often faced with the
following situation. The principal object of interest is geometric, for
example a Lagrangian submanifold $L \subset (W,\omega)$ in a
symplectic manifold. To analyze it, one
{\em chooses} as auxilliary data a suitable almost complex structure
$J$, and studies $J$-holomorphic curves. Due to transversality issues,
there are often many additional choices that need to be made, but
eventually one writes down a chain complex with operations such as the
$\fp_{k,\ell,g}$ described above, where $k$, $\ell$ and $g$ have the expected
meaning in terms of the holomorphic curves. 

Then one is faced with the task of {\em proving independence of the algebraic
structure of all choices made}. One of the standard methods is to
again use holomorphic curves to define morphisms between the algebraic
structures for two different choices, as well as chain
homotopies between them. To organize such proofs, it is therefore
useful to have explicit algebraic descriptions of the structures, the
morphisms and the homotopies that arise. 

The first goal of this paper is to develop the homotopy theory of
$\IBL_\infty$ algebras from this point of view. We follow the
standard approach via obstruction theory, which leads to fairly
explicit formulas. 
In \S \ref{sec:def} we define $\IBL_\infty$-algebras
and their morphisms, and discuss the defining relations from various
points of view. In \S \ref{sec:tech}, we prepare for the homotopy
theory by identifying the obstructions to extending certain partial
structures or partial morphisms inductively. In \S \ref{sec:hom} we
introduce the notion of homotopy and show that it defines an
equivalence relation. We also prove that compositions of homotopic
morphisms are homotopic.

Note that it follows from the defining equation for an
$\IBL_\infty$-algebra that the map $\fp_{1,1,0}:C[1] \to
C[1]$ is a boundary operator, i.e. we have $\fp_{1,1,0} \circ
\fp_{1,1,0}=0$.  
Moreover, it turns out that part of any morphism $\ff : (C,\{\frak p_ 
{k,\ell,g}\}) \to (C',\{\frak p'_{k,\ell,g}\})$ of
$\IBL_\infty$-algebras is a chain map 
$
\frak f_{1,1,0} : (C, \fp_{1,1,0}) \to (C',\fp_{1,1,0}').
$
In \S \ref{sec:whi} we prove the following 
\begin{Theorem}
Let $\ff$ be a morphism of $\IBL_\infty$-algebras such that
$\ff_{1,1,0}$ induces an isomorphism on homology. If $R$ is a field
of characteristic 0,
then $\ff$ is a homotopy equivalence of $\IBL_\infty$-algebras. 
\label{Whi}
\end{Theorem}
In \S \ref{sec:can} we prove that every $\IBL_\infty$-structure can be
pushed onto its homology: 
\begin{Theorem}\label{thm:canonical}
If $R$ is a field of characteristic 0
and $(C,\{\frak p_{k,\ell,g}\})$ is an $\IBL_\infty$-algebra over $R$, then there exists an
$\IBL_\infty$-structure on its homology $H = H(C,\fp_{1,1,0})$ which is homotopy
equivalent to $(C,\{\frak p_{k,\ell,g}\})$.
\end{Theorem}

{\bf Weyl algebras and symplectic field theory. }
In \S \ref{sec:Weyl} we discuss the relationship of
$\IBL_\infty$-structures with the Weyl algebra formalism which is used
in the formulation of symplectic field theory \cite{EGH00}. 
In a nutshell, the relation is as follows.

In symplectic field theory, the information on moduli spaces of holomorphic curves in the
symplectization of a contact manifold is packaged in a certain
Hamiltonian function $\H$, which is a formal power series in
$p$-variables (corresponding to positive asymptotics, or ``inputs'')
and $\hbar$ (corresponding to genus) with polynomial coefficients in
the $q$-variables (corresponding to negative asymptotics, or
``outputs''). The exactness of the symplectic form on the
symplectization forces $\H|_{p=0}=0$. There is a notion of
augmentation in this context, which allows one to change the structure to
also achieve $\H|_{q=0}=0$. Geometrically, augmentations arise from
{\em symplectic fillings} of the given contact manifold. It is this
augmented part of symplectic field theory which is shown to be equivalent to the
$\IBL_\infty$ formalism described here. 

Let us also mention that in \cite{DTT08} the Weyl algebra formalism 
was shown to be equivalent to the structure of an algebra over a
certain properad, at least on the level of objects.
The emergence of $\IBL_\infty$-operations on $S^1$-equivariant
symplectic cohomology (which in view of~\cite{BouOan09} is essentially
equivalent to symplectic field theory) is outlined
in~\cite{Sei08-biased}. 

{\bf Filtrations and Maurer-Cartan elements. }
For various applications one needs the more general notion of a {\em
  filtered $\IBL_\infty$-algebra} which we introduce in \S
\ref{sec:filter}.  
As is common in homotopical algebra, there is a version of the
Maurer-Cartan equation for our structure. We discuss this in \S
\ref{sec:MC}, where we show that Maurer-Cartan elements have many of
the expected properties.

This concludes the basic part of the theory.
The remaining part of the paper gives some ideas how
$\IBL_\infty$-structures arise in algebraic and symplectic topology.
\par


{\bf The dual cyclic bar complex of a cyclic DGA. } 
The relation of string topology to Hochschild cohomology (of, say, the de 
Rham complex) and
to Chen's iterated integrals \cite{Chen73} has already been 
described by various authors, see e.g.
\cite{CoJo99, Kauf04, Mer04, CoVo, Sei06, Fuk05II, Sul07}.
In particular, our operator $\fp_{2,1,0}$ corresponds to the
Gerstenhaber bracket \cite{Ger79} in that case.
Our description below can be regarded as a `higher genus analogue' of 
it.
\par
With de Rham cohomology in mind, let us restrict to  $R=\mathbb R$.
Recall that an $A_{\infty}$-algebra structure on a graded $\mathbb R$-vector 
space $A$ consists of a series of $\mathbb R$-linear maps
$
\frak m_k : A[1]^{\otimes k} \to A[1]
$
for $k \ge 1$, which are 
assumed to satisfy the so called $A_{\infty}$-relations 
(see \S \ref{sec:Cyclic}).
A cyclic $A_{\infty}$-algebra in addition comes with a nondegenerate pairing
$
\langle \,,\, \rangle : A \otimes A 
 \to \mathbb R,
$
of degree $-n$ such that
$$
\langle \frak m_k(x_1,\cdots,x_k), x_0 \rangle
   = (-1)^*\langle \frak m_k(x_0,\cdots,x_{k-1}),x_k 
\rangle
$$
with suitable signs $(-1)^*$. 
The notion of an $A_{\infty}$-algebra was introduced by 
Stasheff \cite{Sta92} and cyclic
$A_{\infty}$-algebras were used in a related context by 
Kontsevich \cite{Kon94}.
Their relation to symplectic topology was
discussed in \cite{Kon94, Fuk93, Che02}.

We can construct a $\text{\rm dIBL}$-algebra (i.e., an
$\IBL_\infty$-algebra whose only nonvanishing terms are $\fp_{1,1,0}$,
$\fp_{1,2,0}$ and $\fp_{2,1,0}$) from a cyclic differential
graded algebra (DGA) as follows.
We let
$
B^{\text{\rm cyc}}_kA
$
be the quotient of $A[1]^{\otimes k}$ by the 
$\mathbb Z_k$-action which cyclically permutes the factors
with signs.
In \S \ref{sec:Cyclic-cochain} and \S \ref{sec:subcomplex} we prove
the following two results.  

\begin{Proposition}\label{prop:structureexists-intro}
Let $(A, \langle \,, \,\rangle, d)$ be a finite dimensional cyclic
cochain complex whose pairing has degree $-n$. Then the shifted
dual cyclic bar complex $(B^{\text{\rm cyc}*}A)[2-n] =
\bigoplus_{k\geq 1}Hom(B^{\text 
{\rm cyc}}_kA,\mathbb R)[2-n]$ carries natural operations
$\fp_{1,1,0} = d$, $\fp_{2,1,0}$ and $\fp_{1,2,0}$ which make it a
$\rm dIBL$-algebra. 
\end{Proposition}

\begin{Theorem}\label{thm:homotopyequiv-intro} 
The {\rm dIBL}-structure in
Proposition~$\ref{prop:structureexists-intro}$ is 
$\IBL_\infty$-homotopy equivalent to the analogous structure
on the dual cyclic bar complex of the cohomology $H(A,d)$. 
\end{Theorem}

So far, the algebra structure on $A$ was not used. In \S
\ref{sec:Cyclic} we take it into account and prove

\begin{Proposition}\label{prop:existoncyc}
Let $A$ be a finite dimensional cyclic DGA whose pairing has degree
$-n$. Then its product $\fm_2$ gives rise to a Maurer-Cartan element
$\fm_2^+$ for the $\text{\rm dIBL}$-structure on the dual cyclic bar
complex of $A$ (completed with respect to its canonical filtration). 
\end{Proposition}

We can twist the $\text{\rm dIBL}$-structure of
Proposition~\ref{prop:structureexists-intro} by the Maurer-Cartan
element from Proposition~\ref{prop:existoncyc} to obtain a 
{\em twisted} filtered $\text{\rm dIBL}$-structure on the dual cyclic
bar complex $(B^{\text{\rm cyc}*}A)[2-n]$. Using (the filtered version
of) Theorem~\ref{thm:canonical}, this structure can be pushed to its
homology with respect to the twisted differential. 
Moreover, using Theorem~\ref{thm:homotopyequiv-intro} and general
homotopy theory of $\IBL_\infty$-algebras we prove 

\begin{Theorem}\label{thm:existoncyc2}
Let $A$ be a finite dimensional cyclic DGA whose pairing has
degree $-n$, and let $H=H(A,d)$ be its 
cohomology. Then there exists a filtered  $\IBL_\infty$-structure on 
$(B^{\text{\rm cyc}*}H)[2-n]$ which is $\IBL_\infty$-homotopy 
equivalent to $(B^{\text{\rm cyc}*}A)[2-n]$ with its twisted filtered 
$\text{\rm dIBL}$-structure. 
In particular, its homology equals Connes' version of cyclic
cohomology of $A$ as defined e.g.~in~\cite{Lod}. 
\end{Theorem}


\begin{rem}
Theorem~\ref{thm:existoncyc2} together with its idea of proof 
using summation over ribbon graphs was explained by the authors 
on several occasions, for example by the second named author 
at the conference `Higher Structures in Geometry and Physics' in Paris
$2007$. At the final stage of completing this paper, the authors 
found that results in~\cite{Bar2} and~\cite{Chen10} seem to be closely
related to Theorem~\ref{thm:existoncyc2}.
\end{rem}


{\bf The de Rham complex and string topology. }
Now consider the de Rham complex $(\Omega(M),d)$ of a closed oriented
manifold $M$. The wedge product and the intersection pairing
$\int_Mu\wedge v$ give $\Omega(M)$ the structure of a cyclic
DGA. However, it is not finite dimensional, so we cannot directly
apply the theory in Sections \S \ref{sec:Cyclic-cochain} and \S
\ref{sec:subcomplex}. 
To remedy this, we introduce in \S \ref{de Rham} the subspaces 
$B^{\text{cyc}*}_k\Omega(M)_{\infty}\subset
B^{\text{cyc}*}_k\Omega(M)$ of operators with smooth kernel and prove 
the following analogues of
Proposition~\ref{prop:structureexists-intro} and
Theorem~\ref{thm:homotopyequiv-intro} 
(see \S \ref{de Rham} for the relevant definitions)

\begin{Proposition}\label{prop:cycIBLde-intro}
Let $M$ be a closed oriented manifold of dimension $n$. Then 
$B^{\text{\rm cyc}*}\Omega(M)_{\infty}[2-n]$ carries the structure of a 
Fr\'echet {\rm dIBL}-algebra. 
\end{Proposition}

\begin{Theorem}\label{homotopyequiv-dR-intro}
The Fr\'echet ${\rm dIBL}$-structure in Proposition~$\ref{prop:cycIBLde-intro}$ is 
$\IBL_\infty$-homotopy equivalent to the analogous structure
on the dual cyclic bar complex of the de Rham cohomology $H_{\rm
  dR}(M)$.  
\end{Theorem}

The triple intersection product $\frak m^+_2(u,v,w) =  (-1)^{*} \int_M
u \wedge v\wedge w$ defines an element $\frak m^+_2 \in B^{\text{\rm
    cyc}*}_3\Omega(M)$ satisfying the equations of a Maurer-Cartan
element. However, $\frak m^+_2$ does {\it not} have a smooth kernel, so
we cannot use it directly to twist the Fr\'echet  {\rm
  dIBL}-structure. Nevertheless, we expect that by pushing the
structure onto the (finite dimensional!) de Rham cohomology
$H_{dR}(M)=H(\Omega(M),d)$ 
one can prove the following analogue of Theorem~\ref{thm:existoncyc2}. 

\begin{Conjecture}\label{deRhamconj-intro}
Let $M$ be a closed oriented manifold of dimension $n$ and
$H=H_{dR}(M)$ its de Rham cohomology. Then there exists a filtered 
$\IBL_\infty$-structure on $B^{\text{\rm cyc}*}H[2-n]$ whose homology
equals Connes' version of cyclic cohomology of the de Rham complex of $M$. 
\end{Conjecture}

Recall that if $M$ is simply connected, then by a theorem of Jones \cite{Jo}
(see also Chen~\cite{Chen73} and Goodwillie~\cite{Good}) 
various versions of cyclic cohomology
of $\Omega(M)$ are isomorphic to suitable versions of
$S^1$-equivariant homology of the free loop space of $M$.  
It is shown in~\cite{CV} that Connes' version of cyclic cohomology is
isomorphic to a ``reduced'' version of $S^1$-equivariant homology of
the free loop space, which differs from standard equivariant homology
(defined via the Borel construction) only by the equivariant homology
of a point.  
We conjecture that under this identification the involutive Lie
bialgebra structure induced by the 
$\IBL_\infty$-structure in Conjecture~\ref{deRhamconj-intro}
agrees with the string bracket and cobracket described by Chas and
Sullivan in~\cite{ChSu99, ChSu04}. This is supposed to be a special
case of such a structure on the equivariant loop space
homology of any closed oriented manifold $M$ \cite{Sul07,CL2}. 
\par

The strategy to prove Conjecture~\ref{deRhamconj-intro} is to mimic 
the proof of Theorem~\ref{thm:existoncyc2} in the de Rham case. Then 
finite sums over basis elements get replaced by multiple integrals
involving the Green kernel associated to a Riemannian metric on
$M$, in a way similar to \cite{Kad82, KoSo01,Fuk03II}
and to perturbative Chern-Simons gauge theory \cite{Wit95,BarNat95}.  
The difficulty in making this rigorous arises from possible divergences
at the diagonal where some integration variables become equal. We hope
to show in future work that this problem can be resolved in a similar
way as in perturbative Chern-Simons gauge theory \cite{AxSi91II}. 

%

{\bf Lagrangian Floer theory. }
Finally, consider an $n$-dimensional closed oriented Lagrangian
submanifold $L$ of a symplectic manifold $(W,\omega)$ (closed or
convex at infinity). 
Then holomorphic curves in $W$ with boundary on $L$ give
rise to a further deformation of the $\IBL_\infty$-structure in
Conjecture~\ref{deRhamconj-intro} associated to $L$. The structure arising from
holomorphic disks has been described in~\cite{Fuk05II}, and the
general structure (in slightly different language) in~\cite{CL2}. 
In the terminology of this paper, it can be described as follows. 

It is proved in~\cite{FOOO06, Fuk10} that moduli spaces of holomorphic disks
with boundary on $L$ give rise to a (filtered) cyclic
$A_\infty$-structure on its de Rham cohomology $H_{dR}(L)$. Moduli
spaces of holomorpic curves of genus zero with several boundary
components should give rise to a solution of an appropriate 
version of Batalin-Vilkovisky Master equation, see \S \ref{sec:Cyclic}.  
Moreover, this data can be further enhanced using holomorphic curves
of higher genus.  
We prove in \S \ref{sec:Cyclic} that Proposition~\ref{prop:existoncyc}
carries over to such $A_\infty$-algebras, so we arrive at the following 

\begin{Conjecture}\label{Lagconj-intro}
Let $L$ be an $n$-dimensional closed oriented Lagrangian
submanifold $L$ of a symplectic manifold $(W,\omega)$ (closed
or convex at infinity) and let 
$H=H_{dR}(M)$ be its de Rham cohomology. Then there exists a filtered
$\IBL_\infty$-structure on $B^{\text{\rm cyc}*}H[2-n]$ whose homology
equals the cyclic cohomology of the cyclic $A_\infty$-structure on $H$
constructed in~\cite{FOOO06, Fuk10}.  
\end{Conjecture}

\begin{Remark}
Suppose that $\pi:W\to D$ is an exact symplectic Lefschetz fibration
over the disk (e.g.~any Stein domain $W$ admits such a
fibration) and let $L_1,\dots,L_m$ be the vanishing cycles. 
A conjecture of Seidel (\cite{Sei09}, see~\cite{BEE09} for further
evidence for this conjecture) asserts that the symplectic
homology of $W$ equals the Hochschild homology of a certain
$A_\infty$-category with objects $L_1,\dots,L_m$. An equivariant
version of this conjecture would equate the $S^1$-equivariant
symplectic cohomology $SH_{S^1}^*(W)$ to the cyclic cohomology of this
$A_\infty$-category. Hence the $\IBL_\infty$-structure on
$SH_{S^1}^*(W)$ mentioned above 
would follow from this conjecture and
an extension of Proposition~\ref{prop:existoncyc} to suitable
$A_\infty$-categories.  
\end{Remark}
\bigskip

\centerline{\bf Acknowledgements}
We thank
A. Cattaneo,
E. Getzler, 
K. M\"unster,
I. Sachs, 
B. Vallette,
and E. Volkov
for stimulating discussions. 
We also thank S. Baranikov, G. Drummond-Cole
and A. Voronov for their comments about the previous version 
of this paper.

\section{Involutive Lie bialgebras up to infinite homotopy}\label{sec:def}

In this section we define involutive Lie bialgebras up to infinite
homotopy, or briefly $\IBL_\infty$-algebras, and morphisms among them.

Let $R$ be a commutative ring with unit that contains $\Q$. 
Let $C= \bigoplus_{k\in\Z} C^k$ be a free graded $R$-module.

It is convenient to introduce the degree shifted version $C[1]$ of $C$
by setting $C[1]^d:=C^{d+1}$. Thus the degrees $\deg c$ in $C$ and
$|c|$ in $C[1]$ are related by
$$
   |c| = \deg c -1. 
$$
We introduce the $k$-fold symmetric product 
$$
    E_k C := \left(C[1] \,\otimes_R \cdots
    \,\otimes_R C[1]\right)/\sim 
$$
as the quotient of the $k$-fold tensor product under the standard
action of the symmetric group $S_k$ permuting the factors with signs,
and the {\em reduced symmetric algebra}
$$
   EC := \bigoplus_{k\geq 1}E_kC. 
$$
Note that we do not include a constant term in $EC$. As usual,
we write the equivalence class of $c_1\otimes\dots\otimes c_k$ in
$E_kC$ as $c_1\cdots c_k$. 

\begin{Remark}\label{rem:tensor}
More precisely, we set $EC := \bigoplus_{k\geq 1}C[1]^{\otimes
  k}/\frak I$,
where $\frak I$ is the two-sided ideal generated by all elements $c\otimes
c'-(-1)^{|c|\,|c'|}c'\otimes c$. Since $R$ contains $\Q$, $EC$ is
canonically isomorphic as a graded $R$-module to the subspace
$$
   \Bigl(\bigoplus_{k\geq 1}C[1]^{\otimes k}\Bigr)^{\rm symm}\subset
   \bigoplus_{k\geq 1}C[1]^{\otimes k} 
$$
of invariant tensors under the action of the symmetric group: An
inverse of the quotient map $\Bigl(\bigoplus_{k\geq 1}C[1]^{\otimes
  k}\Bigr)^{\rm symm}\to EC$ is given by the 
symmetrization map 
$$
   I(c_1\cdots c_k):=
   \frac{1}{k!}\sum_{\rho\in S_k} \eps(\rho)c_{\rho(1)}\otimes\cdots\otimes
  c_{\rho(k)}. 
$$
Here the sign $\eps(\rho)$ (which depends on the $c_i$)
is defined by the equation
$$
    c_{\rho(1)}\cdots c_{\rho(k)}=\eps(\rho)c_1\cdots c_k. 
$$ 
\end{Remark}

We extend any linear map $\phi:E_kC\to E_\ell C$ to a
linear map $\hat \phi:EC \to EC$ by $\hat \phi:=0$ on $E_mC$ for $m<k$ and 
\begin{align}\label{eq:hat}
    \hat \phi(c_1\cdots c_{m})
    &:= \sum_{{\rho\in S_m} \atop {\rho(1) < \dots <
    \rho(k) \atop \rho(k+1) < \dots < \rho(m)}} \hspace{-.7cm}
    \eps(\rho) \phi(c_{\rho(1)} \cdots c_{\rho(k)})
     c_{\rho(k+1)}  \cdots  c_{\rho(m)} \cr
    &= \sum_{\rho\in S_m} \frac{\eps(\rho)}{k!(m-k)!} \phi(c_{\rho(1)}
    \cdots c_{\rho(k)}) c_{\rho(k+1)}  \cdots  c_{\rho(m)}.
\end{align}
for $m\geq k$. Note that $\hat \phi$ maps $E_{k+s}C$ to $E_{l+s}C$ for every $s\geq 0$. 

\begin{rem}
The map $\hat \phi$ is a differential operator of order $\leq k$ and a
``codifferential operator'' of order $\leq\ell$.
In particular, $\hat \phi$ is a derivation if $k=1$ and a coderivation if
$\ell=1$. 
\end{rem}

Now we consider a series of graded $R$-module homomorphisms 
$$
    \frak p_{k,\ell,g} : E_k C \to E_{\ell} C,\qquad k,\ell \ge 1,\
    g\geq 0
$$ 
of degree 
$$
    |\frak p_{k,\ell,g}| = -2d(k+g-1)-1
$$ 
for some fixed integer $d$.
Define the operator
$$
    \hat\fp := \sum_{k,\ell=1}^\infty\sum_{g=0}^\infty
    \hat\fp_{k,\ell,g} \hbar^{k+g-1}\tau^{k+\ell+2g-2}:
    EC\{\hbar,\tau\}\to EC\{\hbar,\tau\},
$$
where $\hbar$ and $\tau$ are formal variables of degree
$$
    |\hbar| := 2d, \quad |\tau|=0,
$$
and $EC\{\hbar,\tau\}$ denotes formal power series in these variables
with coefficients in $EC$. 

\begin{defn}\label{def:IBL}
We say that $(C,\{\frak p_{k,\ell,g}\}_{k,\ell \geq 1,g\geq 0})$ is an
{\em $\IBL_\infty$-algebra of degree $d$} if 
\begin{equation}\label{eq:BL2}
    \hat\fp\circ \hat\fp=0.
\end{equation}
\end{defn}

\begin{rem}
The algebra over the Frobenius properad appearing in~\cite{CMW}
is closely related to the $\IBL_\infty$-algebra.
The Koszul--ness of the former in the sense of~\cite{Vallette} is
proved in~\cite{CMW}. 
It provides a purely algebraic reasoning for  
this structure being a `correct' infinity version of 
an involutive Lie bialgebra. We however do not use this fact in this paper.
\end{rem}

Let us explain this definition from various angles.

(1) One can write equation \eqref{eq:BL2} more explicitly as the
sequence of equations
\begin{equation}\label{eq:BL}
    \sum_{t=2-\min(k,\ell)}^{g+1}
    \sum_{k_1+k_2=k+t \atop {\ell_1+\ell_2=\ell+t \atop g_1+g_2=g+1-t}}
    (\hat \fp_{k_2,\ell_2,g_2}\circ \hat \fp_{k_1,\ell_1,g_1})|_{E_kC} = 0
\end{equation}
for each triple $(k,\ell,g)$ with $k,\ell \geq 1$ and $g\geq 1- \min(k,\ell)$. 
Indeed, equation~\eqref{eq:BL} is the part of the coefficient of
$\hbar^{k+g-1}\tau^{k+\ell+2g-2}$ in \eqref{eq:BL2} mapping $E_kC$ to
$E_\ell C$. 

More appropriately, one should view equation~\eqref{eq:BL} as the
definition of an $\IBL_\infty$-structure, and the formal variables
$\hbar,\tau$ are mere bookkeeping devices that allow us to write this
equation in the more concise form~\eqref{eq:BL2}. Alternatively, we
could also consider equation~\eqref{eq:BL2} on the space
$\prod_{k\geq 1}E_kC\{\hbar\}$. 

(2) It is instructive to think of $\fp_{k,\ell,g}$ as an operation
associated to a compact {\em connected} oriented surface
$S_{k,\ell,g}$ of {\em signature} 
$(k,\ell,g)$, i.e.~with $k$ incoming and $\ell$ outgoing boundary
components and of genus $g$. Then the coefficient $k+\ell+2g-2$ of the
formal variable $\tau$ is the negative Euler characteristic of
$S_{k,\ell,g}$. 
%
%
\begin{figure}[h]
 \labellist
 \endlabellist
  \centering
  \includegraphics[scale=.95]{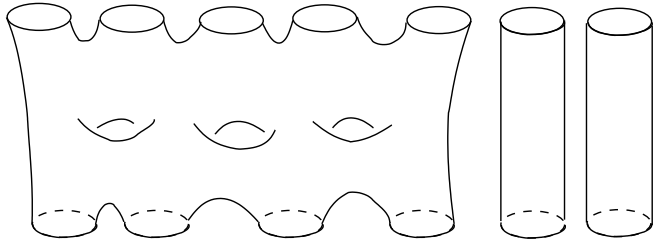}
 \caption{On the left is a pictorial representation of $\fp_{5,4,3}$ by a surface with signature $(5,4,3)$, with incoming boundaries drawn at the top and outgoing boundaries at the bottom. The whole picture would be our graphical representation of the extension $\hat\fp_{5,4,3}:E_{5+2}C \to E_{4+2}C$.}
 \label{fig:surface_klg_with_cylinders}
\end{figure}
It is also useful to think of the identity
$C\to C$ as the operation associated to a {\em trivial cylinder}. 
Then the extension $\hat\fp_{k,\ell,g}:E_{k+r}C\to E_{\ell+r}C$
corresponds to the disjoint union of $S_{k,\ell,g}$ with $r$ trivial
cylinders. Define the genus $g \in \Z$ of a possibly disconnected
surface with $k$ incoming and $\ell$ outgoing boundary components such
that its Euler characteristic equals $2-2g-k-\ell$, so e.g.~adding a
cylinder lowers the genus by one. Then the terms of the form $\hat
\fp_{k_2,\ell_2,g_2}\circ\hat\fp_{k_1,\ell_1,g_1}$ on the left hand
side of~\eqref{eq:BL} correspond to all possible gluings of two
connected surfaces of signatures $(k_1,\ell_1,g_1)$ and
$(k_2,\ell_2,g_2)$, plus the appropriate number of trivial cylinders,
to a {\em possibly disconnected} surface of signature $(k,\ell,g)$.

The following definition will be repeatedly used in inductive
arguments.

\begin{defn}\label{def:order}
We define a linear order on signatures by saying $(k',\ell',g') \prec
(k,\ell,g)$ if one of the following conditions holds:
\begin{enumerate}[\rm (i)]
\item $k'+\ell'+2g' < k+\ell+2g$,
\item $k'+\ell'+2g' = k+\ell+2g$ and $g'>g$, or
\item $k'+\ell'+2g' = k+\ell+2g$ and $g'=g$ and $k'<k$.
\end{enumerate}
\end{defn}
This choice of ordering is explained in Remark~\ref{rem:ordering} below. The sequence of ordered signatures starts with
$$
(1,1,0)\prec (1,2,0) \prec (2,1,0) \prec (1,1,1) \prec (1,3,0) \prec
(2,2,0) \prec (3,1,0) \prec \dots
$$

(3) The preceding discussion suggests that~\eqref{eq:BL} can be
reformulated in terms of gluing to {\em connected} surfaces. For this,
let us denote by $\hat \fp_{k_2,\ell_2,g_2}\circ_s \hat
\fp_{k_1,\ell_1,g_1}$ the part of the composition where exactly
$s$ of the inputs of $\fp_{k_2,\ell_2,g_2}$ are outputs of
$\fp_{k_1,\ell_1,g_1}$.

\begin{lem}\label{lem:order}
Equation~\eqref{eq:BL2} is equivalent to the sequence of equations
\begin{equation}\label{eq:BL4}
    \sum_{s=1}^{g+1}
    \sum_{k_1+k_2=k+s \atop {\ell_1+\ell_2=\ell+s \atop g_1+g_2=g+1-s}}
    (\hat \fp_{k_2,\ell_2,g_2}\circ_s \hat
    \fp_{k_1,\ell_1,g_1})|_{E_kC} = 0, \quad k,\ell \geq 1, g \geq 0.
\end{equation}
Moreover, for $(k,\ell,g)\succ (1,1,0)$ equation~\eqref{eq:BL4} has the form
\begin{equation}\label{eq:BL5}
\hat\fp_{1,1,0}\circ \fp_{k,\ell,g}
+\fp_{k,\ell,g}\circ \hat\fp_{1,1,0} +P_{k,\ell,g} =0,
\end{equation}
where $P_{k,\ell,g}:E_kC \to E_\ell C$ involves only compositions of
terms $\fp_{k',\ell',g'}$ whose signatures satisfy 
$(1,1,0)\prec (k',\ell',g')\prec (k,\ell,g)$. 
\end{lem}
%
%
\begin{figure}[h]
 \labellist
 \endlabellist
  \centering
  \includegraphics[scale=0.9]{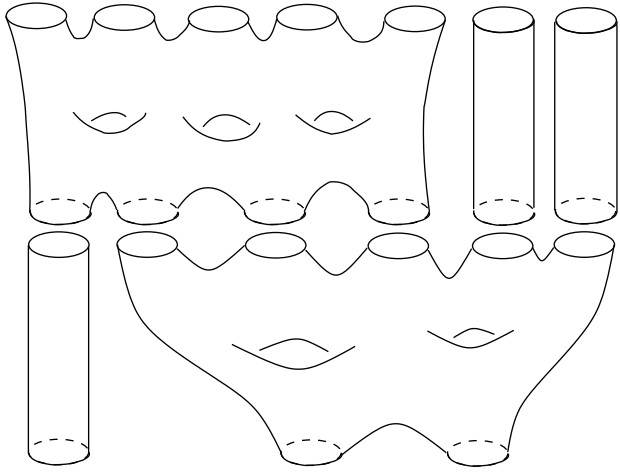}
 \caption{A typical term in $\hat\fp_{5,4,3}\circ_3 \hat\fp_{5,2,2}:E_7C \to E_3C$, appearing on the left hand side of equation \eqref{eq:BL4} for $(k,\ell,g)= (7,3,7)$. }
 \label{fig:surface_gluing}
\end{figure}

\begin{rem}
\label{rem:ordering}
The order $\prec$ on signatures was chosen so that \eqref{eq:BL} for all signatures $(k,l,g) \preceq (K,L,G)$ is equivalent to \eqref{eq:BL4} for the same range of signatures. Other choices are possible.
\end{rem}
\begin{proof}
Abbreviate the left-hand side of~\eqref{eq:BL4} by 
$$
    \fq_{k,\ell,g}:=\sum_{s=1}^{g+1}
    \sum_{k_1+k_2=k+s \atop {\ell_1+\ell_2=\ell+s \atop g_1+g_2=g+1-s}}
    (\hat \fp_{k_2,\ell_2,g_2}\circ_s \hat
    \fp_{k_1,\ell_1,g_1})|_{E_kC}:E_kC\to E_\ell C. 
$$
Note that the terms of the form $\hat
\fp_{k_2,\ell_2,g_2}\circ_s\hat\fp_{k_1,\ell_1,g_1}$ on the right hand
side of this definition correspond to all possible gluings of two
surfaces of signatures $(k_1,\ell_1,g_1)$ and $(k_2,\ell_2,g_2)$ along
$s$ boundary loops (outgoing for the first one and incoming for the
second one) to a connected surface of signature $(k,\ell,g)$. 
Denote by $\hat\fq_{k,\ell,g}$ the usual extension. Then 
\begin{align*}
    \sum_{r=0}^{\min(k,\ell)-1}\hat\fq_{k-r,\ell-r,g+r}
    &= \sum_{r=0}^{\min(k,\ell)-1}\sum_{s=1}^{g+1+r}
    \sum_{k_1+k_2=k+s-r \atop {\ell_1+\ell_2=\ell+s-r \atop g_1+g_2=g+1-s+r}}
    (\hat \fp_{k_2,\ell_2,g_2}\circ_s \hat
    \fp_{k_1,\ell_1,g_1})|_{E_kC} \cr
    &= \sum_{t=2-\min(k,\ell)}^{g+1}
    \sum_{k_1+k_2=k+t \atop {\ell_1+\ell_2=\ell+t \atop g_1+g_2=g+1-t}}
    (\hat \fp_{k_2,\ell_2,g_2}\circ \hat \fp_{k_1,\ell_1,g_1})|_{E_kC}. 
\end{align*}
Here the last equality follows by setting $t=s-r$ and observing that
$$
\hat \fp_{k_2,\ell_2,g_2}\circ_0 \hat \fp_{k_1,\ell_1,g_1}=-
\hat \fp_{k_1,\ell_1,g_1}\circ_0 \hat \fp_{k_2,\ell_2,g_2}.
$$
So the additional terms corresponding to $s=0$ (which were not present in
the line above) appear in cancelling pairs. Since the last
expression agrees with the one in~\eqref{eq:BL}, this shows
that~\eqref{eq:BL2} is equivalent to the sequence of equations
\begin{equation}\label{eq:BL3}
    \sum_{r=0}^{\min(k,\ell)-1}\hat\fq_{k-r,\ell-r,g+r}=0,\quad k,\ell
    \geq 1, g \geq 1-\min(k,\ell), 
\end{equation}
where summands with $g+r<0$ are interpreted as $0$.
Clearly~\eqref{eq:BL4} implies~\eqref{eq:BL3}. The converse
implication follows by induction over the order $\prec$
because~\eqref{eq:BL3} is of the form 
$$
    \fq_{k,\ell,g} + Q_{k,\ell,g} = 0, 
$$
where $Q_{k,\ell,g}$ is a sum of terms $\hat\fq_{k',\ell',g'}$ with
$(k',\ell',g')\prec (k,\ell,g)$. 

For the last statement, consider a term $\hat
\fp_{k_2,\ell_2,g_2}\circ_s\hat\fp_{k_1,\ell_1,g_1}$ appearing in
equation~\eqref{eq:BL4}. Denote by $\chi_{k,\ell,g}:=2-2g-k-\ell$ the
Euler characteristic of a surface of signature $(k,\ell,g)$. Since the
Euler characteristic is additive and only nonpositive Euler
characteristics occur for the admissible triples, we get
$$
   \chi_{k_1,\ell_1,g_1} +\chi_{k_2,\ell_2,g_2} = \chi_{k,\ell,g}
$$
with all terms being nonpositive. If $\chi_{k_1,\ell_1,g_1}<0$ and
$\chi_{k_2,\ell_2,g_2}<0$, then
$(k_1,\ell_1,g_1),\;(k_2,\ell_2,g_2)\prec (k,\ell,g)$ by definition of
the ordering. If $\chi_{k_1,\ell_1,g_1}=0$, then
$(k_1,\ell_1,g_1)=(1,1,0)$ and it follows that
$(k_2,\ell_2,g_2)=(k,\ell,g)$, so we find the first term in
equation~\eqref{eq:BL5}. Similarly, in the case
$\chi_{k_2,\ell_2,g_2}=0$ we find the second term in
equation~\eqref{eq:BL5}. 
\end{proof}

\begin{rem}\label{rem:gen-IBL}
Note that the proof of Lemma~\ref{lem:order} only uses property (i) in
Definition~\ref{def:order}. Moreover, the proof still works if in
Definition~\ref{def:IBL} we allow all triples $(k,\ell,g)$ except the
following ones:
\begin{equation}\label{eq:triples}
   (0,0,0),\ (1,0,0),\ (0,1,0),\ (0,0,1),\ (2,0,0),\ (0,2,0).
\end{equation}
We call such a structure a {\em generalized
  $\IBL_\infty$-structure}. The notion below of an $\IBL_\infty$-morphism 
can be generalized in the same way, and all the theory in Sections 2-6
works for these generalized structures with the exception of the
following discussion in (4). 
\end{rem}

(4) $\IBL_\infty$-algebra stands for {\em involutive bi-Lie algebra up to
infinite homotopy}. To justify this terminology, 
let us spell out equation~\eqref{eq:BL4} for the first few triples
$(k,\ell,g)$.  For $(k,\ell,g)=(1,1,0)$ we find that 
$$
    \fp_{1,1,0}:C\to C
$$ 
has square zero. For $(k,\ell,g)=(2,1,0)$ we find that
$\hat\fp_{2,1,0}$ is a chain map (with respect to $\hat\fp_{1,1,0}$)
whose square is chain homotopic to zero by the relation for
$(k,\ell,g)=(3,1,0)$. It follows that 
$$
    \mu(a,b) := (-1)^{|a|}\fp_{2,1,0}(a,b)
$$
defines a graded Lie bracket up to homotopy on $C$. Similarly, the relations
for $(k,\ell,g)=(1,2,0)$ and $(1,3,0)$ show that 
$$
    \delta(a) := (\iota \otimes \one)\fp_{1,2,0}(a)
$$
is a chain map which defines a graded Lie cobracket up to homotopy on
$C$. Here $\iota$ multiplies a homogeneous element $c\in C$ by
$(-1)^{|c|}$.

Now the relation with $(k,\ell,g)=(1,1,1)$ reads
$$
\fp_{1,1,0}\circ_1\fp_{1,1,1}+\fp_{1,1,1}\circ_1\fp_{1,1,0}+
\fp_{2,1,0}\circ_2\fp_{1,2,0}=0
$$
and yields involutivity of $(\mu,\delta)$ up to homotopy. 
Finally, the relation with $(k,\ell,g)=(2,2,0)$ has the form
$$
\hat\fp_{1,1,0}\circ_1\fp_{2,2,0}+\fp_{2,2,0}\circ_1\hat\fp_{1,1,0}+
\fp_{1,2,0}\circ_1 \fp_{2,1,0} +
\hat\fp_{2,1,0}\circ_1\hat\fp_{1,2,0} = 0, 
$$
yielding compatibility of $\mu$ and $\delta$ up to
homotopy. In summary, $(\mu,\delta)$ induces the structure of an
involutive bi-Lie algebra on the homology $H(C,\p)$.  

(5) Let us consider some special cases of an $\IBL_\infty$-structure. 
If $\fp_{k,\ell,g}=0$ whenever $\ell\geq 2$ or $g>0$, then
\eqref{eq:BL4} is equivalent to the sequence of equations
$$
    \sum_{k_1+k_2=k+1} (\fp_{k_2,1,0}\circ_1 \hat \fp_{k_1,1,0})|_{E_kC}
    = 0, \quad    k=1,2,\dots 
$$
So we recover one of the standard definitions of an $L_\infty$-algebra
(cf. \cite{MSS02}). 
Similarly, if $\fp_{k,\ell,g}=0$ whenever
$k\geq 2$ or $g>0$ we recover the definition of a co-$L_\infty$ structure.

(6) Suppose the following finiteness condition is satisfied:
{\em Given $k\geq 1$, $g\geq 0$ and $a\in E_kC$, the term
$\fp_{k,\ell,g}(a)$ is nonzero for only finitely many $\ell\geq 1$.}
Then we can set $\tau=1$ above and consider $\hat\fp$ as a map
of $EC\{\hbar\}$. This is the case in our main examples
(exact symplectic field theory and string topology); see
e.g.~Section~\ref{sec:Weyl} below.  
\medskip

{\bf Morphisms. }
Next we turn to the definition of morphisms. So let
$(C^+,\{\fp^+_{k,\ell,g}\})$ and $(C^-,\{\fp^-_{k,\ell,g}\})$ be 
two $\IBL_\infty$-algebras of the same degree $d$. 
To a collection of linear maps $f_i:E_{k_i}C^+\to E_{\ell_i}C^-$,
$1\leq i\leq r$, we associate a linear map $f_1\odot\dots\odot
f_r:E_{k_1+\dots+k_r}C^+ \to E_{\ell_1+\dots+\ell_r}C^-$ by 
\begin{align}\label{eq:wedge}
    &f_1\odot\dots\odot f_r(c_1\cdots c_k)\\
    &:= \sum_{{\rho\in S_k}
    \atop { { {\rho(1) < \dots < \rho(k_1)}
    \atop {\dots} }
    \atop {\rho(k_1+\dots+k_{r-1}+1) < \dots < \rho(k)} }
    } \hspace{-1cm}
    \eps(\rho) f_1(c_{\rho(1)} \cdots c_{\rho(k_1)}) \cdots
     f_r(c_{\rho(k_1+\dots+k_{r-1}+1)}\cdots c_{\rho(k)}) \cr
    &= \sum_{\rho\in S_k}
    \frac{\eps(\rho)}{k_1!\cdots k_r!} f_1(c_{\rho(1)} \cdots
    c_{\rho(k_1)}) \cdots
     f_r(c_{\rho(k_1+\dots+k_{r-1}+1)}\cdots c_{\rho(k)}). 
\end{align}
Note that $f_1 \odot f_2 = (-1)^{|f_1||f_2|} f_2 \odot f_1$ if the
$f_i$ are homogeneous of degree $|f_i|$. Now we consider a series of
graded module homomorphisms 
$$
    \ff_{k,\ell,g} : E_k C^+ \to E_{\ell} C^-,\qquad k,\ell \ge 1,\
    g\geq 0
$$ 
of degree 
$$
    |\ff_{k,\ell,g}| = -2d(k+g-1).
$$ 
Define the operator
$$
    \ff := \sum_{k,\ell=1}^\infty\sum_{g=0}^\infty
    \ff_{k,\ell,g} \hbar^{k+g-1}\tau^{k+\ell+2g-2}:
     EC^+\{\hbar,\tau\}\to EC^-\{\hbar,\tau\},
$$
where each $\ff_{k,\ell,g}$ is viewed as a map $EC^+\to E_\ell
C^-\subset EC^-$ by setting it zero on $E_mC^+$ for $m\neq k$.
Furthermore, we introduce the exponential series $e^\ff$ with respect
to the symmetric product, i.e. 
\begin{equation}\label{expff}
\aligned
    e^\ff := \sum_{r=0}^\infty\sum_{{k_i,\ell_i,g_i}\atop {1\leq i\leq
    r}} \frac{1}{r!}&f_{k_1,\ell_1,g_1}\odot\cdots\odot\\
    &f_{k_r,\ell_r,g_r} \hbar^{\sum k_i+\sum g_i-r} \tau^{\sum k_i+\sum
    \ell_i + 2\sum g_i-2r}.
\endaligned
\end{equation}
\begin{defn}
We say that $\{\ff_{k,\ell,g}\}_{k,\ell \geq 1,g\geq 0}$ is an
{\em $\IBL_\infty$-morphism} if 
\begin{equation}\label{eq:mor2}
    e^\ff\hat\fp^+ - \hat\fp^- e^\ff=0.
\end{equation}
\end{defn}

Again, let us explain this definition from various angles.

(1) Equation~\eqref{eq:mor2} is equivalent to requiring that for each
triple $(k,\ell,g)$ with $k,\ell \geq 1$ and $g \geq 1 - \min(k,\ell)$
the equation 
\begin{align}\label{eq:mor}
    &\sum_{r=1}^\ell
    \sum_{ { {\ell_1+\dots+\ell_{r}=\ell}
    \atop {k_1+\dots+k_{r}+k^+=k+\ell^+} }
    \atop {g_1+\dots+g_{r}+g^+=g+r-\ell^+} } \hspace{-.7cm}
    \frac{1}{r!}(\ff_{k_1,\ell_1,g_1}\odot\cdots\odot
    \ff_{k_r,\ell_r,g_r})\circ\hat\fp^+_{k^+,\ell^+,g^+} \cr
    &- \sum_{r=1}^k
    \sum_{ { {k_1+\dots+k_r=k}
    \atop {\ell_1+\dots+\ell_r+\ell^-=\ell+k^-} }
    \atop {g_1+\dots+g_r+g^-=g+r-k^-}  } \hspace{-.7cm}
    \frac{1}{r!}\hat\fp^-_{k^-,\ell^-,g^-}\circ
    (\ff_{k_1,\ell_1,g_1}\odot\cdots\odot \ff_{k_r,\ell_r,g_r}) =0. 
\end{align}
holds as equality between maps $E_kC^+\to E_\ell C^-$.
Indeed, the left hand side of equation~\eqref{eq:mor} is the
corresponding part of the coefficient of $\hbar^{k+g-1}\tau^{k+\ell+2g-2}$ in
$e^\ff\hat\fp^+ - \hat\fp^- e^\ff$.

(2) As before, one thinks of $\ff_{k,\ell,g}$ as an operation
associated to a compact connected oriented surface of signature
$(k,\ell,g)$. Then the terms $(\ff_{k_1,\ell_1,g_1}\odot\cdots\odot 
\ff_{k_r,\ell_r,g_r})\hat\fp^+_{k^+,\ell^+,g^+}$ on the left hand side
of~\eqref{eq:mor} correspond to complete gluings of $r$ connected
surfaces of signatures $(k_i,\ell_i,g_i)$ at their incoming loops to
the outgoing loops of a surface of signature $(k^+,\ell^+,g^+)$, plus
an appropriate number of trivial cylinders, to a {\em possibly disconnected}
surface of signature $(k,\ell,g)$.

(3) Again, it is useful to reformulate~\eqref{eq:mor2} in terms of
gluing to {\em connected} surfaces. For this, let us denote by 
$$
    (\ff_{k_1,\ell_1,g_1}\odot\cdots\odot
    \ff_{k_r,\ell_r,g_r})\circ_{s_1,\dots,s_r}\hat\fp^+_{k^+,\ell^+,g^+}
$$
the part of the composition where precisely $s_i$ of the inputs of
$\ff_{k_i,\ell_i,g_i}$ are outputs of $\fp_{k^+,\ell^+,g^+}$, and
similarly for composition with $\hat\fp_{k^-,\ell^-,g^-}$.

\begin{lem}\label{lem:mor_order}
Equation~\eqref{eq:mor2} is equivalent to the sequence of equations
\begin{align}\label{eq:mor4}
    &\sum_{r=1}^\ell
    \sum_{ {{ {\ell_1+\dots+\ell_{r}=\ell}
    \atop {k_1+\dots+k_{r}+k^+=k+\ell^+} }
    \atop {g_1+\dots+g_{r}+g^+=g+r-\ell^+}}
    \atop {s_1+\dots +s_r=\ell^+ \atop s_i \geq 1}} \hspace{-.7cm}
    \frac{1}{r!}(\ff_{k_1,\ell_1,g_1}\odot\cdots\odot
    \ff_{k_r,\ell_r,g_r})\circ_{s_1,\dots,s_r} \hat\fp^+_{k^+,\ell^+,g^+} \cr
    &- \sum_{r=1}^k
    \sum_{ {{ {k_1+\dots+k_r=k}
    \atop {\ell_1+\dots+\ell_r+\ell^-=\ell+k^-} }
    \atop {g_1+\dots+g_r+g^-=g+r-k^-} }
    \atop {s_1+\dots +s_r=k^- \atop s_i \geq 1}} \hspace{-.7cm}
    \frac{1}{r!}\hat\fp^-_{k^-,\ell^-,g^-}\circ_{s_1,\dots,s_r}
    (\ff_{k_1,\ell_1,g_1}\odot\cdots\odot \ff_{k_r,\ell_r,g_r}) =0, 
\end{align}
for $k,\ell \geq 1$ and $g \geq 0$, where the left hand side is
viewed as a map from $E_kC^+$ to $E_\ell
C^-$. Equation~\eqref{eq:mor4} for a fixed triple $(k,\ell,g)$ has the
form 
$$
    \ff_{k,\ell,g}\circ
    \hat\fp^+_{1,1,0}-\hat\fp^-_{1,1,0}\circ\ff_{k,\ell,g} +
    R_{k,\ell,g}(\ff,\fp^+,\fp^-) = 0,
$$
where the expression $R_{k,\ell,g}(\ff,\fp^+,\fp^-)$ contains only
components $\ff_{k',\ell',g'}$ of $\ff$ with $(k',\ell',g')\prec
(k,\ell,g)$ and only components $\hat\fp^\pm_{k',\ell',g'}$ of
$\hat\fp^\pm$ with $(1,1,0) \prec (k',\ell',g') \preceq (k,\ell,g)$.
Moreover, we have
$$
R_{k,\ell,g}(\ff,\fp^+,\fp^-)= \frac 1 {\ell!} \ff_{1,1,0}^{\odot
   \ell} \circ \hat\fp^+_{k,\ell,g} - \frac 1 {k!} \hat\fp^-_{k,\ell,g}
\circ \ff_{1,1,0}^{\odot k} + \wt R_{k,\ell,g}(\ff,\fp^+,\fp^-),
$$
where the expression $\wt R_{k,\ell,g}(\ff,\fp^+,\fp^-)$ contains only
components $\hat\fp^\pm_{k',\ell',g'}$ with $(k',\ell',g')\prec (k,\ell,g)$.
\end{lem}

Before giving the proof, let us introduce the following notation. For
a map $F:EC^+\{\tau,\hbar\} \to EC^-\{\tau,\hbar\}$, we will denote by
$\langle F\rangle_{k,\ell,g}$  
the part of the coefficient of $\hbar^{k+g-1}\tau^{k+\ell+2g-2}$ which 
corresponds to a map from $E_kC^+$ to $E_\ell C^-$. Then we have for example
\begin{equation}\label{eq:ef}
    \la e^\ff\ra_{k,\ell,g} =
    \sum_{r=1}^{\min(k.\ell)} \sum_{ { {k_1+\dots+k_r=k}
    \atop {\ell_1+\dots+\ell_r=\ell} }
    \atop {g_1+\dots+g_r=g+r-1}  } \hspace{-.7cm}
    \frac{1}{r!}
    \ff_{k_1,\ell_1,g_1}\odot\cdots\odot \ff_{k_r,\ell_r,g_r}.
\end{equation}
Note that this can be non-zero for $g  \geq 1 - \min(k,\ell)$, so in
general negative $g$ are allowed here.
\begin{proof}
We rewrite the term of
the first sum in~\eqref{eq:mor} for fixed $r\geq 1$ as 
\begin{eqnarray*}
\lefteqn{\frac 1 {r!} \sum_{ {{ {\ell_1+\dots+\ell_{r}=\ell}
    \atop {k_1+\dots+k_{r}+k^+=k+\ell^+} }
    \atop {g_1+\dots+g_{r}+g^+=g+r-\ell^+}}\atop 
{s_1+\dots + s_r=\ell^+ \atop s_i \geq 0}}  \hspace{-.7cm} 
(\ff_{k_1,\ell_1,g_1}\odot\cdots\odot
    \ff_{k_r,\ell_r,g_r})\circ_{s_1,\dots,s_r}\hat\fp^+_{k^+,\ell^+,g^+}}\\
&=& \sum_{r'=1}^{r}\frac 1 {r'!(r-r')!}\hspace{-.7cm} \sum_{ {{
    {\ell_1+\dots+\ell_{r'}=\ell'}
    \atop {k_1+\dots+k_{r'}+k^+=k'+\ell^+} }
    \atop {g_1+\dots+g_{r'}+g^+=g'+r'-\ell^+}}\atop {s_1+\dots
    +s_{r'}=\ell^+ \atop s_i \geq 1}} \hspace{-.7cm}
    (\ff_{k_1,\ell_1,g_1}\odot\cdots\odot
    \ff_{k_{r'},\ell_{r'},g_{r'}})\circ_{s_1,\dots,s_{r'}}
    \hat\fp^+_{k^+,\ell^+,g^+} \\
& & \odot\sum_{ { {\ell_{r'+1}+\dots+\ell_r=\ell-\ell'}
    \atop {k_{r'+1}+\dots+k_r=k-k'} } \atop
    {g_{r'+1}+\dots+g_r=g-g'+r-r'}} \ff_{k_{r'+1},\ell_{r'+1},g_{r'+1}}
    \odot \cdots \odot \ff_{k_r,\ell_r,g_r}. 
\end{eqnarray*} 
Here 
we have used the commutation properties of the product $\odot$ as well
as the identity 
\begin{align*}
    &(\ff_{k_1,\ell_1,g_1}\odot\cdots\odot
    \ff_{k_r,\ell_r,g_r})\circ_{s_1,\dots,s_{r-1},0}
    \hat\fp^+_{k^+,\ell^+,g^+} \cr
    &= \Bigl[(\ff_{k_1,\ell_1,g_1}\odot\cdots\odot
    \ff_{k_{r-1},\ell_{r-1},g_{r-1}})\circ_{s_1,\dots,s_{r-1}}
    \hat\fp^+_{k^+,\ell^+,g^+} \Bigr]\odot \ff_{k_r,\ell_r,g_r}. 
\end{align*}
A similar discussion applies to the second sum. 
Note that the left-hand side of
this equation corresponds to all
gluings to possibly disconnected surfaces of signature $(k,\ell,g)$,
while the terms $(\ff_{k_1,\ell_1,g_1}\odot\cdots\odot
\ff_{k_{r'},\ell_{r'},g_{r'}})\circ_{s_1,\dots,s_{r'}} 
\hat\fp^+_{k^+,\ell^+,g^+}$ on the right-hand side correspond to
connected surfaces.

Abbreviate the left-hand side of~\eqref{eq:mor} by $G_{k,\ell,g}$ and
the left-hand side of~\eqref{eq:mor4} by $H_{k,\ell,g}$. Then we have
the relation
$$
    G_{k,\ell,g} = H_{k,\ell,g} + \sum_{ { {k'+k''=k}\atop
    {\ell'+\ell''=\ell} } \atop {g'+g''=g+1}}H_{k',\ell',g'}\odot
    \langle e^\ff \rangle_{k'',\ell'',g''}.
$$
So clearly the sequence of equations~\eqref{eq:mor4} implies
$G_{k,\ell,g}=0$. The converse implication follows by induction over
the order $\prec$ because all terms in the sum on the right hand side
of the above equation involve $(k',\ell',g')\prec (k,\ell,g)$. 
The last statement of the lemma follows as in the proof of
Lemma~\ref{lem:order}.  
\end{proof}

(4) Let us spell out equation~\eqref{eq:mor} for the first few triples
$(k,\ell,g)$. For $(k,\ell,g)=(1,1,0)$ we find that 
$$
    \ff_{1,1,0}:(C^+,\fp^+_{1,1,0})\to (C^-,\fp^-_{1,1,0})
$$ 
is a chain map. The relations for
$(k,\ell,g)=(2,1,0)$ resp.~$(1,2,0)$ show that $\ff_{1,1,0}$ intertwines the
products $\mu^\pm$ resp.~the coproducts $\delta^\pm$ up to
homotopy. In particular, $\ff_{1,1,0}$ induces a morphism of involutive Lie
bialgebras on homology.

(5) In the special cases that $\fp_{k,\ell,g}=0$ and
$\ff_{k,\ell,g}=0$ whenever $\ell\geq 2$ (resp.~$k\geq 2$) or $g>0$ we
recover the definitions of $L_\infty$ (resp.~co-$L_\infty$) morphisms. 

(6) As before with $\hat \fp$, by setting $\tau=1$ the operators $\ff$ and
$e^\ff$ define maps $EC^+\{\hbar\}\to EC^-\{\hbar\}$ if the
following finiteness condition is satisfied: 
{\em Given $k\geq 1$, $g\geq 0$ and $a\in E_kC^+$, the term
   $\ff_{k,\ell,g}(a)$ is nonzero for only finitely many $\ell\geq 1$.}
Again, this condition holds in our main examples (exact symplectic
field theory and string topology).

(7) We say that an $\IBL_\infty$-morphism $\{\ff_{k,\ell,g}\}$ is {\em
linear} if $\ff_{k,\ell,g}=0$ unless $(k,\ell,g)=(1,1,0)$. In this
case its exponential is given by
$$
    e^\ff=\sum_{r=0}^\infty\frac{1}{r!}\ff_{1,1,0}^{\odot r}:c_1\cdots
    c_r\mapsto\ff_{1,1,0}(c_1)\cdots\ff_{1,1,0}(c_r)
$$
and equation~\eqref{eq:mor4} simplifies to
$$
    \frac{1}{\ell!}\ff_{1,1,0}^{\odot\ell}\circ \fp^+_{k,\ell,g} =
    \fp^-_{k,\ell,g}\circ\frac{1}{k!}\ff_{1,1,0}^{\odot k}. 
$$

{\bf Composition of morphisms. }
Consider now two $\IBL_\infty$-morphisms 
\begin{align*}
    \ff^+=\{\ff^+_{k,\ell,g}\}:(C^+,\{\fp^+_{k,\ell,g}\})\to
    (C,\{\fp_{k,\ell,g}\}),\\
    \ff^-=\{\ff^-_{k,\ell,g}\}:(C,\{\fp_{k,\ell,g}\})\to
    (C^-,\{\fp^-_{k,\ell,g}\}). 
\end{align*}
\begin{defn}\label{def:comp}
The {\em composition}
$\ff = \ff^-\diamond \ff^+$
of $\ff^+$ and $\ff^-$ is the unique
$\IBL_\infty$-morphism
$\ff=\{\ff_{k,\ell,g}\}:(C^+,\{\fp^+_{k,\ell,g}\})\to
(C^-,\{\fp^-_{k,\ell,g}\})$ satisfying 
\begin{equation}\label{eq:comp2}
e^\ff = e^{\ff^-}e^{\ff^+}.
\end{equation}
\end{defn}
To see existence and uniqueness of $\ff$, consider the signature
$(k,\ell,g)$ part of $e^{\ff^-}e^{\ff^+}$,
\begin{align*}
&  \la e^{\ff^-}e^{\ff^+}\ra_{k,\ell,g} = \sum_{r^+=1}^k \sum_{r^-=1}^\ell
   \sum_{s=\max(r^+,r^-)}^{r^++r^-+g-1} \cr
&
   \hspace{-1.7cm}
    \sum_{ { { {k_1^++\dots+k_{r^+}^+=k}
    \atop {\ell_1^-+\dots+\ell_{r^-}^-=\ell} }
    \atop {\ell_1^++\dots+\ell_{r^+}^+=k_1^-+\dots+k_{r^-}^-=s} }
    \atop {g_1^++\dots+g_{r^+}^++g_1^-+\dots+g_{r^-}^-=r^++r^-+g-1-s} }
    \hspace{-2.4cm}
    \frac{1}{r^+!r^-!}(\ff^-_{k_1^-,\ell_1^-,g_1^-}\odot\cdots\odot
    \ff^-_{k_{r^-}^-,\ell_{r^-}^-,g_{r^-}^-})\;\circ
    (\ff^+_{k_1^+,\ell_1^+,g_1^+}\odot\cdots\odot
    \ff^+_{k_{r^+}^+,\ell_{r^+}^+,g_{r^+}^+}). 
\end{align*}
Note also that equation~\eqref{eq:ef} has the form
$$
\la e^\ff\ra_{k,\ell,g} = \ff_{k,\ell,g} +
    \sum_{r=2}^{\min\{k,\ell\}}
    \sum_{ { {k_1+\dots+k_r=k}
    \atop {\ell_1+\dots+\ell_r=\ell} }
    \atop {g_1+\dots+g_r=g+r-1} }
    \frac{1}{r!}(\ff_{k_1,\ell_1,g_1}\odot\cdots\odot
    \ff_{k_r,\ell_r,g_r}), 
$$ 
where the conditions in the last summand add up to
$$
    \sum_{i=1}^r(k_i+\ell_i+2g_i-2)=k+\ell+2g-2.
$$
As $r\geq 2$, this easily implies $(k_i,\ell_i,g_i)\prec (k,\ell,g)$
for all $i$, so we can inductively solve the equation $\langle e^\ff
\rangle_{k,\ell,g}= \la e^{\ff^-}e^{\ff^+}\ra_{k,\ell,g}$ for
$\ff_{k,\ell,g}$ to find 
\begin{equation}\label{eq:comp}
  \ff_{k,\ell,g} = \langle e^{\ff^-} e^{\ff^+} \rangle_{k,\ell,g}
   -\sum_{r=2}^{\min\{k,\ell\}}
    \sum_{ { {k_1+\dots+k_r=k}
    \atop {\ell_1+\dots+\ell_r=\ell} }
    \atop {g_1+\dots+g_r=g+r-1} }
    \frac{1}{r!}(\ff_{k_1,\ell_1,g_1}\odot\cdots\odot
    \ff_{k_r,\ell_r,g_r}) .
\end{equation}

Here are some explanations to this definition.

(1) Recall that we think of $\ff_{k,\ell,g}$ as associated to
connected Riemann surfaces with signature $(k,\ell,g)$. The first term
on the right hand side of equation~\eqref{eq:comp} describes all
possible ways to obtain  a (possibly disconnected) Riemann surface of
Euler characteristic $2-2g-k-\ell$ as a complete gluing of pieces
corresponding to the $(k_i^\pm,\ell_i^\pm,g_i^\pm)$. The second term
then subtracts all disconnected configurations.

(2) For $(k,\ell,g)=(1,1,0)$ equation~\eqref{eq:comp} shows that
$\ff_{1,1,0}$ and the $\ff_{1,1,0}^\pm$ are related
by 
$$
    \ff_{1,1,0} = \ff_{1,1,0}^-\circ\ff_{1,1,0}^+. 
$$

(3) In the special cases that $\fp_{k,\ell,g}=0$ and $\ff_{k,\ell,g}=0$
whenever $\ell\geq 2$ (resp.~$k\geq 2$) or $g>0$ we recover the
definitions of composition of $L_\infty$ (resp.~co-$L_\infty$) morphisms.

(4) If $\ff^+$ is linear, then \eqref{eq:comp} simplifies to 
$$
\ff_{k,\ell,g} = \frac 1 {k!} \ff^-_{k,\ell,g} \circ
(\ff^+_{1,1,0})^{\odot k}.
$$
Indeed, if we define $\ff_{k,\ell,g}$ by this equation we find
$$
    \frac{1}{k!}(\ff^-_{k_1,\ell_1,g_1}\odot\cdots\odot
    \ff^-_{k_{r},\ell_{r},g_{r}})
    \circ (\ff^+_{1,1,0})^{\odot k}
    = \ff_{k_1,\ell_1,g_1}\odot\cdots\odot
    \ff_{k_{r},\ell_{r},g_{r}}
$$
and hence
\begin{align*}
    \la e^{\ff^-}e^{\ff^+}\ra_{k,\ell,g}
    &= \sum_{r=1}^\ell
    \sum_{ { k_1+\dots+k_{r}=k
    \atop \ell_1+\dots+\ell_{r}=\ell }
     \atop g_1+\dots+g_{r}=r+g-1 }
    \frac{1}{r!k!}(\ff^-_{k_1,\ell_1,g_1}\odot\cdots\odot
    \ff^-_{k_{r},\ell_{r},g_{r}})\circ 
    (\ff^+_{1,1,0})^{\odot k} \cr
    &= \la e^\ff\ra_{k,\ell,g}. 
\end{align*}
A similar discussion applies if $\ff^-$ is linear. In particular, if
both $\ff^\pm$ are linear then so is their composition.

Finally, we again record a useful observation for later use.
\begin{lem}\label{lem:mor-comp}
For $(1,1,0)\prec (k,\ell,g)$, the component $\ff_{k,\ell,g}$ of the
composition $\ff$ of two $\IBL_\infty$-morphisms $\ff^+$ and $\ff^-$
has the form 
$$
\ff_{k,\ell,g} = \ff^-_{k,\ell,g} \circ \frac 1 {k!}
(\ff^+_{1,1,0})^{\odot k} + \frac 1 {\ell!} (\ff^-_{1,1,0})^{\odot \ell}
\circ \ff^+_{k,\ell,g} + C_{k,\ell,g}(\ff^-,\ff^+),
$$
where $C_{k,\ell,g}(\ff^-,\ff^+)$ contains only components
$\ff^\pm_{k',\ell',g'}$ with $(k',\ell',g')\prec (k,\ell,g)$.
Moreover, if either $\ff^-$ or $\ff^+$ is linear, then
$C_{k,\ell,g}(\ff^-,\ff^+)=0$. 
\hfill $\qed$
\end{lem}
\begin{proof}
The conditions in the first sum in~\eqref{eq:comp} add up to
$$
    \sum_{i=1}^{r^+}(k_i^++\ell_i^++2g_i^+-2)
    + \sum_{i=1}^{r^-}(k_i^-+\ell_i^-+2g_i^--2)
    = k+\ell+2g-2,
$$
from which the first statement easily follows. The last statement
follows from (4) above. 
\end{proof}

\begin{rem}
We leave it to the reader to check that composition of morphisms is
associative. 
\end{rem}

\section{Obstructions}\label{sec:tech}

In this section we prove the main technical proposition underlying the
homotopy theory of $\IBL_\infty$-algebras
(Proposition~\ref{prop:hom-main}). All the results in the following
three sections will be formal consequences of this proposition.

Given free chain complexes $(C,d^C)$ and $(D,d^D)$ over $R$, define
a boundary operator $\delta$ on $Hom(E_kC, E_\ell D)$ by 
$$
\delta \varphi = \hat d^D \,\varphi + (-1)^{|\varphi|+1}\varphi\,
\hat d^C,
$$
where $\hat d^C$ and $\hat d^D$ are the usual extensions of $d^C$ and $d^D$.
This operation is a derivation of composition, in the sense that
for $\varphi \in Hom(E_n B, E_k C)$ and $\psi \in Hom(E_k C,E_\ell D)$
we have 
$$
\delta(\psi \circ \varphi) = (\delta \psi)\circ \varphi +
(-1)^{|\psi|}\psi \circ (\delta \varphi).
$$
Below, we always consider the case that $B$, $C$ and $D$ are (partial)
$\IBL_\infty$-algebras and the boundary operators used in the
definition of $\delta$ are the corresponding structure maps $\fp_{1,1,0}$.
The following proposition identifies the obstructions to inductive
extensions of partially defined $\IBL_\infty$-structures and their morphisms.
\begin{prop}\label{prop:hom-main}
\begin{enumerate}[\rm (1)]
\item Let $\{\fp_{k,\ell,g}:E_kC \to E_{\ell}C\}_{
(k,\ell,g)\prec (K,L,G)\}}$ be a collection of maps that
satisfy the defining relation \eqref{eq:BL4} for
$\IBL_\infty$-structures for all   $(k,\ell,g)\prec
(K,L,G)$. Then the expression $P_{K,L,G} \in Hom(E_KC,E_LC)$
implicitly defined in Lemma~$\ref{lem:order}$ satisfies 
$$
\delta P_{K,L,G}=0.
$$
\item Let $(C,\{\fp^C_{k,\ell,g}\})$ and $(D, \{\fp^D_{k,\ell,g}\})$
be $\IBL_\infty$-algebras, and suppose the collection of maps
$\{\ff_{k,\ell,g}:E_kC \to E_{\ell}D\}_{(k,\ell,g)\prec
(K,L,G)\}}$ satisfies the defining relation \eqref{eq:mor} for
$\IBL_\infty$-morphisms for all   $(k,\ell,g)\prec
(K,L,G)$. Then the expression $R_{K,L,G}(\ff,\fp^C,\fp^D) \in
Hom(E_KC,E_LD)$ implicitly defined in Lem\-ma~$\ref{lem:mor_order}$
satisfies 
$$
\delta R_{K,L,G}(\ff,\fp^C,\fp^D)=0.
$$
\item Assume further that $(B,\{\fp^B_{k,\ell,g}\})$ is another
   $\IBL_\infty$-algebra and the collection $\fg=\{\fg_{k,\ell,g}:E_kB \to
   E_{\ell}C\}_{(k,\ell,g)\prec (K,L,G)}$ also satisfies the
   defining relation \eqref{eq:mor} for morphisms for all
   $(k,\ell,g)\prec (K,L,G)$. Then 
\begin{align*}
R_{K,L,G}(\ff \circ \fg,\fp^B,\fp^D) =& 
\frac 1 {L!} \ff_{1,1,0}^{\odot L}\circ R_{K,L,G}(\fg,\fp^B,\fp^C) 
+ R_{K,L,G}(\ff,\fp^C,\fp^D) \circ\frac 1 {L!} \fg_{1,1,0}^{\odot K} \\
&+ \delta C_{K,L,G}(\ff,\fg),
\end{align*}
where $C_{K,L,G}(\ff,\fg)\in Hom(E_K B,E_L D)$ is the expression
implicitly defined in Lemma~$\ref{lem:mor-comp}$. 
\end{enumerate}
\end{prop}
For the proof we will need some more notation. For three linear maps
$p_i:E_{k_i}C\to E_{\ell_i}C$ and integers $s_{12},s_{13},s_{23}\geq
0$ we denote by 
$$
    \hat p_3\circ_{s_{23},s_{13}}(\hat p_2\circ_{s_{12}}\hat p_1) =
    (\hat p_3\circ_{s_{23}}\hat p_2)\circ_{s_{13},s_{12}}\hat p_1
$$
the part of the composition $\hat p_3\circ\hat p_2\circ\hat p_1$ where
exactly $s_{ij}$ of the inputs of $p_j$ are outputs of $p_i$. The
following properties follow immediately from the definition:
\begin{align*}
    \hat p_3\circ_s(\hat p_2\circ_{s_{12}}\hat p_1)
    &= \sum_{s_{13}+s_{23}=s}
    \hat p_3\circ_{s_{23},s_{13}}(\hat p_2\circ_{s_{12}}\hat p_1), \cr
    (\hat p_3\circ_{s_{23}}\hat p_2)\circ_{s}\hat p_1
    &= \sum_{s_{12}+s_{13}=s}
    (\hat p_3\circ_{s_{23}}\hat p_2)\circ_{s_{13},s_{12}}\hat p_1.
\end{align*}
Note in particular that, since $p_i \circ_0 p_j= (-1)^{|p_i||p_j|}p_j
\circ_0 p_i$, we have
\begin{align*}
    \hat p_3\circ_{0,s_{13}}(\hat p_2\circ_{s_{12}}\hat p_1)
    &= (\hat p_3\circ_0\hat p_2) \circ_{s_{13},s_{12}}\hat p_1) \\
    &= (-1)^{|p_2|\,|p_3|}
    (\hat p_2\circ_0\hat p_3) \circ_{s_{12},s_{13}}\hat p_1) \\
    &= (-1)^{|p_2|\,|p_3|}
    \hat p_2\circ_{0,s_{12}}(\hat p_3\circ_{s_{13}}\hat p_1),
\end{align*}
and similarly
$$
(\hat p_3\circ_{s_{23}}\hat p_2)\circ_{s_{13},0}\hat p_1
    = (-1)^{|p_1|\,|p_2|}
    (\hat p_3\circ_{s_{13}}\hat p_1)\circ_{s_{23},0}\hat p_2. 
$$
Applying these properties in the case where one of the $p_i$ is the
boundary operator $\fp_{1,1,0}$, we obtain 
$$
    \delta(\hat p_2\circ_s\hat p_1) = (\delta\hat p_2)\circ_s\hat p_1 +
    (-1)^{|p_2|}\hat p_2\circ_s (\delta\hat p_1). 
$$

\begin{proof}[Proof of Proposition~\ref{prop:hom-main}]
For the proof of (1), recall from Lemma~\ref{lem:order} that
$P_{k,\ell,g}$ is defined by
$$
    P_{k,\ell,g} := \sum_{s=1}^{g+1}
    \sum_{ { k_1+k_2=k+s \atop {\ell_1+\ell_2=\ell+s \atop
    g_1+g_2=g+1-s} } \atop (k_i,\ell_i,g_i)\neq (1,1,0)}
    (\hat \fp_{k_2,\ell_2,g_2}\circ_s \hat
    \fp_{k_1,\ell_1,g_1})|_{E_kC}
$$
Moreover, in our current notation equation~\eqref{eq:BL4} can be
written as 
$$
\delta \fp_{k,\ell,g} = -  P_{k,\ell,g}.
$$
Since all terms in the definition of $P_{K,L,G}$ satisfy
$(k_i,\ell_i,g_i)\prec (K,L,G)$, using the hypothesis 
we find 
\begin{align*}
    \delta P_{K,L,G}
    &= \sum_{s=1}^{G+1}
    \sum_{ { k_1+k_2=K+s \atop {\ell_1+\ell_2=L+s \atop
    g_1+g_2=G+1-s} } \atop (k_i,\ell_i,g_i)\neq (1,1,0)}
    (\delta\hat\fp_{k_2,\ell_2,g_2}\circ_s \hat\fp_{k_1,\ell_1,g_1} -
    \hat\fp_{k_2,\ell_2,g_2}\circ_s \delta\hat\fp_{k_1,\ell_1,g_1})|_{E_kC} \cr
    &= -\sum_{s=1}^{G+1}
    \sum_{ { k_1+k_2=K+s \atop {\ell_1+\ell_2=L+s \atop
    g_1+g_2=G+1-s} } \atop (k_i,\ell_i,g_i)\neq (1,1,0)}
    (\hat P_{k_2,\ell_2,g_2}\circ_s \hat\fp_{k_1,\ell_1,g_1})|_{E_kC} \cr
    &\ \ \ + \sum_{s=1}^{G+1}
    \sum_{ { k_1+k_2=K+s \atop {\ell_1+\ell_2=L+s \atop
    g_1+g_2=G+1-s} } \atop (k_i,\ell_i,g_i)\neq (1,1,0)}
    (\hat\fp_{k_2,\ell_2,g_2}\circ_s \hat P_{k_1,\ell_1,g_1})|_{E_kC} \cr
    &=: -A + B 
\end{align*}
Geometrically, the terms $A$ and $B$ both correspond to breaking up a
connected Riemann surface of signature $(k,\ell,g)$ in all possible
ways into three non-trivial pieces. To show algebraically that indeed $A=B$, 
in the sum $B$ we rename $(k_2,\ell_2,g_2)$ to $(k_3,\ell_3,g_3)$,
$(k_1,\ell_1,g_1)$ to $(k,\ell,g)$, and we insert the definition of
$P_{k,\ell,g}$ to rewrite it as
\begin{align*}
    B &= \sum_{s=1}^{G+1}
    \sum_{ { k+k_3=K+s \atop {\ell+\ell_3=L+s \atop
    g+g_3=G+1-s} } \atop (k_3,\ell_3,g_3), (k,\ell,g)\neq (1,1,0)}
    \sum_{t=1}^{g+1}
    \sum_{ { k_1+k_2=k+t \atop {\ell_1+\ell_2=\ell+t \atop
    g_1+g_2=g+1-t} } \atop (k_i,\ell_i,g_i)\neq (1,1,0)}
    \hat\fp_{k_3,\ell_3,g_3}\circ_s (\hat\fp_{k_2,\ell_2,g_2}\circ_t
    \hat\fp_{k_1,\ell_1,g_1})|_{E_kC} \cr
    &= \sum_{s,t\geq 1 \atop s+t\leq G+2}
    \sum_{ { k_1+k_2+k_3=K+s+t \atop {\ell_1+\ell_2+\ell_2=L+s+t \atop
    g_1+g_2+g_3=G+2-s-t} } \atop (k_i,\ell_i,g_i)\neq (1,1,0)}
    \hat\fp_{k_3,\ell_3,g_3}\circ_s (\hat\fp_{k_2,\ell_2,g_2}\circ_t
    \hat\fp_{k_1,\ell_1,g_1})|_{E_kC}. 
\end{align*}
We rewrite the term in $B$ for fixed $s,t$ as 
$$
    \sum_{s_{13}+s_{23}=s}
    \sum_{ { k_1+k_2+k_3=K+s+t \atop {\ell_1+\ell_2+\ell_2=L+s+t \atop
    g_1+g_2+g_3=G+2-s-t} } \atop (k_i,\ell_i,g_i)\neq (1,1,0)}
    \hat\fp_{k_3,\ell_3,g_3}\circ_{s_{23},s_{13}}
    (\hat\fp_{k_2,\ell_2,g_2}\circ_t \hat\fp_{k_1,\ell_1,g_1}). 
$$
By the above properties of triple compositions the terms with
$s_{23}=0$ cancel in pairs, and renaming $t=s_{12}$ we obtain
$$
    B = \sum_{s_{12},s_{23}\geq 1,s_{13}\geq 0\atop
    u:=s_{12}+s_{13}+s_{23}\leq G+2}
    \sum_{ { k_1+k_2+k_3=K+u \atop {\ell_1+\ell_2+\ell_2=L+u \atop
    g_1+g_2+g_3=G+2-u} } \atop (k_i,\ell_i,g_i)\neq (1,1,0)}
    \hat\fp_{k_3,\ell_3,g_3}\circ_{s_{23},s_{13}}
    (\hat\fp_{k_2,\ell_2,g_2}\circ_{s_{12}}\hat\fp_{k_1,\ell_1,g_1}). 
$$
By similar discussion and the above associativity of triple
decompositions we see that this expression agrees with $A$, which
proves $\delta P_{K,L,G}=0$.

To prove (2), recall that $R_{K,L,G}=R_{K,L,G}(\ff,\fp^C,\fp^D)$ has
the general form 
\begin{eqnarray*}
R_{K,L,G}&=&
     \sum \frac{1}{r!}(\ff_{k_1,\ell_1,g_1}\odot\cdots\odot
    \ff_{k_r,\ell_r,g_r})\circ_{s_1,\dots,s_r} \hat\fp^C_{k^+,\ell^+,g^+} \\
& & - \sum \frac{1}{r!}\hat\fp^D_{k^-,\ell^-,g^-}\circ_{s_1,\dots,s_r}
    (\ff_{k_1,\ell_1,g_1}\odot\cdots\odot \ff_{k_r,\ell_r,g_r}),
\end{eqnarray*}
where all the terms satisfy $s_i\geq 1$, $(k^\pm,\ell^\pm,g^\pm)\neq
(1,1,0)$ and $(k_i,\ell_i,g_i)\prec (K,L,G)$. 
Hence, using the properties of triple compositions, the hypothesis
$\delta\fp_{k,\ell,g}+P_{k,\ell,g}=0$ for all $(k,\ell,g)$, and the
induction hypothesis $R_{k,\ell,g}=\delta\ff_{k,\ell,g}$ for
$(k,\ell,g)\prec (K,L,G)$, we find that
\begin{align*}
\lefteqn{\delta R_{K,L,G} } \cr
&= \sum \frac{1}{(r-1)!}(\delta\ff_{k_1,\ell_1,g_1}\odot\ff_{k_2,\ell_2,g_2}
\odot \cdots\odot \ff_{k_r,\ell_r,g_r})\circ_{s_1,\dots,s_r}
\hat\fp^C_{k^+,\ell^+,g^+} \cr
&+ \sum \frac{1}{r!}(\ff_{k_1,\ell_1,g_1}\odot\cdots\odot
    \ff_{k_r,\ell_r,g_r})\circ_{s_1,\dots,s_r}
    \delta\hat\fp^C_{k^+,\ell^+,g^+} \cr 
&- \sum \frac{1}{r!}\delta\hat\fp^D_{k^-,\ell^-,g^-}\circ_{s_1,\dots,s_r}
    (\ff_{k_1,\ell_1,g_1}\odot\cdots\odot \ff_{k_r,\ell_r,g_r}) \cr 
&+ \sum \frac{1}{(r-1)!}\hat\fp^D_{k^-,\ell^-,g^-}\circ_{s_1,\dots,s_r}
    (\delta\ff_{k_1,\ell_1,g_1}\odot\ff_{k_2,\ell_2,g_2} \odot
    \cdots\odot \ff_{k_r,\ell_r,g_r}) \cr 
&= \sum \frac{1}{(r-1)!}(R_{k_1,\ell_1,g_1}\odot\ff_{k_2,\ell_2,g_2}
\odot \cdots\odot \ff_{k_r,\ell_r,g_r})\circ_{s_1,\dots,s_r}
\hat\fp^C_{k^+,\ell^+,g^+} \cr
&- \sum \frac{1}{r!}(\ff_{k_1,\ell_1,g_1}\odot\cdots\odot
    \ff_{k_r,\ell_r,g_r})\circ_{s_1,\dots,s_r} P^C_{k^+,\ell^+,g^+} \cr 
&+ \sum \frac{1}{r!}P^D_{k^-,\ell^-,g^-}\circ_{s_1,\dots,s_r}
    (\ff_{k_1,\ell_1,g_1}\odot\cdots\odot \ff_{k_r,\ell_r,g_r}) \cr 
&+ \sum \frac{1}{(r-1)!}\hat\fp^D_{k^-,\ell^-,g^-}\circ_{s_1,\dots,s_r}
    (R_{k_1,\ell_1,g_1}\odot\ff_{k_2,\ell_2,g_2} \odot
    \cdots\odot \ff_{k_r,\ell_r,g_r}).
\end{align*}
Now observe that the first summand contains three kinds of terms:
\begin{enumerate}[(i)]
\item terms of the type
$$
(\ff_{k_1,\ell_1,g_1} \odot \cdots
\odot \ff_{k_r,\ell_r,g_r}) \circ_{s_1,\dots,s_r}
(\hat \fp^C_{k',\ell',g'}\circ_s \hat \fp^C_{k'',\ell'',g''})
$$
with $s>0$ cancelling with the second sum;
\item terms of the type
$$
(\ff_{k_1,\ell_1,g_1} \odot \cdots
\odot \ff_{k_r,\ell_r,g_r}) \circ_{s_1,\dots,s_r}
(\hat \fp^C_{k',\ell',g'}\circ_0 \hat \fp^C_{k'',\ell'',g''})
$$
which appear in cancelling pairs;
\item terms of the type
$$
-\hat\fp^D_{k^-,\ell^-,g^-}\circ (\ff_{k_1,\ell_1,g_1} \odot \cdots
\odot \ff_{k_r,\ell_r,g_r}) \circ \hat \fp^C_{k^+,\ell^+,g^+} 
$$
which appear with opposite sign in the forth sum.
\end{enumerate}
A similar discussion applies to the forth sum, and so in total we find
that $\delta R_{K,L,G}=0$ as claimed.

To prove (3), recall that
\begin{eqnarray*}
C_{K,L,G}(\ff,\fg) &=& (\ff \circ \fg)_{K,L,G} - (\frac 1 {L!}
\ff_{1,1,0}^{\odot L}) \circ \fg_{K,L,G} - \ff_{K,L,G} \circ (\frac 1
    {K!} \fg_{1,1,0}^{\odot K}) \\
&=& \left \langle \sum_{(k_i,\ell_i,g_i) \prec (K,L,G) \atop
          (k'_j,\ell'_j,g'_j) \prec (K,L,G) } \frac 1 {r!r'!}
    (\ff_{k_1,\ell_1,g_1} \odot \dots \odot \ff_{k_r,\ell_r,g_r})
    \circ_{\rm conn} \right. \\
& &  \left. 
(\fg_{k'_1,\ell'_1,g'_1} \odot \dots \odot
    \fg_{k'_{r'},\ell'_{r'},g'_{r'}}) 
\right \rangle_{K,L,G},
\end{eqnarray*}
where $\circ_{\rm conn}$ signifies that we only keep those
compositions which in the geometric picture correspond to a connected
end result of gluing surfaces
\footnote{This process is parallel to the 
relation between a prop and a properad as in \cite[Proposition
  2.1]{Vallette}.}. 
Note also that in each factor
$\ff_{k_1,\ell_1,g_1} \odot \dots \odot \ff_{k_r,\ell_r,g_r}$
(respectively $\fg_{k'_1,\ell'_1,g'_1} \odot \dots \odot
\fg_{k'_{r'},\ell'_{r'},g'_{r'}}$) at least one of the signatures is
different from $(1,1,0)$.

Now since $\delta$ is a derivation of composition, and using the
hypothesis that $\delta \ff_{k,\ell,g} = R_{k,\ell,g}(\ff)$ and
$\delta \fg_{k,\ell,g} = R_{k,\ell,g}(\fg)$ for all $(k,\ell,g) \prec
(K,L,G)$,  we find that
\begin{eqnarray*}
\lefteqn{\delta C_{K,L,G} =}\\
& &\left \langle \sum_{(k_i,\ell_i,g_i) \prec (K,L,G) \atop
          (k'_j,\ell'_j,g'_j) \prec (K,L,G) } \frac 1 {(r-1)!r'!} \right.\\
& &   (R_{k_1,\ell_1,g_1}(\ff) \odot \dots \odot \ff_{k_r,\ell_r,g_r})
    \circ_{\rm conn}  (\fg_{k'_1,\ell'_1,g'_1} \odot \dots \odot
    \fg_{k'_{r'},\ell'_{r'},g'_{r'}}) \\
& & + \sum_{(k_i,\ell_i,g_i) \prec (K,L,G) \atop
          (k'_j,\ell'_j,g'_j) \prec (K,L,G) } \frac 1 {r!(r'-1)!}\\
& &\left. (\ff_{k_1,\ell_1,g_1} \odot \dots \odot \ff_{k_r,\ell_r,g_r})
    \circ_{\rm conn}  (R_{k'_1,\ell'_1,g'_1}(\fg) \odot \dots \odot
    \fg_{k'_{r'},\ell'_{r'},g'_{r'}}) 
\right \rangle_{K,L,G}\\
&=& A + B.
\end{eqnarray*}
Next substitute the corresponding expressions for the $R_{k,\ell,g}$ and
observe that most of the components in $A$ involving 
$\hat\fp^C$ have a corresponding term with opposite sign in $B$. The
only ones remaining are of the form 
$$
(\frac 1 {L!} \ff_{1,1,0}^{\odot L}) \circ \hat\fp_{r',\ell',g'}^C
\circ_{s_1,\dots,s_{r'}} (\fg_{k'_1,\ell'_1,g'_1} \odot \cdots \odot
   \fg_{k_{r'},\ell_{r'},g_{r'}})
$$
with all $s_i>0$. They appear in the summands of the form 
$$
\frac 1 {(r-1)!r'!}(R_{k_1,\ell_1,g_1}(\ff) \odot (\ff_{1,1,0})^{\odot r-1})
    \circ_{\rm conn}  (\fg_{k'_1,\ell'_1,g'_1} \odot \dots \odot
    \fg_{k'_{r'},\ell'_{r'},g'_{r'}})
$$
where $(k_1,\ell_1,g_1)$ is the only signature different from
$(1,1,0)$ in the first factor. In the claimed expression for
$R_{K,L,G}(\ff \circ \fg)$, these terms precisely cancel the terms in $\frac
1 {L!} \ff_{1,1,0}^{\odot L} \circ R_{K,L,G}(\ff)$ involving $\hat \fp^C$.

Similarly, the only terms from $B$ involving $\hat \fp^C$ that remain after
cancellation with corresponding terms in $A$ are those of the form
$$
(\ff_{k_1,\ell_1,g_1} \odot \cdots \odot
   \fg_{k_r,\ell_r,g_r}) \circ_{s_1,\dots,s_r}
   \hat \fp_{k',r,g'}^C \circ (\frac 1 {K!} \fg_{1,1,0}^{\odot K})
$$
with all $s_i>0$, and they precisely cancel the contributions from
$R_{K,L,G}(\ff) \circ \frac 1 {K!} \fg_{1,1,0}^{\odot K}$ involving
$\hat \fp^C$.

Finally note that all the terms in $A$ involving $\hat \fp^D$ and all
the terms in $B$ involving $\hat \fp^B$ also appear in $R_{K,L,G}(\ff
\circ \fg)$. Moreover, the missing pieces are again precisly those
which are supplied by the corresponding terms in
$R_{K,L,G}(\ff) \circ \frac 1 {K!} \fg_{1,1,0}^{\odot K}$ involving
$\hat \fp^D$ and the terms in $\frac 1 {L!} \ff_{1,1,0}^{\odot L}
\circ R_{K,L,G}(\ff)$ involving $\hat \fp^B$. 
\end{proof}

Besides these assertions, we will repeatedly make use of the following
well-known observations. 
\begin{lem}\label{lem:homotopy}
Let $e:(B,\p_B) \to (C,\p_C)$ be a homotopy equivalence of chain
complexes which has a chain homotopy inverse $i:(C,\p_C) \to (B,\p_B)$
such that $ei=id_C$. 
\begin{enumerate}[\rm (1)]
\item If $f:(A,\p_A) \to (B,\p_B)$ is a map of degree $D$
satisfying $\p_B f - (-1)^D f \p_A=0$ and  $ef=0$, then there exists a
map $H:(A,\p_A) \to (B,\p_B)$ of degree $D+1$ such that
$f + \p_B H +(-1)^D H \p_A=0$, and $e H =0$.
\item If $g:(B,\p_B) \to (A,\p_A)$ is a map of degree $D$
satisfying $\p_A g - (-1)^D g \p_B=0$ and $gi=0$, then there exists a
map $H:(B,\p_B) \to (A,\p_A)$ of degree $D+1$ such that
$g + \p_A H +(-1)^D H \p_B=0$ and $Hi =0$.
\end{enumerate}
\end{lem}
\begin{proof}
Choose any homotopy $h$ satisfying
$$
\p_B h + h \p_B = ie -id_B
$$
and set $h'=(\id_B-ie)h(\id_B - ie)$. Then, since $ei=id_C$, 
one straightforwardly checks that $h'$ is also a homotopy between $ie$
and $id_B$, satisfying in addition that $eh'=h'i=0$. 
Now for assertion (1), define $H:=h'f$ and compute that
$$
\p_B H + (-1)^D H\p_A + f= (\p_B h'+ h' \p_B)f + f= (ief-f)+f =0
$$
and $eH=eh'f=0$.

Similarly, for assertion (2), define $H:=gh'$ and compute that
$$
\p_A H + (-1)^DH \p_B + g=g(\p_B h'+ h'\p_B) +g = gie-g + g =0
$$
and $Hi=gh'i=0$ as required.
\end{proof}

%

\section{Homotopy of morphisms}\label{sec:hom}

In this section we define homotopies between $\IBL_\infty$-morphisms and
prove some elementary properties of this relation. We follow the
approaches of \cite{Sul78} and \cite{FOOO06}.

\begin{defn}\label{def:path}
Let $(C,\{\fp_{k,\ell,g}\})$ and $(\fC, \{\fq_{k,\ell,g}\})$
be $\IBL_\infty$-algebras. We say that 
$(\fC, \{\fq_{k,\ell,g}\})$ together with $\IBL_\infty$-morphisms
$\iota:C \to \fC$ and $\eps_0,\eps_1:\fC \to C$ is a
{\em path object for  $C$} if the following hold:
\begin{enumerate}[\rm (a)]
\item $\iota$, $\eps_0$ and $\eps_1$ are linear morphisms (and we
   denote their $(1,1,0)$ parts by the same letters);
\item $\eps_0 \circ \iota = \eps_1 \circ \iota = id_C$;
\item $\iota:C[1] \to \fC[1]$ and $\eps_0,\eps_1:\fC[1] \to C[1]$ are
   chain homotopy equivalences (with   respect to $\fp_{1,1,0}$ and
   $\fq_{1,1,0}$);
\item the map $\eps_0 \oplus \eps_1:\fC[1] \to C[1] \oplus C[1]$ admits
a linear right inverse.
\end{enumerate}
\end{defn}

\begin{prop}\label{prop:model}
For any $\IBL_\infty$-algebra $(C,\{\fp_{k,\ell,g}\})$ there exists a
path object $\fC$.
\end{prop}
\begin{proof}
Define 
$$
\fC := C \oplus C \oplus C[1],
$$ 
with boundary operator $\fq_{1,1,0}: \frak C \to \frak C$ given by
$$
\fq_{1,1,0}(x_0,x_1,y)=(\fp_{1,1,0}(x_0), \fp_{1,1,0}(x_1),
x_1-x_0-\fp_{1,1,0}(y)). 
$$
We define $\iota:C \to \fC$ by $\iota(x)=(x,x,0)$, and
we define $\eps_i:\fC \to C$ and $\eps_i:\fC \to C$ by
$\eps_i(x_0,x_1,y))=x_i$.

With these definitions it is obvious that $\iota$, $\eps_0$ and
$\eps_1$ are chain maps, that $\eps_0\circ\iota=\eps_1\circ\iota=id_C$, and
that $\eps_0\oplus \eps_1$ admits a right inverse.

To prove that $\iota \circ \eps_0$ is homotopic to the identity of
$\frak C$, we define $H:\frak C \to \frak C$ by
$H(x_0,x_1,y)=(0,-y,0)$. Then one checks that 
\begin{align*}
(\fq_{1,1,0}H+ H\fq_{1,1,0})(x_0,x_1,y) &=(0,-\fp_{1,1,0}(y),-y) +
(0,x_0-x_1+\fp_{1,1,0}(y),0) \\
&= (\iota \circ \eps_0 -id_{\frak C})(x_0,x_1,y).
\end{align*}
A similar argument works for $\iota\circ \eps_1$.

To complete the proof, it remains to show that we can construct the
higher operations $\fq_{k,\ell,g}:E_k\fC \to E_\ell\fC$ in
such a way that $\iota$, $\eps_0$ and $\eps_1$ are
$\IBL_\infty$-morphisms.

We proceed by induction on our linear order of the signatures $(k,\ell,g)$.
So assume that $\fq_{k,\ell,g}$ has been constructed for all
$(k,\ell,g) \prec (K,L,G)$, in such a way that 
\begin{enumerate}[(1)]
\item $\frac {1}{\ell!}\iota^{\odot \ell}\fp_{k,\ell,g}=
   \fq_{k,\ell,g} \frac{1}{k!}\iota^{\odot k}:E_{k} C \to
   E_{\ell}\fC$ for all $(k,\ell,g)\prec (K,L,G)$,
\item $\frac {1}{\ell!}\eps_i^{\odot \ell} \fq_{k,\ell,g}=
   \fp_{k,\ell,g}\frac{1}{k!}\eps_i^{\odot k}:E_{k}\fC \to
   E_{\ell}C$ for all $(k,\ell,g)\prec (K,L,G)$ and   $i=0,1$, and 
\item the operations $\fq$ satisfy equation
   \eqref{eq:BL4} for all $(k,\ell,g)\prec (K,L,G)$.
\end{enumerate} 
Define $\fq_{K,L,G}'':=\frac {1}{L!}\iota^{\odot L} \fp_{K,L,G}\frac {1}{K!}
\eps_0^{\odot K}$, and consider
$$
\Gamma:= Q_{K,L,G} + \hat\fq_{1,1,0}\fq_{K,L,G}'' +
\fq_{K,L,G}''\hat\fq_{1,1,0}:E_K\fC \to E_L \fC,
$$
where $Q_{K,L,G}$ is the quadratic expression in the $\fq$'s
implicitly defined in Lemma~\ref{lem:order}. Note that $\Gamma$ is
homogeneous of even degree $-2d(K+G-1)-2$. Using the inductive
assumption (3) and part (1) of Proposition~\ref{prop:hom-main}, one
finds that 
$$
\hat\fq_{1,1,0} \Gamma - \Gamma \hat\fq_{1,1,0} = \hat\fq_{1,1,0} Q_{K,L,G} -
Q_{K,L,G} \hat\fq_{1,1,0} = 0, 
$$
i.e. $\Gamma$ is a chain map. Moreover, it follows from assumption 
(1) above that
$$
Q_{K,L,G} \frac {1}{K!}\iota^{\odot K}= 
\frac {1}{L!}\iota^{\odot L}P_{K,L,G}, 
$$
where $P_{K,L,G}$ is the expression of Lemma~\ref{lem:order} in
the $\fp$'s. We conclude that 
\begin{align*}
\Gamma \frac {1}{K!}\iota^{\odot K}  &= 
(Q_{K,L,G} + \hat\fq_{1,1,0}\fq_{K,L,G}'' +
\fq_{K,L,G}''\hat\fq_{1,1,0}) \frac {1}{K!}\iota^{\odot K}\\ 
&=\frac {1}{L!}\iota^{\odot L} P_{K,L,G} + \hat\fq_{1,1,0}
\fq_{K,L,G}'' \frac {1}{K!}\iota^{\odot K}  +
    \fq_{K,L,G}'' \frac {1}{K!} \iota^{\odot K}\hat\fp_{1,1,0}\\
&= \frac{1}{L!}\iota^{\odot L} (P_{K,L,G} + \hat\fp_{1,1,0}\fp_{K,L,G} +
\fp_{K,L,G}\hat\fp_{1,1,0}) \\
&= 0. 
\end{align*}

Applying part (2) of Lemma~\ref{lem:homotopy} with $e=\frac
{1}{K!}\eps_0^{\odot K}$, $i=\frac {1}{K!}\iota^{\odot K}$
and $g=\Gamma$, we get the existence of a map $H:E_K \frak C \to E_L
\frak C$ such that 
$$
\Gamma + \hat\fq_{1,1,0} H +H \hat \fq_{1,1,0}=0 \quad \text{\rm and} \quad
H \frac {1}{K!}\iota^{\odot K} =0. 
$$
We set 
$$
\fq_{K,L,G}':= \fq_{K,L,G}'' + H.
$$
Then the collection
$\{\fq'_{K,L,G},\{\fq_{k,\ell,g} \,|\, (k,\ell,g)\prec
     (K,L,G)\}\}$ satisfies the inductive assertions (3) and (1)
above for the triple $(K,L,G)$.

We now want to modify $\fq_{K,L,G}'$ in such a way that it will
also satisfy the inductive assertion (2). To proceed, we define
$$
\Gamma_i:=\frac {1}{L!}\eps_i^{\odot L}\fq_{K,L,G}' - \fp_{K,L,G} \frac
{1}{K!} \eps_i^{\odot K}:E_K \fC \to E_L C,
\quad i=0,1, 
$$
which are homogeneous of degree $-2d(K+g-1)-1$. These are the error
terms in (2). Now 
\begin{align*}
\hat\fp_{1,1,0}\Gamma_i+\Gamma_i\hat\fq_{1,1,0} 
=& \hat\fp_{1,1,0} \Gamma_i + (\frac {1}{L!}\eps_i^{\odot L}
\fq_{K,L,G}' - \fp_{K,L,G}\frac {1}{K!} \eps_i^{\odot K})\hat \fq_{1,1,0}\\ 
=& \hat\fp_{1,1,0} \Gamma_i + \frac {1}{L!}\eps_i^{\odot
   L}(-\hat \fq_{1,1,0}\fq'_{K,L,G} - Q_{K,L,G}) \\
  & + (P_{K,L,G}+ \hat\fp_{1,1,0}\fp_{K,L,G}) \frac
{1}{K!}\eps_i^{\odot K}\\ 
=& - \frac {1}{L!}\eps_i^{\odot L}Q_{K,L,G} 
  + P_{K,L,G} \frac{1}{K!}\eps_i^{\odot K}\\ 
=& 0
\end{align*}
by inductive assumption (2). Moreover, by (1) we have
$$
\Gamma_i \frac {1}{K!}\iota^{\odot K}= \frac {1}{L!}\eps_i^{\odot
   L}\fq_{K,L,G}' \frac {1}{K!}\iota_i^{\odot K} - \fp_{K,L,G} = 0.
$$
So again by part (2) of Lemma~\ref{lem:homotopy} above, there exist even maps
$\chi_i:E_K\fC \to E_L C$, $i=0,1$ of degree $-2d(K+g-1)$ such that 
$$
\Gamma_i+ \hat\fp_{1,1,0}\chi_i - \chi_i \hat \fq_{1,1,0}=0 \quad \text{\rm and } 
\chi_i \frac 1 {K!}\iota^{\odot K} =0.
$$
Choosing a right inverse $\rho:E_LC \oplus E_LC \to E_L\fC$ to
$\frac 1 {L!} \eps_0^{\odot L} \oplus \frac 1 {L!} \eps_1^{\odot L}$,
we find a linear lift $\chi=\rho\circ (\chi_0\oplus \chi_1): E_K \fC \to
E_L \fC$ such that $\chi_i= \frac{1}{L!} \eps_i^{\odot L} \circ
\chi$ and $\chi \frac {1}{K!}\iota^{\odot K} =0$. Now we
define 
$$
\fq_{K,L,G}:= \fq_{K,L,G}' + \hat\fq_{1,1,0}\chi - \chi \hat
\fq_{1,1,0},
$$
which is easily seen to satisfies all three inductive assumptions.
\end{proof}
\begin{rem}
Note that in the inductive construction of the $\IBL_\infty$-structure
in the above proof we did not make use of the specific form of the
chain complex $(\fC, \fq_{1,1,0})$ satisfying {\rm (a)-(d)}, so one could
have started equally well with any other chain model for $\fC$. 
\end{rem}
\begin{prop}\label{prop:homotopy-transfer}
Let $C$ and $D$ be $\IBL_\infty$-algebras, and let $\fC$ and $\fD$ be
path objects for $C$ and $D$, respectively. Let $\ff:C \to D$ be a
morphism. Then there exists a morphism $\fF:\fC \to \fD$ such that
the diagram
$$
\begin{CD}
C @>\iota^C>> \fC @>\eps_i^C>> C \\
@V{\frak f}VV       @V{\frak F}VV      @V{\frak f}VV\\
D @>>\iota^D> \fD @>>\eps_i^D> D 
\end{CD}
$$
commutes for both $i=0$ and $i=1$.
\end{prop}
\begin{proof}
The proof is inductive and similar in structure to the proof of
Proposition~\ref{prop:model}.

{\em Step 1:} We construct a chain map $\fF_{1,1,0}:(\fC,\fq^C_{1,1,0})
\to (\fD,\fq^D_{1,1,0})$ satisfying the required relations.

Set $F':= \iota^D \ff_{1,1,0}\eps_0^C:\fC \to \fD$, and note that
$F'\iota^C-\iota^D\ff_{1,1,0}=0$ as required. Similarly, 
$\Gamma_0= \eps_0^D F' - \ff_{1,1,0}\eps_0^C=0$, but $\Gamma_1:=
\eps_1^D F'-\ff_{1,1,0}\eps_1^C\neq 0$. One checks that 
$$
\fp_{1,1,0}^D \Gamma_1-\Gamma_1 \fq_{1,1,0}^C=0, \quad \text{\rm and } \Gamma_1\iota^C=0.
$$
Hence by part (2) of Lemma~\ref{lem:homotopy}, there exists a chain
homotopy $E_1:\fC \to D$ such that
$$
\Gamma_1 + \fp_{1,1,0}^D E_1 + E_1 \fq_{1,1,0}^C = 0 \quad \text{\rm and
} E_1\iota^C=0.
$$
Choosing a right inverse to $\eps_0 \oplus \eps_1:\fD \to D \oplus D$,
we construct a lift $E: \fC \to \fD$ of $0 \oplus E_1$ such that
$\eps_0^D \circ E =0$, $\eps_1^D\circ E=E_1$ and $E \circ \iota^C=0$. Then 
$$
\fF_{1,1,0}:= F' + \fq_{1,1,0}^DE + E \fq_{1,1,0}^C:\fC \to \fD
$$
is the required chain map.

{\em Step 2:} We now proceed by induction on our linear order of
signatures $(k,\ell,g)$. So suppose we have already constructed maps
$\fF_{k,\ell,g}:E_{k}\fC \to E_{\ell} \fD$ for all
$(k,\ell,g)\prec (K,L,G)$ such that: 
\begin{enumerate}[(1)]
\item $\frac {1}{\ell!}(\iota^D)^{\odot \ell} \ff_{k,\ell,g} =
   \fF_{k,\ell,g}\frac {1}{k!}(\iota^C)^{\odot k}$ for all
   $(k,\ell,g)\prec (K,L,G)$, 
\item $\frac {1}{\ell!}(\eps_i^D)^{\odot \ell}\fF_{k,\ell,g} =
   \ff_{k,\ell,g}\frac {1}{k!}(\eps_i^C)^{\odot k}$  for $i=0,1$
   and for all $(k,\ell,g)\prec (K,L,G)$, and 
\item the defining equation \eqref{eq:mor} for morphisms holds for
   $\fF$, $\fq^C$ and $\fq^D$ for all $(k,\ell,g)\prec (K,L,G)$.
\end{enumerate}
Consider the expression $R_{K,L,G}(\fF,\fq^C,\fq^D):E_K\fC \to
E_L \fD$ as defined in the statement of Lemma~\ref{lem:mor_order}.
Using the inductive assumptions and part (2) of
Proposition~\ref{prop:hom-main}, one proves that
\begin{align*}
&\hat\fq^D_{1,1,0}R_{K,L,G}(\fF,\fq^C,\fq^D) +
   R_{K,L,G}(\fF,\fq^C,\fq^D) \hat\fq^C_{1,1,0} = 0\\ 
&\frac 1{L!}(\iota^D)^{\odot L} R_{K,L,G}(\ff,\fp^C,\fp^D) = 
R_{K,L,G}(\fF,\fq^C,\fq^D) \frac 1 {K!} (\iota^C)^{\odot k} \quad
\text{\rm and}\\ 
&\frac 1 {L!} (\eps_i^D)^{\odot L} R_{K,L,G}(\fF,\fq^C,\fq^D) = 
R_{K,L,G}(\ff,\fp^C,\fp^D) \frac 1 {K!} (\eps_i^C)^{\odot k} \quad
\text{\rm for }i=0,1. 
\end{align*}
Now, as in the proof of Proposition~\ref{prop:model}, define
$\fF''_{K,L,G}:= \frac {1}{L!}(\iota^D)^{\odot   L}\ff_{K,L,G}\frac 1
{K!}(\eps_0^C)^{\odot K}:E_K \fC \to E_L \fD$, and 
observe that
$$
F:= R_{K,L,G}(\fF,\fq^C,\fq^D) + \hat\fq_{1,1,0}^D\fF''_{K,L,G} -
\fF''_{K,L,G}\hat\fq_{1,1,0}^C:E_k\fC \to E_L \fD
$$
is of odd degree and satisfies $\hat\fq_{1,1,0}^D F+ F\hat\fq_{1,1,0}^C=0$ and
$F \frac{1}{K!}(\iota^C)^{\odot K} = 0$. Hence, by part (1) of
Lemma~\ref{lem:homotopy}, we find  $H: E_K \fC \to E_L \fD$ of even
degree such that 
$$
F + \fq_{1,1,0}^D H - H \fq_{1,1,0}^C = 0, \quad \text{\rm and }
H \frac 1 {K!}(\iota^C)^{\odot K} =0.
$$
Then the  map $\fF'_{K,L,G}:=\fF''_{K,L,G}+ H:E_K\fC \to E_L\fD$
satisfies the conditions (1) and (3) in the inductive assumption.

To achieve (2), consider the maps of even degree 
$$
\Gamma_i:=
\frac 1 {L!}(\eps_i^D)^{\odot L}\fF'_{K,L,G}- \ff_{K,L,G} \frac 1 {K!}(\eps_i^C)^{\odot K}:E_K \fC \to E_L D, \quad i=0,1
$$ 
and compute that 
$$
\fp_{1,1,0}^D \Gamma_i - \Gamma_i \fq_{1,1,0}^C = 0, \quad \text{\rm and }
\Gamma_i\frac 1 {K!}(\iota^C)^{\odot K} =0.
$$
Hence, by part (2) of Lemma~\ref{lem:homotopy}, there exist maps
$E_i:E_K\fC \to E_L D$ of odd degree such that
$$
\Gamma_i + \fp_{1,1,0}^D E_i + E_i \fq_{1,1,0}^C = 0 \quad \text{\rm and }
E_i\frac 1 {K!}(\iota^C)^{\odot K} =0.
$$
With a right inverse $\rho_L:E_LD \oplus E_LD \to E_L\fD$ to $(\eps_0^D)^{\odot
L} \oplus (\eps_1^D)^{\odot L}$, we define a linear extension $E=\rho_L\circ(E_0\oplus E_1): E_K\fC \to E_L\fD$ such that $\frac 1 {L!}(\eps_i^D)^{\odot L} \circ E=E_i$ and $E \circ \frac 1 {K!}(\iota^C)^{\odot K}=0$. Then 
$$
\fF_{K,L,G}:= \fF'_{K,L,G} + \fq_{1,1,0}^DE + E
\fq_{1,1,0}^C:E_K\fC \to E_L\fD 
$$
has the required properties. This completes the inductive step and
hence the proof of the proposition.
\end{proof}
We now come to the main definition of this section.
\begin{defn}\label{def:homotopy}
We say that two $\IBL_\infty$-morphisms $\ff_0:C \to D$ and $\ff_1:C \to D$ 
are {\em homotopic} if for some path object $\fD$ for $D$ there
exists a morphism $\fF:C \to \fD$ such that $\eps_0\diamond \fF=\ff_0$ and
$\eps_1\diamond \fF=\ff_1$. We call such an $\fF$  {\em a homotopy}
between $\ff_0$ and $\ff_1$. 
\end{defn}
\begin{prop}\label{prop:homotopy-equiv}
The notion of homotopy has the following properties:
\begin{enumerate}[\rm (a)]
\item A homotopy between $\ff_0$ and $\ff_1$ exists for some path
   object for $D$ if and only if it exists for all path objects for $D$. 
\item Homotopy of morphisms is an equivalence relation. 
\item If $\ff_0:B \to C$ and $\ff_1:B \to C$ are homotopic and 
$\fg_0:C \to D$ and $\fg_1:C \to D$ are homotopic, then
$\fg_0 \diamond \ff_0$ and $\fg_1 \diamond \ff_1$ are homotopic.
\end{enumerate}
\end{prop}
\begin{proof}
To prove (a), suppose $\fF:C \to \fD$ is a homotopy between $\ff_0:C
\to D$ and $\ff_1:C \to D$, and let $\fD'$ be any other path object for
$D$. Applying Proposition~\ref{prop:homotopy-transfer} to the
identity of $D$ and the two path objects $\fD$ and $\fD'$, we obtain an
$\IBL_\infty$-morphism $\fI:\fD \to \fD'$. Setting $\fF':=\fI \diamond
\fF$, one verifies that 
$$
\eps_i' \diamond \fF' = \eps_i' \diamond \fI \diamond \fF = \eps_i \diamond \fF = \ff_i
$$
as required.

We next prove (b). To see that $\ff:C \to D$ is homotopic to itself,
consider any path object $\fD$ for $D$ and set $\fF:= \iota \circ \ff$.

To see that the relation is symmetric, note that if
$(\fD,\iota,\eps_0,\eps_1)$ is a path object for $D$, then
$(\fD,\iota,\eps_0',\eps_1')$ is also a path object, where $\eps_0'=\eps_1$
and $\eps_1'=\eps_0$.

To prove transitivity of the relation, suppose $\ff_0$ and $\ff_1$ are
homotopic via a homotopy $\fF^1: C \to \fD^1$ and $\ff_1$ and $\ff_2$
are homotopic via a homotopy $\fF^2:C \to \fD^2$.

We define a new path object $\fD$ for $D$ as follows. As a
vector space, set
$$
    \fD := \{ (d_1,d_2) \in \fD^1 \oplus \fD^2 \,|\,
    \eps^1_1(d_1)=\eps^2_0(d_2) \}.
$$
We define the map $\iota:D \to \fD$ by $\iota(x)=(\iota^1(x),\iota^2(x))$ and
the maps $\eps_i:\fD \to D$ by $\eps_0(d_1,d_2):=\eps_0(d_1)$ and
$\eps_1(d_1,d_2)=\eps_1(d_2)$.  To construct the structure maps
$\fq_{k,\ell,g}$, first note that the two projections $\pi^i:\fD \to \fD^i$
determine a projection $\pi_k:E_k\fD \to E_k\fD^1 \oplus E_k\fD^2$
which surjects onto $\fP_k:= \{(g^1,g^2) \in E_k\fD^1 \oplus
E_k\fD^2\,|\, \frac 1 {k!}(\eps^1_1)^{\odot k}(g^1)=\frac 1
{k!}(\eps^2_0)^{\odot k}(g^2)\}$. In particular, $\pi_k$ admits a right
inverse $\rho_k:\fP_k \to E_k\fD$.

Now we define $\fq_{k,\ell,g}:E_k \fD \to E_\ell \fD$ by 
$$
\fq_{k,\ell,g}:= \rho_\ell \circ
(\fq_{k,\ell,g}^1\oplus\fq_{k,\ell,g}^2) \circ \pi_k.
$$
Observe that, by construction, the defining relations
\eqref{eq:BL} for $\fq_{k,\ell,g}$ follow from the defining relations of
$\fq_{k,\ell,g}^i$, and similarly the properties (a)-(d) of a path object
can be easily checked using the corresponding properties of the
$\fD^i$. So we have proven that $\fD$ is a path object for $D$.

We now define a homotopy between $\ff_0$ and $\ff_2$ to be the
morphism $\fF: C \to \fD$ whose component
$\fF_{k,\ell,g}:E_k C \to E_\ell \fD$ is defined as
$\rho_\ell \circ (\fF^1_{k,\ell,g} \oplus \fF^2_{k,\ell,g})$. One
straightforwardly checks that this is indeed a morphism with the
required properties, and this completes the proof of part (b).

We now prove part (c). First let $\fG:C \to \fD$ be a homotopy between
$\fg_0$ and $\fg_1$. Then $\fG\diamond \ff_0:B \to \fD$ is a homotopy
between $\fg_0 \diamond \ff_0$ and $\fg_1 \diamond \ff_0$. So by part (b),
it suffices to prove the claim for $\fg_0=\fg_1=:\fg$.

Applying Proposition~\ref{prop:homotopy-transfer} to $\fg:C \to D$, we
obtain a map $\fG: \fC \to \fD$ with $\eps^D_i \diamond \fG = \fg \diamond
\eps^C_i$. Now let $\fF:B \to \fC$ be a homotopy between $\ff_0$ and
$\ff_1$ and set $\fH:=\fG \diamond \fF: B \to \fD$. Then from the
definitions one checks that
$$
\eps_i^D \diamond \fH  = \eps_i^D \diamond \fG \diamond \fF = \fg \diamond
\eps^C_i \diamond \fF = \fg \diamond \ff_i,
$$
so that $\fH$ is the required homotopy between $\fg \diamond \ff_0$ and
$\fg \diamond \ff_1$.
\end{proof}
We say that an $\IBL_\infty$-morphism $\ff:C \to D$ is a {\em homotopy
equivalence} if there exists an $\IBL_\infty$-morphism $\fg:D \to C$
such that $\ff \diamond \fg$ and $\fg \diamond \ff$ are each homotopic to
the respective identity map. From the above discussion, we get the following
immediate consequence. 
\begin{corollary}\label{Cor:comp}
A composition of homotopy equivalences is a homotopy equivalence. In
particular, homotopy equivalence is an equivalence relation. $\qed$
\end{corollary}

\par\medskip
\section{Homotopy inverse}\label{sec:whi}

In this section we prove Theorem \ref{Whi}, following the scheme of the
corresponding argument in \cite[\S4.5]{FOOO06} in the
$A_\infty$-case. 
\begin{lem}\label{lem:inverse1}
Consider $\IBL_\infty$-algebras $(C,\{\fp^C_{k,\ell,g}\})$ and
$(D,\{\fp^D_{k,\ell,g}\})$, and let $(\fD,\{\fq^D_{k,\ell,g}\})$ be a
path object for $D$. Suppose $\fh = \{\fh_{k,\ell,g}:E_kC \to E_\ell
\fD\}_{(k,\ell,g)\prec (K,L,G)}$ satisfies \eqref{eq:mor} for
all $(k,\ell,g)\prec (K,L,G)$. Then
$$
[R_{K,L,G}(\eps_0\diamond \fh,\fp^C,\fp^D)]=[R_{K,L,G}(\eps_1 \diamond
   \fh,\fp^C,\fp^D)] \in H_*(Hom(E_K C, E_L D),\delta). 
$$
\end{lem}
\begin{proof}
Observe that, since the $\eps_i$ are linear $\IBL_\infty$-morphisms, by
part (3) of Proposition~\ref{prop:hom-main} we have 
$$
R_{K,L,G}(\eps_i \diamond \fh ,\fp^C,\fp^D) =
\frac 1 {L!}\eps_i^{\odot L}\circ R_{K,L,G}(\fh,\fp^C,\fq^D), \quad
i=0,1.
$$
Now the two maps $E_i:(Hom(E_K C,E_L \fD),\delta) \to
(Hom(E_K C, E_L D), \delta)$ given by $E_i(\varphi) = \frac 1 {L!}
\eps_i^{\odot L} \circ \varphi$ are both homotopy inverses to the same
map $I:(Hom(E_K C, E_L D), \delta) \to (Hom(E_K C, E_L \fD),
\delta)$, $I(\psi)= \frac 1 {L!}\iota^{\odot L} \circ \psi$, so they
induce the same map in homology. This proves the claim.
\end{proof}
The proof of Theorem~\ref{Whi} will be an easy consequence of the
following observation.
\begin{prop}\label{prop:inverse2}
Let $\ff:D \to C$ be an $\IBL_\infty$-morphism such that
$\ff_{1,1,0}:(D,\fp^D_{1,1,0}) \to (C,\fp^C_{1,1,0})$ is a
quasi-isomorphism of chain complexes. Then there exists an
$\IBL_\infty$-morphism $\fg:C \to D$ such that $\fg \diamond \ff$ is homotopic to
the identity of $D$.
\end{prop}
\begin{proof}
{\em Step 1:} One first constructs a chain map
$\fg_{1,1,0}:(C,\fp^C_{1,1,0}) \to (D,\fp^D_{1,1,0})$ which is a chain
homotopy inverse to $\ff_{1,1,0}$, together with a homotopy
$\fh_{1,1,0}: D \to \fD$ between $\fg_{1,1,0}\circ \ff_{1,1,0}$ and
the identity of $D$. This is completely standard.

{\em Step 2:} Now we proceed by induction on our linear order of
signatures $(k,\ell,g)$. Suppose we have constructed maps
$\fg_{k,\ell,g}:E_{k} C \to E_{\ell}D$ and
$\fh_{k,\ell,g}:E_{k}D \to E_{\ell} \fD$ for all
$(k,\ell,g)\prec (K,L,G)$ such that
\begin{enumerate}[(i)]
\item $\frac 1 {\ell!} \eps_0^{\odot \ell} \circ \fh_{k,\ell,g} = 0$ for
   all $(1,1,0) \prec (k,\ell,g) \prec (K,L,G)$ and $\eps_0 \circ \fh_{1,1,0}
   = \id_D$,
\item $\frac 1 {\ell!} \eps_1^{\odot \ell} \circ \fh_{k,\ell,g} =
   (\fg \diamond \ff)_{k,\ell,g}$ for all $(k,\ell,g) \prec
   (K,L,G)$, 
\item $\fh$ satisfies \eqref{eq:mor} for all $(k,\ell,g) \prec
   (K,L,G)$ and
\item $\fg$ satisfies \eqref{eq:mor} for all $(k,\ell,g) \prec
   (K,L,G)$.
\end{enumerate}
By inductive assumption (i), $\eps_0 \diamond \fh$ is the identity of
$D$, which is clearly an $\IBL_\infty$-morphism. So by
part (3) of Proposition~\ref{prop:hom-main} 
$$
\frac 1 {L!} (\eps_0^{\odot L})_* [R_{K,L,G}(\fh, \fp^D, \fq^D) ]
= [R_{K,L,G}(\eps_0 \diamond \fh ,
   \fp^D,\fp^D)]= 0 
$$
in $H_*(Hom(E_kD, E_L\fD),\delta)$.
Applying part (1) of Lemma~\ref{lem:homotopy} with $i= \frac 1 {L!}
\iota^{\odot L}:E_L D \to E_L \fD$, $e= \frac 1 {L!}\eps_0^{\odot L}:
E_L \fD \to E_L D$ and $f=R_{K,L,G}(\fh, \fp^D, \fq^D)$, we 
obtain  $S:E_K D \to E_L \fD$ such that
$$
R_{K,L,G}(\fh, \fp^D, \fq^D) = \delta S
$$
and $\frac 1 {L!} \eps_0^{\odot L}\circ S=0$. Note that 
\begin{align*}
\delta (\frac 1 {L!}\eps_1^{\odot L}\circ S) 
=& \frac 1 {L!}\eps_1^{\odot L} R_{K,L,G}(\fh, \fp^D, \fq^D)\\
=& R_{K,L,G}( \eps_1 \diamond \fh , \fp^D, \fp^D) \\
=& R_{K,L,G}( \fg \diamond \ff, \fp^D, \fp^D) \\
=& \frac 1 {L!} \fg_{1,1,0}^{\odot L} R_{K,L,G}(\ff, \fp^D,
\fp^C) + R_{K,L,G}( \fg , \fp^C, \fp^D) \circ \frac 1 {K!}
\ff_{1,1,0}^{\odot K} + \delta C_{K,L,G}(\fg, \ff),
\end{align*}
where we used part (3) of Proposition~\ref{prop:hom-main} in the last step.
Since $\ff$ is an $\IBL_\infty$-morphism, we conclude that
$$
[R_{K,L,G}( \fg , \fp^C, \fp^D) \circ \frac 1 {K!}
\ff_{1,1,0}^{\odot K}] = 0 \in H_*(Hom(E_K D, E_L D), \delta).
$$
Since $\frac 1 {K!} \ff_{1,1,0}^{\odot K}$ is a homotopy equivalence,
this implies that
$$
R_{K,L,G}( \fg , \fp^C, \fp^D) = \delta T
$$
for some $T:E_K C \to E_L D$. Now consider 
$$
F = T \circ \frac 1 {K!} \ff_{1,1,0}^{\odot K} + \frac 1 {L!}
\fg_{1,1,0}^{\odot L} \circ \ff_{K,L,G} + C_{K,L,G}(\fg,\ff )
- \frac 1 {L!} \eps_1^{\odot L}\circ S : E_K D \to E_L D,
$$
and note that by construction we have $\delta F=0$.

Observe also that precomposition with
$\frac 1 {K!} \ff_{1,1,0}^{\odot K}$ induces a homotopy equivalence
from $(Hom(E_K C,E_L D),\delta)$ to $(Hom(E_K D,E_L D),\delta)$.
In particular, there exists $G:E_K C \to E_L D$ with $\delta G=0$ and 
$$
[F + G \circ \frac 1 {K!} \ff_{1,1,0}^{\odot K}]=0 \in H_*(Hom(E_K
D,E_L D),\delta)).
$$
This in turn means that we can find $H_1:E_K D \to E_L D$ such that
$$
\delta H_1 = F + G \circ \frac 1 {K!} \ff_{1,1,0}^{\odot K}.
$$
Since $\eps_0 \oplus \eps_1:\fD \to D \oplus D $ admits a right inverse, we
find a lift $H:E_K D \to E_L \fD$ such that $\frac 1 {L!}
\eps_1^{\odot L} \circ H= H_1$ and $\frac 1 {L!} \eps_0^{\odot L}
\circ H=0$. Now set 
$$
\fg_{K,L,G}:= T + G \quad \text{\rm and} \quad
\fh_{K,L,G}:= S + \delta H.
$$
Let us check that they satisfy properties (i)-(iv) for
$(K,L,G)$. For (i), observe that by construction 
$$
\frac 1 {L!} \eps_0^{\odot L} \circ (S+\delta H)=
\frac 1 {L!} \eps_0^{\odot L} \circ S + \delta (\frac 1 {L!}
\eps_0^{\odot L} \circ H)= 0 
$$ 
as required. For (ii), we check that
\begin{align*} 
\frac 1 {L!} \eps_1^{\odot L} \circ (S+\delta H)=
& \frac 1 {L!} \eps_1^{\odot L} \circ S+ \delta H_1\\
=& (T+G) \circ \frac 1 {K!} \ff_{1,1,0}^{\odot K} + \frac 1 {L!}
\fg_{1,1,0}^{\odot L} \circ \ff_{K,L,G} + C_{K,L,G}(\fg,\ff )\\
=& (\fg \diamond \ff)_{K,L,G},
\end{align*}
where in the last step we used Lemma~\ref{lem:mor-comp}.
For (iii), we check that
$$
\delta \fh_{K,L,G}= \delta S = R_{K,L,G}(\fh,\fp^D,\fq^D)
$$ 
as required.
Similarly, for (iv) we observe that
$$
\delta \fg_{K,L,G}= \delta T = R_{K,L,G}(\fg,\fp^C,\fp^D).
$$
This completes the induction step and hence the proof of 
Proposition~\ref{prop:inverse2}. 
\end{proof}

We now conclude this section by proving Theorem~\ref{Whi}. Indeed, 
let $\ff:D \to C$ be an $\IBL_\infty$-morphism such that $\ff_{1,1,0}$
induces an isomorphism in homology. Then by
Proposition~\ref{prop:inverse2} there exists an $\IBL_\infty$-morphism
$\fg:C \to D$ such that $\fg \diamond \ff$ is homotopic to the identity
of $D$. As $\fg_{1,1,0}$ induces the inverse isomorphism in homology,
we can apply Proposition~\ref{prop:inverse2} again to construct
another $\IBL_\infty$-morphism $\ff':D \to C$ such that $\ff' \diamond \fg$
is homotopic to the identity of $C$. Now it follows that
$$
\ff \sim \ff' \diamond \fg \diamond \ff \sim \ff',
$$
so that by Proposition~\ref{prop:homotopy-equiv}(c) we conclude that
$\ff \diamond \fg$ is also homotopic to the identity of $C$. In other
words, $\ff$ and $\fg$ are homotopy inverses of each other. This
completes the proof of Theorem~\ref{Whi}.

\par\medskip
\section{Canonical model}\label{sec:can}

In this section we prove the following statement, which is
Theorem~\ref{thm:canonical} from the Introduction. We assume that the
ground ring $R$ is a field containing $\Q$. 

\begin{thm}
\label{thm:can}
Suppose $(C,\{\fp_{k,\ell,g}\})$ is an $\IBL_\infty$-algebra.
Then there exist operations $\{\fq_{k,\ell,g}\}$ on
its homology $H:=H_*(C,\fp_{1,1,0})$ giving it the structure of an
$\IBL_\infty$-algebra such that there exists a homotopy
equivalence $\ff: (H,\{\fq_{k,\ell,g}\}) \to (C,\{\fp_{k,\ell,g}\})$.
\end{thm}

\begin{proof}
Fix a cycle-choosing embedding $\ff_{1,1,0}:H \to C$ and a splitting
$C= H\oplus B \oplus A$, where we identify $H$ with its image under
$\ff_{1,1,0}$ and where $B=\im\, \fp_{1,1,0}$. Denote by
$\pi:C \to H$ the projection along $B \oplus A$, and by $h:C \to C$
the map which vanishes on $H \oplus A$ and is equal to the inverse of
$\fp_{1,1,0}:A \stackrel{\cong}{\to} B$ on $B$. Then we have $ \pi
\ff_{1,1,0}=\id_H$ 
and
$$
\fp_{1,1,0}h + h \fp_{1,1,0} = \id_C - \ff_{1,1,0} \pi,
$$
so that $\ff_{1,1,0}:(H,\fq_{1,1,0}=0) \to (C,\fp_{1,1,0})$ is a chain
homotopy equivalence. Now as usual we argue by induction on our linear
order of signatures $(k,\ell,g)$. So assume that we have defined
$\fq_{k,\ell,g}$ and $\ff_{k,\ell,g}$ for all
$(k,\ell,g)\prec (K,L,G)$ such that
\begin{enumerate}[(i)]
\item the $\fq$'s satisfy \eqref{eq:BL4} for all $(k,\ell,g)\prec
   (K,L,G)$, 
\item the $\fq$'s and the $\ff$'s satisfy \eqref{eq:mor4} for all
   $(k,\ell,g)\prec (K,L,G)$. 
\end{enumerate}
Consider the expression $\wt R_{K,L,G}(\ff,\fq,\fp):E_KH \to
E_L C$ appearing in the second statement of Lemma~\ref{lem:mor_order},
and define
$$
\fq_{K,L,G}:= \frac 1 {L!}\pi^{\odot L} \left(\fp_{K,L,G}\,\frac 1 {K!} 
\ff_{1,1,0}^{\odot K} -\wt R_{K,L,G}(\ff,\fq,\fp)\right). 
$$
We claim that with this (or any other) definition the
$\{\fq_{k,\ell,g}\}_{(k,\ell,g) \preceq (K,L,G)}$ satisfy equation
\eqref{eq:BL4} for $(K,L,G)$. In fact, since $\fq_{1,1,0}=0$, we need
only show that the quadratic expression $Q_{K,L,G}$ 
in the $\hat\fq_{k,\ell,g}$ defined in Lemma~\ref{lem:order} vanishes,
which does not involve $\fq_{K,L,G}$.

As in the proof of Proposition~\ref{prop:hom-main}, we use the notation
$\langle A\rangle_{k,\ell,g}$ to denote the part of the coefficient of
$\hbar^{k+g-1}\tau^{k+\ell+2g-2}$ in some map $A:EH \otimes
R\{\tau,\hbar\} \to EH \otimes R\{\tau,\hbar\}$ which
corresponds to the part mapping $E_kH$ to $E_\ell H$. 
Define
$$
\hat \fq':= \sum_{(k,\ell,g) \prec (K,L,G)} \hat
\fq_{k,\ell,g}\hbar^{k+g-1}\tau^{k+\ell+2g-2}: EH \to EH,
$$
and note that hypothesis (i) implies that $\langle \hat \fq' \hat
\fq'\rangle_{k,\ell,g} = 0$ for all $(k,\ell,g) \prec (K,L,G)$, so
that the claim $Q_{K,L,G}=0$ is equivalent to $\langle \hat\fq' \hat
\fq'\rangle_{K,L,G}=0$.

We also define
$$
\hat \fp':= \sum_{(k,\ell,g) \prec (K,L,G)} \hat
\fp_{k,\ell,g}\hbar^{k+g-1}\tau^{k+\ell+2g-2}: EC \to EC
$$
and 
$$
\ff' := \sum_{(k,\ell,g) \prec (K,L,G)} 
\ff_{k,\ell,g}\hbar^{k+g-1}\tau^{k+\ell+2g-2}: EH \to EC.
$$

Induction hypothesis (ii) implies that
$$
\langle e^{\ff'} \hat \fq' - \hat \fp' e^{\ff} \rangle_{k,\ell,g} = 0
$$
for all $(k,\ell,g) \prec (K,L,G)$. But this, together with $\langle
\hat \fq' \hat \fq'\rangle_{k,\ell,g} = 0$ for all $(k,\ell,g) \prec
(K,L,G)$, implies 
\begin{eqnarray*}
\langle \hat\fq' \hat \fq'\rangle_{K,L,G}
&=& \frac 1 {L!} \pi^{\odot L} \frac 1 {L!} \ff_{1,1,0}^{\odot L} 
\langle \hat\fq' \hat \fq'\rangle_{K,L,G}\\
&=& \frac 1 {L!} \pi^{\odot L} \langle e^{\ff'} \hat\fq' \hat
\fq'\rangle_{K,L,G} \\
&=& \frac 1 {L!} \pi^{\odot L} \langle \hat\fp' e^{\ff'} \hat
\fq'\rangle_{K,L,G} \\
&=& \frac 1 {L!} \pi^{\odot L} \langle \hat\fp' \hat \fp' e^{\ff'}
\rangle_{K,L,G} \\ 
&=& -\frac 1 {L!} \pi^{\odot L} (\hat\fp_{1,1,0} \fp_{K,L,G} +
\fp_{K,L,G} \hat \fp_{1,1,0}) \frac 1 {L!} \ff_{1,1,0}^{\odot L}\\
&=& 0
\end{eqnarray*}
since $\pi \fp_{1,1,0} = \fp_{1,1,0} \ff_{1,1,0} = 0$. This proves
that the $\fq$'s satisfy relation~\eqref{eq:BL4} for $(K,L,G)$.

Now we apply part (2) of Proposition~\ref{prop:hom-main}
to find that
$$
\hat \fp_{1,1,0} \left(\frac 1 {K!} \fp_{K,L,G}
\ff_{1,1,0}^{\odot K} -\wt R_{K,L,G}(\ff,\fq,\fp)\right)= \delta
R_{K,L,G}(\ff,\fq,\fp)=0. 
$$
Since $\frac 1 {L!}\ff_{1,1,0}^{\odot L}:E_L H \to E_L C$ 
is a chain homotopy inverse to $\pi_L$, we can choose a chain
homotopy between $\frac 1 {L!}\ff_{1,1,0}^{\odot L}
\fq_{K,L,G}$ and $\frac 1 {K!} \fp_{K,L,G} \ff_{1,1,0}^{\odot
K}- \wt R_{K,L,G}(\ff,\fp,\fq)$ and denote it by 
$\ff_{K,L,G}$. Then by construction we have 
$$
-\hat\fp_{1,1,0}\ff_{K,L,G} 
+\frac 1 {L!}\ff_{1,1,0}^{\odot L}\fq_{K,L,G} 
- \frac 1 {K!} \fp_{K,L,G}\ff_{1,1,0}^{\odot K}
+ \wt R_{K,L,G}(\ff,\fp,\fq)=0,
$$
which according to Lemma~\ref{lem:mor_order} proves property (ii) for
the induction. 
\end{proof}

Suppose $(H, \{\fq_{k,\ell,g}\})$ is an $\IBL_\infty$-algebra, where
$H=H_*(C, \p)$ is the homology of some chain complex $(C,\p)$, and so
inparticular $\fq_{1,1,0}=0$. Let $\ff=\ff_{1,1,0}:H \to C$ and $\pi:C
\to H$ be maps as in the above proof. Then we get the structure of an
$\IBL_\infty$-algebra on $C$ by setting $\fp_{1,1,0}:=\p$ and 
$$
\fp_{k,\ell,g} := \frac 1 {\ell!} \ff^{\odot \ell} \fq_{k,\ell,g}\frac
1 {k!} \pi^{\odot k}
$$
for $(1,1,0)\prec (k,\ell,g)$. Since $\pi \ff= \id_H$ and $\pi
\fp_{1,1,0} = \fp_{1,1,0} \ff = 0$, the identities for
$(C,\fp_{k,\ell,g})$ easily follow from those for $H$, and $\ff$ and
$\pi$ are linear homotopy equivalences inverse to each other.

Now suppose $f:(C, \p^C) \to (D,\p^D)$ is a chain homotopy
equivalence and suppose $C$ has the structure of an
$\IBL_\infty$-algebra $(C, \{\fp^C_{k,\ell,g}\})$ with
$\fp^C_{1,1,0}=\p^C$. By Theorem~\ref{thm:can}, this structure can be
projected to an $\IBL_\infty$-structure $(H, \{\fq_{k,\ell,g}\})$ on
the homology $H=H(C, \p^C)\stackrel{\ff}{\cong} H(D, \p^D)$ and then
lifted according to the above discussion. So we have
\begin{corollary}
Let $(C,\{\fp_{k,\ell,g}\})$ be an $\IBL_\infty$-algebra, $(D,\p^D)$ a
chain complex, and $f:(C, \p^C) \to (D,\p^D)$ a chain homotopy
equivalence. Then there exists an $\IBL_\infty$-structure
$\{\fq_{k,\ell,g}\}$ on $D$ with $\fq_{1,1,0}=\p^D$ and a chain
homotopy equivalence of $\IBL_\infty$-algebras
$\ff_{k,\ell,g}:(C,\{\fp_{k,\ell,g}\})\to (D,\{\fq_{k,\ell,g}\})$ with
$\ff_{1,1,0}=f$. $\hfill \qed$
\end{corollary}

\section{Relation to differential Weyl algebras}\label{sec:Weyl}

In this section we will explain the relation between the
$\IBL_\infty$-formalism and the formalism of differential Weyl algebras
used to describe symplectic field theory (SFT) for contact manifolds
in~\cite{EGH00}.  

{\bf Objects.} 
Fix a ground ring $R$ containing $\Q$, and fix some index set
$\PP$ (which corresponds to the set of periodic orbits in SFT). 
Consider the Weyl algebra $\W$ of power series in
variables $\{p_{\gamma}\}_{\gamma \in \PP}$ and $\hbar$ with
coefficients polynomial over $R$ in variables
$\{q_{\gamma}\}_{\gamma\in \PP}$. Each variable comes with an
integer grading, and we assume that 
$$
|p_\gamma| + |q_\gamma|=|\hbar|=2d
$$
for some integer $d$ and all $\gamma\in \PP$. (In SFT, $d=n-3$ for a
contact manifold of dimension $2n-1$.) $\W$ comes equipped
with an associative product $\star$ in which all variables commute
according to their grading except for $p_\gamma$ and $q_\gamma$
corresponding to the same index $\gamma$, for which we have
$$
p_\gamma \star q_\gamma - (-1)^{|p_\gamma||q_\gamma|}q_\gamma \star
p_\gamma = \kappa_\gamma \hbar
$$
for some integers $\kappa_\gamma \geq 1$ (which correspond to
multiplicities or periodic orbits in SFT).

A homogeneous element $\H \in \frac 1 {\hbar} \W$ of degree $-1$
satisfying the {\em master equation}
\begin{equation}\label{eq:weyl1}
\H \star \H = 0
\end{equation}
is called a {\em Hamiltonian}, and the pair $(\W,\H)$ is called a {\em
   differential Weyl algebra of degree $d$}. Indeed, the commutator
with $\H$ is then a derivation of $(\W,\star)$ of square 0.

We will impose two further restrictions on our Hamiltonians $\H$,
namely
\begin{equation}\label{eq:H-exact}
\H|_{p=0}=0 \quad \text{\rm and} \quad \H|_{q=0}=0.
\end{equation}

\begin{rem}
In SFT, the first condition in~\eqref{eq:H-exact} is always
satisfied, and the second one can be arranged if the theory admits an
augmentation, which for example is always the case if the
corresponding contact manifold admits a symplectic filling.
\end{rem}

Note that, under our restrictions, $\H$ can be expanded as
\begin{equation}\label{formula:H}
   \H = \sum_{k,\ell \geq 1, g \geq 0} H_{k,\ell,g} \hbar^{g-1},
\end{equation}
where $H_{k,\ell,g}$ is the part of the coefficient of $\hbar^{g-1}$
which has degree $k$ in the $p$'s and degree $\ell$ in the $q$'s.

Consider now the free $R$-module $C$ generated by the elements
$q_\gamma$ for $\gamma\in\PP$, and graded by the
degrees $\deg(q_\gamma):=|q_\gamma|+1$. Then $EC=\bigoplus_{k\geq 1}E_kC$,
defined as in Section~\ref{sec:def}, is the non-unital commutative
algebra of polynomials in the variables $\{q_\gamma\}_{\gamma \in
  \PP}$ without constant terms. 
We can represent $\W$ as differential operators acting on the left on
$EC\{\hbar\}$ by the replacements
$$
   p_\gamma\longrightarrow\hbar\kappa_\gamma
\overrightarrow{\frac \p {\p q_{\gamma}}}.
$$
Then the Hamiltonian $\H$ determines operations
\begin{equation}\label{eq:H-p}
   \fp_{k,\ell,g} := \frac 1 {\hbar^k}
   \overrightarrow{H_{k,\ell,g}}: E_kC \to E_\ell C.
\end{equation}
The fact that the coefficients of $\H$ are polynomial in the
$q_\gamma$'s translates into
\begin{equation}\label{eq:p-finite}
\begin{aligned}
   &\text{Given $k\geq 1$, $g\geq 0$ and $a\in E_kC$, the term
   $\fp_{k,\ell,g}(a)$} \cr 
   &\text{is nonzero for only finitely many $\ell\geq 1$.}
\end{aligned}
\end{equation}
%
Conversely, $\H$ can be recovered from the operations $\fp_{k,\ell,g}$
by 
\begin{equation}\label{eq:p-H}
   H_{k,\ell,g} = \sum_{\gamma_1,\dots,\gamma_k\in\PP}\frac{1}
   {\kappa_{\gamma_1}\cdots \kappa_{\gamma_k}}
    \fp_{k,\ell,g}(q_{\gamma_1}\cdots q_{\gamma_k})p_{\gamma_1}\cdots
    p_{\gamma_k}. 
\end{equation}

\begin{prop}
Equations~\eqref{eq:H-p} and~\eqref{eq:p-H} define a one-to-one
correspondence between differential Weyl algebras
satisfying~\eqref{eq:H-exact} and $\IBL_\infty$-algebras
satisfying~\eqref{eq:p-finite} (both of degree $d$).    
\end{prop}

\begin{proof}
In the present context, the operator $\hat \fp$ appearing in 
Definition~\ref{def:IBL} can be written as
$$
   \hat \fp = \sum_{k,\ell,g} \overrightarrow{H_{k,\ell,g}}
   \hbar^{g-1}:EC\{\hbar\}\to EC\{\hbar\}. 
$$
(The condition~\eqref{eq:p-finite} allows us to set $\tau=1$ in
   $\hat\fp$. ) 
It is easily checked that $\H \star \H=0$ is
equivalent to $\hat \fp \circ \hat \fp=0$.
\end{proof}


{\bf Morphisms.} Next suppose $(\W^+,\H^+)$ and $(\W^-,\H^-)$ are
differential Weyl algebras of the same degree $d$ with indexing sets
$\PP^+$ and $\PP^-$. Let $\D$ denote the graded commutative associative
algebra of power series in the $p^+$ and $\hbar$ with coefficients
polynomial in the $q^-$. By definition, a {\em morphism between the
differential Weyl algebras} is an element $\F \in  \frac 1 {\hbar} \D$ 
satisfying 
\begin{equation}\label{eq:weyl2}
   e^{-\F} (\overrightarrow{\H^-}e^\F - e^\F \overleftarrow{\H^+}) = 0.
\end{equation}
Here $\H^+$ acts on $e^\F$ from the right by replacing each
$q^+_{\gamma}$ by $\hbar\kappa_\gamma \overleftarrow{\frac \p {\p
p^+_{\gamma}}}$, and the expression is to be viewed as an equality of
elements of $\frac 1 {\hbar} \D$. Again, we impose the additional
condition that 
\begin{equation}\label{eq:F-exact}
   \F|_{p^+=0}=0 \quad \text{\rm and} \quad \F|_{q^-=0}=0.
\end{equation}

\begin{rem}
In SFT, the first condition in~\eqref{eq:F-exact} is satisfied for potentials
coming from \emph{exact} cobordisms, and the second one can be
arranged in the augmented case. Moreover, the potential of a
general (augmented) symplectic cobordism can also be viewed as a morphism in the
above sense by splitting off the part $A=\F|_{p^+=0}$, which is treated
as a Maurer-Cartan element in the $\IBL_\infty$-algebra associated
to the negative end; cf.~\cite{CL2}. The remaining part $\F
-A$ then gives a morphism from $(\W^+,\H^+)$ to the twisted version
$(\W^-, \H_A^-)$.
\end{rem}

As with $\H$ above, we expand $\F$ as
\begin{equation}\label{formula:F}
    \F = \sum_{k,\ell,g} F_{k,\ell,g} \hbar^{g-1},
\end{equation}
and define operators $\overrightarrow{F_{k,\ell,g}}:E_kC^+ \to
E_\ell C^-$ by substituting $p_\gamma^+$ by $\hbar\kappa_\gamma
\overrightarrow{\frac \p {\p q_{\gamma}^+}}$. In this way, we get maps 
\begin{equation}\label{eq:F-f}
   \ff_{k,\ell,g}:=\frac 1
   {\hbar^k}\overrightarrow{F_{k,\ell,g}}:E_kC^+ \to E_\ell C^-. 
\end{equation}
satisfying the condition
\begin{equation}\label{eq:f-finite}
\begin{aligned}
   &\text{Given $k\geq 1$, $g\geq 0$ and $a\in E_kC^+$, the term
   $\ff_{k,\ell,g}(a)$} \cr 
   &\text{is nonzero for only finitely many $\ell\geq 1$.}
\end{aligned}
\end{equation}
%
Again, $\F$ can be recovered from the operations $\ff_{k,\ell,g}$
by 
\begin{equation}\label{eq:f-F}
   F_{k,\ell,g} = \sum_{\gamma_1,\dots,\gamma_k\in\PP^+}\frac{1}
   {\kappa_{\gamma_1}\cdots \kappa_{\gamma_k}}
    \ff_{k,\ell,g}(q_{\gamma_1}^+\cdots q_{\gamma_k}^+)p_{\gamma_1}^+\cdots
    p_{\gamma_k}^+. 
\end{equation}

\begin{prop}
Equations~\eqref{eq:F-f} and~\eqref{eq:f-F} define a one-to-one
correspondence between morphisms of differential Weyl algebras
satisfying~\eqref{eq:F-exact} and morphisms of $\IBL_\infty$-algebras
satisfying~\eqref{eq:f-finite}.    
\end{prop}
\begin{proof}
Again one checks easily that equation~\eqref{eq:weyl2} translates into
equation~\eqref{eq:mor2} relating the exponential of 
$$
   \ff = \sum_{k,\ell,g} \ff_{k,\ell,g} \hbar^{k+g-1}.
$$
and the operators $\hat\fp^\pm$ (where again we have set
$\tau=1$). For the computation, it is useful to 
observe that for any monomial $Q$ in the
$q^+$  we have
$$
   e^\ff(Q) = \bigl(\overrightarrow{e^\F} Q\bigr)\Bigl|_{q^+=0}.
$$
Moreover, $\overrightarrow{e^{\F}\overleftarrow{\H^+}}=
\overrightarrow{e^{\F}}\circ \overrightarrow{\H^+}$, and similarly 
$\overrightarrow{\overrightarrow{\H^-}e^\F}=
\overrightarrow{\H^-}\circ \overrightarrow{e^\F}$, which follows
easily from the definitions. 
\end{proof}


The composition $\F^- \diamond \F^+$ of morphisms $\F^+$ from
$(\W^+,\H^+)$ to $(\W,\H)$ and $\F^-$ from $(\W,\H)$ to $(\W^-,\H^-)$
is the morphism $\F$ from $(\W^+,\H^+)$ to $(\W^-,\H^-)$ defined as
the unique solution of 
$$
e^\F = (e^{\F^-})\star(e^{\F^+})|_{q=p=0}.
$$
Here the star product is with respect to the middle variables $p$ and
$q$, and one checks that indeed $\F \in \frac 1 {\hbar} \D$ as required.
We leave it to the reader to check that this agrees with composition of
$\IBL_\infty$-morphisms.

{\bf Homotopies.} 
For the discussion of homotopies it is convenient to extend the
definitions of Weyl algebras and morphisms between them from ordinary
ground rings $R$ to differential graded ground rings $(\widehat
R,\dd)$. 

In general, a Weyl algebra $\WW$ over a differential graded ring
$(\widehat R,\dd)$ consists of power series in variables
$\{p_\gamma\}_{\gamma \in \PP}$ and $\hbar$ with coefficients
polynomial over $\widehat R$ in variables $\{q_\gamma\}_{\gamma \in
  \PP}$, with the same grading and commutation relations as before. A
{\em Hamiltonian} in this context is now a homogeneous element $\HH
\in \frac 1\hbar \WW$ of degree $-1$ satisfying the generalized master
equation  
\begin{equation}\label{eq:weyl1-graded}
   \dd \HH + \HH \star \HH =0, 
\end{equation}
where $\dd$ is the differential in the ring, as well as our standing
assumption 
$$
   \HH|_{p=0}=0, \quad \HH|_{q=0}=0.
$$
We let $\widehat C$ be the free graded $\widehat R$-module generated
by the elements $q_\gamma$ and $E\widehat C = \bigoplus _{k\ge 1}
E_k\widehat C$, where the tensor products are taken over the
differential graded ring $\widehat R$. Representing elements of $\WW$
as differential operators as before, the generalized master equation
\eqref{eq:weyl1-graded} ensures that the operations  
\begin{align*}
   \fp_{k,\ell,g} := \frac 1{\hbar^k}
   \overrightarrow{\wh{H}_{k,\ell,g}}&: E_k\wh C \to E_\ell \wh C , 
\end{align*}
together with the differential $\dd$ on the coefficients, determine a
differential $\dd+\wh\fp$ on $E\wh C\{\hbar\}$ which squares to zero.  

Given differential graded Weyl algebras $(\WW^+,\HH^+)$ over the
differential graded ring  $\widehat R^+$ and $(\WW^-,\HH^-)$ over the
differential graded ring $\widehat R^-$, we let $\wh\D$ denote the power
series in the $p^+$ with coefficients polynomial over $\widehat R^-$
in the $q^-$. A {\em morphism} between the differential graded Weyl
algebras now consists of a morphism of differential graded rings
$\rho:\widehat R^+ \to \widehat R^-$ and an element $\G \in \frac 1
\hbar \wh\D$ satisfying 
\begin{equation} 
   e^{-\G}\left( \dd e^\G + \overrightarrow{\HH^-}e^\G -
   e^\G\overleftarrow{\rho(\HH^+)}\right) =0, 
\end{equation}
as well as
$$
   \G|_{p^+=0}=0, \quad \G|_{q^-=0}=0.
$$
As above, this induces a morphism $\fg$ from $(E\wh C^+,\wh\fp^+)$ to
$(E\wh C^-,\wh\fp^-)$ satisfying
$$
   (\dd+\wh\fp^-)e^\fg - e^\fg(\dd+\wh\fp^+) = 0. 
$$

We will apply these generalizations as follows. Associated to a given
ground ring (without differential) $R$, we now introduce the {\em
  differential graded} ring $(R[s,ds],\dd)$ where  
$$
   |s|=0 , \quad |ds|=-1 , \quad \text{\rm and } \dd(s)=ds, \quad
   \dd(ds)=0. 
$$ 
Thinking of elements in $R[s,ds]$ as polynomial functions $f(s,ds)$
with values in $R$, we have morphisms 
$$
   R \stackrel{j}{\longrightarrow} (R[s,ds],\dd)
   \stackrel{e_i}{\longrightarrow} R, \quad i=0,1 
$$
defined by 
$$
   j(r)=r, \quad e_i(f(s,ds))=f(i,0).
$$
They satisfy $e_i \circ j = \id_R$ and $j \circ e_i \sim
\id_{R[s,ds]}$, where a chain homotopy $H:R[s,ds]\to R[s,ds]$ with
$\dd H+H\dd=\id-je_i$ is given by the integration map
$g(s)+h(s)ds\mapsto\int_i^sh(s)ds$. 

Now given any differential Weyl algebra $(\W,\H)$ over the ring
(without differential) $R$ generated by $\{p_\gamma\}_{\gamma \in\PP}$ and
$\{q_\gamma\}_{\gamma \in\PP}$, we consider the new differential Weyl
algebra $(\WW,\HH)$ over $R[s,ds]$, where $\WW= R[s,ds]\otimes
\W\equiv \W[s,ds]$, and we view $\HH=\H$ as an element independent of
$s$ and $ds$. The generalized master equation \eqref{eq:weyl1-graded}
for $\HH$ follows directly from the corresponding equation
\eqref{eq:weyl1} for $\H$. 

Note that $(\W[s,ds],\HH)$ corresponds to an $\IBL_\infty$-algebra
whose underlying $R[s,ds]$-module is $C[s,ds]=R[s,ds]\otimes C$.  
Since we take tensor products over $R[s,ds]$, we have
$$
   E(C[s,ds])=(EC)[s,ds],
$$
so notations are not too ambiguous. Graded commutativity inserts the
usual signs when $ds$ is moved past some $q_\gamma$ or $p_\gamma$. As
already mentioned before, the differential takes the form
$\dd+\wh\fp:EC[s,ds]\{\hbar\} \to EC[s,ds]\{\hbar\}$, where $\wh\fp$
is induced from $\HH=\H$ as above. 


Associated to the above ring morphisms $j:R \to (R[s,ds],\dd)$ and
$e_i:(R[s,ds],\dd) \to R$ there are morphisms $\J:(\W,\H) \to
(\W[s,ds],\HH)$ and $\E_i:(\W[s,ds],\HH) \to (\W,\H)$ which act as the
identity on $C$. Explicitly, the associated power series in all three
cases is 
$$
\frac 1\hbar \sum_{\gamma\in \PP} q_\gamma p_\gamma,
$$
and the nontrivial part comes from the action on the coefficients.
Note that they are linear morphisms of Weyl algebras, in the sense
that they equal their $(1,1,0)$-terms and satisfy 
$$
e^{-\J}(\overrightarrow{\HH}e^\J - e^\J \overleftarrow{\H}) = 0, \quad
e^{-\E_i}(\overrightarrow{\H}e^{\E_i} - e^{\E_i} \overleftarrow{\HH}) = 0.
$$
Moreover, $\E_i \diamond \J= \id_{\W}$ for $i=0,1$.

\begin{defn} \label{def:Weylhomotopy}
We define a {\em homotopy between two morphisms} $\F_0$ and $\F_1$
from $(\W^+,\H^+)$ to $(\W^-,\H^-)$ as a morphism $\G$ from
$(\W^+,\H^+)$ to $(\W^-[s,ds],\HH^-)$
which on coefficients corresponds to the inclusion $j:R \to
(R[s,ds],\dd)$, and such that 
$$
  \E_i^- \diamond \G = \F_i, \quad i=0,1.
$$ 
\end{defn}

According to our definitions, such a $\G$ is a power series
in the $p^+$ with coefficients polynomial in the $q^-$, $s$ and
$ds$. Therefore it can be written in the form
\begin{equation}\label{defbackslashG}
   \G = \F(q^-,s,p^+) + ds\, \K(q^-,s,p^+).
\end{equation}
It satisfies the equation
\begin{align*}
   0 =\, & e^{-\G}\left(\dd e^\G + \overrightarrow{\HH^-} e^{\G} -
   e^{\G}\overleftarrow{\H^+}\right) \\
   =\, & e^{-\F(s)}(1-ds\,\K(s))\left(ds
   \frac {\p}{\p s} e^{\F(s)} +
   \overrightarrow{\HH^-}(e^{\F(s)}(1+ds\, \K(s))) \right.\\ 
   &\left.\phantom{e^{-\F(s)}(1-ds\K(s))} - (e^{\F(s)}(1+ds\,
   \K(s)))\overleftarrow{\H^+})\right), 
\end{align*}
which splits into the two equations
\begin{eqnarray}
0 &=& e^{-\F(s)}\left(\overrightarrow{\HH^-} e^{\F(s)} -
e^{\F(s)}\overleftarrow{\H^+}\right) \quad \text{\rm and}\notag\\
0 &=& e^{-\F(s)}\left(\frac {\p}{\p s}e^{\F(s)}  -
\overrightarrow{\HH^-}(\K(s)e^{\F(s)}) -
(e^{\F(s)}\K(s))\overleftarrow{\H^+}\right) \notag\\ 
& & - e^{-\F(s)}\K(s)\left(\overrightarrow{\HH^-} e^{\F(s)} -
e^{\F(s)}\overleftarrow{\H^+}\right) \notag\\
&= &e^{-\F(s)}\left(\frac {\p}{\p s}e^{\F(s)}  - \overrightarrow{[\HH^-,\K(s)]}
e^{\F(s)} - e^{\F(s)}\overleftarrow{[\K(s),\H^+]}\right) .\notag
\end{eqnarray}

Note that the second equation and the fact that the
$\H^\pm$ are Hamiltonians imply that 
$$
\frac {\p}{\p s}\left(e^{-\F(s)} \left(\overrightarrow{\HH^-} e^{\F(s)} -
e^{\F(s)}\overleftarrow{\H^+}\right)\right) = 0.
$$
So together with the inital condition that $\F(0)$ is a morphism it
implies the first equation. 

Summarizing the above discussion (and recalling $\HH^-=\H^-$), we see:

\begin{lem}\label{homotWeyl}
Two Weyl algebra morphisms $\F_0,\F_1:(\W^+,\H^+) \to (\W^-,\H^-)$ are
homotopic in the sense of Definition~\ref{def:Weylhomotopy} if 
and only if there exists 
$$
   \G = \F(q^-,s,p^+) + ds\, \K(q^-,s,p^+)
$$
such that
$$
\F(q^-,0,p^+) = \F_0(q^-,p^+), \qquad 
\F(q^-,1,p^+) = \F_1(q^-,p^+)
$$
and 
\begin{equation}\label{eq:BVhomot}
0 =
\frac {\p}{\p s}e^{\F(s)}  - \overrightarrow{[\H^-,\K(s)]}
e^{\F(s)} - e^{\F(s)}\overleftarrow{[\K(s),\H^+]}.
\end{equation}
\end{lem}

\begin{rem}\label{SFThomot}
In SFT, one works in the slightly more general context of not
necessarily augmented morphisms, and equation \eqref{eq:BVhomot} is
taken as the definition of homotopy between morphisms of Weyl
algebras, cf. \cite[p.~629]{EGH00}.
\end{rem}

Now we have the following:

\begin{prop}\label{prop:IBLWeylehomot}
Consider two differential Weyl algebras $(\W^+,\H^+)$ and $(\W^-,\H^-)$, 
and denote by $(C^+,\{\fp_{k,\ell,g}^{+}\})$ and
$(C^-,\{\fp_{k,\ell,g}^{-}\})$ 
the corresponding $\IBL_\infty$-algebras, respectively.
Let $\F_0,\F_1:(\W^+,\H^+) \to (\W^-,\H^-)$ be 
Weyl algebra morphisms and denote by
$\frak f^{(0)} = \{\frak f^{(0)}_{k,\ell,g}\}$ and $\frak f^{(1)} =
\{\frak f^{(1)}_{k,\ell,g}\}$ the corresponding morphisms
$(C^+,\{\fp_{k,\ell,g}^{+}\}) \to (C^-,\{\fp_{k,\ell,g}^{-}\})$ of
$\IBL_\infty$-algebras, respectively. 
\par
Then $\F_0$ is homotopic to $\F_1$ in the sense of Definition \ref{def:Weylhomotopy} if and only if 
$\frak f^{(0)}$ is homotopic to $\frak f^{(1)}$  in the sense of Definition \ref{def:homotopy}.
\end{prop} 

The proof uses the following lemma.

\begin{lem}\label{lem:IBLWeylehomot}
Let $(C,\{\fp_{k,\ell,g}\})$ be an $\IBL_\infty$-algebra over $R$, and
let $(C[s,ds], \{\fp_{k,\ell,g}\})$ be the corresponding
$\IBL_\infty$-algebra over $(R[s,ds],\dd)$. 
Let $(\fC, \{\fq_{k,\ell,g}\},\iota,\eps_i)$ be a path object for
$(C,\{\fp_{k,\ell,g}\})$. 
Then there exists a morphism
$$
\fa=\{\fa_{k,\ell,g}\}: (\fC,\{\fq_{k,\ell,g}\})  \to (C[s,ds], \{\fp_{k,\ell,g}\}),
$$ 
corresponding to $j:R \to R[s,ds]$ on coefficients,  
which makes the following diagram commute:
$$
\begin{CD}
C@>\iota>> \fC @>\eps_i>> C \\
@V{\text{\rm id}}VV       @V{\fa}VV      @V{\text{\rm id}}VV\\
C@>>{j}>  C [s,ds] @>>{e_i}> C
\end{CD}
$$
Similarly, there is a morphism
$$
\fb=\{\fb_{k,\ell,g}\}: (C[s,ds], \{\fp_{k,\ell,g}\}) \to (\fC,\{\fq_{k,\ell,g}\}) ,
$$
corresponding to $e_0:R[s,ds] \to R$ on coefficients,
which makes the following diagram commute:
$$
\begin{CD}
C@>>{j}>  C [s,ds] @>>{e_i}> C\\
@V{\text{\rm id}}VV       @V{\fb}VV      @V{\text{\rm id}}VV\\
C@>\iota>> \fC @>\eps_i>> C 
\end{CD}
$$
\end{lem}

\begin{proof}
$C[s,ds]$ is {\it not} a path object for $C$ in the sense of
Definition~\ref{def:path} (they are even defined over different
rings), but still it is true that  
\begin{enumerate}[\rm (a)]
\item $j$, $e_0$ and $e_1$ are linear morphisms (and we
   denote their $(1,1,0)$ parts by the same letters);
\item $e_i \circ j = \id_C$ and $j\circ e_i\sim \id_{C[s,ds]}$;
\item $j:C \to C[s,ds]$ and $e_i:C[s,ds] \to C$ are
   homotopy equivalences (of chain complexes over $R$ with
   differentials $\fp_{1,1,0}$ and $\dd+\fp_{1,1,0}$, respectively);
\item the map $e_0 \oplus e_1:C[s,ds] \to C \oplus C$ admits
the linear right inverse $(c_0,c_1)\mapsto c_0+(c_1-c_0)s$.
\end{enumerate}
So while the lemma is {\it not} a particular case of  Proposition
\ref{prop:homotopy-transfer}, the proof there can be adapted to the
present situation. 
\end{proof}

\begin{proof}[Proof of Proposition~\ref{prop:IBLWeylehomot}]
If $\G = \F(q^-,s,p^+) + ds\, \K(q^-,s,p^+)$ is a homotopy 
between $\F_0$ and $\F_1$ in the sense of
Definition~\ref{def:Weylhomotopy}, and $\mathfrak G$ is the
corresponding $\IBL_\infty$-morphism from $C^+$ to $C^-[s,ds]$, then
the composition $\mathfrak H:=\fb \diamond \mathfrak G$ is the
required morphism from $C^+$ to $\fC^-$ with $\eps_i\diamond \mathfrak
H=\ff^{(i)}$. 

Similarly, if $\mathfrak H$ is a homotopy in the sense of
Definition~\ref{def:homotopy}, then $\mathfrak G:=\fa \diamond
\mathfrak H$ is an $\IBL_\infty$-morphism 
from $C^+$ to $C^-[s,ds]$ whose Weyl algebra translation satisfies
Definition~\ref{def:Weylhomotopy}. 
\end{proof}
\par\bigskip
\par\medskip

\section{Filtered $\IBL_\infty$-structures}\label{sec:filter}

For many applications the notion of an $\IBL_\infty$-structure needs to be
generalized to that of a filtered $\IBL_\infty$-structure. In this
section we define this generalization and extend our previous results
to this case. This refinement is also necessary for the discussion of
Maurer-Cartan elements in Section~\ref{sec:MC}. 

{\bf Filtrations. }
A {\em filtration} on a commutative ring $R$ is a family of
subrings
$$
   R \supset R_\lambda \supset R_\mu ,\quad
   \lambda\leq\mu, \qquad \bigcup_{\lambda\in\R}R_\lambda = R,\qquad
   R_\lambda\cdot R_\mu\subset R_{\lambda+\mu}. 
$$
Such a filtration is equivalent to a {\em valuation}\footnote{
Commonly, a valuation is required to satisfy $\|r\|\geq 0$ and
$\|rr'\|=\|r\|+\|r'\|$, but we will not need these stronger conditions.} 
$\|\ \|:R\to\R\cup\{\infty\}$ satisfying $\|0\|=\infty$ and 
$$
   \|r+r'\| \geq \min\{\|r\|,\|r'\|\},\qquad \|rr'\|\geq \|r\|+\|r'\|
$$
for $r,r'\in R$. The two notions are related by
$$
   \|r\| = \sup\{\lambda\mid r\in R_\lambda \},\qquad R_\lambda
   =\{r\;\bigl|\; \|r\|\geq\lambda\}.
$$
The {\em trivial filtration} on a ring $R$ is defined by 
the trivial valuation $\|r\|=0$ for all $r\neq 0$. 
%
Another example of a filtered ring is the universal Novikov ring
considered later in this section. 

A {\em filtration} on an $R$-module $C$ is a family of $R$-linear
subspaces  
$$
   C \supset \FF^\lambda C\supset \FF^\mu C,\quad
   \lambda\leq\mu, \qquad \bigcup_{\lambda\in\R}\FF^\lambda = C,\qquad
   R_\lambda\cdot\FF^\mu C\subset\FF^{\lambda+\mu}C.  
$$
Again, this is equivalent to a {\em valuation} $\|\
\|:C\to\R\cup\{\infty\}$\footnote{
The filtration degree $\|\ \|$ should not be confused with the grading
$|\ \ |$ on $C$, which plays no role in this section.}  
satisfying 
$$
   \|c+c'\| \geq \min\{\|c\|,\|c'\|\},\qquad \|rc\|\geq\|r\|+\|c\|
$$
for $c,c'\in C$ and $r\in R$, where the two notions are related by
$$
   \|c\| = \sup\{\lambda\mid c\in\FF^\lambda C\},\qquad \FF^\lambda
   C=\{c\;\bigl|\; \|c\|\geq\lambda\}.
$$
If $C$ is an $R$-algebra, we require in addition that $\FF^\lambda
C\cdot\FF^\mu C\subset\FF^{\lambda+\mu}C$, or equivalently, $\|c\cdot
c'\|\geq\|c\|+\|c'\|$. 

The {\em completion} of $C$ with respect to $\FF$ is the completion
with respect to the metric $d(c,c'):=e^{-\|c-c'\|}$, i.e., the $R$-module 
$$\aligned
  \wh C 
   &:= \left\{\sum_{i=1}^\infty c_i\,\left|\, c_i\in C,\ \lim_{i\to\infty}\|c_i\|=\infty\right\} \right.\cr
   &= \left\{\sum_{i=1}^\infty c_i\,\left|\, c_i\in C,\ \#\{i\mid
   c_i\notin\FF^\lambda C\}<\infty \text{ for all }\lambda\in\R\right\}\right.. 
\endaligned$$
Note that $\wh C$ inherits a filtration from $C$. 

For example, the completion of a direct sum 
$
   C = \bigoplus_{k\geq 0} C^k
$
with respect to the filtration 
$
   \FF^\lambda C :=
   \bigoplus_{k\geq\lambda}C^k
$
is the direct product 
$
   \wh C = \prod_{k\geq 0} C^k
$
with the induced filtration 
$$
   \FF^\lambda\wh C = \{(c_k) \in \wh C \;\bigl|\; c_k=0 \text{ 
    for all }k<\lambda\}.  
$$

A linear map $f:C\to C'$ between filtered $R$-modules is called {\em
  filtered} if it satisfies
$$
   f(\FF^\lambda C)\subset \FF^{\lambda+K}C'
$$
for some constant $K\in\R$. In this case we call the largest such
constant the {\em (filtration) degree} of $f$ and denote it by
$\|f\|$. 

Filtrations on $C$ and $C'$ induce filtrations on the direct sum and
tensor product by
\begin{align*}
   \FF^\lambda(C\oplus C') &:= \FF^\lambda C\oplus\FF^\lambda C', \cr
   \FF^\lambda(C\otimes C') &:=
   \bigoplus_{\lambda_1+\lambda_2=\lambda}\FF^{\lambda_1}C
   \otimes\FF^{\lambda_2}C'.  
\end{align*}
Given several filtrations $\FF^\lambda_jC$ on $C$, we denote by $\wh C$
the completion with respect to the filtration
$\FF^\lambda C:=\bigcup_j\FF^\lambda_jC$. 

A filtration on $C$ induces filtrations on the symmetric products
$$
   \FF^\lambda(C[1]\otimes_R\cdots \otimes_R C[1])/\sim \ :=
   \bigoplus_{\lambda_1+\dots+\lambda_k=\lambda} (\FF^{\lambda_1}C[1]
   \otimes_R \cdots\otimes_R \FF^{\lambda_k}C[1])/\sim
$$
We denote by $\wh E_kC$ the completion of the
$k$-fold symmetric product $(C[1]\otimes_R\cdots\otimes_R C[1])/\sim$
with respect to this filtration. 
We now also include the case $\wh E_0C:= R$ (where $R$ is assumed
complete). 

Note that the symmetric algebra $\bigoplus_{k\geq 0}\wh E_kC$ has two
filtrations: the one induced by $\FF$, and the filtration by the sets
$\bigoplus_{k\geq\lambda}\wh E_kC$. We denote by $\wh EC$ the completion of
$\bigoplus_{k\geq 0}\wh E_kC$ with respect to these two filtrations. Thus
elements in $\wh EC$ are infinite sums $\sum_{i=1}^\infty c_i$ such that
$c_i\in \FF^{\lambda_i}E_{k_i}C$ with
$$
   \lim_{i\to\infty}\max\{k_i,\lambda_i\}=\infty.
$$ 
Given a map $p:\wh E_kC\to \wh E_\ell
C$ of finite filtration degree $\|p\|$, formula~\eqref{eq:hat} extends to
the completion to define a map $\wh p:\wh EC\to \wh EC$. 

For the remainder of this section, $R$ will denote a complete filtered
commutative ring, and $C$ a filtered $R$-module. 

{\bf Filtered $\IBL_\infty$-algebras. }
Consider now a collection of maps 
$$
  \fp_{k,\ell,g}: \wh E_kC\to \wh E_\ell C,\qquad k,\ell,g\geq 0
$$
of finite filtration degrees satisfying
\begin{equation}\label{eq:p-deg}
   \|\fp_{k,\ell,g}\|\geq \gamma\,\chi_{k,\ell,g} \text{ for all
   }k,\ell,g.
\end{equation}
Here $\gamma\geq 0$ is a fixed constant and 
$$
   \chi_{k,\ell,g} := 2-2g-k-\ell
$$
is the Euler characteristic of a Riemann surface of genus $g$ with $k$
positive and $\ell$ negative boundary components. 
Note that, in contrast to the unfiltered case, we allow $k=0$ and
$\ell=0$. 
In Section~\ref{sec:Cyclic} we will use $\gamma=2$. 

Define
$$
    \wh\fp := \sum_{k,\ell,g=0}^\infty
    \wh\fp_{k,\ell,g} \hbar^{k+g-1}:
    \wh EC\{\hbar\}\to \frac{1}{\hbar}\wh EC\{\hbar\},
$$
where $\wh EC\{\hbar\}$ denotes the space of power series in $\hbar$ with
coefficients in $\wh EC$. To see that $\wh\fp$ is well-defined, note that
elements of $\wh EC\{\hbar\}$ are given by $c=\sum_{\ell',g'\geq
  0}c_{\ell',g'}\hbar^{g'}$ with $c_{\ell',g'}\in \wh E_{\ell'}C$. Then
$$
   \wh\fp(c) = \sum_{\ell''\geq 0,\;g''\geq -1}c_{\ell'',g''}\hbar^{g''}
$$
with
$$
   c_{\ell'',g''} = \sum_{{k+g+g'-1=g''} \atop
     {\ell+\ell'-k=\ell''}}\wh\fp_{k,\ell,g}(c_{\ell',g'}) \in
     \wh E_{\ell''}C. 
$$
We need to show that for $\ell''$, $g''$ and $\|c_{\ell'',g''}\|$ bounded
from above only finitely many terms can appear on the right hand
side. Since $k,g,g'\geq 0$, the relation $k+g+g'-1=g''$ bounds $k,g,g'$
in terms of $g''$. Then the relation $\ell+\ell'-k=\ell''$ bounds
$\ell,\ell'\geq 0$ in terms of $g'',\ell''$. In particular,
$\chi_{k,\ell,g}$ is bounded. So
$$
   \|\wh\fp_{k,\ell,g}(c_{\ell',g'})\| \geq
   \|\fp_{k,\ell,g}\|+\|c_{\ell',g'}\| \geq \gamma\chi_{k,\ell,g} +
   \|c_{\ell',g'}\|
$$ 
bounds $\|c_{\ell',g'}\|$ from above in terms of $\ell''$, $g''$ and
$\|c_{\ell'',g''}\|$, so by convergence of $c$ only finitely many such
terms appear. 

\begin{defn}\label{def:IBL-fil}
A {\em filtered $\IBL_\infty$-structure of bidegree
$(d,\gamma)$} on a filtered graded $R$-module $C$ is a collection of
maps  
$$
  \fp_{k,\ell,g}: \wh E_kC\to \wh E_\ell C,\qquad k,\ell,g\geq 0
$$
of grading degrees $-2d(k+g-1)-1$ and filtration degrees
satisfying~\eqref{eq:p-deg},
where the inequality is strict for the following triples $(k,\ell,g)$:
\begin{equation}\label{eq:f-deg2}
   (0,0,0),\ (1,0,0),\ (0,1,0),\ (0,0,1),\ (2,0,0),\ (0,2,0),
\end{equation}
such that 
\begin{equation*}
   \wh{\frak p}\circ  \wh{\frak p} = 0. 
\end{equation*}
A filtered $\IBL_\infty$-structure is called {\em strict} if
$\fp_{k,\ell,g}=0$ unless $k,\ell\geq 1$. 
\end{defn}
We remark that the coefficient of $\hbar^{-2}$ of the
$\wh E_0C \to \wh E_0C$ component of $\wh{\frak p}\circ  \wh{\frak p}$ is 
$\fp_{0,0,0}^2$. So $\fp_{0,0,0}= 0$ automatically.
Inductively, it follows that $\fp_{0,0,g}=0$ for all $g\ge 0$.

\begin{Remark}
We could absorb the constant $\gamma$ by giving the variable $\hbar$
filtration degree 
$$
   \|\hbar\|_{\rm new} := 2\gamma
$$
and shifting the filtration on $C$ by 
$$
   \|c\|_{\rm new} := \|c\|+\gamma. 
$$
Then we obtain
\begin{align*}
   \|\fp_{k,\ell,g}(x_1\cdots x_k)\hbar^{k+g-1}\|_{\rm new} 
   &= \|\fp_{k,\ell,g}(x_1\cdots x_k)\| + \gamma(2k+2g-2+\ell) \cr 
   &\geq \|x_1\cdots x_k\| + \gamma(2-2g-k-\ell) + \gamma(2k+2g-2+\ell)
   \cr 
   &= \|x_1\cdots x_k\|_{\rm new},
\end{align*}
hence $\wh\fp:\wh EC\{\hbar\}\to \frac{1}{\hbar}\wh EC\{\hbar\}$ has
filtration degree $\|\wh\fp\|_{\rm new}\geq 0$. In this paper we will not use
this filtration, but rather carry out inductional proofs explicitly by
considering connected surfaces as in the discussion preceding
Definition~\ref{def:IBL-fil}.  
\end{Remark}

{\bf Filtered $\IBL_\infty$-morphisms. }
Consider two filtered $R$-modules $(C^\pm,\fp^\pm)$ and a collection
of maps  
$$
   \ff_{k,\ell,g}:\wh E_kC^+\to
   \wh E_\ell C^-, \qquad k,\ell,g\geq 0
$$
of finite filtration degrees satisfying
\begin{equation}\label{eq:f-deg1}
   \|\ff_{k,\ell,g}\|\geq \gamma\,\chi_{k,\ell,g} \text{ for all }k,\ell,g, 
\end{equation}
where the inequality is strict for the triples in~\eqref{eq:f-deg2}.
Here $\gamma\geq 0$ is the fixed constant from above, and again
we allow $k=0$ and $\ell=0$. 

\begin{lem}\label{lem:filt1}
Given $(C^\pm,\fp^\pm)$ and $\ff_{k,\ell,g}$ as above, there exist
unique collections of maps $\fq^\pm_{k,\ell,g}:\wh E_kC^+\to \wh E_\ell C^-$
satisfying~\eqref{eq:p-deg} such that
$$
   e^\ff\wh\fp^+=\wh\fq^+,\
   \wh\fp^-e^\ff=\wh\fq^-:\wh EC^+\{\hbar\}\to \wh
   EC^-\{\hbar,\hbar^{-1}\}.   
$$
where $\wh\fq^{\pm}$ is defined by
\begin{equation*}
\aligned
    \wh\fq^{\pm} := \sum_{r=1}^\infty\sum_{{k_i,\ell_i,g_i}\atop {1\leq i\leq
    r}} \frac{1}{(r-1)!}&f_{k_1,\ell_1,g_1}\odot\cdots\odot f_{k_{r-1},\ell_{r-1},g_{r-1}}\odot
    \frak q^{\pm}_{k_r,\ell_r,g_r} \hbar^{\sum k_i+\sum g_i-r}. 
\endaligned
\end{equation*}
\end{lem}

\begin{proof}
Let us consider the composition $\wh\fp^-e^\ff=\wh\fq^-$, the other
one being analogous. According to Lemma~\ref{lem:mor_order}, the map
$\fq^-_{k,\ell,g}$ is given by the sum 
\begin{equation}\label{eq:filt1}
    \sum_{r\geq 0}
    \sum_{ {{ {k_1+\dots+k_r=k}
    \atop {\ell_1+\dots+\ell_r+\ell^--k^-=\ell} }
    \atop {g_1+\dots+g_r+g^-+k^--r=g} }
    \atop {s_1+\dots +s_r=k^- \atop s_i \geq 1}} \hspace{-.7cm}
    \frac{1}{r!}\hat\fp^-_{k^-,\ell^-,g^-}\circ_{s_1,\dots,s_r}
    (\ff_{k_1,\ell_1,g_1}\odot\cdots\odot \ff_{k_r,\ell_r,g_r})
\end{equation}
corresponding to complete gluings of $r$ connected surfaces of
signatures $(k_i,\ell_i,g_i)$ at their outgoing ends to the ingoing
ends of a connected surface of signature $(k^-,\ell^-,g^-)$, plus an
appropriate number of trivial cylinders, to obtain a {\em connected}
surface of signature $(k,\ell,g)$. In particular, for each term in
this sum the Euler characteristics satisfy 
$$
   \chi_{k,\ell,g} = \chi_{k^-,\ell^-,g^-} +
   \sum_{i=1}^r\chi_{k_i,\ell_i,g_i}.  
$$
Let us write $\{1,\dots,r\}$ as the disjoint union $I\cup J\cup K$,
where $i\in I$ iff $(k_i,\ell_i,g_i)$ is one of the triples
in~\eqref{eq:f-deg2}, $i\in J$ iff $\chi_{k_i,\ell_i,g_i}<0$, and
$i\in K$ iff $(k_i,\ell_i,g_i)=(1,1,0)$. Let $\delta>0$ be such that
$\|\ff_{k_i,\ell_i,g_i}\|-\gamma\,\chi_{k_i,\ell_i,g_i}\geq\delta$ for
all $i\in I$. Then the filtration
conditions on $\fp^-$ and $\ff$ imply
\begin{align}\label{eq:filt2}
    &\|\fq^-_{k,\ell,g}\| - \gamma\,\chi_{k,\ell,g} \cr
   &\geq (\|\fp^-_{k^-,\ell^-,g^-}\| - \gamma\,\chi_{k^-,\ell^-,g^-}) +
   \sum_{i=1}^r (\|\ff_{k_i,\ell_i,g_i}\| - \gamma\,\chi_{k_i,\ell_i,g_i}) \cr
   &\geq \sum_{i\in I} (\|\ff_{k_i,\ell_i,g_i}\| - \gamma\,\chi_{k_i,\ell_i,g_i}) \cr
   &\geq \delta|I| \geq 0.
\end{align} 
This shows that $|I|$ is uniformly bounded, where we say that a
quantity is {\em uniformly bounded} if it is bounded from above 
in terms of $k,\ell,g$ and $\|\fq^-_{k,\ell,g}\|$.  It follows that  
$$
   \sum_{j\in J}-\chi_{k_j,\ell_j,g_j} = -\chi_{k,\ell,g} +
   \chi_{k^-,\ell^-,g^-} + \sum_{i\in I}\chi_{k_i,\ell_i,g_i}
$$
is uniformly bounded. Since each term $2g_j+k_j+\ell_j-2$ on the
left-hand side is $\geq 1$, this provides uniform bounds on $|J|$ as 
well as all the $g_j,k_j,\ell_j$ for $j\in J$. Finally, the fact that
each $i\in K$ contributes $1$ to the sum $k_1+\dots+k_r=k$ yields a
uniform bound on $|K|$. Hence the number of terms in the sum
in~\eqref{eq:filt1} is uniformly bounded, which proves convergence of  
$\fq^-_{k,\ell,g}$ with respect to the
filtration. Inequality~\eqref{eq:filt2} shows that $\fq^-_{k,\ell,g}$
satisfies~\eqref{eq:p-deg}.  
\end{proof}

In view of the preceding lemma, the following definition makes sense.  

\begin{defn}\label{def:morphifil}
A {\em filtered $\IBL_\infty$-morphism} between filtered
$\IBL_\infty$-algebras $(C^\pm,\{\fp^\pm_{k,\ell,g}\})$ is a collection
of maps  
$$
   \ff_{k,\ell,g}:\wh E_kC^+\to
   \wh E_\ell C^-, \qquad k,\ell,g\geq 0
$$
of grading degrees $-2d(k+g-1)$ and filtration degrees
satisfying~\eqref{eq:f-deg1} and~\eqref{eq:f-deg2} such that   
\begin{equation}\label{eq:mor3}
   e^\ff\wh p^+-\wh p^- e^\ff = 0.
\end{equation}
A filtered $\IBL_\infty$-morphism $\ff$ is called {\em strict} if
$\ff_{k,\ell,g}=0$ unless $k,\ell\geq 1$.  
\end{defn}

Note that for a strict filtered $\IBL_\infty$-morphism or
structure, condition~\eqref{eq:f-deg2} is vacuous. 
\medskip

{\bf Composition of filtered $\IBL_\infty$-morphisms. }
Consider now two filtered $\IBL_\infty$-morphisms 
\begin{align*}
    \ff^+=\{\ff^+_{k,\ell,g}\}:(C^+,\{\fp^+_{k,\ell,g}\})\to
    (C,\{\fp_{k,\ell,g}\}),\\
    \ff^-=\{\ff^-_{k,\ell,g}\}:(C,\{\fp_{k,\ell,g}\})\to
    (C^-,\{\fp^-_{k,\ell,g}\}). 
\end{align*}

\begin{lem}\label{lem:filtered-comp}
There exists a unique filtered $\IBL_\infty$-morphism
$\ff=\{\ff_{k,\ell,g}\}:(C^+,\{\fp^+_{k,\ell,g}\})\to
(C^-,\{\fp^-_{k,\ell,g}\})$ satisfying 
\begin{equation*}
e^\ff = e^{\ff^-}e^{\ff^+}.
\end{equation*}
\end{lem}

We call $\ff$ the {\em composition} of $\ff^+$ and $\ff^-$. 

\begin{proof}
According to the discussion following Definition~\ref{def:comp}, 
the map $\ff_{k,\ell,g}$ is given by the sum 
\begin{align}\label{eq:filt3}
   \hspace{-1.7cm}
    \sum_{ { { {k_1^++\dots+k_{r^+}^+=k}
    \atop {\ell_1^-+\dots+\ell_{r^-}^-=\ell} }
    \atop {\ell_1^++\dots+\ell_{r^+}^+=k_1^-+\dots+k_{r^-}^-} }
    \atop {\sum g_i^+ + \sum g_i^- + \sum \ell_i^++ - r^+-r^-+1 = g} }
    \hspace{-1.8cm}
    \frac{1}{r^+!r^-!}
    (\ff^-_{k_1^-,\ell_1^-,g_1^-}\odot\cdots\odot
    \ff^-_{k_{r^-}^-,\ell_{r^-}^-,g_{r^-}^-})\circ
    (\ff^+_{k_1^+,\ell_1^+,g_1^+}\odot\cdots\odot
    \ff^+_{k_{r^+}^+,\ell_{r^+}^+,g_{r^+}^+}). 
\end{align}
corresponding to complete gluings of $r^+$ connected surfaces of
signatures $(k_i^+,\ell_i^+,g_i^+)$ at their outgoing ends to the
ingoing ends of $r^-$ connected surfaces of signatures
$(k_i^-,\ell_i^-,g_i^-)$ to obtain a {\em connected} surface of
signature $(k,\ell,g)$. In particular, for each term in this sum the
Euler characteristics satisfy 
$$
   \chi_{k,\ell,g} = \sum_{i=1}^{r^+}\chi_{k_i^+,\ell_i^+,g_i^+} +
   \sum_{i=1}^{r^-}\chi_{k_i^-,\ell_i^-,g_i^-}.   
$$
Let us write $\{1,\dots,r^\pm\}$ as the disjoint union $I^\pm\cup
J^\pm\cup K^\pm$, where $i\in I^\pm$ iff
$(k_i^\pm,\ell_i^\pm,g_i^\pm)$ is one of the triples 
in~\eqref{eq:f-deg2}, $i\in J^\pm$ iff
$\chi_{k_i^\pm,\ell_i^\pm,g_i^\pm}<0$, and $i\in K^\pm$ iff
$(k_i^\pm,\ell_i^\pm,g_i^\pm)=(1,1,0)$. Let $\delta>0$ be such that 
$\|\ff^\pm_{k_i^\pm,\ell_i^\pm,g_i^\pm}\|-\gamma\,\chi_{k_i^\pm,\ell_i^\pm,g_i^\pm}\geq\delta$
for all $i\in I^\pm$. Then the filtration
conditions on $\ff^\pm$ imply
\begin{align}\label{eq:filt4}
    &\|\ff_{k,\ell,g}\| - \gamma\,\chi_{k,\ell,g} \cr
   &\geq  \sum_{i=1}^{r^+} (\|\ff^+_{k_i^+,\ell_i^+,g_i^+}\| -
   \gamma\,\chi_{k_i^+,\ell_i^+,g_i^+}) +
   \sum_{i=1}^{r^-} (\|\ff^-_{k_i^-,\ell_i^-,g_i^-}\| -
   \gamma\,\chi_{k_i^-,\ell_i^-,g_i^-}) \cr 
   &\geq  \sum_{i\in I^+} (\|\ff^+_{k_i^+,\ell_i^+,g_i^+}\| -
   \gamma\,\chi_{k_i^+,\ell_i^+,g_i^+}) +
   \sum_{i\in I^-} (\|\ff^-_{k_i^-,\ell_i^-,g_i^-}\| -
   \gamma\,\chi_{k_i^-,\ell_i^-,g_i^-}) \cr 
   &\geq \delta (|I^+| + |I^-|) \geq 0.
\end{align} 
This shows that $|I^+|$ and $|I^-|$ are uniformly bounded,
i.e.~bounded from above  
in terms of $k,\ell,g$ and $\|\ff_{k,\ell,g}\|$. It follows that  
$$
   \sum_{j\in J^+}-\chi_{k_j^+,\ell_j^+,g_j^+} + \sum_{j\in
     J^-}-\chi_{k_j^-,\ell_j^-,g_j^-} = -\chi_{k,\ell,g} +
   \sum_{i\in I^+}\chi_{k_i^+,\ell_i^+,g_i^+} + \sum_{i\in
     I^-}\chi_{k_i^-,\ell_i^-,g_i^-}  
$$
is uniformly bounded. Since each term $2g_j^\pm+k_j^\pm+\ell_j^\pm-2$
on the left-hand side is $\geq 1$, this provides uniform bounds on
$|J^+|$ and $|J^-|$ as well as all the $g_j^\pm,k_j^\pm,\ell_j^\pm$
for $j\in J^\pm$. Finally, the fact that 
each $i\in K^+$ contributes $1$ to the sum $k_1^++\dots+k_{r^+}^+=k$ and
each $i\in K^-$ contributes $1$ to the sum
$\ell_1^-+\dots+\ell_{r^-}^-=\ell$ yields uniform bounds on $|K^+|$
and $|K^-|$. Hence the number of terms in the sum
in~\eqref{eq:filt3} is uniformly bounded, which proves convergence of  
$\fq^-_{k,\ell,g}$ with respect to the
filtration. Inequality~\eqref{eq:filt4} shows that $\ff_{k,\ell,g}$
satisfies~\eqref{eq:f-deg1}, where the inequality is strict if $I^+$
or $I^-$ is nonempty. If $I^+$ and $I^-$ are both empty, then either
$\chi_{k,\ell,g}<0$ (if $J^+$ or $J^-$ are nonempty), or (if $J^+$ and
$J^-$ are both empty) $(k_i^\pm,\ell_i^\pm,g_i^\pm)=(1,1,0)$ for all
$i$ and hence $(k,\ell,g)=(1,1,0)$ (the corresponding connected
surface is a gluing of cylinders and hence a cylinder). This shows
that $\ff_{k,\ell,g}$ also satisfies~\eqref{eq:f-deg2}. 
\end{proof}

{\bf Homotopies of filtered $\IBL_\infty$-algebras. }
Consider a subset $G$ of $\R_{\ge 0}$ such that
\begin{enumerate}
\item $g_1,g_2 \in G$ implies $g_1+g_2\in G$;
\item $0 \in G$;
\item $G$ is a discrete subset of $\R$.
\end{enumerate}
We call such $G$ a {\em discrete submonoid}, and we will write it as 
\begin{equation}\label{Gname}
G = \{\lambda_0,\lambda_1,\dots\},
\end{equation}
where $\lambda_j < \lambda_{j+1}$ and $\lambda_0 =0$.

\begin{defn}\label{def:ggapped}
We say that filtered $\IBL_\infty$-algebra of bidegree
$(d,\gamma)$ is $G$-{\em gapped} if the operations $\fp_{k,\ell,g}$
can be written as 
$$
   \fp_{k,\ell,g} = \sum_{j=0}^\infty \fp_{k,\ell,g}^j:\wh E_kC\to\wh
   E_\ell C,
$$
where the filtration degrees of $\fp_{k,\ell,g}^j$ satisfy
$$
   \|\fp_{k,\ell,g}^j\|-\gamma\chi_{k,\ell,g}\geq\lambda_j\in G,
$$
and $\fp_{k,\ell,g}^0=0$ for the triples $(k,\ell,g)$
in~\eqref{eq:f-deg2}. We call an $\IBL_\infty$ algebra {\em gapped} if it 
is $G$-gapped for some discrete submonoid $G \subset \R_{\ge 0}$.
\end{defn}

We define a linear ordering on {\em extended signatures}
$(j,k,\ell,g)\in\N_0^4$ by
saying $(j',k',\ell',g') \prec (j,k,\ell,g)$ if either $j'<j$, or
$j'=j$ and $(k',\ell',g') \prec (k,\ell,g)$ in the sense of
Definition~\ref{def:order}. 
\begin{rem}
As with the original ordering in Definition~\ref{def:order}, this is only 
one of several possible choices. 
\end{rem}
Now we have the following analogue of Lemma~\ref{lem:order}. 
\begin{lem}\label{lem:orderfil}
For a gapped filtered $\IBL_\infty$-algebra $(C,\fp_{k,\ell,g})$ the 
condition $\hat\fp \circ \hat\fp = 0$ is equivalent to 
$\fp_{1,1,0}^0\circ\fp_{1,1,0}^0=0$, together with the sequence of relations
$$
\hat\fp_{1,1,0}^0\circ \fp_{k,\ell,g}^j
   +\fp_{k,\ell,g}^j\circ \hat\fp_{1,1,0}^0 + P_{k,\ell,g}^j 
   + R_{k,\ell,g}^j=0
$$
as maps from $\wh E_kC$ to $\wh E_\ell C$ for all extended signatures 
$(j,k,\ell,g)\succ (0,1,1,0)$, 
where $P_{k,\ell,g}^j:\wh E_kC \to \wh E_\ell C$ involves only
compositions of terms $\fp_{k',\ell',g'}^{j'}$ whose extended signatures
satisfy $(0,1,1,0)\prec (j',k',\ell',g')\prec (j,k,\ell,g)$, and
$\|R_{k,\ell,g}^j\| > \lambda_j+\gamma\chi_{k,\ell,g}$.  

\end{lem}

\begin{proof}
Recall that the left hand side of relation \eqref{eq:BL4} is a sum of terms
$\hat\fp_{k_2,\ell_2,g_2}\circ_s\hat\fp_{k_1,\ell_1,g_1}$ which correspond to
gluings of two connected surfaces of signatures
$\sigma_i=(k_i,\ell_i,g_i)$ along $s\geq 1$ boundary components to a
connected surface of signature $\sigma=(k,\ell,g)$. We fix $j\ge 0$ and 
combine all terms in this sum of filtration degree $>\lambda_j+\gamma\chi_\sigma$ into one summand, which we denote by $R^j_{k,\ell,g}$. 
Next consider a term with
$$
  \| \hat\fp_{k_2,\ell_2,g_2}^{j_2}\circ_s\hat\fp_{k_1,\ell_1,g_1}^{j_1} \| \le 
  \lambda_j+\gamma\chi_{\sigma}
$$
Then
$$
   \lambda_{j_2}+\gamma\chi_{\sigma_2} +
   \lambda_{j_1}+\gamma\chi_{\sigma_1} \leq \|\hat\fp_{\sigma_2}^{j_2}\| +
   \|\hat\fp_{\sigma_1}^{j_1}\| \leq
   \|\hat\fp_{\sigma_2}^{j_2}\circ_s\fp_{\sigma_1}^{j_1}\| \leq
   \lambda_{j}+\gamma\chi_{\sigma}. 
$$
Since $\chi_{\sigma_2}+\chi_{\sigma_1}=\chi_{\sigma}$, this implies 
$$
   \lambda_{j_2}+\lambda_{j_1}\leq\lambda_j.
$$
%
If $j=0$, then $j_1=j_2=0$ and, by the last condition in
Definition~\eqref{def:ggapped}, $\sigma_1$ and $\sigma_2$ are none of
the triples in~\eqref{eq:f-deg2}. Thus Lemma~\ref{lem:order} implies
that either $\sigma_1=(1,1,0)$ and $\sigma_2=\sigma$, or
$\sigma_2=(1,1,0)$ and $\sigma_1=\sigma$, or
$\sigma_1,\sigma_2\prec\sigma$.  

If $j>0$, then either $j_1,j_2<j$, or $j_1=j$ and $j_2=0$, or $j_2=j$
and $j_1=0$. In the first case we are done, so consider the second
case (the third case is analogous). Then
$(j_2,\sigma_2)\prec(j,\sigma)$ and $\sigma_2$ is none of
the triples in~\eqref{eq:f-deg2}. In particular, $\chi_{\sigma_2}\leq
0$, and thus  
$$
   -\chi_{\sigma_1}\leq -\chi_{\sigma_1}-\chi_{\sigma_2} = -\chi_{\sigma}. 
$$
If $\chi_{\sigma_2}<0$, this yields
$(j_1,\sigma_1)\prec(j,\sigma)$ and we are done. If
$\chi_{\sigma_2}=0$, then $\sigma_2=(1,1,0)$ and it follows that
$(j_1,\sigma_1)=(j,\sigma)$. 
\end{proof}

For a filtered $\IBL_\infty$-algebra $(C,\{\fp_{k,\ell,g}\})$, the
composition ${\fp}_{1,1,0}\circ {\fp}_{1,1,0}$ may be nonzero due to
the presence of $\fp_{0,1,0}$ or $\fp_{1,0,0}$. If the
$\IBL_\infty$-algebra is {\em strict}, then these terms are not present
and we get a chain complex $(C,\fp_{1,1,0})$. 
We define path objects in the category of (gapped) filtered
$\IBL_{\infty}$-structures in the same way as in Definition 
\ref{def:path}, except that we require the morphisms $\iota$, $\eps_0$ 
and $\eps_1$ to have filtration degree $0$, and in condition (c) the maps 
$\fp^0_{1,1,0}$ and $\fq^0_{1,1,0}$ replace $\fp_{1,1,0}$ and $\fq_{1,1,0}$.
The proofs of the following two propositions are now
completely analogous to those of Proposition~\ref{prop:model} and
Proposion~\ref{prop:homotopy-transfer}, using 
induction over the linear order on extended signatures and
Lemma~\ref{lem:orderfil}. 

\begin{prop}\label{prop:modelfil}
For any gapped filtered $\IBL_\infty$-algebra
$(C,\{\fp_{k,\ell,g}\})$ there exists a path object $\fC$ that is
gapped. \qed
\end{prop}

\begin{prop}\label{prop:homotopy-transfer2}
Let $C$ and $D$ be gapped filtered $\IBL_\infty$-algebras, and let $\fC$ 
and $\fD$ be gapped path objects for $C$ and $D$, respectively. Let 
$\ff:C \to D$ be a gapped filtered $\IBL_\infty$-morphism. Then there exists 
a gapped  filtered $\IBL_\infty$-morphism $\fF:\fC \to \fD$ such that
the diagram
$$
\begin{CD}
C @>\iota^C>> \fC @>\eps_i^C>> C \\
@V{\frak f}VV       @V{\frak F}VV      @V{\frak f}VV\\
D @>>\iota^D> \fD @>>\eps_i^D> D 
\end{CD}
$$
commutes for both $i=0$ and $i=1$. \qed
\end{prop}

We define the notion of a homotopy between gapped filtered
$\IBL_\infty$-algebras in the same way as in
Definition~\ref{def:homotopy}. Then Proposition
\ref{prop:homotopy-equiv} and Corollary \ref{Cor:comp} can be
generalized to the gapped filtered case in the same way. 

Now
Proposition~\ref{prop:inverse2} and Theorem~\ref{thm:can} have the
following analogues in the strict filtered case. (More generally,
they hold for gapped filtered $\IBL_\infty$-algebras with
$\fp_{k,\ell,g}=0$ for the triples in~\eqref{eq:f-deg2}.)

\begin{prop}\label{prop:filt-inverse}
Let $\ff:(C,\fp)\to (D,\fq)$ be a strict gapped filtered
$\IBL_\infty$-morphism such that $\ff_{1,1,0}:(C,\fp_{1,1,0}) \to
(D,\fq_{1,1,0})$ is a chain homotopy equivalence. Then $\ff$
is a filtered $\IBL_\infty$-homotopy equivalence. \qed
\end{prop}

\begin{thm}\label{thm:canfil}
Suppose $(C,\{\fp_{k,\ell,g}\})$ is a strict gapped filtered
$\IBL_\infty$-algebra. 
Then there exist operations $\{\fq_{k,\ell,g}\}$ on its homology
$H:=H_*(C,\fp_{1,1,0})$ giving it the structure of a strict gapped
filtered $\IBL_\infty$-algebra such that there exists a
gapped homotopy equivalence $\ff: (H,\{\fq_{k,\ell,g}\}) \to
(C,\{\fp_{k,\ell,g}\})$. \qed
\end{thm}

\begin{rem}
Proposition~\ref{prop:filt-inverse} continues to hold in the nonstrict
gapped case provided that $\ff^0_{1,1,0}$ is a chain homotopy
equivalence with respect to $\fp^0_{1,1,0}$ and $\fq^0_{1,1,0}$, and
similarly for Theorem~\ref{thm:canfil}. 
\end{rem}

{\bf Filtered $\IBL_\infty$-algebras over the universal
  Novikov ring.} \\
In applications to symplectic geometry (both to SFT and Lagrangian
Floer theory), the $\IBL_\infty$-algebra that is expected to appear has 
coefficients in a Novikov ring and has a filtration by energy
(that is, the symplectic area of pseudo-holomorphic curves).
Here we explain the algebraic part of this story and 
show that various results in the previous sections have analogues in this setting.
Let $\K$ be a field of characteristic $0$ (for example $\Q$).

\begin{defn}
The {\em universal Novikov ring} $\Lambda_0$ consists of formal sums
\begin{equation}\label{novikovelement}
   a= \sum_{i=0}^{\infty} a_i T^{\lambda_i},
\end{equation}
where $a_i \in \K$, $\lambda_i \in \R_{\ge 0}$ such that $\lambda_{i+1}
> \lambda_i$ and $\lim_{i\to \infty} \lambda_i = +\infty$.
It is a commutative ring with the obvious sum and product. 
The $T$-adic valuation 
$$
   \|a\|_T := \inf \{\lambda_i \mid a_i \ne 0\}\in\R_{\geq 0}\cup\{+\infty\}
$$
turns $\Lambda_0$ into a complete filtered ring. 
$\Lambda_0$ is a local ring whose unique maximal ideal $\Lambda_+$ is
the subset of elements (\ref{novikovelement}) such that $\lambda_i >
0$ for all $i$ with $a_i \ne 0$. Note that $\Lambda_0/\Lambda_+ = \K$.
\end{defn}
\par
Let now $\overline C$ be a $\K$-vector space. 
Recall that the tensor product $\overline C \otimes_\K \Lambda_0$
consists of finite sums $\sum_{i=1}^Nx_i\otimes a_i$, where
$x_i\in\overline C$ and $a_i\in\Lambda_0$. In particular, it contains
{\it finite} sums 
$$
   \sum_{i=1}^N x_i T^{\lambda_i},
$$
where $x_i \in \overline C$ and $\lambda_i \in \R_{\ge 0}$. We denote by 
$$
   C := \overline C \,\widehat{\otimes}_{\K}\, \Lambda_0
$$
the space of possibly infinite sums
\begin{equation}\label{novikovelement2}
   x = \sum_{i=1}^{\infty} x_i T^{\lambda_i}
\end{equation}
such that $x_i \in \overline C$, $\lambda_i \in \R_{\ge 0}$, and
$\lim_{i\to \infty} \lambda_i = +\infty$. 

Hereafter in this section we shall only consider $\Lambda_0$-modules
that are obtained as $C=\overline C \,\widehat{\otimes}_\K\, \Lambda_0$
for some $\overline C$. Then $C$ has a valuation defined by
$$
   \|x\|_T := \inf \{\lambda_i \mid a_i \ne 0\},
$$
which turns $C$ into a complete filtered $\Lambda_0$-module.  
On such $C$, the notion of a {\em filtered $\IBL_\infty$-structure over
$\Lambda_0$} is now defined as above, with the ring $R$ replaced by
$\Lambda_0$ and the constant $\gamma=0$. 
Note that the operations $\fp_{k,\ell,g}$ descend to the quotient
$\overline{C}\cong C/(\Lambda_+\cdot C)$ to give 
$(\ol C,\ol\fp_{k,\ell,g})$ the structure of a generalized
$\IBL_\infty$-algebra over $\Lambda_0/\Lambda_+=\K$, which we call the
{\em reduction} of $(C,\fp_{k,\ell,g})$.

As above, consider a discrete sub-monoid $G =
\{\lambda_0,\lambda_1,\dots\}$, 
where $\lambda_j < \lambda_{j+1}$ and $\lambda_0 =0$. 
Let $\overline C_i$ ($i=1,2$) be two $\K$-vector spaces and 
$C_i =  \overline C_i\widehat{\otimes} \Lambda_0$. A
$\Lambda_0$-linear map
$$
F : C_1 \to C_2
$$
is said to be $G$-{\em gapped} if there exist $\K$-linear maps 
$F_j : \overline C_1 \to \overline C_2$ for each $\lambda_j \in G$ 
such that
$$
   F = \sum T^{\lambda_j} F_j,
$$
where we extend $F_j$ to $C_1 \to C_2$ by $\Lambda_0$-linearity.
Note that the $F_j$ are uniquely determined by $F$, and 
a filtered $\IBL_\infty$-algebra over $\Lambda_0$ is $G$-gapped
in the sense of Definition~\ref{def:ggapped} if all the operations
$\hat\fp_{k,\ell,g}$ are $G$-gapped. 

We define path objects in the category of filtered
$\IBL_{\infty}$-structures over $\Lambda_0$ in the same way as in 
Definition~\ref{def:path}. Then Propositions~\ref{prop:model}
and~\ref{prop:homotopy-transfer} have analogues in this category. 
Thus we obtain a notion of homotopy between gapped filtered
$\IBL_\infty$-algebras over $\Lambda_0$ with the same properties as in
Section \ref{sec:hom}. 
Now Proposition~\ref{prop:inverse2} and Theorem~\ref{thm:can} have the
following analogues, which improve Proposition~\ref{prop:filt-inverse}
and Theorem~\ref{thm:canfil} in this setting and are proved
analogously. 

\begin{prop}\label{prop:filt-inverse2}
Let $\ff:(C,\fp)\to (D,\fq)$ be a gapped filtered
$\IBL_\infty$-morphism between gapped filtered
$\IBL_\infty$-algebras over the universal Novikov ring
$\Lambda_0$. Suppose that its reduction $\ol\ff_{1,1,0}:(\ol
C,\ol\fp_{1,1,0}) \to (\ol D,\ol\fq_{1,1,0})$ is a chain homotopy
equivalence. Then $\ff$ is a filtered $\IBL_\infty$-homotopy
equivalence. \qed
\end{prop}

\begin{thm}
\label{thm:canfil2}
Suppose $(C,\{\fp_{k,\ell,g}\})$ is a gapped filtered
$\IBL_\infty$-algebra over the universal Novikov ring
$\Lambda_0$. Set $\overline H := H(\overline C,\overline{\fp}_{1,1,0})$ 
and $H := \overline H \widehat{\otimes}\Lambda_0$.
Then there exist operations $\{\fq_{k,\ell,g}\}$ on
$H$ giving it the structure of a gapped filtered 
$\IBL_\infty$-algebra over $\Lambda_0$ such that there exists a
gapped homotopy equivalence $\ff: (H,\{\fq_{k,\ell,g}\}) \to
(C,\{\fp_{k,\ell,g}\})$. \qed 
\end{thm}

The main example in this paper is the dual cyclic bar complex of a
cyclic DGA, which is discussed in
Sections~\ref{sec:Cyclic-cochain} and~\ref{sec:Cyclic}.

\section{Maurer-Cartan elements}\label{sec:MC}

In this section we discuss Maurer-Cartan elements and the resulting
twisted $\IBL_\infty$-structures. With the applications in the
following sections in mind, we formulate the discussion for strict
filtered $\IBL_\infty$-algebras. 
However, most statements in this section
remains true if we drop the word ``strict'' throughout. 

Let $(C,\fp=\{\fp_{k,\ell,g}\})$ be a strict filtered
$\IBL_\infty$-algebra of bidegree $(d,\gamma)$. 
Consider a collection of elements 
$$
    \fm_{\ell,g}\in \wh E_\ell C,\qquad \ell\geq 1,g\geq 0
$$
of grading degrees 
$$
   |\fm_{\ell,g}|_{grading} = -2d(g-1)
$$ 
and filtration degrees $\|\fm_{\ell,g}\|$ satisfying
\begin{equation}\label{eq:energycondforMC}
   \|\fm_{\ell,g}\|\geq\gamma\,\chi_{0,\ell,g} \text{ for all }\ell,g,
\end{equation}
where the inequality is strict for the pairs $(\ell,g)=(1,0)$ 
and
$(2,0)$. Define the grading degree zero element
$$
    \fm := \sum_{\ell\geq 1,g\geq
      0}\fm_{\ell,g}\hbar^{g-1}\in\frac{1}{\hbar}\wh EC\{\hbar\}.
$$

\begin{Definition}
$\{\fm_{\ell,g}\}_{\ell\geq 1,g\geq 0}$ is called a {\em
     Maurer-Cartan element} in $(C,\{\fp_{k,\ell,g}\})$ if
\begin{equation}\label{eq:MC}
    \hat\fp(e^{\fm}) = 0. 
\end{equation}
\end{Definition}

Here we view $\fm$ as a filtered $\IBL_\infty$-morphism from the
trivial $\IBL_\infty$-algebra ${\bf 0}$ to $(C,\fp)$ whose
$(0,\ell,g)$ term sends $1\in 
R=\wh E_0{\bf 0}$ to $\fm_{\ell,g}\in\wh E_\ell C$. Then the Maurer-Cartan 
equation~\eqref{eq:MC} is just equation~\eqref{eq:mor3} for the
corresponding filtered $\IBL_\infty$-morphism. 
In view of this observation and Lemma~\ref{lem:filt1}, the left hand
side of (\ref{eq:MC}) converges with respect to the metric induced by
the filtration. 

%
For later reference, we record the following observation.
\begin{lem}\label{lem:MC}
Suppose that $(C,\{\fp_{k,\ell,g}\})$ is a filtered $\text{\rm dIBL}$-algebra,
i.e.~its only nonvanishing terms are $d=\fp_{1,1,0}$, $\fp_{2,1,0}$
and $\fp_{1,2,0}$, and the only nonvanishing term in $\fm$ is
$\fm_{1,0}$. Then the Maurer-Cartan equation~\eqref{eq:MC} is
equivalent to 
\begin{equation}\label{eq:MC2}
    d\fm_{1,0} + \frac{1}{2}\fp_{2,1,0}(\fm_{1,0}, \fm_{1,0}) = 0,\qquad
    \fp_{1,2,0}(\fm_{1,0})=0. 
\end{equation}
\end{lem}

\begin{proof}
We compute 
\begin{align*}
   \hat\fp(e^\fm) 
   &= (\hat\fp_{1,1,0} + \hat\fp_{2,1,0}\hbar +
   \hat\fp_{1,2,0})(e^{\hbar^{-1}\fm_{1,0}})\\
   &= \left(\fp_{1,1,0}(\fm_{1,0}) + \frac 1 2
   \fp_{2,1,0}(\fm_{1,0}, \fm_{1,0})  
   + \fp_{1,2,0}(\fm_{1,0})\right) (\hbar^{-1} e^\fm),
\end{align*}
implying the equivalence.
\end{proof}

{\bf Twisted $\IBL_\infty$-structures. }
The next proposition shows that a Maurer-Cartan element gives
rise to a ``twisted'' $\IBL_\infty$-structure.  

\begin{prop}\label{prop:MC}
Let $\{\fm_{\ell,g}\}_{\ell\geq 1,g\geq 0}$ be a Maurer-Cartan element
in the strict filtered $\IBL_\infty$-algebra
$(C,\{\fp_{k,\ell,g}\})$. Then there exists a unique strict filtered
$\IBL_\infty$-structure $\{\fp^\fm_{k,\ell,g}\}_{k,\ell\geq 1,g\geq 0}$
on $C$ satisfying
\begin{equation}\label{eq:twist}
    \wh{\fp^\fm} = e^{-\fm}\hat\fp(e^{\fm}\cdot):\wh EC\{\hbar\} \to
    \wh EC\{\hbar\}. 
\end{equation}
\end{prop}

\begin{proof}
The map $\fp^\fm_{k,\ell,g}:\wh E_kC\to \wh E_\ell C$ is given by the sum 
\begin{equation}\label{eq:filt5}
    \sum_{r\geq 0}
    \sum_{ { {k^-\geq k,\;\ell^-\leq\ell}
    \atop {\ell_1+\dots+\ell_r+\ell^--k^-=\ell-k} }
    \atop {g_1+\dots+g_r+g^-+k^--k-r=g} } \hspace{-.7cm}
    \frac{1}{r!}\hat\fp_{k^-,\ell^-,g^-}
    (\fm_{\ell_1,g_1}\cdots\fm_{\ell_r,g_r}\cdot)_{conn}
\end{equation}
corresponding to complete gluings of $r$ connected surfaces of
signatures $(0,\ell_i,g_i)$, plus $k$ trivial cylinders, 
at their outgoing ends to the ingoing
ends of a connected surface of signature $(k^-,\ell^-,g^-)$, plus an
appropriate number of trivial cylinders, to obtain a {\em connected}
surface of signature $(k,\ell,g)$. In particular, for each term in
this sum the Euler characteristics satisfy 
$$
   \chi_{k,\ell,g} = \chi_{k^-,\ell^-,g^-} +
   \sum_{i=1}^r\chi_{k_i,\ell_i,g_i}.  
$$
Let us write $\{1,\dots,r\}$ as the disjoint union $I\cup J$,
where $i\in I$ iff $(\ell_i,g_i)$ equals $(1,0)$ or $(2,0)$, and $i\in
J$ iff $\chi_{0,\ell_i,g_i}<0$. Let $\delta>0$ be such that
$\|\fm_{\ell_i,g_i}\|-\gamma\,\chi_{0,\ell_i,g_i}\geq\delta$ for
all $i\in I$. As in the proof of Lemma~\ref{lem:filt1}, the filtration 
conditions on $\fp$ and $\fm$ imply
\begin{align}\label{eq:filt6}
   \|\fp^\fm_{k,\ell,g}\| - \gamma\,\chi_{k,\ell,g} 
   &\geq \delta|I| \geq 0.
\end{align} 
This shows that $|I|$ is uniformly bounded, i.e.~bounded from above 
in terms of $k,\ell,g$ and $\|\fp^\fm_{k,\ell,g}\|$. Then the equation
for the Euler characteristics provides uniform bounds on $|J|$ as 
well as all the $\ell_j,g_j$ for $j\in J$. Finally, the equation
$\ell_1+\dots+\ell_r+\ell^--k^-=\ell-k$ provides a uniform bound on
$k^-$, which proves convergence of  
$\fp^\fm_{k,\ell,g}$ with respect to the
filtration. Inequality~\eqref{eq:filt6} shows that $\fp^\fm_{k,\ell,g}$
satisfies~\eqref{eq:p-deg}. The equation
$\wh{\fp^\fm}\circ\wh{\fp^\fm} = 0$ follows immediately from
$\wh{\fp^\fm} = e^{-\fm}\hat\fp(e^{\fm}\cdot)$ and
$\hat\fp\circ\hat\fp=0$. 
\end{proof}

\begin{Remark}
(1) Note that, although $\frak m$ contains negative powers of $\hbar$,
  the map $\wh{\fp^\fm}$ does not contain negative powers of $\hbar$.

(2) In the case of filtered $\IBL_{\infty}$-algebras over the
universal Novikov ring $\Lambda_0$, condition
(\ref{eq:energycondforMC}) just says that $\frak m_{1,0},\frak
m_{2,0} \equiv 0 \mod \Lambda_+$. 
\end{Remark}

{\bf Push-forward of Maurer-Cartan elements. }
Next consider a strict filtered $\IBL_\infty$-morphism
$\ff=\{\ff_{k,\ell,g}\}:(C,\{\fp_{k,\ell,g}\})\to
(D,\{\fq_{k,\ell,g}\})$ between strict filtered $\IBL_\infty$-algebras  
and a Maurer-Cartan element $\{\fm_{\ell,g}\}$ in
$(C,\{\fp_{k,\ell,g}\})$. The interpretation of Maurer-Cartan elements
as filtered $\IBL_\infty$-morphisms from the trivial
$\IBL_\infty$-algebra and Lemma~\ref{lem:filtered-comp} immediately
imply  

\begin{lem}\label{lem:pushforward}
There exists a unique Maurer-Cartan element
$\{\ff_*\fm_{\ell,g}\}$ in $(D,\{\fq_{k,\ell,g}\})$ satisfying
$$
    e^\ff(e^\fm) = e^{\ff_*\fm}.
$$ \qed
\end{lem}

We call $\{\ff_*\fm_{\ell,g}\}$ the {\em push-forward} of the
Maurer-Cartan element $\{\fm_{\ell,g}\}$ under the morphism
$\{\ff_{k,\ell,g}\}$.

\begin{prop}\label{prop:MC-homotopy}
In the situation of Lemma~$\ref{lem:pushforward}$, 
there exists a unique strict filtered $\IBL_\infty$-morphism 
$\{\ff^\fm_{k,\ell,g}\}$ from $(C,\{\fp^\fm_{k,\ell,g}\})$ to
$(D,\{\fq^{\ff_*\fm}_{k,\ell,g}\})$ satisfying
$$
    e^{\ff^\fm} = e^{-\ff_*\fm}e^\ff(e^\fm\cdot):EC\{\hbar\}\to
    ED\{\hbar\},
$$
Moreover, if $\{\ff_{k,\ell,g}\}$ is a strict gapped filtered
$\IBL_\infty$-homotopy equivalence, then so is
$\{\ff^\fm_{k,\ell,g}\}$.  
\end{prop}

\begin{proof}
As usual, we translate the equation $e^{\ff^\fm} =
e^{-\ff_*\fm}e^\ff(e^\fm\cdot)$ for disconnected surfaces into one for
connected surfaces. This shows that the map $\ff^\fm_{k,\ell,g}:E_kC\to
E_\ell D$ is given by the sum  
\begin{align}\label{eq:filt7}
   \hspace{-1.7cm}
    \sum_{ { { 
    \ell_1^-+\dots+\ell_{r^-}^-=\ell }
    \atop {\sum k_i^- - \sum\ell_i^+ = k} }
    \atop {\sum g_i^+ + \sum g_i^- + \sum \ell_i^++ - r^+-r^-+1 = g} }
    \hspace{-1.8cm}
    \frac{1}{r^+!r^-!}
    (\ff_{k_1^-,\ell_1^-,g_1^-}\odot\cdots\odot
    \ff_{k_{r^-}^-,\ell_{r^-}^-,g_{r^-}^-})
    (\fm_{\ell_1^+,g_1^+}\cdots
    \fm_{\ell_{r^+}^+,g_{r^+}^+}\cdot)_{conn} 
\end{align}
corresponding to complete gluings of $r^+$ connected surfaces of
signatures $(0,\ell_i^+,g_i^+)$, plus $k$ trivial cylinders, at their
outgoing ends to the ingoing ends of $r^-$ connected surfaces of
signatures $(k_i^-,\ell_i^-,g_i^-)$ to obtain a {\em connected}
surface of signature $(k,\ell,g)$. In particular, for each term in
this sum the Euler characteristics satisfy 
$$
   \chi_{k,\ell,g} = \sum_{i=1}^{r^+}\chi_{0,\ell_i^+,g_i^+} +
   \sum_{i=1}^{r^-}\chi_{k_i^-,\ell_i^-,g_i^-}.   
$$
Let us set $k_i^+:=0$ and write $\{1,\dots,r^\pm\}$ as the disjoint
union $I^\pm\cup J^\pm\cup K^\pm$, where $i\in I^\pm$ iff
$(k_i^\pm,\ell_i^\pm,g_i^\pm)$ is one of the triples 
in~\eqref{eq:f-deg2}, $i\in J^\pm$ iff
$\chi_{k_i^\pm,\ell_i^\pm,g_i^\pm}<0$, and $i\in K^\pm$ iff
$(k_i^\pm,\ell_i^\pm,g_i^\pm)=(1,1,0)$. Note that $K^+=\emptyset$. 
Let $\delta>0$ be such that
$\|\ff_{k_i^-,\ell_i^-,g_i^-}\|-\gamma\,\chi_{k_i^-,\ell_i^-,g_i^-}\geq\delta$
for all $i\in I^-$ and
$\|\fm_{\ell_i^+,g_i^+}\|-\gamma\,\chi_{0,\ell_i^+,g_i^+}\geq\delta$
for all $i\in I^+$. As in the proof of Lemma~\ref{lem:filtered-comp},
the filtration conditions on $\ff$ and $\fm$ imply
\begin{align}\label{eq:filt8}
   \|\ff^\fm_{k,\ell,g}\| - \gamma\,\chi_{k,\ell,g} 
   &\geq \delta (|I^+| + |I^-|) \geq 0.
\end{align} 
This shows that $|I^+|$ and $|I^-|$ are uniformly bounded,
i.e.~bounded from above in terms of $k,\ell,g$ and
$\|\ff^\fm_{k,\ell,g}\|$. Then the equation for the Euler characteristics
provides uniform bounds on
$|J^+|$ and $|J^-|$ as well as all the $k_j^-,\ell_j^\pm,g_j^\pm$
for $j\in J^\pm$. Finally, the fact that 
each $i\in K^-$ contributes $1$ to the sum
$\ell_1^-+\dots+\ell_{r^-}^-=\ell$ yields a uniform bound on
$|K^-|$. Hence the number of terms in the sum 
in~\eqref{eq:filt7} is uniformly bounded, which proves convergence of  
$\ff^\fm_{k,\ell,g}$ with respect to the
filtration. Inequality~\eqref{eq:filt8} shows that $\ff^\fm_{k,\ell,g}$
satisfies~\eqref{eq:f-deg1}, where the inequality is strict if $I^+$
or $I^-$ is nonempty. If $I^+$ and $I^-$ are both empty then either
$\chi_{k,\ell,g}<0$ (if $J^+$ or $J^-$ are nonempty), or (if $J^+$ and
$J^-$ are both empty) $(k_i^-,\ell_i^-,g_i^-)=(1,1,0)$ for all
$i$ and hence $(k,\ell,g)=(1,1,0)$. This shows
that $\ff^\fm_{k,\ell,g}$ also satisfies~\eqref{eq:f-deg2}. 
That $\ff^\fm$ defines an $\IBL_\infty$-morphism now follows from
\begin{align*}
    \wh{\fq^{\ff_*\fm}}e^{\ff^\fm}
    &= e^{-\ff_*\fm}\hat\fq e^{\ff_*\fm}e^{\ff^\fm}
    = e^{-\ff_*\fm}\hat\fq e^\ff (e^\fm\cdot) \cr
    &= e^{-\ff_*\fm}e^\ff \hat\fp (e^\fm\cdot)
    = e^{-\ff_*\fm}e^\ff (e^\fm \wh{\fp^\fm}\cdot) \cr
    &= e^{\ff^\fm}\wh{\fp^\fm}.
\end{align*}
Note that
$$
   \ff^\fm_{1,1,0} = \ff_{1,1,0} +
   \underbrace{\sum_{k=1}^\infty\frac{1}{k!}\ff_{k+1,1,0}(\fm_{1,0}\cdots\fm_{1,0}\cdot)}_{=:F^+},  
$$
where the $k$-th term in the sum has filtration degree at least
$k\delta>0$. 

Finally, suppose that $\{\ff_{k,\ell,g}\}$ is a strict gapped filtered
$\IBL_\infty$-homotopy equivalence. Then $\ff_{1,1,0}$ is a
chain homotopy equivalence. 
A standard spectral sequence argument using the filtration 
(cf. \cite[chapter 3]{McC01}) now
shows that $\ff^\fm_{1,1,0}$ is also a chain homotopy equivalence. 
By Proposition~\ref{prop:filt-inverse}, this implies that
$\{\ff^\fm_{k,\ell,g}\}$ is a filtered $\IBL_\infty$-homotopy equivalence.  
\end{proof}

{\bf Gauge equivalence of Maurer-Cartan elements.} 
We conclude this section with a brief discussion of gauge equivalence,
which will be important for geometric applications in SFT and
Lagrangian Floer theory. 

\begin{defn}\label{def;Gequiv}
Let $\fm_0$, $\fm_1$ be Maurer-Cartan elements of a strict $G$-gapped
filtered $\IBL_{\infty}$-algebra $C$, and let $\frak C$ be a path
object for $C$. We say $\fm_0$ is {\em gauge equivalent} to $\fm_1$ 
if there exists a Maurer-Cartan element $\frak M$
of $\frak C$ such that
$$
(\eps_0)_* \frak M = \fm_0, 
\qquad 
(\eps_1)_* \frak M = \fm_1.
$$
\end{defn}

\begin{prop}\label{prop:gaugeMC}
\begin{enumerate}[\rm (1)]
\item The notion of gauge equivalence is independent of 
the choice of the path object $\frak C$.
\item
Gauge equivalence is an equivalence relation.
\item
Let $\frak f$, $\frak g$ be strict $G$-gapped filtered
$\IBL_{\infty}$-morphisms from $C$ to $D$, and $\fm_0$, $\fm_1$ be
Maurer-Cartan elements of $C$. If $\frak f$ is homotopic to $\frak
g$ and $\fm_0$ is gauge equivalent to $\fm_1$, then $\frak f_*\fm_0$
is gauge equivalent to $\frak g_*{\fm_1}$.
\end{enumerate}
\end{prop}

Using the fact that a Maurer-Cartan element is identified with a
morphism from $\text{\bf 0}$, Proposition~\ref{prop:gaugeMC}
immediately follows from the $G$-gapped filtered version of
Proposition~\ref{prop:homotopy-equiv}. 

\begin{prop}
Let $\fm_0$, $\fm_1$ be gauge equivalent Maurer-Cartan elements of a
strict $G$-gapped filtered $\IBL_{\infty}$-algebra
$(C,\{\fp_{k,\ell,g}\})$. Then $(C,\{\fp^{\fm_0}_{k,\ell,g}\})$ is
homotopy equivalent to $(C,\{\fp^{\fm_1}_{k,\ell,g}\})$.
\end{prop}

\begin{proof}
Let $(\frak C,\{\fq_{k,\ell,g}\})$ be a path object and 
$\frak M$ be as in Definition \ref{def;Gequiv}.
By Proposition~\ref{prop:MC-homotopy}, the $\eps_i$ induce morphisms 
$\eps_i^{\fm_i}(\frak C,\{\fq^{\frak M}_{k,\ell,g}\})
\to (C,\{\fp^{\fm_i}_{k,\ell,g}\})$, which are homotopy equivalences
since the $\eps_i$ are. The proposition now follows from the
$G$-gapped filtered version of Corollary \ref{Cor:comp}.
\end{proof}

\section{The dual cyclic bar complex of a cyclic cochain
   complex}\label{sec:Cyclic-cochain} 

In this section we show that the dual cyclic bar complex of a cyclic
cochain complex carries a natural $\text{\rm dIBL}$-structure, i.e., a
$\IBL_{\infty}$-structure such that $\mathfrak p_{k,\ell,g}
= 0$ unless $(k,\ell,g) \in\{ (1,1,0), (2,1,0), (1,2,0)\}$.
\par

{\bf The ${\rm dIBL}$ structure on the dual cyclic bar complex. } 
Let $(A=\bigoplus_kA^k,d)$ be a $\Z$-graded cochain complex over
$\R$. We assume that $\dim A$ is finite. Let $n$ be a positive integer
and  
$$
   \langle \cdot,\cdot \rangle : \bigoplus_kA^k \otimes A^{n-k} \to \R
$$
a nondegenerate bilinear form, which we extend by zero to the rest of
$A\otimes A$. 

\begin{Definition}\label{cyccochaincmx}
$(A,\la\,,\, \ra,d)$ is called a {\em cyclic cochain complex} if 
\begin{gather*}
    \langle dx,y \rangle + (-1)^{\deg x-1}\langle x,dy \rangle = 0, \cr
    \langle x,y \rangle + (-1)^{(\deg x-1)(\deg y-1)}\langle y,x
    \rangle = 0.
\end{gather*}
\end{Definition}
We define the {\em cyclic bar complex} 
$$
   B_k^{\text{\rm cyc}} A := \underbrace{A[1] \otimes \cdots \otimes
   A[1]}_{\text{$k$ times}} / \sim 
$$
as the quotient of the $k$-fold tensor product under the action of
$\Z_k$ by cyclic permutations with signs. As explained in
Remark~\ref{rem:tensor}, $B_k^{\text{\rm cyc}} A$ is isomorphic to the
subspace of invariant tensors under the cyclic group action. 
We introduce the {\em dual cyclic bar complex} 
$$\aligned
B_k^{\text{\rm cyc} *}A& := Hom(B_k^{\text{\rm cyc}} A,\R), \\
B^{\text{\rm cyc} *}A &:= \bigoplus_{k=1}^{\infty} B_k^{\text{\rm
    cyc} *}A. 
\endaligned$$
To avoid confusion, we will denote the degree in $A$ by $\deg x$ and
the degree in $A[1]$ by
$$
    |x| = \deg x-1.
$$
An element $\varphi \in B_k^{\text{\rm cyc} *}A$ is homogeneous of degree $D$ if
$\varphi(x_1 \otimes \cdots \otimes x_k) = 0$ whenever $\sum |x_i|
\neq D$. 
The coboundary operator $d$ induces a boundary operator on
$B^{\text{\rm cyc} *}A$ in the obvious way, which we denote by $\dd$:
$$
   (\dd\varphi)(x_1,\dots,x_k):=\sum_{j=1}^k(-1)^{|x_1|+\dots+|x_{j-1}|} 
   \varphi(x_1,\dots,x_{j-1},dx_j,x_{j+1},\dots x_k). 
$$
Note that the coboundary operator on $A$ has degree $+1$, so the induced 
boundary operator $\dd$ on $B^{\text{\rm cyc} *}A$ has degree $-1$. 

We will now construct two operations
\begin{gather*}
    \mu:B^{\text{\rm cyc} *}A \otimes B^{\text{\rm cyc} *}A \to B^{\text{\rm cyc} *}A, \cr
    \delta:B^{\text{\rm cyc} *}A \to B^{\text{\rm cyc} *}A \otimes B^{\text{\rm cyc} *}A
\end{gather*}
of degree $|\mu|=|\delta|=2-n$, which together with the differential $\dd$ will give rise to a $\text{\rm dIBL}$-structure. It 
suffices to define these operations on homogeneous elements in $B^{\text{\rm cyc} *}A =
\bigoplus_{k\geq 1}B_k^{\text{\rm cyc} *}A$, respectively
$$
   B^{\text{\rm cyc} *}A\otimes B^{\text{\rm cyc} *}A 
   = \bigoplus_{k_1,k_2\geq 1}B_{k_1}^{\text{\rm cyc} *}A\otimes
   B_{k_2}^{\text{\rm cyc} *}A 
   = \bigoplus_{k_1,k_2\geq 1}Hom(B_{k_1}^{\text{\rm cyc}}A\otimes
   B_{k_2}^{\text{\rm cyc}}A,\R). 
$$
 
Let $e_i$ be a homogeneous basis of $A$ and set
$$
\eta_i:= |e_i| = \deg e_i - 1.
$$
Let $e^i$ be the dual basis of $A$ with
respect to the pairing $\la\ ,\ \ra$, i.e. 
$$
    \la e_i,e^j\ra = \delta_i^j. 
$$
We set 
$$
    g_{ij} := \la e_i,e_j\ra,\quad g^{ij} := \la e^i,e^j\ra. 
$$
Note that
$$
     g_{ij} = (-1)^{\eta_i \eta_j+1} g_{ji}.
$$
In the following we use Einstein's sum convention. Then
$g_{ik}g^{jk}=\delta_i^j$, i.e.~$(g^{ij})$ is the transpose of the
inverse matrix of $(g_{ij})$.
Note that $\deg e_i + \deg e^i=n$, so $g^{ij} \neq 0$ implies that
$\eta_i + \eta_j = n-2$. If $\tilde e_j = \tau^i_je_i$ is another
basis, then its dual basis is $\tilde e^i = \sigma^i_je^j$ with
$(\sigma^i_j)$ the inverse matrix of $(\tau^i_j)$ and the pairing in
the new basis is given by
\begin{gather}\label{eq:transf}
    \tilde g_{\alpha\beta} = \la\tilde e_\alpha,\tilde e_\beta\ra =
    \tau^a_\alpha g_{ab}\tau^b_\beta, \qquad
    g_{ab} = \sigma^\alpha_a \tilde g_{\alpha\beta}\sigma^\beta_b, \cr
    \tilde g^{\alpha\beta} = \la\tilde e^\alpha,\tilde e^\beta\ra =
    \sigma^\alpha_a g^{ab}\sigma^\beta_b, \qquad
    g^{ab} = \tau^a_\alpha \tilde g^{\alpha\beta}\tau^b_\beta. 
\end{gather}
Finally, we introduce the notation
$$
d e_i = \sum_j d^j_i e_j,
$$
for the coboundary operator.
\begin{lem}\label{lem:edge1} The following identities hold:
\begin{eqnarray}
d e^a =& (-1)^{\eta_a}d_c^a e^c\\
(-1)^{\eta_{a}}d_a^{a'} g^{ab} + g^{a'b'}d_{b'}^b=&0. \label{eq:littleg}
\end{eqnarray}
\end{lem}

\begin{proof}
To prove the first equation, we compute the coefficient of $e^c$ in $d\,e^a$ as
\begin{align*}
\langle e_c, de^a \rangle &= (-1)^{\eta_c+1} \langle de_c,e^a \rangle\\
&= (-1)^{\eta_c+1} d_c^{c'}\langle e_{c'},e^a \rangle\\
&= (-1)^{\eta_c+1} d_c^{a}.
\end{align*}
Since the degrees of $e^c$ and $e^a$, and hence also the degrees of $e_c$ and $e_a$ differ by one, the first claim follows.

To prove the second claim, we again use the cyclic relation
$$
\langle de^{a'}, e^b  \rangle = (-1)^{|e^{a'}|-1} \langle e^{a'} , de^b\rangle 
= (-1)^{\eta_{a'}+n-3} \langle e^{a'} , de^b\rangle.
$$
Using the first equation, we find
$$
\langle de^{a'}, e^b  \rangle = (-1)^{\eta_{a'}} d_a^{a'} g^{ab}
$$ 
and
$$
\langle e^{a'} , de^b\rangle = (-1)^{\eta_b} d_{b'}^b g^{a'b'}.
$$
Putting things together and noting that $\eta_a+\eta_b\equiv n-2$, we obtain the result.
\end{proof}

An element $\varphi \in B_k^{\text{\rm cyc} *}A$ is determined by its coefficients
$$
   \varphi_{i_1\cdots i_k} =
   (-1)^{\eta_{i_k} \cdot \sum_{j=1}^{k-1}\eta_{i_j}} \varphi_{i_k
     i_1\cdots i_{k-1}}. 
$$
The boundary operator on $B^{\text{\rm cyc} *}A$ acts on these
coefficients by
$$
(\dd\varphi)_{i_1,\dots,i_k} = \sum_{j=1}^k\sum_a (-1)^{\eta_{i_1}+ \dots +
  \eta_{i_{j-1}}} d^a_{i_j} \varphi_{i_1,\dots,i_{j-1},a,i_{j+1},\dots,i_k}.
$$
Now we are ready to define a bracket and a cobracket on $B^{\text{\rm cyc} *}A$. 
For the bracket, 
let $\varphi^1 \in B_{k_1+1}^{\text{\rm cyc} *}A$, $\varphi^2 \in B_{k_2+1}^{\text{\rm cyc} *}A$, 
$k_1,k_2 \ge 0$, $k_1+k_2\ge 1$. We define $\mu(\varphi^2,\varphi^2)\in
B_{k_1+k_2}^{\text{\rm cyc} *}A$ by
\begin{equation}
    \mu(\varphi^1,\varphi^2)_{i_1\cdots i_{k_1+k_2}} :=
    \sum_{a,b}\sum_{c=1}^{k_1+k_2}(-1)^{\eta+\eta_a}
    g^{ab} \varphi^1_{a i_c \cdots i_{c+k_1-1}}\varphi^2_{b i_{c+k_1} \cdots i_{c-1}},
\end{equation}
where 
\begin{equation*}
    \eta =\eta(i_1,\dots,i_{k_1+k_2};a;b;c)= 
    \sum_{r=1}^{c-1}\sum_{s=c}^{k_1 
    + k_2} \eta_{i_r} \eta_{i_s} \,+\,
    \eta_b \cdot \sum_{t=c}^{c+k_1-1} \eta_{i_t}
\end{equation*}
is the sign needed to move 
$(e_a,e_{i_c},\dots,e_{i_{c+k_1-1}},e_b,e_{i_{c+k_1}},\dots,e_{i_{c-1}})$
to the order $(e_a,e_b,e_{i_1},\dots,e_{i_{k_1+k_2}})$. 
Here and hereafter we put $i_{k_1+k_2+m} = i_m$ etc. Using
$\sum_{t=c}^{c+k_1-1} \eta_{i_t}=|\varphi^1|-\eta_a$ one verifies 
$$
    \eta(i_2,\dots,i_{k_1+k_2},i_1;a;b;c) -
    \eta(i_1,\dots,i_{k_1+k_2};a;b;c+1) = \eta_{i_1} \cdot \sum_{s=2}^{k_1+k_2}
    \eta_{i_s}, 
$$
so $\mu(\varphi^1,\varphi^2)$ picks up the correct signs under cyclic
permutation to define an element in $B_{k_1+k_2}^{\text{\rm cyc} *}A$. Note that 
the operation $\mu:B^{\text{\rm cyc} *}A \otimes B^{\text{\rm cyc} *}A \to B^{\text{\rm cyc} *}A$ has
degree $2-n$. The independence of the basis $e_i$ follows from the
transformation law~\eqref{eq:transf}, using the equivalent definition
for $x_i\in A$ 
%
%
\begin{figure}[h]
 \labellist
  \pinlabel $e_a$ [bl] at 45 51
  \pinlabel $e_b$ [br] at 100 51
  \pinlabel $x_c$ [b] at 27 93
  \pinlabel $x_{c+k_1-1}$ [tr] at 27 10
  \pinlabel $x_{c+k_1}$ [tl] at 110 8
  \pinlabel $x_{c-1}$ [b] at 101 98
 \endlabellist
  \vspace{.5cm}
  \centering
  \includegraphics[scale=1.2]{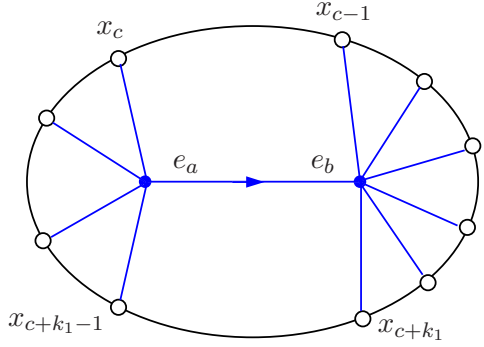}
 \caption{An illustration of formula \eqref{eq:mu} for $\mu$ which also fits into a more general graphical approach taken below. The vertex on the left is first, the one on the right is second.}
 \label{fig:p210}
\end{figure}
\begin{align}\label{eq:mu}
    & \mu(\varphi^1,\varphi^2)(x_1,\dots,x_{k_1+k_2}) \cr
    &= \sum_{a,b}\sum_{c=1}^{k_1+k_2} (-1)^{\eta_a+\eta}
    g^{ab} \varphi^1(e_a,x_c,\dots,x_{c+k_1-1})\varphi^2(e_b,x_{c+k_1},
    \dots,x_{c-1}). 
\end{align}

\par
For the cobracket, note that an element $\psi \in B_{k_1}^{\text{\rm
    cyc} *}A \otimes B_{k_2}^{\text{\rm cyc} *}A  
\cong Hom(B_{k_1}^{\text{\rm cyc}}A\otimes B_{k_2}^{\text{\rm
    cyc}}A,\R)$ is determined by the coefficients
$$
    \psi_{i_1\cdots i_{k_1};j_1\cdots j_{k_2}}
    = \psi\bigl((e_{i_1}\cdots e_{i_{k_1}}) \otimes (e_{j_1}\cdots e_{j_{k_2}})\bigr).
$$
Now let $\varphi \in B_k^{\text{\rm cyc}*}A$, $k\ge 4$. We define
$$
    \delta(\varphi) \in \bigoplus_{k_1+k_2= k-2}
    B_{k_1}^{\text{\rm cyc}*}A \otimes B_{k_2}^{\text{\rm cyc}*}A
$$
by
\begin{equation}
    (\delta\varphi)_{i_1\cdots i_{k_1};j_1\cdots j_{k_2}} :=
    \sum_{a,b}\sum_{c=1}^{k_1}\sum_{c'=1}^{k_2} (-1)^{\eta+\eta_a} g^{ab}\varphi_{ai_c\cdots
    i_{c-1}bj_{c'}\cdots j_{c'-1}}, 
\end{equation}
where
\begin{align}\label{eq:eta}
    \eta &=
    \eta(i_1,\dots,i_{k_1};j_1,\dots,j_{k_2};a;b;c,c') \cr
    &= \sum_{r=1}^{c-1}\sum_{s=c}^{k_1} \eta_{i_r} \eta_{i_s} \,+\,
    \sum_{r'=1}^{c'-1}\sum_{s'=c'}^{k_2} \eta_{j_{r'}} \eta_{j_{s'}} \,+\,
    \eta_b \cdot \sum_{t=1}^{k_1} \eta_{i_t}
\end{align}
is the sign needed to move $(e_a,e_{i_{c+1}},\dots,e_{i_{c+k_1}},e_b,e_{j_{c'+1}},\dots,e_{j_{c'+k_2}})$
to the order $(e_a,e_b,e_{i_1},\dots,e_{i_{k_1}},e_{j_1},\dots,e_{j_{k_2}})$. 

%
%
\begin{figure}[h]
 \labellist
  \pinlabel $e_a$ [br] at 42 100
  \pinlabel $e_b$ [bl] at 88 100
  \pinlabel $x_c$ [r] at 52 80
  \pinlabel $x_{c-1}$ [l] at 76 80
  \pinlabel $y_{c'}$ [bl] at 99 123
  \pinlabel $y_{c'-1}$ [br] at 35 123
 \endlabellist
  \vspace{.2cm}
  \centering
  \includegraphics[scale=1.2]{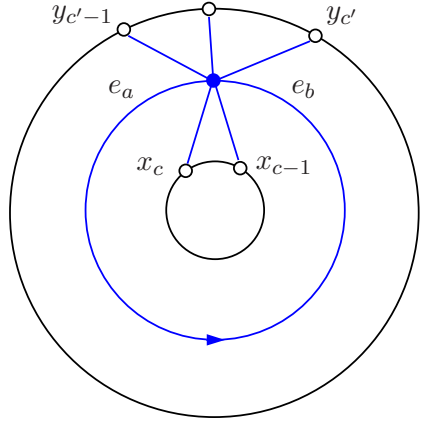}
 \caption{A graphical illustration of formula \eqref{eq:delta} for $\delta$. The inner boundary component is the first and is oriented clockwise, whereas the outer one is second and oriented counterclockwise.}
 \label{fig:p120}
\end{figure}
The operation $\delta: B^{\text{\rm cyc} *}A \to B^{\text{\rm cyc} *}A \otimes B^{\text{\rm cyc} *}A$ also has degree $2-n$, and independence of the basis follows from the
equivalent definition in terms of $x_i\in A$ 
\begin{align}\label{eq:delta}
    & (\delta\varphi)(x_1\cdots x_{k_1}\otimes y_1\cdots y_{k_2}) \cr
    &= \sum_{a,b}\sum_{c=1}^{k_1}\sum_{c'=1}^{k_2} (-1)^{\eta_a+\eta}
    g^{ab}\varphi(e_a,x_c,\dots,x_{c-1},e_b,y_{c'},\dots,y_{c'-1}). 
\end{align}
Denote by $\tau:B^{\text{\rm cyc} *}A \otimes B^{\text{\rm cyc} *}A \to B^{\text{\rm cyc} *}A \otimes
B^{\text{\rm cyc} *}A$ the map permuting the factors with sign, i.e. 
$$
\tau(\varphi \otimes \psi) = (-1)^{|\varphi||\psi|} \psi \otimes \varphi
$$
whenever $\varphi$ and $\psi$ are homogeneous elements of $B^{\text{\rm cyc} *}A$.

\begin{lem}\label{lem:mu-delta}
The operations $\delta$ and $\mu$ introduced above satisfy
\begin{enumerate}  [\rm (1)]
\item $\mu \tau = (-1)^{n-3}\mu$,
\item $\tau \delta = (-1)^{n-3}\delta$.
\end{enumerate}
\end{lem}

\begin{proof}
The assertions can be proven by tedious but straightforward
computations. We will later see equivalent assertions within a larger
``graphical calculus'' (cf. Remark~\ref{rem:proof-mu-delta}), so we do
not give details here.  
\end{proof}
\par

We will use these operations to define a $\text{\rm dIBL}$-structure
on $B^{\text{\rm cyc} *}A$. 
To fit with the conventions used in \S~\ref{sec:def}, we
consider the degree shifted version 
$$
    \BC:=(B^{\text{\rm cyc}*}A)[2-n]. 
$$
For the shifted degrees, $\mu$ has degree $2(2-n)$ and $\delta$ has
degree $0$. We define the boundary operator 
$$
    \fp_{1,1,0}:=\dd:E_1\BC \to E_1\BC 
$$
so that
\begin{equation}
(\fp_{1,1,0} \varphi)_{i_1\dots i_k} = \sum_{j=1}^k \sum_a
(-1)^{\eta_{i_1} + \dots + \eta_{i_{j-1}}} d^a_{i_j} \varphi_{i_1\dots
   i_{j-1} a i_{j+1} \dots i_k}.
\end{equation}
Next, we define maps
$$
P_k: (B^{\rm cyc*}A)^{\otimes k} \to (B^{\rm cyc*}A)[3-n]^{\otimes k}=\BC[1]^{\otimes k}
$$
by
\begin{equation}\label{eq:Pk}
P_k(c_1 \otimes \cdots \otimes c_k) := (-1)^{(n-3)\sum_{i=1}^k (k-i)|c_i|} c_1 \otimes \cdots \otimes c_k.
\end{equation}
The sign can be interpreted in terms of formal variables
$s_1\dots,s_k$ of degree $n-3$ as the sign for the change of order
$$
   s_1c_1\cdots s_kc_k \longrightarrow s_1\cdots s_k\, c_1\cdots c_k. 
$$
The maps $P_k$ conjugate the $(n-3)$-twisted action of the permutation
group $S_k$ on $(B^{\rm cyc*}A)^{\otimes k}$, given for $\sigma \in
S_k$ by 
$$
\sigma \cdot (\varphi^1 \otimes \cdots \otimes \varphi^k) :=
(-1)^{\eta+(n-3)\sgn(\sigma)} \varphi^{\sigma(1)} \otimes \cdots
\otimes \varphi^{\sigma(k)}, 
$$
with the usual action by signed permutations on $\BC[1]$, i.e. the
following diagram commutes: 
\begin{equation*}
\begin{CD}
   (B^{\rm cyc*}A)^{\otimes k} @>\sigma>> (B^{\rm cyc*}A)^{\otimes k} \\
   @V{P_k}VV @V{P_k}VV \\
   \BC[1]^{\otimes k} @>\sigma>> \BC[1]^{\otimes k},
\end{CD}
\end{equation*}
where the action of $\sigma$ is $(n-3)$-twisted on the first line
(with respect to degrees in $B^{\rm cyc*}A$) and standard on the
second line (with respect to degrees in $\BC[1]=B^{\rm cyc*}A[3-n]$).

According to Lemma~\ref{lem:mu-delta}, $\mu$ and $\delta$ are
symmetric operations on $BC^{\rm cyc*}A$ with respect to the
$(n-3)$-twisted action of the permutation group, so defining \footnote{The combinatorial factor $\frac 12$ in the definition of $\fp_{1,2,0}$, which seems rather unmotivated here, will be explained in a more general context below.}
\begin{align*}
   \fp_{2,1,0} &:= P_1 \circ \mu \circ P_2^{-1}: 
   \BC[1] \otimes \BC[1] \to \BC[1], \cr
   \fp_{1,2,0} &:= \frac 12 P_2 \circ \delta \circ P_1^{-1}:
   \BC[1] \to \BC[1] \otimes \BC[1],
\end{align*}
makes these operations symmetric with respect to the usual action of permutations on $\BC[1]^{\otimes k}$.
Hence $\fp_{2,1,0}$ descends to a map $\fp_{2,1,0}:E_2\BC \to
E_1\BC$ of degree $-2(n-3)-1$ and $\fp_{1,2,0}$ 
descends to a map $\fp_{1,2,0}:E_1\BC \to E_2\BC$ of degree $-1$.
Explicitly, the operations are given by
\begin{align}\label{eq:p-mu}
   \fp_{2,1,0}(\phi,\psi) &= (-1)^{(n-3)|\phi|}\mu(\phi,\psi), \cr
   \fp_{1,2,0}(\phi)(x,y) 
   &= \frac 12 (-1)^{(n-3)|x|} \delta(\phi)(x,y). 
\end{align}


The following result corresponds to
Proposition~\ref{prop:structureexists-intro} from the Introduction. 

\begin{Proposition}\label{prop:structureexists}
$\bigl(\BC=(B^{\text{\rm cyc}*}A)[2-n],\fp_{1,1,0},\fp_{1,2,0},\fp_{2,1,0}\bigr)$
is a \text{\rm dIBL}-algebra of degree $d=n-3$.
\end{Proposition}

\begin{Remark}
Proposition~\ref{prop:structureexists} was 
known among researchers in string topology, see e.g.~\cite{Chas,Gonz}.
\end{Remark}

We will prove this proposition below, after introducing an alternative point of view on the construction of the operations. This will allow us to largely avoid messy computations and instead concentrate on describing the underlying organizing principles.  We hope that will make the proofs in this section more transparent.

{\bf Defining maps using ribbon graphs. }
As promised, we now develop the graphical calculus underlying the
constructions in this section. It generalizes similar constructions of
$A_{\infty}$ or 
$L_{\infty}$ structures and homotopy equivalences among them, which 
are used to construct canonical models and are based on summation over
ribbon {\it trees} (see \cite{KoSo01}, \cite[Subsection 5.4.2]{FOOO06}).

By a {\em ribbon graph} we mean a finite connected graph $\Gamma$ with a cyclic
ordering of the (half-)edges incident to each vertex. We denote by $d(v)$
the {\em degree} of a vertex $v$, i.e.~the number of (half-)edges incident to
$v$. Let the set of vertices be decomposed as 
$$
    C_0(\Gamma) = C_0^{\text{\rm int}}(\Gamma) \cup C_0^{\text{\rm
    ext}}(\Gamma) 
$$
into {\em interior} and {\em exterior} vertices, where all exterior
vertices have degree 1 (interior vertices can have degree $1$ or
higher). We assume that all our ribbon graphs have at least one interior vertex.

Such a graph $\Gamma$ can be thickened 
to a compact oriented surface $\Sigma_\Gamma$ with boundary in a unique way 
(up to orientation preserving diffeomorphism) such that 
$\Gamma\cap\p\Sigma_\Gamma=C_0^\ext(\Gamma)$.  
See Figure~\ref{fig:Gamma}, where interior vertices are drawn as
$\bullet$ and exterior vertices as $\circ$. We denote by $s(b)$ the
number of exterior vertices on the boundary component $b$.  
%
%
\begin{figure}[h]
 \centering
  \includegraphics[scale=.83]{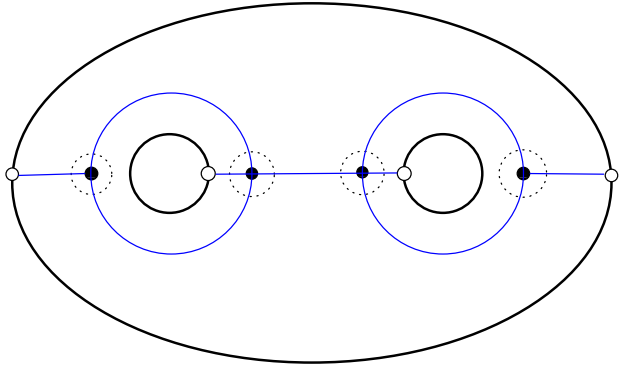}
 \caption{The surface $\Sigma_\Gamma$ associated to a very simple
   ribbon graph $\Gamma$ with four interior and four exterior
   vertices.} 
 \label{fig:Gamma}
\end{figure}

We assume that each graph has at least one interior vertex. Moreover, we
impose the following condition:

{\em Each boundary component of $\Sigma_\Gamma$ contains at least one
exterior vertex of $\Gamma$}.

The set of edges is decomposed as 
$$
    C_1(\Gamma) = C_1^{\text{\rm int}}(\Gamma) \cup C_1^{\text{\rm
    ext}}(\Gamma) 
$$
into {\em interior} and {\em exterior} edges, where 
an edge is called {\em exterior} if and only if it contains an exterior
vertex.

The {\em signature} of $\Gamma$ is $(k,\ell,g)$, where
$k=|C_0^\inn(\Gamma)|$ is the number of interior vertices, $\ell$ is the 
number of boundary components of $\Sigma_\Gamma$, and $g$ is its genus. 
For example, the graph in Figure~\ref{fig:Gamma} has signature $(4,3,0)$. 

An {\em automorphism} of a ribbon graph $\Gamma$ is an automorphism of the underlying graph which is required to preserve the cyclic ordering around vertices, but can permute vertices and edges (and hence also boundary components). For example, the ribbon graph in Figure~\ref{fig:Gamma} has a unique nontrivial automorphism (in the picture it is given by rotation around the center by the angle $\pi$). We denote the group of automorphisms of a ribbon graph $\Gamma$ by $\Aut (\Gamma)$.

%
%
\begin{figure}[h]
  \centering
  \includegraphics[scale=.83]{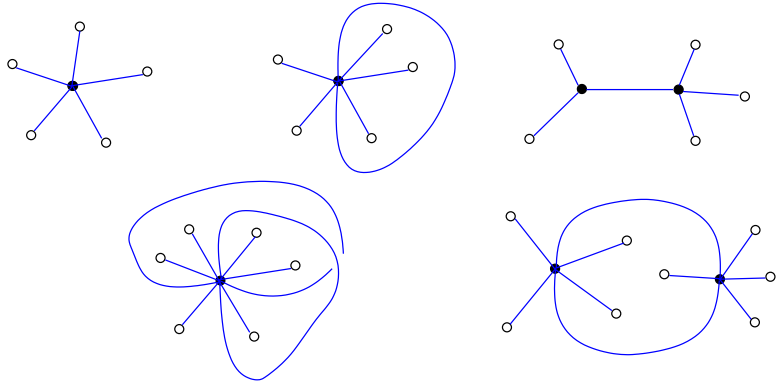}
 \caption{Examples of ribbon graphs from $RG_{1,1,0}$, $RG_{1,2,0}$, $RG_{2,1,0}$, $RG_{1,1,1}$ and $RG_{2,2,0}$.}
 \label{fig:ribbon}
\end{figure}

We denote by $RG_{k,\ell,g}$ the 
set of isomorphism classes of ribbon graphs of
signature $(k,\ell,g)$. Figure~\ref{fig:ribbon} gives some examples of 
such graphs, where we use the same convention of labelling interior vertices by $\bullet$ and exterior vertices by $\circ$.

Let us make two comments on this definition.

(1) Since all surfaces of the same signature are diffeomorphic, the set $RG_{k,\ell,g}$ can alternatively
be described as follows. Fix a connected compact oriented surface
$\Sigma$ of genus $g$ with $\ell$ boundary components. 
Then $RG_{k,\ell,g}$ is the set of isomorphism classes of (connected) graphs $\Gamma$ embedded in $\Sigma$ that satisfy the following conditions:
\begin{itemize}
\item $\Gamma\cap \partial \Sigma$ consists of degree 1 vertices. 
We call it the set of exterior vertices and write $C_0^{\text{\rm
ext}}(\Gamma)$. All the other vertices are called interior and make up a set 
$C_0^{\text{\rm int}}(\Gamma)$.
\item Each connected component $D_i$ of
   $\Sigma \setminus \Gamma$ is an open disk such that $\ol{D}_i \cap
   \partial \Sigma$ is a (nonempty) arc.
\item $\Gamma \cap \partial_b \Sigma$ is nonempty for each boundary
   component $\partial_b \Sigma$.
\end{itemize}
In this description, two graphs are isomorphic if there is an
orientation preserving diffeomorphism of $\Sigma$ mapping one to the
other.

(2) 
Given a ribbon graph $\Gamma$ of signature $(k,\ell,g)$, let
$\Sigma_\Gamma^-$ be the surface obtained from $\Sigma_\Gamma$ by
removing a small disk around each interior vertex. (In Figure~\ref{fig:Gamma}, 
these small disks are indicated by dotted circles). 
Viewing the
$\ell$ original boundary components as outgoing and the $k$ new
boundary components as incoming, $\Sigma_\Gamma^-$ is a surface of
signature $(k,\ell,g)$ as considered in
Section~\ref{sec:def}. $\Gamma\cap\Sigma_\Gamma^-$ is a collection of
disjoint arcs on $\Sigma_\Gamma^-$ starting and ending on the boundary
such that every boundary component meets some arc. Identifying
$\Sigma_\Gamma^-$ with a model surface $\Sigma^-$, each ribbon graph
of signature $(k,\ell,g)$ thus induces such a collection of arcs on
$\Sigma^-$.

To each interior edge of our ribbon graphs we will associate a
decomposable 2-tensor $T= T^{ab}e_a \otimes e_b$ (using Einstein's sum
convention) which has the symmetry property 
\begin{equation}\label{eq:tensorsymmetry}
T^{ab}=(-1)^{\eta_a\eta_b + (n-3)}T^{ba}.
\end{equation}

\begin{rem}
In this section we will only use the tensor
$T^{ab}=(-1)^{\eta_a}g^{ab}$. In the following section we will 
introduce another such tensor. 
\end{rem}

Given a ribbon graph $\Gamma \in R_{k,\ell,g}$ with such tensors
associated to its interior edges, we want to define a map
$$
   F_\Gamma: (\BC[1])^{\otimes k} \to (\BC[1])^{\otimes \ell}. 
$$
In order to do that, we will make additional choices.

\begin{Definition}\label{def:labelling}
A {\em labelling} of a ribbon graph $\Gamma$ consists of
\begin{itemize}
\item a numbering of the interior vertices by $1,\dots,k$; 
\item a numbering of the boundary components of $\Sigma_\Gamma$ by
   $1,\dots,\ell$; 
\item a numbering of the (half)-edges incident to a given vertex $v$
  by $1,\dots,d(v)$, which is compatible with the previously given
  cyclic ordering; 
\item a numbering of the exterior vertices on the $b$-th boundary
  component by $1,\dots,s(b)$, which is compatible with the cyclic
  order induced from the orientation of the surface $\Sigma_\Gamma$.  
\end{itemize}
\end{Definition}

The first two of these choices induce a (suitable equivalence class of
a) choice of ordering and orientation of the interior edges. One
procedure to make a consistent such choice will be described in
Definition~\ref{def:edge-ordering} below. 
Here we only state that for
every labelled tree with two vertices and one interior edge the
resulting orientation of that edge points from the first to the second
vertex, and for every labelled graph with one interior vertex and
one interior edge (which is necessarily a loop at that vertex) the
resulting orientation of the edge is such that the first boundary
component is to the left of the edge. These conventions were already 
illustrated in Figures~\ref{fig:p210} and \ref{fig:p120}, where examples 
of contributing graphs are drawn on their ribbon surfaces.

Now imagine basis elements $e_{\beta(b,1)},\dots,e_{\beta(b,s(b))}$ feed into the exterior edges incident to each of the boundary components $b$. For an oriented interior edge labelled by $T^{ab}e_a \otimes e_b$, label the half-edge of the starting point with $e_a$ and the half-edge of the endpoint with $e_b$.
Then around each interior vertex $v$, the $j$-th incident half-edge comes labelled with a basis vector $e_{\alpha(v,j)}$, and we can define
\begin{align}\label{Fdefstatesum}
    &\ \ \ (F_\Gamma(\varphi^1 \otimes \cdots \otimes \varphi^k))_{\beta(1,1)\cdots \beta(1,s(1)); \cdots ; \beta(\ell,1)\cdots \beta(\ell,s(\ell))}\\
    &:= \frac 1 {\ell!\,|\Aut(\Gamma)|\,\prod_v d(v)}\sum(-1)^\eta \left(\prod_{l\in
    C^1_\inn(\Gamma)} T^{a_lb_l}
    \prod_{v \in C^0_{\text{int}}(\Gamma)} \varphi^v_{\alpha(v,1) \cdots \alpha(v,d(v))}\right),
\notag
\end{align}
where the sum is over all possible ways of making the choices mentioned above, 
and for each interior edge $l\in C_1^{\rm int}(\Gamma)$ we also sum over all $a_l,b_l$ ranging in the index set of the chosen basis of $A$.
The sign exponent $$\eta=\eta_1+\eta_2$$ is determined as follows. With
all the choices that we have made, we can write all the involved basis
elements $e_i$ in two different orders: 

{\em Edge order. }
$$
   \prod_{t\in C_1^\inn(\Gamma)} e_{a_t}e_{b_t} 
   \prod_{b=1}^\ell e_{\beta(b,1)}\cdots e_{\beta(b,s(b))}. 
$$
Note that this depends on the ordering of the interior edges, the
orientation of the interior edges, the ordering of the boundary
components, and the ordering of the vertices on each boundary
component. 

{\em Vertex order. }
$$
    \prod_{v \in C_0^{\text{int}}(\Gamma)} e_{\alpha(v,1)}
    \cdots e_{\alpha(v,d(v))}.
$$
Note that this depends on the ordering of the interior vertices, and
the ordering of the half-edges incident to each interior vertex. Now $(-1)^{\eta_1}$
is defined as the sign needed to move the edge order to the vertex
order, according to the $A[1]$-degrees $\eta_i=|e_i|=\deg e_i-1$. 
The other part of the sign exponent is determined as above by viewing
the map as a composition  
$$
\BC[1]^{\otimes k} \stackrel{P_k^{-1}}{\longrightarrow} (B^{\rm cyc*}A)^{\otimes k}
\stackrel{\widetilde{F}_\Gamma}{\longrightarrow} (B^{\rm cyc*}A)^{\otimes \ell}
\stackrel{P_\ell}{\longrightarrow} \BC[1]^{\otimes \ell}.
$$
Now $\eta_2$ is the part of the sign exponent coming from the
conjugation with $P_k$ and $P_\ell$ as in~\eqref{eq:Pk}, i.e.
$$
\eta_2= (n-3)\left(\sum_{v=1}^k (k-v)|\varphi^v| + \sum_{b=1}^\ell (\ell-b)|x^b|\right),
$$
where $x^b=e_{\beta(b,1)}\cdots e_{\beta(b,s(b))}$ is the word
associated to the $b$-th boundary component. 

We now discuss some consequences of this definition, where the second one
depends on additional properties of a specific choice for $T$: 
\begin{enumerate}[(1)]
\label{symmetrylist}
\item Reversing the orientation of an interior edge yields a change in $\eta_1$ of
   ${\eta_a\eta_b+n-3}$ from replacing $T^{ab}$ by
   $T^{ba}$ (cf. \eqref{eq:tensorsymmetry}), and another change in $\eta_1$ of $\eta_a\eta_b$ from interchanging
  $e_a$ and $e_b$. Since $\eta_2$ is unchanged, reversing the
  orientation of an interior edge yields the total
  sign $(-1)^{n-3}$. 
\item[($2_g$)] With the specific choice of
  $T^{ab}=(-1)^{\eta_a}g^{ab}$, interchanging the order of two
  adjacent interior edges leads to a sign $(-1)^{n-2}$ due to the change of
  $\eta_1$ from interchanging the corresponding pairs of basis
  vectors in the edge order, because $\eta_a+\eta_b=n-2$ whenever
  $g^{ab}\neq 0$. 
\setcounter{enumi}{2}
\item Changing the ordering of the half-edges at an interior vertex by a
  cyclic permutation yields the same sign twice, once from the change
  of the coefficient $\varphi^v_{\alpha(v,1)
     \cdots \alpha(v,d(v))}$ in~\eqref{Fdefstatesum} (because
  $\varphi^v\in B^{\text{\rm cyc}*}_{d(v)}A$), and once from the
  cyclic permutation of the corresponding basis vectors $e_{\alpha(v,1)}\cdots e_{\alpha(v,d(v))}$. So
     definition~\eqref{Fdefstatesum} does not depend on the ordering
     (compatible with the cyclic order) of the half-edges at an interior
     vertex.  
\item Changing the ordering of the vertices on the $b$-th boundary
  component by a cyclic permutation yields the sign obtained from the
  cyclic permutation of the corresponding basis vectors
  $e_{\beta(b,1)}\cdots e_{\beta(b,s(b))}$. As the sum
  in~\eqref{Fdefstatesum} extends over all orderings 
  (compatible with the boundary orientations) of
  the exterior vertices on each boundary component, it defines a map 
$$
    B^{\text{\rm cyc}*}_{d(1)}A \otimes \cdots \otimes B^{\text{\rm cyc}*}_{d(k)}A
    \to B^{\text{\rm cyc}*}_{s(1)}A \otimes \cdots \otimes B^{\text{\rm cyc}*}_{s(\ell)}A. 
$$
\item Interchanging the order of two adjacent boundary components
  of total $B^{\text{\rm cyc}*}A$-degrees $t_1,t_2$ yields a change of  $t_1t_2$ in $\eta_1$ from permuting the corresponding basis vectors in the edge order. It also yields a change of $(n-3)(t_1+t_2)$ in $\eta_2$. 
\item Interchanging the order of two adjacent interior vertices
  $v,w$ yields a change of $|\varphi^v|\,|\varphi^w|$ in $\eta_1$ from swapping the corresponding basis vectors in the vertex order. It also yields a change of $(n-3)(|\varphi^v|+|\varphi^w|)$ in $\eta_2$.
%
\end{enumerate}
Now we {\em define} $p_{2,1,0}:\BC[1] \otimes \BC[1] \to \BC[1]$ by summation over all graphs $\Gamma\in RG_{2,1,0}$, that is graphs with two interior vertices and exactly one interior edge connecting them.
According to (1) and (6) above, reversing the order of the interior vertices and the orientation of the interior edge simultaneously has the effect of changing the sign exponent by $(|\varphi^1|+n-3) (|\varphi^2|+n-3)$, which gives exactly the correct sign for the standard action of $S_2$ on $\BC[1] \otimes \BC[1]$. So $p_{2,1,0}$ is symmetric in its inputs, and descends to a map 
$$
\fp_{2,1,0}:E_2\BC \to E_1\BC.
$$
Similarly, we {\em define} $p_{1,2,0}:\BC[1] \to \BC[1] \otimes \BC[1]$ by summation over all graphs $\Gamma\in RG_{1,2,0}$, i.e. graphs with one interior vertex and one interior edge (which then necessarily is a loop at that vertex) Then according to (1) and (5) above, the map takes values in the invariant part of $\BC[1] \otimes \BC[1]$, which we identify with $E_2\BC$ to get a map
$$
\fp_{1,2,0}:E_1\BC \to E_2\BC.
$$
\begin{rem}\label{rem:proof-mu-delta}
Explicitly, the maps just defined are given by
\begin{align*}
   \fp_{2,1,0}(\varphi^1,\varphi^2)(x) 
   &= \sum_{ab}(-1)^{\eta_b|x_{(1)}|+(n-3)|\phi|+\eta_a}
   g^{ab}\varphi^1(e_a,x_{(1)})\varphi^2(e_b,x_{(2)}) + \text{cyclic}, \cr
   \fp_{1,2,0}(\varphi)(x,y) 
   &= \frac 12 \sum_{ab}(-1)^{\eta_b|x|+(n-3)|\varphi_{(1)}|+\eta_a}
   g^{ab}\varphi(e_a,x,e_b,y) + \text{cyclic}. 
\end{align*}
One sees that these definitions agree with the ones previously given. In particular, our sign considerations have validated the symmetry properties of $p_{2,1,0}$ and $p_{1,2,0}$, which are equivalent to the assertions of Lemma~\ref{lem:mu-delta}.
\end{rem}

We are now ready to prove Proposition~\ref{prop:structureexists}.

\begin{proof}[Proof of Proposition~\ref{prop:structureexists}]
The maps $\fp_{1,1,0}$, $\fp_{1,2,0}$ and $\fp_{2,1,0}$ extend
uniquely to maps $\hat\fp_{1,1,0}$ (both a derivation and a
coderivation), $\hat\fp_{1,2,0}$ (a derivation) and  $\hat\fp_{2,1,0}$
(a coderivation), all defined on $E\BC$, respectively.
It remains to prove that the following maps vanish (with $\circ_s$
defined as in Section~\ref{sec:def}):
$$
\aligned
\fp_{1,1,0} \circ_1 \fp_{2,1,0} + \fp_{2,1,0} \circ_1
\hat\fp_{1,1,0}&:E_2 \BC \to E_1\BC\\
\hat\fp_{1,1,0} \circ_1 \fp_{1,2,0} + \fp_{1,2,0} \circ_1
\fp_{1,1,0}&:E_1\BC \to E_2\BC\\
\fp_{2,1,0} \circ_1 \hat\fp_{2,1,0}&: E_3\BC \to
E_1\BC\qquad\text{(Jacobi)}\\
\hat\fp_{1,2,0} \circ_1 \fp_{1,2,0}&: E_1\BC \to
E_3\BC\qquad\text{(co-Jacobi)}\\
\fp_{1,2,0} \circ_1 \fp_{2,1,0} + 
\hat\fp_{2,1,0} \circ_1 \hat\fp_{1,2,0}&: E_2\BC \to E_2\BC\qquad\text{(Drinfeld)}\\
\fp_{2,1,0} \circ_2 \fp_{1,2,0}&: E_1\BC \to E_1\BC\qquad\text{(involutivity)}.
\endaligned
$$
\par
To discuss the first equation, i.e. compatibility of the bracket with the boundary map, we write out the equation more explicitly:
\begin{align*}
   \fp_{1,1,0}(\fp_{2,1,0}(\phi,\psi))
   + \fp_{2,1,0}(\fp_{1,1,0}\phi,\psi)
   + (-1)^{(|\phi|+(n-3))}\fp_{2,1,0}(\phi,\fp_{1,1,0}\psi)=0.
\end{align*}
Let us first look at corresponding terms in the first and second summands:
The $\eta_1(\fp_{2,1,0})$ for the second term differs from $\eta_1(\fp_{2,1,0})$ for the first term by $\eta_b$ because of the difference of degrees of the first argument of $\fp_{2,1,0}$. For the same reason,  $\eta_2$ of these two terms differs by $n-3$. Finally, the sign exponent in the application of $\fp_{1,1,0}$ in these two terms differs by $\eta_a$ because in the second term the differential has to be moved past the additional argument $e_a$ in the first slot. This gives a total difference in sign exponents of
$$
\eta_a+\eta_b+(n-3)=1,
$$
so that corresponding terms cancel. 

One can similarly compare corresponding terms in the first and third summands:
Here $\eta_2$ is the same for both summands, but for the application of $d$ there is an additional contribution of $|\phi|-\eta_a$ in the first term (from moving $d$ past the part of the inputs which feeds into $\phi$) and an additional contribution of $\eta_b$ in the third term from moving $d$ past the $e_b$ in the first slot. Together with the external exponent, we again get exactly the total difference of $1$ as above, meaning that corresponding terms again cancel.

Finally, there are two more terms, one each in the second and third summands, coming from applying the differential $d$ to $e_a$ and $e_b$ respectively. Lumping the arguments that feed into $\phi$ and $\psi$ together into words $x$ and $y$ respectively, and ignoring the identical part of the sign coming from cyclic permutation of the inputs, we see that the two terms are of the form
\begin{align*}
     \lefteqn{(-1)^{\eta_b|x|+(n-3)(|\phi|+1)}\bar g^{ab}
   \phi(de_a,x)\psi(e_b,y)=}&\cr
   &(-1)^{\eta_b|x|+(n-3)(|\phi|+1)}\bar g^{ab}d_a^{a'}
   \phi(e_{a'},x)\psi(e_b,y), \qquad\text{and}\cr
   \lefteqn{(-1)^{\eta_{b'}|x|+(n-3)(|\phi|+1) + |\phi|}\bar g^{a'b'}
   \phi(e_{a'},x)\psi(de_{b'},y)=}\cr
   &(-1)^{\eta_{b'}|x|+(n-3)(|\phi|+1) + |\phi|}\bar g^{a'b'}d_{b'}^b
   \phi(e_{a'},x)\psi(e_b,y).
\end{align*}
Using $\eta_b=\eta_{b'}+1$, $\eta_{a'}=\eta_a+1$, and
$|\phi|=|x|+\eta_a+1$, we see that these terms cancel using equation \eqref{eq:littleg} of Lemma~\ref{lem:edge1}.

A similar discussion proves the compatibility of $\fp_{1,1,0}$ and $\fp_{1,2,0}$.

The remaining four equations are the Jacobi, co-Jacobi, Drinfeld and involutivity relations, respectively.
To prove them, we will argue by cut-and-paste techniques on suitable ribbon graphs and their associated surfaces.

{\bf Jacobi and co-Jacobi. }
Writing out the Jacobi identity $\fp_{2,1,0} \circ_1 \hat{\fp}_{2,1,0}=0$ yields the equation
\begin{align}
\fp_{2,1,0} (\fp_{2,1,0}(\phi,\psi),\theta)+
(-1)^{(|\phi|+n-3)(|\psi|+|\theta|)}\fp_{2,1,0}(\fp_{2,1,0}(\psi,\theta),\phi)+ &\notag\\
(-1)^{(|\theta|+n-3)(|\phi|+|\psi|)}\fp_{2,1,0}(\fp_{2,1,0}(\theta,\phi),\psi)&=0.
\label{eq:jacobi_explicit}
\end{align}
The possible configurations of interior edges for these compositions are depicted in Figure~\ref{fig:jacobi_3terms}, where we have left off the exterior edges for clarity.
%
%
\begin{figure}[ht!]
 \labellist
  \pinlabel $\phi$ [b] at 5 48
  \pinlabel $\psi$ [b] at 55 12
  \pinlabel $\theta$ [b] at 107 48 
  \pinlabel $\phi$ [b] at 128 48
  \pinlabel $\theta$ [b] at 179 12
  \pinlabel $\psi$ [b] at 230 48 
  \pinlabel $\psi$ [b] at 249 48
  \pinlabel $\phi$ [b] at 299 12
  \pinlabel $\theta$ [b] at 351 48 
 \endlabellist
  \centering
  \includegraphics[scale=.83]{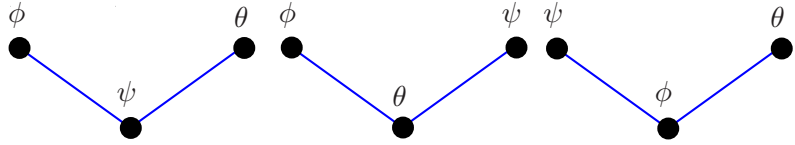}
 \caption{Possible interior parts of the graph in the compositions for the Jacobi identity.}
 \label{fig:jacobi_3terms}
\end{figure}
The first term in the Jacobi identity contains contributions from the first and third graphs, the second term from the first and second graphs, and the last term from the second and third graphs. 
To compare the signs of the contributions to the first and second summand in the Jacobi identity corresponding to the first graph in Figure~\ref{fig:jacobi_3terms}, we (use the symmetry properties of $\fp_{2,1,0}$ to) rewrite the second summand as 
$$
(-1)^{(|\phi|+n-3)}\fp_{2,1,0}(\phi,\fp_{2,1,0}(\psi,\theta)).
$$
The two possible compositions are depicted schematically in Figure~\ref{fig:jacobi_cutting}.
%
%
\begin{figure}[ht!]
 \labellist
  \pinlabel $\phi$ [bl] at 22 115
  \pinlabel $\psi$ [bl] at 102 115
  \pinlabel $\theta$ [bl] at 186 115 
 \endlabellist
  \centering
  \includegraphics[scale=.83]{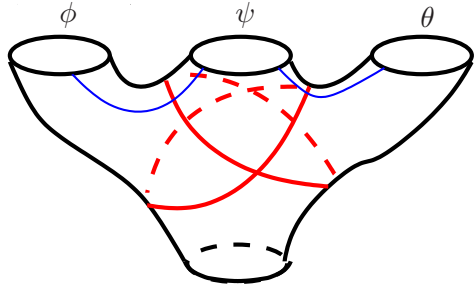}
 \caption{Two of the composition in the Jacobi identity that yield the same overall operation up to sign (which corresponds to the first graph in Figure~\ref{fig:jacobi_3terms}).}
 \label{fig:jacobi_cutting}
\end{figure}
Consider first the composition corresponding to the cut separating $\phi$ and $\psi$ from $\theta$, which contributes to the first summand in \eqref{eq:jacobi_explicit}. Reading the figure from top to bottom, one sees that successively applying the two operations the sum of the $\eta_1$-parts of the sign yield the correct sign to switch from the vertex order
$$
\prod_{v=1}^3 e_{\alpha(v,1)} \cdots e_{\alpha(v,d(v))}
$$
to the edge order
$$
e_{a_1}e_{b_1}e_{a_2}e_{b_2}e_{i_1}\cdots e_{i_{k_1+k_2+k_3}},
$$
where $e_{a_1}$ and $e_{b_1}$ are the labels of the left edge (connecting $\phi$ and $\psi$) and $e_{a_2}$ and $e_{b_2}$ are attached to the ends of the right edge (connecting $\psi$ and $\theta$).

On the other hand, consider the other possible cut. The $\eta_1$-part of the sign for the first operation now corresponds to moving from the above vertex order to
$$
e_{\alpha(1,1)} \cdots e_{\alpha(1,k_1+1)}e_{a_2}e_{b_2}e_{\gamma(1)}\cdots e_{\gamma(k_1+k_2+1)},
$$
where $e_{\gamma(i)}$ are the labels of the intermediate exterior vertices created in the cutting process. Moving $e_{a_2}e_{b_2}$ to the front yields an extra sign of $|\phi|(n-2)$, and then the $\eta_1$-part of the sign for applying the second operation yields the reordering into
$$
e_{a_2}e_{b_2}e_{a_1}e_{b_1}e_{i_1}\cdots e_{i_{k_1+k_2+k_3}}.
$$
Comparing this with the previous outcome, we need an additional exponent of $n-2$ to exchange $e_{a_1}e_{b_1}$ with $e_{a_2}e_{b_2}$. So the difference in the $\eta_1$-component of the sign is $|\phi|(n-2)+(n-2)$.
The sum of the $\eta_2$-terms for the two operations in the first case is 
$$
(n-3)|\phi|+(n-3)(|\phi|+|\psi|+2-n)=(n-3)|\psi|,
$$
and in the second case it is
$$
(n-3)(|\phi|+|\psi|).
$$
In total, the difference in sign exponent of the two ways of producing this output is (also taking into account the ``external sign'' of the second term)
$$
\underbrace{|\phi|(n-2)+n-2}_{\text{difference in }\eta_1} +\underbrace{(n-3)|\phi|}_{\text{difference in }\eta_2}+|\phi|+n-3 = 1,
$$
so these terms cancel. Similar discussions apply to the other pairs of terms, and also to the co-Jacobi identity.
%
%
\begin{figure}[ht!]
  \centering
  \includegraphics[scale=.83]{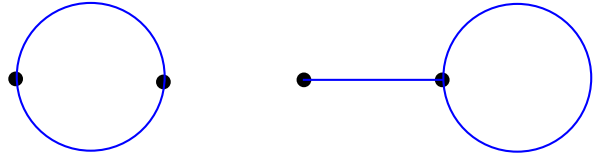}
 \caption{Possible interior parts of the graph in the compositions for the Drinfeld identity. The right hand graph comes in four flavours, depending on the orientation of the edges, yielding the four terms in $\hat{\fp}_{2,1,0} \circ_1 \hat{\fp}_{1,2,0}$, which also appear in $\fp_{1,2,0} \circ \fp_{2,1,0}$. The left hand term represents the self-cancelling part of $\fp_{1,2,0} \circ \fp_{2,1,0}$.}
 \label{fig:drinfeld_graphs}
\end{figure}

{\bf Drinfeld. }
We next prove the Drinfeld compatibility between bracket and cobracket. 
Using the common short hand notation $\fp_{1,2,0}\phi= \phi_{(1)} \otimes \phi_{(2)}$ etc, it takes the explicit form
\begin{align*}
0= &\,\fp_{1,2,0} \circ \fp_{2,1,0}(\phi,\psi) \\
&+(-1)^{|\phi_{(1)}|+n-3} \phi_{(1)} \otimes \fp_{2,1,0}(\phi_{(2)},\psi) \\
&+(-1)^{(|\phi_{(2)}|+n-3)(|\psi|+n-3)} \fp_{2,1,0}(\phi_{(1)},\psi) \otimes \phi_{(2)}\\
&+(-1)^{(|\phi|+n-3)(|\psi|+n-3)+|\psi_{(1)}|+n-3} \psi_{(1)} \otimes \fp_{2,1,0}(\psi_{(2)},\phi) \\
&+(-1)^{|\phi|+n-3} \fp_{2,1,0}(\phi,\psi_{(1)})\otimes \psi_{(2)}.
\end{align*}
This time, the possible interior parts of the underlying graphs are shown in Figure~\ref{fig:drinfeld_graphs}.
Let us consider the configuration in Figure~\ref{fig:drinfeld}, which contributes to both $\fp_{1,2,0} \circ \fp_{2,1,0}(\phi,\psi)$ and 
$\phi_{(1)} \otimes \fp_{2,1,0}(\phi_{(2)},\psi) $.
%
%
\begin{figure}[ht!]
  \centering
  \includegraphics[scale=.83]{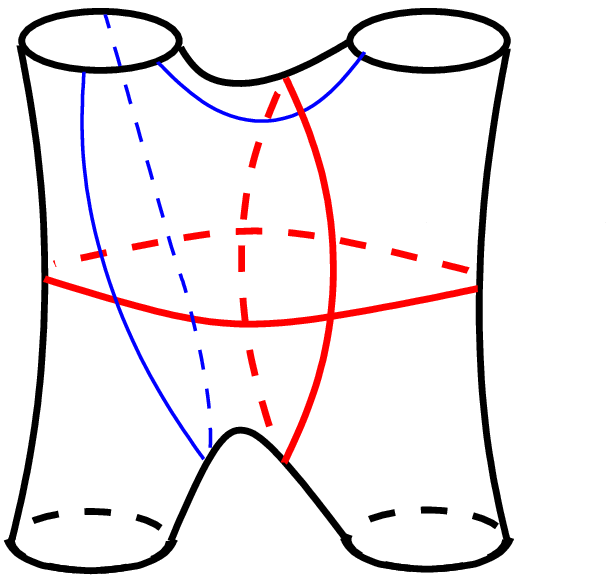}
 \caption{A possible configuration of interior edges appearing in the Drinfeld compatibility equation and the two gluing that give rise to it.}
 \label{fig:drinfeld}
\end{figure}
The $\eta_1$-part of the sign of the composition $\fp_{1,2,0}\circ \fp_{2,1,0}(\phi,\psi)$ (the horizontal cut) corresponds to moving from the vertex order
$$
\prod_{v=1}^2 e_{\alpha(v,1)} \cdots e_{\alpha(v,d(v))}
$$
to the edge order
$$
e_{a_1}e_{b_1}e_{a_2}e_{b_2}\prod_{b=1}^2 e_{\beta(b,1)} \cdots e_{\beta(b,s(b))},
$$
and the $\eta_2$-part of the sign exponent is simply $(n-3)(|\phi|+t_1)$, where $t_1$ is the total $B^{\rm cyc*}A$-degree of the first output (which turns out to be $|\phi_{(1)}|$ from the second point of view).

Let us now consider the vertical cut, which corresponds to $(-1)^{|\phi_{(1)}|+n-3} \phi_{(1)} \otimes \fp_{2,1,0}(\phi_{(2)},\psi)$. The $\eta_1$ part of the sign for $\fp_{1,2,0}(\phi)$ corresponds to moving from the above vertex order to
$$
e_{a_2}e_{b_2} e_{\beta(1,1)} \cdots e_{\beta(1,s(1))} e_{\gamma(1)} \cdots e_{\gamma(r)} e_{\alpha(2,1)} \cdots e_{\alpha(2,d(2))},
$$
and the $\eta_1$ part of the sign of the $\fp_{2,1,0}$-part of the operation allows one to move this to
$$
e_{a_2}e_{b_2} e_{\beta(1,1)} \cdots e_{\beta(1,s(1))} e_{a_1}e_{b_1} e_{\beta(2,1)} \cdots e_{\beta(2,s(2))}.
$$
To get to the same order as in the first case, we need to move $e_{a_1}e_{b_1}$ to the front, yielding an extra contribution to the sign exponent of $(n-2)(|\phi_{(1)}|+n-2)$. This time, the $\eta_2$-part of the sign exponent equals $(n-3)(|\phi_{(1)}|+ |\phi_{(2)}|)$. Together with the ``external sign'', the total difference in sign exponents is
$$
\underbrace{n-2 +|\phi_{(1)}|(n-2)}_{\text{difference in }\eta_1}
+\underbrace{(n-3)(|\phi|+|\phi_{(2)}|)}_{\text{difference in }\eta_2}+|\phi_{(1)}|+n-3 = 1,
$$
because $|\phi|=|\phi_{(1)}|+|\phi_{(2)}|+2-n$. Hence the two contributions cancel. Similar discussions apply to the other three terms involving the second graph in Figure~\ref{fig:drinfeld_graphs}.

To complete the proof of Drinfeld compatibility, it remains to discuss the contributions of the first graph in Figure~\ref{fig:drinfeld_graphs}. 
In Figure~\ref{fig:drinfeld_cancelling}, we show the two possible gluings which yield this configuration, this time with the numbering of vertices and the orientations of the edges included. The latter are chosen so that the outer boundary component is the first output in both cases. Notice that here the difference in sign comes solely from changing the order of the edges and changing the orientation of one of them, which according to points ($2_g$) and (1) in the discussion of signs above yield sign exponents of $n-2$ and $n-3$, respectively, for a total difference of $1$ as needed for cancellation.
%
%
\begin{figure}[ht!]
 \labellist
  \small\hair 2pt
  \pinlabel $1$ [b] at 35 109
  \pinlabel $2$ [b] at 154 109
  \pinlabel $1$ [b] at 268 90
  \pinlabel $2$ [b] at 392 90
 \endlabellist
  \centering
  \includegraphics[scale=.73]{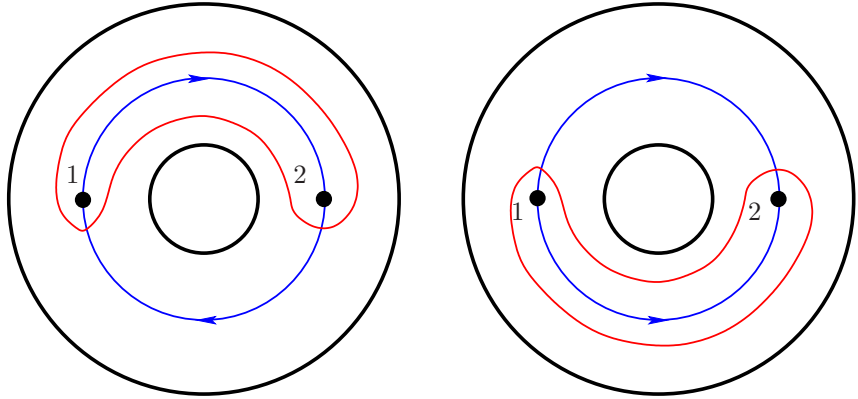}
 \caption{The two ways of obtaining the same output from the first graph in Figure~\ref{fig:drinfeld}, with orientations of edges given.}
 \label{fig:drinfeld_cancelling}
\end{figure}

{\bf Involutivity. }
The involutivity relation follows from an analogous argument, but this time applied to the underlying graph depicted in Figure~\ref{fig:involutive}, which gives rise to a genus one surface. 
%
%
\begin{figure}[ht!]
  \centering
  \includegraphics[scale=.83]{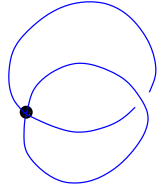}
 \caption{The graph corresponding to the composition in the involutivity relation.}
 \label{fig:involutive}
\end{figure}
Here the orientation of the first edge determines the order of the outputs of the first operation, which by our conventions forces the orientation of the second edge. Again, switching the order of the edges forces the reversal of one of the edge orientations for consistency.

This concludes our proof of Proposition~\ref{prop:structureexists}.
\end{proof}

In the above proof we used only rather elementary graphs from our
``graphical calculus''. General ribbon graphs will make their
appearance in the following section. 
\medskip

\section{The \text{\rm dIBL} structure associated to a subcomplex. } 
\label{sec:subcomplex}

In this section we relate the \text{\rm dIBL} structure associated to 
$(A,\la\ ,\ \ra,d)$ to that of suitable subcomplexes $B \subset A$
having the same homology. Our main result of this section states that
the $\text{\rm dIBL}$-structure on such a subcomplex is $\text{\rm
  IBL}_\infty$-homotopy equivalent to the original one. We closely
follow~\cite[Section 6.4]{KoSo01}. 

Let $(A,\la\ ,\ \ra,d)$ be as above, i.e.~
\begin{equation}\label{eq:A}
    \langle dx,y \rangle + (-1)^{|x|}\langle x,dy \rangle = 0, \qquad
    \langle x,y \rangle = -(-1)^{|x| |y|}\langle y,x \rangle.
\end{equation}
Let $B\subset A$ be a subcomplex
such that the restriction of $\la\ ,\ra$ to $B$ is nondegenerate. This
means that $B$ is the image of a projection $\Pi:A\to A$ satisfying
$\Pi^2=\Pi$ and
\begin{equation}\label{eq:Pi}
    \Pi d=d\Pi,\qquad \langle\Pi x, y\rangle = \langle
    x, \Pi y\rangle. 
\end{equation}
We assume in addition the existence of a chain homotopy $G : A^* \to
A^{*-1}$ such that 
\begin{equation}\label{eq:G}
    dG + Gd = \Pi - \id, \qquad
    \langle G x, y\rangle = (-1)^{|x|}\langle x,  G y\rangle.
\end{equation}
Note that conditions~\eqref{eq:G} imply conditions~\eqref{eq:Pi}. 
It follows that the inclusion $i:B\to A$ and the projection $p:A\to B$
are chain homotopy inverses of each other, in particular they induce
isomorphisms on cohomology. We will be mostly interested in the case
that $B$ is isomorphic to the cohomology of $(A,d)$, which is possible
due to the following

\begin{lem}\label{lem:B}
There exists a subcomplex $B\subset \ker d\subset A$ satisfying
conditions~\eqref{eq:Pi} and~\eqref{eq:G} such that 
$$
    \ker d = \im d\oplus B. 
$$ 
\end{lem}

\begin{proof}
The proof is a straightforward exercise in linear algebra. 
Note first that the orthogonal complements with respect to $\la\ ,\
\ra$ satisfy
$$
    (\im d)^\perp = \ker d,\qquad (\ker d)^\perp = \im d. 
$$
We will construct subspaces $B\subset\ker d$ and $C\subset A$ with
the following properties: 
\begin{equation}\label{eq:BC}
    A = \ker d\oplus C,\qquad \ker d = \im d\oplus B, \qquad C
    \perp B\oplus C. 
\end{equation}
Given such subspaces, it follows from $\ker d\perp\im d$ that $B \perp
\im d\oplus C$. Let $\Pi:A\to A$ be the orthogonal projection 
onto $B$ and define $G:A\to A$ with respect to the decomposition
$A=\im d\oplus B\oplus C$ by 
$$
    G(dz,b,c) := (0,0,-z),\qquad c,z\in C,\ b\in B. 
$$
Then it is easy to verify that $\Pi$ and $G$ satisfy
conditions~\eqref{eq:Pi} and~\eqref{eq:G}.

Subspaces $B,C$ satisfying~\eqref{eq:BC} can be constructed
directly. A more conceptual argument is based on the following

{\em Claim. }There exist linear operators 
$$
    *: A^k\to A^{n-k},\qquad k\in\Z,
$$
such that $(\cdot,\cdot) := \la\cdot,*\cdot\ra$
is a positive definite inner product on $A$ and
$$
    *^2 = (-1)^{k(n-k)+n}\id:A^k\to A^k. 
$$ 
To construct $*$, suppose first $k<n/2$. Pick any metric $(\ ,\ )$ on
$A^k$. It induces an isomorphism
$$
    I:A^k\to(A^k)^*,\qquad y\mapsto(\cdot,y). 
$$ 
Similarly, $\la\ ,\ \ra$ induces an isomorphism
$$
    J:A^{n-k}\to(A^k)^*,\qquad y\mapsto\la\cdot,y\ra. 
$$
Denote by $(\ ,\ )$ the induced metric on $A^{n-k}$ via the
isomorphism $I^{-1}J:A^{n-k}\to A^k$, i.e.
$$
    (x,y) := (I^{-1}Jx,I^{-1}Jy),\qquad x,y\in A^{n-k}. 
$$
Define $*$ on $A^k\oplus A^{n-k}$ by
$$
    *:=J^{-1}I:A^k\to A^{n-k},\qquad *:=(-1)^{k(n-k)+n-3}I^{-1}J:A^{n-k}\to A^k. 
$$
Then $*^2=(-1)^{k(n-k)+n}\id$ on $A^k$. For $x,y\in A^k$ we compute
\begin{gather*}
   \la x,*y\ra = \la x,J^{-1}Iy\ra = (JJ^{-1}Iy)(x) = (Iy)(x) =
   (x,y), \cr
   \la *x,**y\ra = (-1)^{k(n-k)+n}\la *x,y\ra = \la y,*x\ra = (y,x) =
   (x,y) = (*x,*y). 
\end{gather*}
Thus $(\cdot,\cdot) = \la\cdot,*\cdot\ra$,
so $*$ has the desired properties. 

For $n$ even and $k=n/2$ we distinguish two cases. If $k$ is even,
then $\la\ ,\ \ra$ is symmetric on $A^k$ and there exists a basis $(e_i)$ of
$A^k$ with $\la e_i,e_j\ra=\pm\delta_{ij}$; then $*e_i:=\pm e_i$ has
the desired properties. If $k$ is odd, then $\la\ ,\ \ra$ is
symplectic on $A^k$ and there exists a symplectic basis $(e_i,f_i)$ of
$A^k$ with $\la e_i,e_j\ra=0=\la f_i,f_j\ra$ and $\la
e_i,f_j\ra=\pm\delta_{ij}$; then $*e_i:=f_i$, $*f_i:=-e_i$ has 
the desired properties. This proves the claim. Note that so far we
have not used the operator $d$.

From the claim the lemma follows by standard Hodge theory
arguments: Define the adjoint operator $d^*:A^{k+1}\to A^k$ of $d$ by 
$$
    (d^*x,y) = (x,dy),\qquad x\in A^{k+1},y\in A^k. 
$$
It follows that
$$
    d^* = \pm*d*:A^{k+1}\to A^k
$$
and 
$$
    \la d^*x,y\ra = \pm\la x,d^*y\ra,\qquad x\in A^{k+1},y\in
    A^{n-k}. 
$$
Define the Laplace operator $\Delta:=dd^*+d^*d$ and 
$$
    B := \ker\Delta = \ker d\cap \ker d^*,\qquad
    C := \im d^*.
$$
Then the decomposition $A = \im d\oplus B \oplus C$
is orthogonal with respect to $(\ ,\ )$ and
satisfies~\eqref{eq:BC}. Note that the operator $G$ is explicitly
given by
$$
    G = -d^*\Delta^{-1} = -\Delta^{-1}d^*,
$$
where $\Delta^{-1}$ is zero on $B$ and the inverse of $\Delta$ on $\im
d\oplus C$. This concludes the proof of Lemma~\ref{lem:B}. 
\end{proof}

Given a choice of basis $e_i$ of $A$ and the dual basis $e^i$, we set
$$
G^{ab}:=\langle Ge^a,e^b\rangle.
$$
Then we have $G^{ab} \neq 0$ only if $\eta_a+\eta_b=n-3$. Moreover,
from \eqref{eq:G} and Definition \ref{cyccochaincmx} one deduces the
symmetry properties 
\begin{equation}\label{eq:symG}
G^{ba}= (-1)^{\eta_a\eta_b+n-3}G^{ab}.
\end{equation}

\begin{lem}\label{lem:edge2}
The equation $dG+Gd=\Pi-\id$ translates into the identity
\begin{equation}\label{edge2}
d^a_{a'} G^{a'b} + (-1)^{\eta_a} d^b_{b'} G^{ab'} = (-1)^{\eta_a} g^{\bar a \bar b}-(-1)^{\eta_a} g^{ab},
\end{equation}
where in the first term on the right hand side the bar signifies that we take the inner product of the images in the subcomplex, i.e.
$$
   g^{\bar a\bar b} := \langle \Pi e^a, \Pi e^b\rangle.
$$
\end{lem}

\begin{proof}
We do a straightforward computation:
\begin{align*}
\langle dG e^a,e^b \rangle &= (-1)^{|e^a|}\langle Ge^a,de^b\rangle \\
&= (-1)^{|e^a|+\eta_b}d_{b'}^b\langle Ge^a,e^{b'}\rangle = d_{b'}^b G^{ab'}
\end{align*}
(since we need $|e^a|=|e_b|$ for the term to be nonzero) and 
\begin{align*}
\langle Gd e^a,e^b \rangle &= (-1)^{\eta_a}d_{a'}^a\langle Ge^{a'},e^b\rangle \\
&= (-1)^{\eta_a}d_{a'}^aG^{a'b}. 
\end{align*}
Now multiplying the equation
$$
\langle (dG+Gd)e^a,e^b \rangle = \langle\Pi e^a,e^b\rangle - \langle e^a,e^b\rangle
$$
by $(-1)^{\eta_a}$ gives the claim.
\end{proof}

Now we return to $(A,\la\ ,\ \ra,d)$ and a subcomplex $B\subset A$
satisfying conditions~\eqref{eq:Pi} and~\eqref{eq:G}. We denote the
induced structures on $B$ by $d^B$ and $\ \la\ ,\ \ra$. They again
satisfy condition~\eqref{eq:A}. Therefore,
Proposition~\ref{prop:structureexists} equips $(B^{\text{\rm cyc}*}B)[2-n]$ with a
$\text{\rm dIBL}$ algebra structure.

The next theorem, which corresponds to
Theorem~\ref{thm:homotopyequiv-intro} from the Introduction, is the
main result of this section. 

\begin{Theorem}\label{thm:homotopyequiv}
There exists an $\IBL_{\infty}$-homotopy equivalence 
$$
    \frak f : (B^{\text{\rm cyc}*}A)[2-n] \to (B^{\text{\rm cyc}*}B)[2-n]
$$
such that $\frak f_{1,1,0} : (B^{\text{\rm cyc}*}A)[2-n] \to (B^{\text{\rm cyc}*}B)[2-n]$ 
is the map induced by the dual of the inclusion $i:B \to A$.
\end{Theorem}


\begin{proof}
We provide two proofs of Theorem~\ref{thm:homotopyequiv}.
We first give a short proof.
Take a chain map $j:A \to B$ which is orthogonal with respect to
the inner product and is a left inverse to $i:B \to A$.
Set $\frak g_{1,1,0} := j^*:(B^{\text{\rm cyc}*}B)[2-n] \to
(B^{\text{\rm cyc}*}A)[2-n]$ and $\frak g_{k,\ell,g} := 0$ for
$(k,\ell,g) \ne (1,1,0)$. 
It is easy to see from the definition that this defines an 
$\IBL_{\infty}$-morphism $(B^{\text{\rm cyc}*}B)[2-n] \to
(B^{\text{\rm cyc}*}A)[2-n]$ such that $\frak g_{1,1,0}$ induces an
isomorphism on homology. Therefore, by Theorem~\ref{Whi}, 
it has a homotopy inverse $\frak f$. We can arrange $\frak
f_{1,1,0} = i^*$ by choosing $\frak f_{1,1,0}$ this way in the first
step of the proof of Proposition~\ref{prop:inverse2}. 
Let us emphasize that the same proof does not work in the opposite
direction, with the inclusion $i$ in place of $j$. This explains the
appearance of the nontrivial terms $\frak f_{k,\ell,g}$ in the second
proof. 
\par\medskip

We next discuss another proof of Theorem~\ref{thm:homotopyequiv},
which will occupy most of the remainder of this section.
This proof provides an explicit description of the map $\frak f$.  
We think this explicit description is interesting 
because of its relation to perturbative Chern-Simons theory,
as we explain in Section~\ref{de Rham} during the discussion of 
Conjecture~\ref{deRhamconj-intro}. Also, it is likely to be useful for
the generalization of Theorem ~\ref{thm:homotopyequiv} to the 
case when $A$ has infinite dimension.
\par
We will construct $\frak f$ by summation over general ribbon graphs. 
Similar constructions using ribbon trees are well-known, see
e.g.~\cite{KoSo01} and~\cite[Subsection 5.4.2]{FOOO06}. 

Since the $G^{ab} $ satisfy the symmetry relation
\eqref{eq:tensorsymmetry}, we can apply the procedure described in the previous section to associate a map  
$$
f_\Gamma: (B^{\rm cyc*}A[2-n])^{\otimes k} \to (B^{\rm cyc*}B[2-n])^{\otimes
  \ell}
$$ 
to any ribbon graph $\Gamma\in RG_{k,\ell,g}$ via the formula
\eqref{Fdefstatesum}, i.e. 
\begin{align}\label{def:fGamma}
    &\ \ \ (f_\Gamma(\varphi^1 \otimes \cdots \otimes
  \varphi^k))_{\overrightarrow{\beta(1)};\cdots;
    \overrightarrow{\beta(\ell)}}:=\\ 
    & \frac 1 {\ell !\,|\Aut(\Gamma)|\,\prod_v d(v)}\sum(-1)^\eta \left(\prod_{t\in
    C_1^\inn(\Gamma)} G^{a_tb_t}
    \prod_{v \in C_0^{\text{int}}(\Gamma)} \varphi^v_{\alpha(v,1) \cdots \alpha(v,d(v))}\right),
\notag
\end{align}
with the conventions as before. Recall that in this definition we sum
over all labellings of $\Gamma$ in the sense of
Definition~\ref{def:labelling}. 

The signs also depend on choices of an ordering and orientations for
the interior edges. We will now specify these in such a way that
expression~\eqref{def:fGamma} becomes independent of these additional
choices.  

Given a ribbon graph $\Gamma\in RG_{k,\ell,g}$, we have its associated ribbon surface $\Sigma_\Gamma$, which is already determined by the subgraph $\Gamma_{\rm int} \subset \Gamma$ of interior edges. Collapsing the boundary components of $\Sigma_\Gamma$ to points results in a closed oriented surface $\widehat{\Sigma}_\Gamma$ of genus $g$, which comes with a cell decomposition whose vertices and edges correspond to the vertices and edges of $\Gamma_{\rm int}$, and whose 2-cells correspond bijectively to the boundary components of $\Sigma_\Gamma$. Denote by $\Gamma_{\rm int}^*$ the dual graph, whose vertices correspond to the boundary components of $\Sigma_\Gamma$ and whose edges are transverse to those of $\Gamma_{\rm int}$. 

Choose a maximal tree $T \subset \Gamma_{\rm int}$, which will have $k-1$
edges, and a maximal tree $T^*\subset \Gamma_{\rm int}^*$ disjoint from $T$, 
which will have $\ell-1$ edges. Denote by $\Gamma' \subset \Gamma_{\rm int}$ 
the subgraph of edges from $T$ and of edges dual to those of $T^*$.
The graph $\Gamma_{\rm int}$ has exactly $2g$ further edges. When added to 
$\Gamma'$, each one of them determines a unique cycle in 
$\Gamma_{\rm int} \subset \Sigma_\Gamma$, and these cycles 
form a basis for $H_1(\widehat{\Sigma}_\Gamma)$. 

\begin{Definition}[ordering and orientation of edges]\label{def:edge-ordering}
In formula \eqref{def:fGamma}, we allow any ordering and orientation of 
interior edges which arises from choices of $T$ and $T^*$ according to the following rules:
\begin{itemize}
\item The oriented edges $e_1,\dots,e_{k-1}$ are the edges of $T$, oriented 
away from vertex 1 and numbered such that $e_i$ ends at 
vertex $k+1-i$. In other words, they are numbered in {\em decreasing order} 
of the vertex they point to. 
\item We orient the edges of $T^*$ away from the first boundary component 
and label them in {\em increasing order} of the boundary component they 
point to, so that $e^*_{k+s-2}$ points to the boundary component $s$. The 
oriented edges $e_{k},\dots,e_{k+\ell-2}$ are obtained as the dual edges to 
the $e^*_i$, oriented so that the pair $\{e^*_i,e_i\}$ defines the orientation 
of the surface $\Sigma_\Gamma$. 
\item Finally we choose the order and orientation of the remaining edges 
$e_{k+\ell-1},\dots,e_{k+\ell+2g-2}$ compatible with the symplectic structure 
on $H_1(\widehat{\Sigma}_\Gamma)$ corresponding to the intersection pairing. 
\end{itemize}
\end{Definition}



%
%
\begin{figure}[ht!]
 \labellist
  \pinlabel $1$ [r] at 106 88
  \pinlabel $2$ [l] at 162 88
  \pinlabel $1$ [t] at 73 101
  \pinlabel $2$ [t] at 198 101
  \pinlabel $3$ [t] at 140 161
  \pinlabel $e_1$ [b] at 131 88
  \pinlabel $e_2$ [b] at 195 124
  \pinlabel $e_3$ [b] at 72 124
  \pinlabel $e_2^*$ [t] at 162 137
  \pinlabel $e_3^*$ [t] at 110 137
 \endlabellist
  \centering
  \includegraphics[scale=.83]{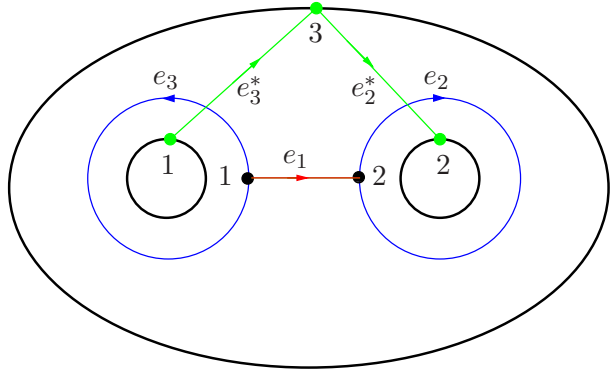}
 \caption{For this particular graph $\Gamma\in RG_{2,3,0}$, the choices of trees $T$ (in red) and $T^*$ (in green) are unique (exterior edges were omitted for clarity). We also show the numbering and orientation of the interior edges resulting from the given numbering of interior vertices and boundary components.}
 \label{fig:trees}
\end{figure}

Of course, the orientations and order of the edges obtained in this way depend on the choices of the trees $T$ and $T^*$. Note that for $g=0$ the tree $T^*$ is uniquely determined by the choice of $T$, and that the conventions here agree with those used for the graphs in $RG_{2,1,0}$ and $RG_{1,2,0}$ in the definition of $\fp_{2,1,0}$ and $\fp_{1,2,0}$.

\begin{lem}\label{lem:edge-convention}
Let $\Gamma,\wt\Gamma\in RG_{k,\ell,g}$ correspond to the same graph,
with the numberings of the interior vertices and boundary components
differing by permutations $\sigma\in S_k$ and $\tau\in S_\ell$,
respectively. Consider 
pairs of maximal trees $(T,T^*)$ and $(\wt T,\wt T^*)$ as above
corresponding to $\Gamma$ and $\wt\Gamma$, respectively, and their
induced orderings and orientations of edges. Let $r$ be the number of
edges whose orientations differ in the two conventions, and let
$\rho \in S_{k+\ell+2g-2}$ denote the permutation realizing the
relabelling of the edges. Then 
$$
   \sgn(\sigma)\sgn(\tau)\sgn(\rho) = (-1)^r.
$$
\end{lem}

This lemma is proved in Appendix~\ref{sec:or}, where our convention
for the orientation and ordering of the interior edges is
reinterpreted in terms of orientations on the singular chain complex of
a surface. 

The symmetry properties (1)-(6) described on
page~\pageref{symmetrylist} above also apply to $f_\Gamma$, with
the exception of (2$_g$), which is replaced by 
\begin{enumerate}
\item[(2$_G$)] With the specific choice of $T^{ab}=G^{ab}$,
  interchanging the order of two adjacent edges leads to a sign
  $(-1)^{n-3}$ from interchanging the corresponding pairs of basis
  vectors in the edge order, because $\eta_a+\eta_b=n-3$ whenever
  $G^{ab}\neq 0$. 
\end{enumerate}

It follows from Lemma~\ref{lem:edge-convention} (with
$\wt\Gamma=\Gamma$ but different pairs of trees) and the symmetry
properties (1) and (2$_G$) that the expression for $f_\Gamma$ 
is independent of the choice of maximal trees $(T,T^*)$ used to write
down \eqref{def:fGamma}.  
We define $f_{k,\ell,g}:(B^{\rm cyc*}A[3-n])^{\otimes k} \to (B^{\rm cyc*}B[3-n])^{\otimes \ell}$ as
\begin{equation}\label{eq:f_klg}
   f_{k,\ell,g} := (-1)^{n-3}\!\!\!\!\!\sum_{\Gamma \in RG_{k,\ell,g}} f_\Gamma\;.
\end{equation}
The symmetry property (4) ensures that $f_{k,\ell,g}$ indeed lands in
$(B^{\rm cyc*}B[3-n])^{\otimes \ell}$, i.e., each tensor factor in the
output is cyclically symmetric. 
It follows from Lemma~\ref{lem:edge-convention} (with $\sigma=\id$
and $\tau$ a transposition) and the symmetry properties (5) and
(2$_G$) that $f_{k,\ell,g}$ actually lands in the
invariant subspace (under the action of the symmetric group $S_\ell$) of
$(B^{\rm cyc*}B[3-n])^{\otimes \ell}$.  
Similarly, it follows from Lemma~\ref{lem:edge-convention} (with 
$\sigma$ a transposition and $\tau=\id$) and the symmetry properties
(6) and (2$_G$) that $\ff_{k,\ell,g}$ descends to the quotient
$E_kB^{\rm cyc*}A$ of 
$(B^{\rm cyc*}A[3-n])^{\otimes k}$ under the action of $S_k$. We now define
$$
   \ff_{k,\ell,g}:=\pi \circ f_{k,\ell,g}\circ I: E_kB^{\rm cyc*}A \to
   E_\ell B^{\rm cyc*}B,
$$
where as in Remark~\ref{rem:tensor} the map 
$I:E_kB^{\rm cyc*}A \to (B^{\rm cyc*}A[3-n])^{\otimes k}$ is the inverse of the projection
$\pi:(B^{\rm cyc*}A[3-n])^{\otimes k} \to E_kB^{\rm cyc*}A$, given by 
$$
I(c_1\cdots c_k) = \frac 1{k!} \sum_{\rho\in S_k} \eps(\rho)c_{\rho(1)} \otimes \cdots \otimes c_{\rho(k)},
$$
and similarly $\pi:(B^{\rm cyc*}B[3-n])^{\otimes \ell} \to E_\ell
B^{\rm cyc*}B$ is the projection to the quotient. 
Note that, since $f_{k,\ell,g}$ is symmetric in the inputs, the
symmetrization and the factor $1/k!$ in $I$ are actually unneccesary
and will not appear in formulae below. 
Note also that we try to distinguish in the notation between $f$ and
$\frak f$.  

\begin{Remark}
The global sign $(-1)^{n-3}$ in definition~\eqref{eq:f_klg} will be
needed for the signs to work out at the end of the proof of Claim 4
below (and similarly for Claim 5). We do not have a conceptual
explanation for this sign. 
\end{Remark}

We claim that $\ff= \{\ff_{k,\ell,g}\}$ is the required IBL$_\infty$-homotopy equivalence. To understand this, we first note that $RG_{1,1,0}$ consists of trees $T_r$ with only one interior vertex and $r\ge 1$ exterior vertices, and that each such tree induces the map
$$
f_{T_r}:B_r^{\rm cyc*}A[3-n] \to B_r^{\rm cyc*}B[3-n]
$$
which is dual to the inclusion. It follows that the map 
$$
\ff_{1,1,0}= \sum_r f_{T_r}: B^{\rm cyc*}A[3-n] \to B^{\rm cyc*}B[3-n]
$$
is induced by the dual of the inclusion $B \to A$ and hence a chain homotopy equivalence. It remains to prove that $\ff= \{\ff_{k,\ell,g}\}$ is an IBL$_\infty$ morphism, since then it follows from Theorem~\ref{Whi} that $\ff$ is a homotopy equivalence.

To prove that assertion, we start by considering the difference
$$
\ff_\Gamma \circ \fp_{1,1,0} - \fq_{1,1,0} \circ \ff_\Gamma
$$
for a fixed graph $\Gamma\in RG_{k,\ell,g}$, where for clarity we denote the restriction of the boundary operator $\fp_{1,1,0}$ to the subcomplex by $\fq_{1,1,0}$.
 
\emph{{\bf Claim 1. }All the terms of $\fq_{1,1,0} \circ \ff_\Gamma$
  appear in $\ff_\Gamma \circ \fp_{1,1,0}$ with the same sign.} 

Here and below, we use the notation
$\alpha(v)=(\alpha(v,1),\dots,\alpha(v,d(v))$ and
$\beta(b)=(\beta(b,1),\dots,\beta(b,s(b))$ for the indices associated
to an interior vertex $v$ or a boundary component $b$, respectively.  

\begin{proof}[Proof of Claim 1]
To prove the claim, consider an exterior edge in $\Gamma\in RG_{k,\ell,g}$ going from vertex $i$ to boundary component $j$. In $(\ff_\Gamma \circ \hat\fp_{1,1,0})(\varphi^1,\dots,\varphi^k)_{\beta(1);\dots;\beta(\ell)}$ this contributes
\begin{align*}
(-1)^\epsilon\sum (-1)^\eta \prod G^{a_t,b_t} \prod_{v<i} \varphi^v_{\alpha(v)} d^a_{\alpha(i,r)}\varphi^i_{\alpha'(i)a\alpha''(i)}\prod_{v>i} \varphi^v_{\alpha(v)},
\end{align*}
where $\alpha(i)=(\alpha'(i),\alpha(i,r)\alpha''(i))$, and the sign exponents are as follows:
\begin{itemize}
\item the external sign exponent, coming from the application of $\hat\fp_{1,1,0}$, is $\eps=\sum_{v<i}(|\varphi^v|+(n-3)) + \eta_{\alpha'(i)}$.
\item $\eta=\eta_1+\eta_2$, where $\eta_1$ is the sign exponent corresponding to the permutation
\begin{align*}
\prod_t e_{a_t}e_{b_t} \prod_b e_{\beta(b)} \quad \longrightarrow \quad
\prod_v e_{\alpha(v)}
\end{align*}
and $\eta_2$ is given by
$$
(n-3) \left( \sum_v (k-v)|\varphi^v| +(k-i) + \sum_b (\ell-b) |x^b|\right).
$$
\end{itemize} 
In the corresponding term in $(\hat\fq_{1,1,0} \circ\ff_\Gamma)(\phi^1,\dots,\phi^k)_{\beta(1);\dots;\beta(\ell)}$, we assume that the considered edge corresponds to the point numbered $s$ on the $j$th boundary component and write $\beta(j)=(\beta'(j)\beta(j,s)\beta''(j))$. The the sign exponents are as follows:
\begin{itemize}
\item the external sign exponent, coming from the application of $\hat\fq_{1,1,0}$, is $\eps=\sum_{b<j}(|\varphi^v|+(n-3)) + \eta_{\beta'(j)}$.
\item $\eta=\eta_1+\eta_2$, where $\eta_1$ is the sign exponent corresponding to the permutation
\begin{align*}
\prod_t e_{a_t}e_{b_t} \prod_{b<j} e_{\beta(b)}\,e_{\beta'(j)}e_ae_{\beta''(j)}\,\prod_{b>j}e_{\beta(b)} \quad \longrightarrow \quad
\prod_{v<i} e_{\alpha(v)}\,e_{\alpha'(i)}e_ae_{\alpha''(i)}\,\prod_{v>i} e_{\alpha(v)}
\end{align*}
and $\eta_2$ is given by
$$
(n-3) \left( \sum_v (k-v)|\varphi^v| + \sum_b (\ell-b) |x^b|+(\ell-j)\right).
$$
\end{itemize}
To compare the $\eta_1$-part to the previous one, imagine bringing $e_a$ to the front, replacing it by $e_{\alpha(i,r)}=e_{\beta(j,s)}$, and moving it back to its place. Doing this on both sides relates the second permutation to the first permutation, so we have
$$
\eta_1(\hat\fq_{1,1,0}\circ \ff_\Gamma)- \eta_1(\ff_\Gamma\circ\hat\fp_{1,1,0}) \equiv
(n-3)(k+\ell-2) + \sum_{b<j}|x^b|+\eta_{\beta'(j)} + \sum_{v<i}|\phi^v| + \eta_{\alpha'(i)},
$$
where the first summand reflects the fact that the number of edges of $\Gamma$ is $(k+\ell-2) \mod 2$. Combining this with
$$
\eta_2(\ff_\Gamma\circ\hat\fp_{1,1,0})-\eta_2(\hat\fq_{1,1,0}\circ \ff_\Gamma)
=(n-3)\left((k-i)-(\ell-j)\right)
$$
and
\begin{align*}
\lefteqn{\eps(\ff_\Gamma\circ\hat\fp_{1,1,0}) -\eps(\hat\fq_{1,1,0}\circ \ff_\Gamma)}\\
=&(i-1)(n-3) + \sum_{v<i}|\phi^v|+ \eta_{\alpha'(i)} - \left((j-1)(n-3)+\sum_{b<j}|x^b| + \eta_{\beta'(j)}\right),
\end{align*}
we conclude that the total sign exponents are congruent modulo $2$,
proving Claim 1.
\end{proof}

In the remaining terms in $\ff_\Gamma \circ \fp_{1,1,0}$ the
differential $d$ is applied to one of the labels at an interior vertex
coming from an interior edge. 

\emph{{\bf Claim 2. }Given a graph $\Gamma$ and an interior edge $e_{t_0}$ of $\Gamma$, the contributions coming from the differential acting on the two ends of the edge $e_{t_0}$ have the correct relative signs to combine to yield the left hand side of \eqref{edge2}.}

\begin{proof}
In the proof of this claim, one needs to consider two cases: either the edge $e_{t_0}$ connects two different vertices, or it is a loop. We will treat the second case in detail, the first case is handled the same way.

So assume $e_{t_0}\in C_1^{\rm int}(\Gamma)$ is a loop at the $i$th vertex leaving the vertex as the half-edge numbered $r_1$ and coming back as the half-edge $r_2>r_1$ (the other case $r_2<r_1$ could be handled similarly). Then the relevant terms come from 
\begin{align*}
\lefteqn{(-1)^\eps \ff_\Gamma(\varphi^1,\dots,\fp_{1,1,0}\varphi^i,\dots,\varphi^k)_{\beta(1);\dots;\beta(\ell)}}\\
=& (-1)^\eps \sum (-1)^\eta \prod_t G^{a_tb_t} \prod_{v<i}\varphi^v_{\alpha(v)} (\fp_{1,1,0}\varphi^i)_{\alpha(i)} \,\prod_{v>i}\varphi^v_{\alpha(v)},
\end{align*}
where the sign exponents are as follows:
\begin{itemize}
\item the external exponent is $\epsilon= \sum_{v<i} (|\varphi^v|+(n-3))$.
\item $\eta=\eta_1+\eta_2$, where $\eta_1$  is the sign of the permutation moving
\begin{align}\label{eq:perm1}
\prod_t e_{a_t}e_{b_t} \prod_b e_{\beta(b)} \quad \longrightarrow \quad
\prod_v e_{\alpha(v)},
\end{align}
and $\eta_2$ is given by
$$
(n-3) \left( \sum_v (k-v)|\varphi^v| +(k-i) + \sum_b (\ell-b) |x^b|\right).
$$
\end{itemize}
Freezing all other coefficients, the two relevant terms here are
\begin{align*}
\lefteqn{\sum_{a_{t_0}}(-1)^\eta\prod_t G^{a_tb_t} \prod_{v<i}\varphi^v_{\alpha(v)} \,(-1)^{\sum_{r<r_1} \eta_{\alpha(i,r)}} d^{a'}_{a_{t_0}}\varphi^i_{\alpha'(i)a'\alpha''(i)\alpha(i,r_2)\alpha'''(i)} \,\prod_{v>i}\varphi^v_{\alpha(v)}}\\
&= \sum_{a'}(-1)^{\eta+\sum_{r<r_1} \eta_{\alpha(i,r)}} d^{a}_{a'}G^{a'b}\prod_{t\ne t_0} G^{a_tb_t} \prod_{v<i}\varphi^v_{\alpha(v)} \, \varphi^i_{\alpha'(i)a\alpha''(i)b\alpha'''(i)} \,\prod_{v>i}\varphi^v_{\alpha(v)}
\end{align*}
and 
\begin{align*}
\lefteqn{\sum_{b_{t_0}}(-1)^\eta\prod_t G^{a_tb_t} \prod_{v<i}\varphi^v_{\alpha(v)} \,(-1)^{\sum_{r<r_2} \eta_{\alpha(i,r)}} d^{b'}_{b_{t_0}}\varphi^i_{\alpha'(i)\alpha(i,r_1)\alpha''(i)b'\alpha'''(i)} \,\prod_{v>i}\varphi^v_{\alpha(v)}}\\
&= \sum_{b'}(-1)^{\eta+\sum_{r<r_2} \eta_{\alpha(i,r)}} d^b_{b'}G^{ab'}\prod_{t\ne t_0} G^{a_tb_t} \prod_{v<i}\varphi^v_{\alpha(v)} \, \varphi^i_{\alpha'(i)a\alpha''(i)b\alpha'''(i)} \,\prod_{v>i}\varphi^v_{\alpha(v)}
\end{align*}
where we renamed the variables for better comparison. For the purposes of computing $\eta_1$, the degrees $\eta_{\alpha(i,r_2)}$ differ by one in the two expressions, because in the first setting $\alpha(i,r_2)=b$ but in the second setting $\alpha(i,r_2)=b'$. Imagine doing the permutation \eqref{eq:perm1} in stages,
$$
\prod_t e_{a_t}e_{b_t} \prod_b e_{\beta(b)} \longrightarrow 
\prod_{v<i}e_{\alpha(v)} \, e_{\alpha'(i)} e_{a_{t_0}}e_{b_{t_0}}e_{\alpha''(i)}e_{\alpha'''(i)}\,\prod_{v>i} e_{\alpha(v)} \longrightarrow 
\prod_v e_{\alpha(v)}.
$$
The first stage will give the same sign in both cases, because $\eta_{a_{t_0}}+\eta_{b_{t_0}}=n-3$ for both. The difference in the second stage will be 
$$
\sum_{r_1<r<r_2}\eta_{\alpha(i,r)}.
$$
In total, we get a difference in sign exponent of $\eta_{a}$ (because
$\alpha(i,r_1)=a$ in the second case), which is exactly what is needed
to produce the right hand side of \eqref{edge2}. This finishes the
proof of the Claim 2 when $e_{t_0}$ is a loop, the other case being similar.
\end{proof}

Claim 2 motivates the following definition. 

\begin{Definition}\label{def:fGamma-e}
Given an edge $e \in
C_1^{\rm int}(\Gamma)$, we define maps $\ff_{\Gamma,e}^\id$ and
$\ff_{\Gamma,e}^\Pi$ by a formula analogous to \eqref{def:fGamma}, 
with the following modifications:
\begin{itemize}
\item to the edge $e$ we associate $(-1)^{\eta_a}g^{ab}$ (for
  $f_{\Gamma,e}^\id$) resp.~$(-1)^{\eta_a} g^{\bar a\bar b}$ (for
  $f_{\Gamma,e}^\Pi$) in place of $G^{ab}$, with $g^{\bar a\bar b}$ defined
  in Lemma~\ref{lem:edge2};
\item the sign $\eta$ gets replaced by $\eta+(n-3)(t_0-k)$, where $e=e_{t_0}$ in the chosen ordering of the edges.
\end{itemize}
\end{Definition}

This choice of sign makes the definition independent of the ordering
of the edges: Interchanging $e$ with an adjacent edge does not change
$\eta_1$ (because the basis vectors assigned to the edge $e$ have total degree
$n-2$, while those associated to the other edges have total degree
$n-3$), but it yields sign $(-1)^{n-3}$ from replacing $t_0$ by
$t_0\pm 1$.  So symmetry property ($2_G$) still holds and
Lemma~\ref{lem:edge-convention} yields independence of the pair of
trees $(T,T^*)$ defining the edge ordering. 

Summing over all interior edges $e \in C^1_{\rm int}(\Gamma)$, we get maps 
$$
f_{\Gamma}^\id:= \sum_{e\in C_1^{\rm int}(\Gamma)}f_{\Gamma,e}^\id
\quad \text{ \rm and } \quad
f_{\Gamma}^{\Pi}:= \sum_{e\in C_1^{\rm int}(\Gamma)}f_{\Gamma,e}^\Pi,
$$ 
respectively. In analogy to~\eqref{eq:f_klg}, we sum over graphs
$\Gamma \in RG_{k,\ell,g}$ to define maps  
\begin{equation}\label{eq:f_klg_e}
\ff^\id_{k,\ell,g}:= (-1)^{n-3}\!\!\!\!\!\sum_{\Gamma\in RG_{k,\ell,g}} f_\Gamma^\id
\quad \text{ \rm and } \quad
\ff^\Pi_{k,\ell,g}:= (-1)^{n-3}\!\!\!\!\!\sum_{\Gamma\in RG_{k,\ell,g}} f_\Gamma^\Pi,
\end{equation}
respectively.

To prove Theorem~\ref{thm:homotopyequiv}, it remains to prove
equations \eqref{eq:mor4} for each triple $(k,\ell,g)$. This is the
content of the following sequence of claims. 

\emph{{\bf Claim 3. }With the definitions above, for each $\Gamma \in RG_{k,\ell,g}$ we have
$$
   \ff^\Pi_{\Gamma} - \ff^\id_{\Gamma}
   = \ff_\Gamma\circ\hat\fp_{1,1,0} - \hat\fq_{1,1,0}\circ\ff_\Gamma,
$$
and so in particular
\begin{equation}\label{eq:homotopyequiv1}
   \ff^\Pi_{k,\ell,g} - \ff^\id_{k,\ell,g}
   = \ff_{k,\ell,g}\circ\hat\fp_{1,1,0} - \hat\fq_{1,1,0}\circ\ff_{k,\ell,g}
\end{equation}
for all $(k,\ell,g) \succeq (1,1,0)$.
}

\emph{{\bf Claim 4. }We have
\begin{align}
\ff^\Pi_{k,\ell,g} =& \;\hat{\fq}_{1,2,0} \circ \ff_{k,\ell-1,g} 
   + \hat\fq_{2,1,0}\circ_2\ff_{k,\ell+1,g-1}\notag\\
   &+ \frac 12 \sum_{{k_1+k_2=k \atop \ell_1+\ell_2=\ell+1 }\atop  g_1+g_2=g}\hat\fq_{2,1,0}\circ_{1,1}
   (\ff_{k_1,\ell_1,g_1}\odot\ff_{k_2,\ell_2,g_2}). \label{eq:homotopyequiv2}
\end{align}
}
\emph{{\bf Claim 5. }We have
\begin{align}
   \ff^\id_{k,\ell,g}
   =& \;\ff_{k-1,\ell,g} \circ \hat\fp_{2,1,0} 
+ \ff_{k+1,\ell,g-1} \circ_2 \hat\fp_{1,2,0} \notag\\
&+ \frac 12 \sum_{k_1+k_2=k+1\atop {\ell_1+\ell2=\ell\atop g_1+g_2=g}} (\ff_{k_1,\ell_1,g_1} \odot \ff_{k_2,\ell_2,g_2}) \circ_{1,1} \hat\fp_{1,2,0}.\label{eq:homotopyequiv3}
\end{align}
}

\begin{proof}[Proof of Claim 3]
Claim 3 essentially follows from Claims 1 and 2: By Claim 1, the right hand side
is the sum of terms where the differential $d$ is applied to both ends
of each interior edge. In view of Claim 2 we can apply
Lemma~\ref{lem:edge2} to convert it instead into the sum of terms
where a particular interior edge is labelled with either $(-1)^{\eta_a}g^{\bar a\bar b}$ of
$(-1)^{\eta_a}g^{ab}$, which correspond to the terms on the left hand side. 
It remains to check that the signs match. 

So let $\Gamma \in RG_{k,\ell,g}$ be 
given and suppose the edge $e$ in $\Gamma$ runs from vertex $i$ to vertex $j$ (the case of a loop is treated similarly). 
For definiteness we assume $i<j$, and for simplicity we also assume that the corresponding half-edges come first in the respective orders at these vertices. 
Since by Claim 2 the relative signs of the two terms corresponding to these two half-edges are correct, we just consider the term coming from the half-edge at vertex $i$.
The relevant term in $f_\Gamma \circ \hat\fp_{1,1,0}$ is then of the form
\begin{align*}
\lefteqn{(-1)^\eps f_\Gamma(\varphi^1,\dots,\fp_{1,1,0}\varphi^i,\dots,\varphi^k)_{\beta(1);\dots,\beta(\ell)} }\\
&= (-1)^\eps \sum (-1)^{\eta_1+\eta_2} \prod_t G^{a_tb_t}\Bigl(\prod_{v<i}\varphi^v_{\alpha(v)}\Bigr) (\fp_{1,1,0}\varphi^i)_{\alpha(i)} \prod_{v>i} \varphi^v_{\alpha(v)}
\end{align*}
with $\eps = \sum_{v<i} (|\varphi^v|+(n-3))$, $\eta_1$ the sign for permuting
\begin{align*}
\prod_t e_{a_t}e_{b_t} \prod_b e_{\beta(b)}  \quad \longrightarrow \quad \prod_v e_{\alpha(v)},
\end{align*}
and
$$
\eta_2 = (n-3) \left(\sum_v (k-v)|\varphi^v|+(k-i) + \sum_b (\ell-b)|x^b|\right).
$$
In $f_\Gamma^\id$ (and $f_\Gamma^\Pi$) we have to replace $e_{\alpha(i,1)}$ 
by a basis element whose degree is smaller by 1. One easily checks that this changes $\eta_1$ by $(t_0-1)(n-3) + \sum_{v<i}|\varphi^v|$.
At the same time, $\eta_2$ changes by $(k-i)(n-3)$, since for $f_\Gamma^\id$ we compute $\eta_2$ with the arguments $\varphi^1,\dots,\varphi^k$. Together with the external sign $\eps$, the total sign difference is $(t_0+k)(n-3)$, which exactly fits the extra sign added in the definition of $f_\Gamma^\id$. This proves of Claim 3.
\end{proof}

\begin{proof}[Proof of Claim 4] 
We start with an explanation of the combinatorial factors. 
Throughout this discussion, we use the notation $\wGamma$ for the graphs with a marked edge $e$ appearing in $\ff_{k,\ell,g}^\id$ and $\ff_{k,\ell,g}^\Pi$, and $\Gamma$ or $\Gamma_1$ and $\Gamma_2$ for the graphs appearing in the expressions on the right hand side of \eqref{eq:homotopyequiv2}.

\begin{rem}{\em (Automorphisms) }
A term associated to an edge $e$ of a ribbon graph $\wGamma \in RG_{k,\ell,g}$
appears in $\ff^\Pi_{k,\ell,g}$ with the combinatorial coefficient  
$$
\frac 1 {\ell !\,|\Aut(\wGamma)|\,\prod_v d(v)}.
$$ 
In order to avoid considerations of how automorphisms of graphs behave 
under gluings, it is convenient to consider {\em labelled ribbon graphs}, i.e.,
ribbon graphs together with a labelling in the sense of
Definition~\ref{def:labelling}. Since a labelled graph has no
automorphisms preserving the labelling, the automorphism group of an
unlabelled graph $\wGamma$ acts freely on its labellings. So in the sum
over all labellings of a graph $\wGamma$ each isomorphism class of
labelled graphs appears $|\Aut(\wGamma)|$ times, and we can replace
the sum 
$$
   \sum_{\wGamma\in RG_{k,\ell,g}}\frac{1}{|\Aut(\wGamma)|}\sum_{\text{\rm labellings of }\wGamma}
$$
in the definition of $f_{k,\ell,g}$ by the sum over all isomorphism
classes of labelled ribbon graphs of signature $(k,\ell,g)$ without
the factor $1/|\Aut(\wGamma)|$. This is what we will do in the following 
discussion.
\end{rem}

For a graph $\wGamma\in RG_{k,\ell,g}$ with distinguished edge $e$ as above 
we now consider the new ribbon graph $\Gamma$ obtained by cutting $e$ into two halfs and viewing their endpoints as new exterior vertices $v',v''\in \Gamma$. We have three distinct cases:
\begin{enumerate}[(i)]
\item The dual edge $e^*$ connects distinct boundary components ($e$
  itself could be a loop, or it could connect distinct vertices). In
  this case, $ \Gamma$ is necessarily still connected and $\Gamma \in
  RG_{k,\ell-1,g}$, and the contribution to $\ff_{\wGamma}^\Pi$ will
  correspond to a term in $\hat\fq_{1,2,0} \circ
  \ff_{k,\ell-1,g}$. See Figure~\ref{fig:case_i}.
%
%
\begin{figure}[h]
\labellist
 \pinlabel $e$ [b] at 38 76
\endlabellist
\centering
\includegraphics{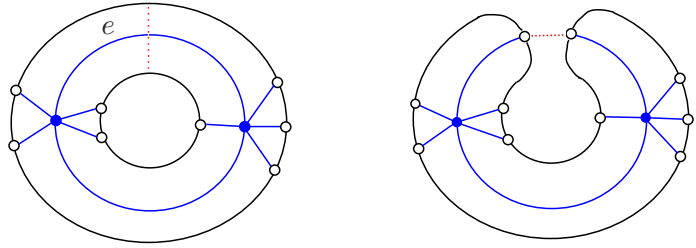}
\caption{On the left we have an example of a graph $\wh \Gamma \in RG_{2,2,0}$ with a marked edge $e$ and its dual edge $e^*$ (dotted) as in case (i), both drawn on the ribbon surface $\Sigma_{\wh \Gamma}$. On the right one sees the graph $\Gamma\in RG_{2,1,0}$ obtained from cutting open $e$, with its ribbon surface. The dotted line connects the new exterior vertices which will be reconnected by an edge in the corresponding term in $\fq_{1,2,0} \circ \ff_{2,1,0}$.}
\label{fig:case_i}
\end{figure}
\item The dual edge $e^*$ is a loop connecting some boundary component
  to itself, and $\Gamma$ is  still connected. Then $\Gamma \in
  RG_{k,\ell+1,g-1}$ and the contribution to $\ff_{\wGamma}^\Pi$ will
  correspond to a term in $\hat\fq_{2,1,0} \circ
  \ff_{k,\ell+1,g-1}$. See Figure~\ref{fig:case_ii}. 
%
%
\begin{figure}[h]
\labellist
 \pinlabel $e$ [b] at 38 76
\endlabellist
\centering
\includegraphics{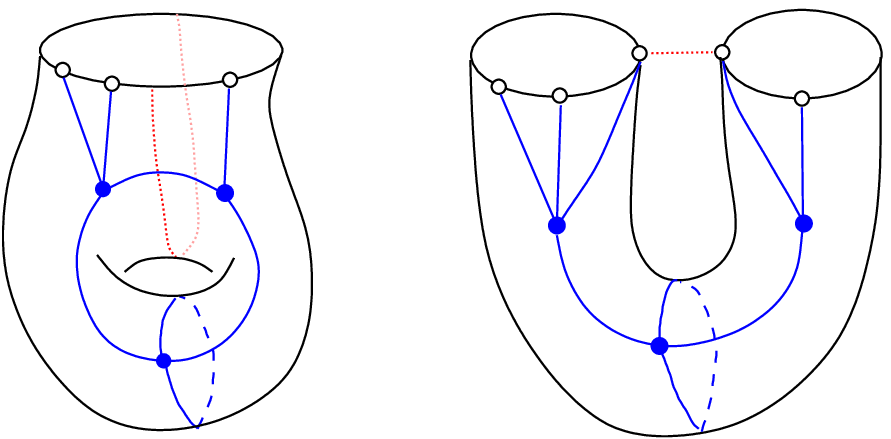}
\caption{On the left we have an example of a graph $\wh \Gamma \in RG_{3,1,1}$ with a marked edge $e$ and its dual edge $e^*$ (dotted) as in case (ii), both drawn on the ribbon surface $\Sigma_{\wh \Gamma}$. On the right one sees the graph $\Gamma\in RG_{3,2,0}$ obtained from cutting open $e$, with its ribbon surface. The dotted line connects the new exterior vertices which will be reconnected by an edge in the corresponding term in $\fq_{2,1,0} \circ \ff_{3,2,0}$.}
\label{fig:case_ii}
\end{figure}
\item The dual edge $e^*$ is a loop connecting some boundary component
  to itself, and $\Gamma =\Gamma_1 \coprod \Gamma_2$ is
  disconnected. This time the contribution to $\ff_{\wGamma}^\Pi$ will
  correspond to a term in $\hat\fq_{2,1,0} \circ (\ff_{k_1,\ell_1,g_1}
  \odot \ff_{k_2,\ell_2,g_2})$ for suitable $(k_i,\ell_i,g_i)$
  corresponding to our two graphs $\Gamma_i$. See Figure~\ref{fig:case_iii}.
%
%
\begin{figure}[h]
\labellist
 \pinlabel $e$ [b] at 101 196
\endlabellist
\centering
\includegraphics{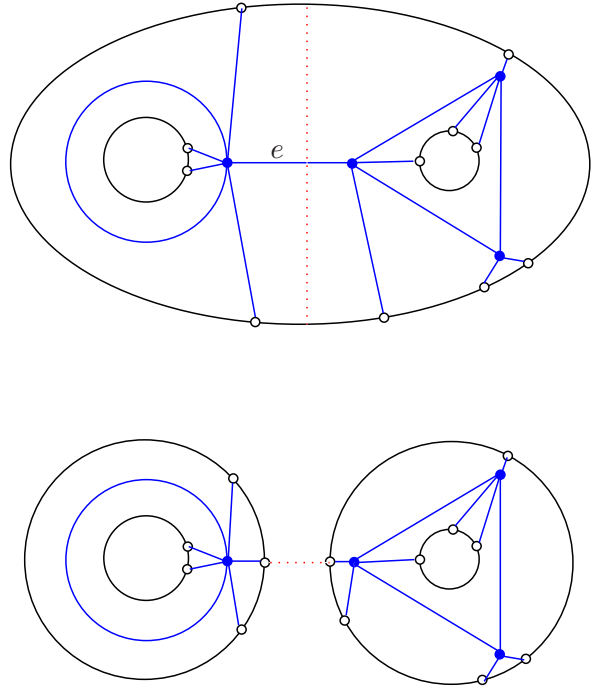}
\caption{On the top we have an example of a graph $\wh \Gamma \in RG_{4,3,0}$ with a marked edge $e$ and its dual edge $e^*$ (dotted) as in case (iii), both drawn on the ribbon surface $\Sigma_{\wh \Gamma}$. On the bottom one sees the graphs $\Gamma_1\in RG_{1,2,0}$ and $\Gamma_2\in RG_{3,2,0}$ obtained from cutting open $e$, with their ribbon surfaces. The dotted line connects the new exterior vertices which will be reconnected by an edge in the corresponding term in $\fq_{2,1,0} \circ (\ff_{1,2,0} \odot \ff_{3,2,0})$.}
\label{fig:case_iii}
\end{figure}
\end{enumerate}

We now discuss the combinatorial factors $\frac 1 {\ell !\prod_v
  d(v)}$. Consider first a composition $f_{\Gamma'}\circ 
f_\Gamma$ corresponding to a complete gluing of the interior vertices
of $\Gamma'$ to the boundary components of $\Gamma$. Recall that in
the definition of $f_\Gamma$ we sum over all labellings of $\Gamma$,
and similarly for $f_{\Gamma'}$. Thus each term in the map 
$f_{\wh\Gamma}$ corresponding to a graph $\wh\Gamma$ obtained by
gluing $\Gamma$ and $\Gamma'$ appears $\ell !\prod_{w=1}^{k'}d'(w)$
times in $f_{\Gamma'}\circ f_\Gamma$. Combining this with the
combinatorial factors of $f_\Gamma$ and $f_{\Gamma'}$, we see that
each term in $f_{\wh\Gamma}$ appears with the correct combinatorial factor
$$
   \Bigl(\frac{1}{\ell !\prod_{v=1}^{k}d(v)}\Bigr) \Bigl(\frac{1}{\ell'
     !\prod_{w=1}^{k'}d'(w)}\Bigr) \Bigl(\ell !\prod_{w=1}^{k'}d'(w)\Bigr) =
   \frac{1}{\ell' !\prod_{v=1}^{k}d(v)}. 
$$
The same discussion also applies to an incomplete gluing where one of
the maps, say $f_\Gamma$, is replaced by an extended map $\wh
f_\Gamma$. 
To see this, let us again abbreviate $\BC:=B^{\rm  cyc*}A[2-n]$ and
consider a map $f:\BC[1]^{\otimes k_1}\to \BC[1]^{\otimes\ell_1}$. 
Its extension to a map $\wh f:\BC[1]^{\otimes k}\to
\BC[1]^{\otimes\ell}$, with $k=k_1+r$ and $\ell=\ell_1+r$ for some
$r\geq 1$, is defined by
\begin{align}\label{eq:hat2}
   \wh f(c_1,\dots,c_k) := \frac{\ell_1!}{\ell!}
   \sum_{\sigma\in S_{k_1,r}}\sum_{\tau\in S_{\ell_1,r}}\sum_{\rho\in S_r}
   \tau\Bigl(
   &f(c_{\sigma(1)}\otimes\cdots\otimes c_{\sigma(k_1)}) \cr
   &\otimes c_{\sigma(k_1+\rho(1))}\otimes\cdots \cdots\otimes c_{\sigma(k_1+\rho(r))}\Bigr). 
\end{align}
Here $S_{p,r}$ denotes the set of $(p,r)$-shuffles, i.e., permutations
$\sigma\in S_{p+r}$ with $\sigma(1)<\cdots<\sigma(p)$ and
$\sigma(p+1)<\cdots<\sigma(p+r)$. Since the number of $(p,r)$-shuffles
is $|S_{p,r}|={p+r\choose r}$, the combinatorial factor can be written
as 
$$
   \frac{\ell_1!}{\ell!} = \frac{1}{r!{\ell\choose r}} =
   \frac{1}{|S_r|\,|S_{\ell_1,r}|}. 
$$
When considered as a map into the quotient under permutations
$E_\ell\BC=\BC[1]^{\otimes\ell}/\sim$, the averaging over $\tau\in
S_{\ell_1,r}$ and $\rho\in S_r$ in the definition of $\wh f$ can be
dropped and we recover our earlier definition~\eqref{eq:hat}. 
Now we apply the extension~\eqref{eq:hat2} to the map $f_\Gamma$
associated to a ribbon graph $\Gamma\in RG_{k_1,\ell_1,g}$. Then the 
resulting map $\wh f_\Gamma:\BC[1]^{\otimes k}\to
\BC[1]^{\otimes\ell}$ descends to the quotient $E_k\BC$ and lands in
the invariant part of $\BC[1]^{\otimes\ell}$, a property it shares
with the map $f_\Gamma$ itself. Moreover, the combinatorial factors in
the definition of $f_\Gamma$ and in~\eqref{eq:hat2} combine to the
combinatorial factor for $\wh f_\Gamma$:
$$
   \Bigl(\frac{1}{\ell_1!\prod_{v=1}^{k_1}d(v)}\Bigr)
   \Bigl(\frac{\ell_1!}{\ell!}\Bigr) =
   \frac{1}{\ell!\prod_{v=1}^{k_1}d(v)}.  
$$
Using this, we see that in the composition $f_{\Gamma'}\circ\wh f_\Gamma$
corresponding to an incomplete gluing each term in $f_{\wh\Gamma}$ 
appears with the correct combinatorial factor
$$
   \Bigl(\frac{1}{\ell !\prod_{v=1}^{k_1}d(v)}\Bigr) \Bigl(\frac{1}{\ell'
     !\prod_{w=1}^{k'}d'(w)}\Bigr) \Bigl(\ell !\prod_{b=1}^{\ell_1}s(b)\Bigr) =
   \frac{1}{\ell' !\prod_{v=1}^{k}d(v)}. 
$$
Here $s(b)$ is the number of vertices on the $b$-th boundary component
of $\Sigma_\Gamma$ and we have used the fact that $\ell-\ell_1=k-k_1$
factors in $\prod_{w=1}^{k'}d'(w)$ combine with
$\prod_{v=1}^{k_1}d(v)$ to give $\prod_{v=1}^{k}d(v)$ in the
denominator, while the remaining $\ell_1$ terms cancel
$\prod_{b=1}^{\ell_1}s(b)$. A similar (in fact, easier) argument
applies in the case of an incomplete gluing of the type $\wh
f_{\Gamma'}\circ f_\Gamma$. 

Consider now a gluing of two labelled graphs $\Gamma,\Gamma'$. Suppose
that each interior vertex of $\Gamma'$ is glued to a boundary component of
$\Gamma$, but some boundary components of $\Gamma$ may remain
free. Such a gluing is described uniquely in terms of the following
{\em gluing data}:
\begin{itemize}
\item an injective map $\lambda:\{1,\dots,k'\}\to\{1,\dots,\ell\}$
  such that $d'(v)=s(\lambda(v))$ for all $v=1,\dots,k'$, where $d'(v)$
  is the degree of the vertex $v$ of $\Gamma'$ and $s(b)$ is the number
  of vertices on the $b$-th boundary component of $\Gamma$;
\item for each $v=1,\dots,k'$ a bijection
  $c_v:\{1,\dots,d'(v)\}\to\{1,\dots,s(\lambda(v))\}$ preserving the
  cyclic orders. 
\end{itemize}
Note that the resulting ribbon graph $\Gamma\#\Gamma'$ naturally
inherits a labelling from those of $\Gamma$ and $\Gamma'$. Moreover,
different gluing data give rise to different isomorphism classes of
labelled graphs $\Gamma\#\Gamma'$. This shows that the sums over
isomorphism classes of labelled graphs on both sides in Claim 4 agree
without further combinatorial factors due to automorphisms. 

The preceding considerations show that in all three cases the
combinatorial factors of the corresponding terms on both sides
of~\eqref{eq:homotopyequiv2} match. Here the additional factor
$1/2$ in case (iii) is due to the fact that for each split graph
$\Gamma=\Gamma_1\amalg\Gamma_2$ the two terms
$\ff_{\Gamma_1}\odot\ff_{\Gamma_2}$ and
$\ff_{\Gamma_2}\odot\ff_{\Gamma_1}$ of the right-hand side
of~\eqref{eq:homotopyequiv2} correspond to the same term on the
left-hand side. 
This finishes the discussion of combinatorial factors. 

{\bf Signs. }
We now discuss the signs involved in formula
\eqref{eq:homotopyequiv2}, starting with case (i) above. So let
$\wGamma \in RG_{k,\ell,g}$ be given and assume for simplicity that
the special edge $e=e_{t_0}$ separates the boundary components
labelled $1$ and $2$ and is dual to the first edge of $T^*$. 
%
%
\begin{figure}[h]
\labellist
 \pinlabel $e$ [l] at 30 71
 \pinlabel $e^*$ [l] at 48 36
 \pinlabel $1$ [t] at 14 2
 \pinlabel $2$ [t] at 58 2
 \pinlabel $3$ [t] at 76 103
 \pinlabel $\ell$ [t] at 184 103
\endlabellist
\centering
\includegraphics{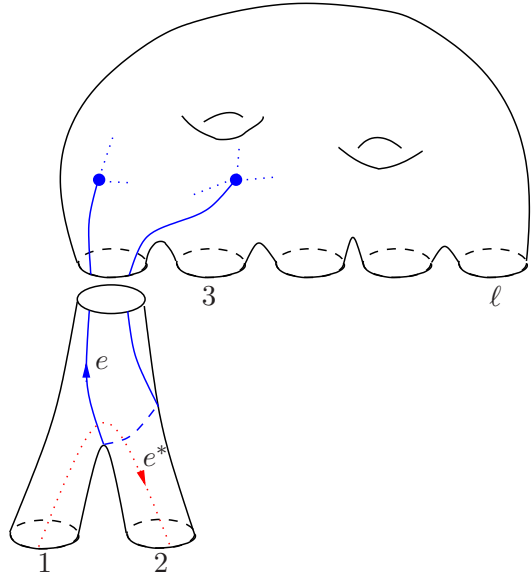}
\vspace{.3cm}
\caption{The relevant situation for the term of $\hat\fq_{1,2,0} \circ f_\Gamma$ in case (i) for which the signs discussed. We have drawn the edge $e=e_{t_0}$ in blue and the dual edge $e^*$ dotted in red. We only show the endpoints of $e$ (which could also coincide), since the remaining part of the graph $\Gamma$ is irrelevant to the discussion.}
\label{fig:claim4_i}
\end{figure}
According to our edge ordering conventions in Definition~\ref{def:edge-ordering}
we then have $t_0=k$. As a consequence, the sign of the corresponding
term in $\ff_{k,\ell,g}^\Pi$ according to
Definition~\ref{def:fGamma-e} is just the ``usual'' sign
$(-1)^{\eta_1(\ff^\Pi)+\eta_2(\ff^\Pi)}$, where $\eta_1(\ff^\Pi)$ is
the sign of the permutation
$$
   \prod_t e_{a_t}e_{b_t} \prod_b e_{\beta(b)} \quad \longrightarrow \quad \prod_v e_{\alpha(v)}
$$
and 
$$
   \eta_2(\ff^\Pi)= (n-3) \left(\sum_{v=1}^k (k-v)|\varphi^v| + \sum_{b=1}^\ell (\ell-b)|x^b|) \right).
$$
To understand the signs in the corresponding term in $\hat\fq_{1,2,0} \circ \ff_{k,\ell-1,g}$, let $\Gamma$ as before be the graph obtained from $\wGamma$ by cutting $e_{t_0}$ in half and adding exterior vertices at the new end points. We can assume that the resulting new boundary component is labelled $1$, and all other boundary components have their labelling decreased by $1$, and that the trees $T(\Gamma)$ and $T^*(\Gamma)$ agree with $T(\wGamma)$ and $T^*(\wGamma)\setminus e^*_{t_0}$. Then $\eta_1(f_\Gamma)$ is the sign of the permutation
$$
   \Bigl(\prod_{t\neq t_0}e_{a_t}e_{b_t}\Bigr) e_{\beta'(1)}\prod_{b\ge 3} e_{\beta(b)} \quad \longrightarrow \quad \prod_v e_{\alpha(v)},
$$
where $\beta'(1)$ is the vector of labels at the new first boundary component, and $\eta_1(\fq)$ is the sign of the permutation
$$
e_{a_{t_0}}e_{b_{t_0}}e_{\beta(1)}e_{\beta(2)} \quad \longrightarrow \quad  e_{\beta'(1)}.
$$
Since permuting $e_{a_{t_0}}e_{b_{t_0}}$ past the other pairs of edge elements does not introduce signs (because $\eta_{a_{t_0}}+\eta_{b_{t_0}}=n-2$ and the sum equals $n-3$ for all other edges), we see that the sum of these terms matches $\eta_1(\ff^\Pi)$.

Denoting by $\bar x \in B^{\rm cyc*}B$ the element associated to the new first boundary component of $\Gamma$, we have
$$
\eta_2(f_\Gamma)= (n-3) \left(\sum_{v=1}^k (k-v)|\varphi^v| + (\ell-2)|\bar x| +\sum_{b\ge 3} (\ell-b)|x^b|) \right)
$$
and
$$
\eta_2(\fq)= (n-3)|x^1|.
$$
Since $|\bar x|=|x^1|+|x^2|+n-2$, we see that their sum also matches $\eta_2(\ff^\Pi)$, so the two terms appear on both sides with the same sign.

%
%
\begin{figure}[ht]
\labellist
 \pinlabel $e$ [t] at 110 41
 \pinlabel $k_1+1$ [bl] at 147 129
 \pinlabel $1$ [t] at 12 94
 \pinlabel $\ell_1-1$ [t] at 49 94
 \pinlabel $\ell_1$ [t] at 108 2
 \pinlabel $\ell_1+1$ [t] at 166 94
 \pinlabel $\ell=\ell_1+\ell_2-1$ [tl] at 235 94
\endlabellist
\centering
\includegraphics[scale=1]{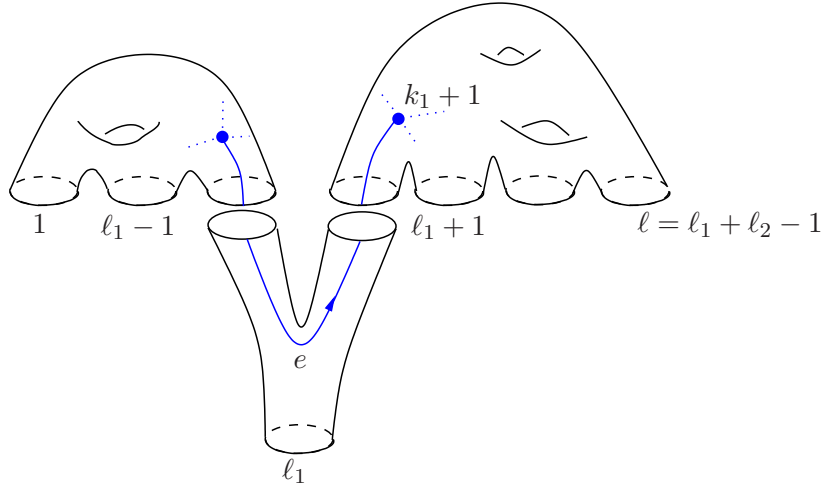}
\vspace{.3cm}
\caption{The assumptions for the term of $\hat\fq_{2,1,0} \circ (f_{\Gamma_1} \otimes F_{\Gamma_1})$ in case (iii). Again we have drawn the edge $e=e_{t_0}$, but not the remaining parts of the graphs $\Gamma_1$  and $\Gamma_2$.}
\label{fig:claim4_iii}
\end{figure}
Finally we discuss case (iii), leaving the slightly easier case (ii)
to the reader. Here we start again with $\wGamma\in RG_{k,\ell,g}$,
and assume that cutting $e=e_{t_0}$ results in graphs $\Gamma_1\in
RG_{k_1,\ell_1,g_1}$ and $\Gamma_2\in RG_{k_2,\ell_2,g_2}$. For
simplicity, we also assume that the vertices of $\Gamma_1$ are
labelled $1,\dots,k_1$ and the boundary components it inherits from
$\wGamma$ have labels $1,\dots,\ell_1-1$, with the new one being the
last component (labelled $\ell_1$), whereas the new boundary component of $\Gamma_2$ is
its first (again labelled $\ell_1$), followed by the inherited
boundary components labelled $\ell_1+1,\dots,\ell=\ell_1+\ell_2-1$. 
Note that $e_{t_0}$ will necessarily belong to $T$. For convenience we
assume that $e_{t_0}$ ends at vertex $k_1+1\in \Gamma_2\subset
\wGamma$, so that according to our conventions in
Definition~\ref{def:edge-ordering} we have $t_0=k_2$. In
particular, the sign of this term in $\ff_{k,\ell,g}^\Pi$ according to
Definition~\ref{def:fGamma-e} is $\eta_1+\eta_2+(n-3)(k+k_2)$, where
$\eta_1$ and $\eta_2$ are standard as before. 

The corresponding term in $\hat\fq_{2,1,0} \circ( f_{\Gamma_1} \otimes f_{\Gamma_2})$  involves the following signs. Denoting by $\beta'(\ell_1)$ and $\beta''(\ell_1)$ the vectors of labels occuring at the new boundary components of $\Gamma_1$ and $\Gamma_2$, respectively, the $\eta_1$-parts of the signs of $\fq_{2,1,0}$, $f_{\Gamma_1}$ and $f_{\Gamma_2}$ are respectively the signs of the permutations
$$
e_{a_{t_0}}e_{b_{t_0}}e_{\beta(\ell_1)} \quad \longrightarrow \quad
e_{\beta'(\ell_1)}e_{\beta''(\ell_1)}, 
$$
\mbox{}
$$
\prod_{t\in C_1^{\rm int}(\Gamma_1)} e_{a_t}e_{b_t}
\Bigl(\prod_{b<\ell_1} e_{\beta(b)}\Bigr) e_{\beta'(\ell_1)} \quad
\longrightarrow \quad \prod_{v\le k_1} e_{\alpha(v)} 
$$
and
$$
\Bigl(\prod_{t\in C_1^{\rm int}(\Gamma_2)} e_{a_t}e_{b_t}\Bigr)
e_{\beta''(\ell_1)}\prod_{b>\ell_1} e_{\beta(b)} \quad \longrightarrow
\quad \prod_{v\ge k_1+1} e_{\alpha(v)}. 
$$
To understand the difference in the $\eta_1$-part of the sign on both
sides, we write the permutation corresponding to
$\eta_1(\ff_{\wh\Gamma}^\Pi)$ in stages, as 
\begin{align*}
\lefteqn{\prod_t e_{a_t}e_{b_t} \prod_b e_{\beta(b)}}\\ 
&\quad \stackrel{(1)}\longrightarrow \quad \prod_{t\in C_1^{\rm int}(\Gamma_2)} e_{a_t}e_{b_t} \Bigl(\prod_{t\in C_1^{\rm int}(\Gamma_1)} e_{a_t}e_{b_t}\Bigr) e_{a_{t_0}}e_{b_{t_0}} \prod_b e_{\beta(b)}\\
&\quad \stackrel{(2)}\longrightarrow \quad \prod_{t\in C_1^{\rm int}(\Gamma_2)} e_{a_t}e_{b_t} \prod_{t\in C_1^{\rm int}(\Gamma_1)} e_{a_t}e_{b_t} \Bigl(\prod_{b<\ell_1} e_{\beta(b)}\Bigr) e_{a_{t_0}}e_{b_{t_0}} \prod_{b\ge \ell_1} e_{\beta(b)}\\
&\quad \stackrel{(3)}\longrightarrow \quad \prod_{t\in C_1^{\rm int}(\Gamma_2)} e_{a_t}e_{b_t} \prod_{t\in C_1^{\rm int}(\Gamma_1)} e_{a_t}e_{b_t} \Bigl(\prod_{b<\ell_1} e_{\beta(b)}\Bigr) e_{\beta'(\ell_1)}e_{\beta''(\ell_1)} \prod_{b> \ell_1} e_{\beta(b)}\\
&\quad \stackrel{(4)}\longrightarrow \quad \prod_{t\in C_1^{\rm int}(\Gamma_2)} e_{a_t}e_{b_t} \Bigl(\prod_{v\le k_1} e_{\alpha(v)}\Bigr) e_{\beta''(\ell_1)} \prod_{b> \ell_1} e_{\beta(b)}\\
&\quad \stackrel{(5)}\longrightarrow \quad \prod_{v\le k_1} e_{\alpha(v)} \Bigl(\prod_{t\in C_1^{\rm int}(\Gamma_2)} e_{a_t}e_{b_t}\Bigr) e_{\beta''(\ell_1)} \prod_{b> \ell_1} e_{\beta(b)}\\
&\quad \stackrel{(6)}\longrightarrow \quad \prod_v e_{\alpha(v)}.
\end{align*}
The sign exponents are as follows:
\begin{enumerate}[{in} (1)]
\item it is $(n-3)(\ell_2-1)(k_1+\ell_1)$ for moving the edges of $T^*(\Gamma_2)$ and $H_1(\Gamma_2)$ past the edges of $\Gamma_1$,
\item it is $(n-2)\sum_{b<\ell_1} |x^b|$ for moving the edge $e_{t_0}$ past these boundary components,
\item it is simply $\eta_1(\fq)$,
\item it is simply $\eta_1(f_{\Gamma_1})$,
\item it is $(n-3)(k_2+\ell_2)\sum_{v\le k_1}|\varphi^v|$ for moving the edges of $\Gamma_2$ past the inputs of $f_{\Gamma_1}$, and
\item it is simply $\eta_1(f_{\Gamma_2})$.
\end{enumerate}
Here the sign exponent in (1) may require some explanation. Recall that
according to Definition~\ref{def:edge-ordering} the interior edges are
ordered as follows: edges in $T_2$ (in reverse order),
$e_{t_0}=e_{k_2}$, edges in $T_1$ (in reverse order), edges dual to
$T_1^*$, edges dual to $T_2^*$, and finally the remaining edges
generating $H_1(\Sigma_{\wh\Gamma})$. Since the sign exponent of the variables
on $e_{t_0}$ is $n-2$ and the sign exponents on all other edges are $n-3$, and
the edges generating $H_1(\Sigma_{\wh\Gamma})$ appear in pairs and
thus do not contribute to the sign, the sign exponent for moving the
edges in $\Gamma_2$ to the first position (in the correct order) comes
from moving $T_2^*$ past $T_1$ and $T_1^*$, hence equals
$(n-3)(\ell_2-1)(k_1+\ell_1)$. Note that for $i=1,2$ the orientations
and orderings of the edges in $\Gamma_i$ induced by the trees $T_i,T_i^*$ 
according to Definition~\ref{def:edge-ordering} (oriented away from the first
vertex resp.~boundary component) agree with those induced by the
trees $T,T^*$ in $\Gamma$ because $e_{t_0}$ was assumed to end at the
first vertex $k_1+1$ of $\Gamma_2$. Had $e_{t_0}$ ended at a different
vertex $k_1+s$, then the total change in sign comparing $T(\Gamma)$ 
with $T(\Gamma_2)$ would contribute a sign exponent $(n-3)(s-1)$, which 
would cancel with the same change in the extra sign in 
Definition~\ref{def:fGamma-e}. 

In total, the difference in sign exponents for $\eta_1$ is
$$
(n-3)\left( (\ell_2-1)(k_1+\ell_1) + \sum_{b<\ell_1}|x^b| + (k_2+\ell_2)\sum_{v\le k_1}|\varphi^v|\right)+\sum_{b<\ell_1}|x^b|.
$$
Let us denote by $x_1^{\ell_1}$ and $x_2^{\ell_1}$ the terms at the
new boundary components of $f_{\Gamma_1}$ and $f_{\Gamma_2}$ (or
equivalently, the inputs of $\fq_{2,1,0}$). Replacing $\sum_{v\le
  k_1}|\varphi^v|$ using the relation  
$$
\sum_{b<\ell_1}|x^b|+|x_1^{\ell_1}|+(k_1+\ell_1)(n-3) \equiv \sum_{v\le k_1}|\varphi^v| \mod 2,
$$
the preceding sign exponent can be rewritten as
\begin{align}
\label{eq:diff_eta_1}
(n-3) \left(  k_2\sum_{v\le k_1}|\varphi^v| + (\ell_2+1)\sum_{b<\ell_1}|x^b| + \ell_2 |x_1^{\ell_1}| -(k_1+\ell_1)\right)+\sum_{b<\ell_1}|x^b|.
\end{align}
Next we note that
\begin{align*}
\eta_2(\fq)&= (n-3)|x_1^{\ell_1}|,\\
\eta_2(f_{\Gamma_1})&= (n-3)\left(\sum_{v\le k_1} (k_1-v)|\varphi^v| + \sum_{b<\ell_1} (\ell_1-b)|x^b|\right) \quad \text{\rm and}\\
\eta_2(f_{\Gamma_2})&= (n-3)\left(\sum_{v> k_1} (k-v)|\varphi^v| + (\ell-\ell_1)|x_2^{\ell_1}| + \sum_{b>\ell_1} (\ell-b)|x^b|\right).
\end{align*}
\bigskip
It follows (using $|x^{\ell_1}|=|x_1^{\ell_1}|+|x_2^{\ell_1}|+n-2$)  that the total sign difference in $\eta_2$ can be written as
\begin{align}
\label{eq:diff_eta_2}
(n-3) \left(  k_2\sum_{v\le k_1}|\varphi^v| + (\ell_2-1)\sum_{b<\ell_1}|x^b| + 
\ell_2 |x_1^{\ell_1}|\right).
\end{align}
Comparing with \eqref{eq:diff_eta_1}, we see that this cancels the first three terms there, so we get a total sign difference in $\eta_1+\eta_2$ of
$$
(n-3) (k_1+\ell_1)+\sum_{b<\ell_1}|x^b|.
$$
Finally, note that the external sign of $\wh\fq_{2,1,0}$ (from moving
$\fq_{2,1,0}$ past the outputs $x^1,\dots,x^{\ell_1-1}$) contributes
$(n-3)(\ell_1-1)+ \sum_{b<\ell_1}|x^b|$ and the extra sign of
$\ff^\Pi$ contributes $(n-3)(k+k_2)$ to exponents. We conclude that
the total difference in sign exponents in the two terms is 
$$
(n-3)(\ell_1-1+k+k_2+k_1+\ell_1)=(n-3).
$$
Now recall that the definition~\eqref{eq:f_klg} of $f_{k,\ell,g}$ in
terms of the $f_\Gamma$ involves a global sign $(-1)^{n-3}$, and
similarly for the definition~\eqref{eq:f_klg_e} of
$f_{k,\ell,g}^\Pi$. Since in~\eqref{eq:homotopyequiv2} the last term
is quadratic in $f$ and all other terms are linear, this cancels the
sign difference that we just computed.  
This concludes the proof of Claim 4.
\end{proof}

\begin{proof}[Proof of Claim 5]
The proof of Claim 5 is analogous to that of Claim 4. This time, given a graph $\wh \Gamma \in RG_{k,\ell,g}$ with a marked edge $e$, the graph $\Gamma$ (or more precisely its ribbon surface $\Sigma_\Gamma$) is obtained by cutting out a neighborhood of the edge $e$ in $\Sigma_{\wh \Gamma}$ and collapsing each resulting new ``boundary'' component to a new vertex. Again there are three cases to consider:

\begin{enumerate}[(i)]
\item If the edge $e$ connects different vertices, then the resulting
  graph $\Gamma \in RG_{k-1,\ell,g}$ is obtained simply by collapsing
  the edge $e$ in $\wh \Gamma$. Here the contribution to $\ff_{\wh
    \Gamma}^\id$ corresponds to a term in $\ff_{k-1,\ell,g} \circ
  \wh\fp_{2,1,0}$. See Figure~\ref{fig:case_iv}. 
%
%
\begin{figure}[h]
\labellist
 \pinlabel $e$ [b] at 109 196
\endlabellist
\centering
\includegraphics{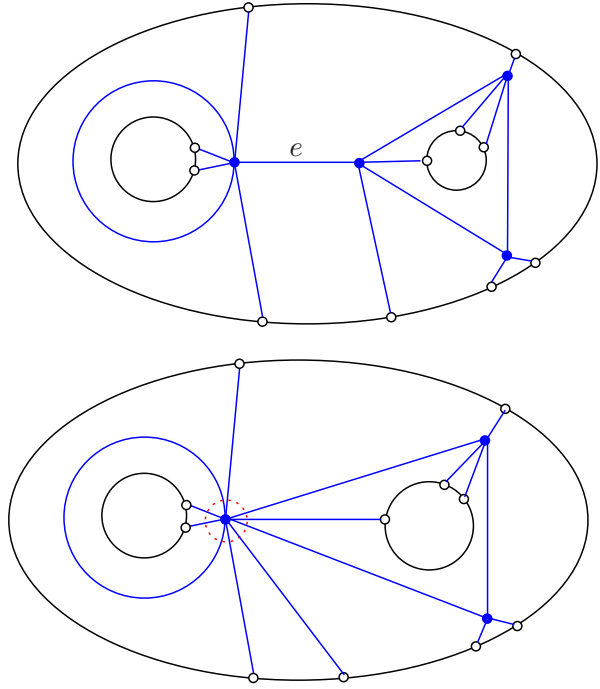}
\caption{On the top we have the graph $\wh \Gamma \in RG_{4,3,0}$ from Figure~\ref{fig:case_iii} with the same marked edge $e$, which now corresponds to case (i) in Claim 5, both drawn on the ribbon surface $\Sigma_{\wh \Gamma}$. On the bottom one sees the graph $\Gamma\in RG_{3,3,0}$ obtained by contracting $e$, drawn on its ribbon surface. The dotted circle marks the new vertex which receives the output from $\fp_{2,1,0}$ in the corresponding term in $\ff_{3,3,0} \circ \hat \fp_{2,1,0}$.}
\label{fig:case_iv}
\end{figure}
\item If the edge $e$ is a loop at a vertex $v$ in $\wh \Gamma$ such
  that $\wh \Gamma \setminus \{v,e\}$ is connected, then the graph
  $\Gamma \in RG_{k+1,\ell,g-1}$ is obtained by deleting $e$ and
  splitting the vertex $v$ into two new vertices $v',v''\in \Gamma$
  whose incident half-edges correspond to the two (ordered)
  collections of half-edges of $\wh \Gamma$ incident to $v$ which form
  the complement of $e$. The contribution to $\ff_{\wh \Gamma}^\id$
  will correspond to a term in $\ff_{k+1,\ell,g-1} \circ \hat
  \fp_{1,2,0}$. See Figure~\ref{fig:case_v}. 
%
%
\begin{figure}[h]
\labellist
 \pinlabel $e$ [r] at 50 11
\endlabellist
\centering
\includegraphics{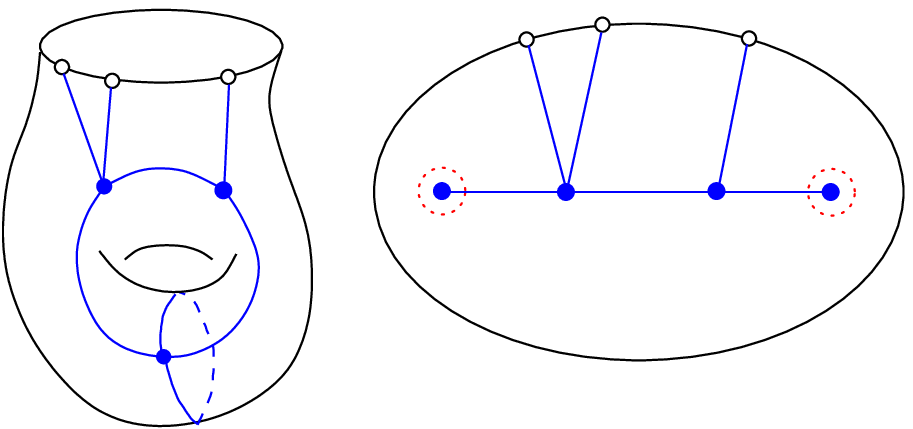}
\caption{On the left we have the graph $\wh \Gamma \in RG_{3,1,1}$
  with a marked edge $e$ as in Figure~\ref{fig:case_ii}, which now
  corresponds to case (ii) in Claim 5, drawn on its ribbon surface $\Sigma_{\wh \Gamma}$. On the right one sees the resulting graph $\Gamma\in RG_{4,1,0}$ obtained by removing $e$ and splitting its vertex, drawn on its ribbon surface. The dotted circles mark the new vertices which receive the output from $\fp_{1,2,0}$ in the corresponding term in $\ff_{4,1,0}  \circ \hat \fp_{1,2,0}$.}
\label{fig:case_v}
\end{figure}
\item If the edge $e$ is a loop at a vertex $v$ in $\wh \Gamma$ such
  that removing $v$ from $\wh \Gamma \setminus e$ disconnects the
  remaining graph, then $\Gamma = \Gamma_1 \sqcup \Gamma_2$ is
  disconnected, and the contribution to $\ff_{\wh \Gamma}^\id$ will
  correspond to a term in $(\ff_{k_1,\ell_1,g_1} \odot
  \ff_{k_2,\ell_2,g_2}) \circ \hat\fp_{1,2,0}$. See Figure~\ref{fig:case_vi}. 
%
%
\begin{figure}[h]
\labellist
 \pinlabel $e$ [b] at 65 226
\endlabellist
\centering
\includegraphics{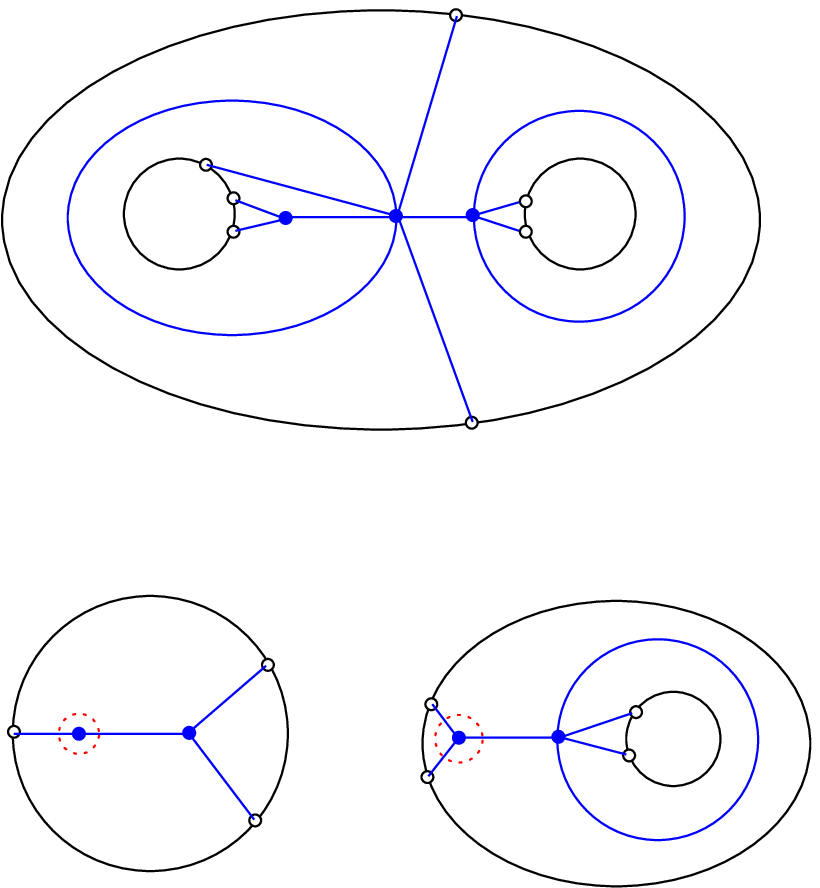}
\caption{On the top we have a graph $\wh \Gamma \in RG_{3,3,0}$ with a
  marked edge $e$ corresponding to case (iii) in Case 5, both drawn on
  the ribbon surface $\Sigma_{\wh \Gamma}$. On the bottom one sees the
  resulting graphs $\Gamma_1\in RG_{2,1,0}$ and $\Gamma_2\in
  RG_{2,2,0}$ obtained by removing $e$ and splitting its vertex, drawn
  on their ribbon surfaces. The dotted circles mark the new vertices
  which receive the output from $\fp_{1,2,0}$ in the corresponding
  term in $(\ff_{2,1,0} \odot \ff_{2,2,0}) \circ \hat \fp_{1,2,0}$.} 
\label{fig:case_vi}
\end{figure}
\end{enumerate}
As the discussion of combinatorial factors and signs follows the lines
of argument used in Claim 4, we leave the remaining details to the
reader. 
\end{proof}

This concludes the proof of Claims 1--5, and thus of
Theorem~\ref{thm:homotopyequiv}. 
\end{proof}


{\bf The filtration on the dual cyclic bar complex. }
Consider a cyclic cochain complex $(A,\langle\,,\,\rangle,d)$ as in
Section~\ref{sec:Cyclic-cochain}. The space 
$$
   B^{\text{\rm cyc}*}A = \bigoplus_{k\geq 1} B^{\text{\rm cyc}*}_kA
$$
carries a natural filtration by the subspaces
$$
   \FF^\lambda B^{\text{\rm cyc}*}A :=
   \bigoplus_{k\geq\lambda}B^{\text{\rm cyc}*}_kA. 
$$
Its completion with respect to this filtration,
$$
   \wh{B^{\text{\rm cyc}*}A} = \prod_{k\geq 1} B^{\text{\rm cyc}*}_kA
   = Hom(B^{\text{\rm cyc}}A,\R),
$$
inherits a filtration by the subspaces
$$
   \FF^\lambda\wh{B^{\text{\rm cyc}*}A} := \{\varphi\in \wh{B^{\text{\rm
         cyc}*}A}\;\bigl|\; \varphi|_{B^{\text{\rm cyc}}_kA}=0 \text{ 
    for all }k<\lambda\}.  
$$
Recall that the operations $\fp_{1,1,0}$, $\fp_{2,1,0}$, $\fp_{1,2,0}$
defining the \text{\rm dIBL} structure and the operations $\ff_{k,\ell,g}$ in
Theorem~\ref{thm:homotopyequiv} are defined by summation over 
certain ribbon graphs. Consider a ribbon graph $\Gamma$ with $k$
interior vertices of degrees $d(1),\dots,d(k)$ and $s=s(1)+\dots+s(\ell)$
exterior vertices distributed on the $\ell$ boundary components of the
corresponding ribbon surface $\Sigma$. Then the number of exterior
edges equals $s$ and the number of interior edges equals $(d-s)/2$,
where $d=d(1)+\dots+d(k)$ (since each interior edge meets precisely two
interior vertices). It follows that
$$
   2-2g-\ell = \chi(\Sigma) = \chi(\Gamma) = |C^0_{\text{\rm int}}| +
   |C^0_{\text{\rm ext}}| - |C^1_{\text{\rm int}}| - |C^1_{\text{\rm ext}}| = k + s  -
   \frac{d-s}{2} - s, 
$$
so $\ff_{k,\ell,g}$ has filtration degree
$$
   \|\ff_{k,\ell,g}\| = s-d = 2(2-2g-k-\ell) = 2\chi_{k,\ell,g}. 
$$
Similarly, $\fp_{k,\ell,g}$ has filtration degree $\|\fp_{k,\ell,g}\| =
2\chi_{k,\ell,g}$ for $(k,\ell,g)$ equal to $(1,1,0)$, $(2,1,0)$ or
$(1,2,0)$. Thus the operations $\fp_{k,\ell,g}$ and $\ff_{k,\ell,g}$
are filtered in the sense of Definitions~\ref{def:IBL-fil}
and~\ref{def:morphifil} with $\gamma=2$. Note that they are
$\N_0$-gapped in the sense of Definition~\ref{def:ggapped} for the discrete submonoid
$\N_0=\{0,1,2,\dots\}\subset\R_{\geq 0}$, where all the higher terms 
$\fp_{k,\ell,g}^j$ and $\ff_{k,\ell,g}^j$ vanish for $j\geq 1$. 
Hence Proposition~\ref{prop:structureexists},
Theorem~\ref{thm:homotopyequiv} and
Proposition~\ref{prop:filt-inverse} imply

\begin{Corollary}\label{cor:structureexists}
The operations $\fp_{1,1,0},\fp_{1,2,0},\fp_{2,1,0}$ in
Proposition~$\ref{prop:structureexists}$ induce on 
$\BC=B^{\text{\rm cyc}*}A[2-n]$ a {\em strict $\N_0$-gapped filtered
\text{\rm dIBL}-structure of bidegree $(d,\gamma)=(n-3,2)$}. Moreover, the 
operations $\ff_{k,\ell,g}$ in Theorem~$\ref{thm:homotopyequiv}$
induce a strict $\N_0$-gapped filtered $\text{\rm IBL}_\infty$-homotopy equivalence
between $B^{\text{\rm cyc}*}A[2-n]$ and $B^{\text{\rm cyc}*}B[2-n]$. 
\qed
\end{Corollary}

This corollary will be the basis for our discussion of Maurer-Cartan
elements in the following section.

\section{The dual cyclic bar complex of a cyclic
  DGA}\label{sec:Cyclic} 

In this section we show that for a cyclic DGA the $\text{\rm dIBL}$-structure on its
dual cyclic bar complex comes with a natural Maurer-Cartan
element. This gives rise to a twisted $\text{\rm dIBL}$-structure on the dual
cyclic bar complex, and thus to a twisted $\IBL_\infty$-structure on
its homology. In this way, we prove Proposition \ref{prop:existoncyc} and 
Theorem \ref{thm:existoncyc2} from the Introduction.
\par

We begin by considering cyclic $A_{\infty}$-structures.

\begin{Definition}\label{cycAinfty}
Let $(A,\langle\,,\,\rangle,d)$ be a cyclic cochain complex with
pairing of degree $-n$ and set $\frak
m_1:=d$. A series of operations
$$
   \frak m_k : A[1]^{\otimes k} \to A[1],\qquad k\geq 2
$$
of degree $1$ is said to define a {\em cyclic $A_{\infty}$ structure}
on $(A,\langle\,,\,\rangle,d)$ if for each $r\in\N$ the following holds:
\begin{equation}\label{eq:cyc1}
    \sum_{k+\ell=r+1\atop k,\ell\geq 1}\sum_{c=1}^{r+1-\ell} (-1)^*\frak
    m_k(x_1,\cdots,\frak m_\ell(x_c,\cdots,x_{c+\ell-1}),\cdots,x_{r}) = 0, 
\end{equation}
where $* = \deg x_1 +\cdots +\deg x_{c-1} + c-1$, and
\begin{equation}\label{eq:cyc2}
    \langle \frak m_k(x_1,\cdots,x_k),x_0\rangle = (-1)^{**} \langle \frak
    m_k(x_0,x_1,\cdots,x_{k-1}),x_k\rangle,
\end{equation}
where $** = (\deg x_0 +1)(\deg x_1 +\cdots +\deg x_k + k)$.
\end{Definition}

\begin{Remark}\label{ex:DGA}
We will refer to a cyclic $A_\infty$-algebra $(A,\langle\,,\,\rangle, \{\fm_k\})$ with $\fm_k=0$ for $k \ge 3$ as a {\em cyclic DGA}. Note that, as usual, to go from $\fm_2$ to a multiplication on $A$ which makes it a differential graded algebra in the usual sense involves adding signs; see~\eqref{intpairing} below for a possible convention. 
\end{Remark}

For operations $\frak m_k : A[1]^{\otimes k} \to A[1]$ we define $\frak m_k^+ :
A[1]^{\otimes(k+1)} \to \R$ by\footnote{The global sign $(-1)^{n-2}$
  is inserted to make Proposition~\ref{MCAinfty} below true. We do not
  have a conceptual explanation for this.} 
\begin{equation}
    \frak m^+_k(x_0,x_1,\cdots,x_k) := (-1)^{n-2}\langle \frak
    m_k(x_0,\cdots,x_{k-1}), x_k\rangle.
\end{equation}
Then $\frak m_k$ satisfies~\eqref{eq:cyc2} if and only if $\frak m^+_k
\in B^{\text{\rm cyc}*}_{k+1}A$. In this case we obtain an element 
\begin{equation}\label{defofm+}
    \frak m^+ := \sum_{k\geq 2} \frak m^+_k \in \wh\BC :=
    \wh{B^{\text{\rm cyc}*}A}[2-n]. 
\end{equation}
Note that $\fm^+$ is homogeneous of degree $n-3$ in $\wh{B^{\text{\rm cyc}*}A}$
and so it has degree $2(n-3)$ when viewed as an element of
$\wh{E_1}\BC=\wh{\BC}[1]=\wh{B^{\text{\rm cyc}*}A}[3-n]$. Moreover,
$\fm^+=\sum_{k\geq 2} \frak m^+_k$ has filtration degree at least $3$
with respect to the degree  $k$ in $B_k^{{\rm cyc}*}A$, 
so it satisfies the grading and filtration conditions for a
Maurer-Cartan element in the filtered \text{\rm dIBL}-algebra (of
bidegree $(n-3,2)$)
$$
    \bigl(\BC=B^{\text{\rm cyc}*}A[2-n],
      \fp_{1,1,0},\fp_{1,2,0},\fp_{2,1,0}\bigr)
$$
from Corollary~\ref{cor:structureexists}. 

\begin{Proposition}\label{MCAinfty}
Let $\{\frak m_k\}_{k\geq 2}$ satisfy~\eqref{eq:cyc2} of Definition
$\ref{cycAinfty}$. Then it satisfies~\eqref{eq:cyc1} if and only if
$$
    \fp_{1,1,0} \frak m^+ + \frac{1}{2}{\mathfrak p}_{2,1,0}(\frak m^+,\frak m^+) = 0 
$$
in $\wh{E_1}\BC$.
\end{Proposition}

\begin{proof}
Consider $k\geq 2$ and any element $\phi\in B_{\ell+1}^{{\rm cyc}*}A$,
and set $r:=k+\ell$. Since $|\fm^+|=n-3$, the relation~\eqref{eq:p-mu}
between $\fp_{2,1,0}$ and $\mu$ yields
$$
   \fp_{2,1,0}(\fm_k^+,\phi) = (-1)^{n-3}\mu(\fm_k^+,\phi).
$$
The formula~\eqref{eq:mu} for $\mu$ then yields
$$
   \fp_{2,1,0}(\fm_k^+,\phi)(x_1,\dots,x_r) =
   \sum_{c=1}^r(-1)^{\nu_c}b_k(x_c,\dots,x_{c-1}), 
$$
with
\begin{align*}
   &b_k(x_1,\dots,x_r) \cr
   &= \sum_{a,b}(-1)^{\eta_a+\eta_b(|x_1|+\cdots+|x_k|)+n-3}g^{ab}
   \fm_k^+(e_a,x_1,\dots,x_k)\phi(e_b,x_{k+1},\dots,x_r)
\end{align*}
and the sign
\begin{equation}\label{eq:nu}
   \nu_c = (|x_1|+\cdots+|x_{c-1}|)(|x_c|+\cdots+|x_r|). 
\end{equation}
Abbreviating $x:=(x_1,\dots,x_k)$ and using the symmetries of
$\fm_k^+$ and $\la\ ,\ \ra$, we compute 
\begin{align*}
   \fm_k^+(e_a,x)
   &= (-1)^{\eta_a|x|}\fm_k^+(x,e_a) \cr
   &= (-1)^{\eta_a|x|+n-2}\la\fm_k(x),e_a\ra \cr
   &= (-1)^{\eta_a|x|+n-2+\eta_a(|x|+1)+1}\la e_a,\fm_k(x)\ra \cr
   &= (-1)^{\eta_a+n-3}\la e_a,\fm_k(x)\ra.
\end{align*}
Using the relation
$$
   \sum_{a,b}g^{ab}\la e_a,z\ra e_b = z
$$
we obtain 
$$
   \sum_{a,b}(-1)^{\eta_a+n-3}g^{ab}\fm_k^+(e_a,x) e_b = \fm_k(x). 
$$
Next note that in the formula for $b_k$ we have the relations
$\eta_b=\eta_a+(n-2) = (|x|+n-3)+(n-2)=|x|+1$ (mod $2$), so the term
$\eta_b|x|$ is even and can be dropped from the sign
exponent. Inserting the previous formula, we obtain
$$
   b_k(x_1,\dots,x_r) =
   \phi\Bigl(\fm_k(x_1,\dots,x_k),x_{k+1},\dots,x_r\Bigr),
$$
and therefore
\begin{align}\label{eq:twisted}
   &\fp_{2,1,0}(\fm_k^+,\phi)(x_1,\dots,x_r) \cr
   &= \sum_{c=1}^r(-1)^{\nu_c}
   \phi\Bigl(\fm_k(x_c,\dots,x_{c+k-1}),x_{c+k},\dots,x_{c-1}\Bigr). 
\end{align}
Let us now insert $\phi=\fm_\ell^+$ with $\ell\geq 2$ in this formula
and consider a summand with $1\leq c\leq\ell$. Then
$x_r$ appears in the argument of $\fm_\ell^+$ and we can rewrite the
summand as
\begin{align*}
   &(-1)^{\nu_c}
  \fm_\ell^+\Bigl(\fm_k(x_c,\dots,x_{c+k-1}),x_{c+k},\dots,x_{c-1}\Bigr) \cr
   &=(-1)^{*}
  \fm_\ell^+\Bigl(x_1,\dots,\fm_k(x_c,\dots,x_{c+k-1}),x_{c+k},\dots,x_{r}\Bigr) \cr
   &=(-1)^{*+n-2}
  \Bigl\la\fm_\ell\Bigl(x_1,\dots,\fm_k(x_c,\dots,x_{c+k-1}),x_{c+k},\dots,x_{r-1}\Bigr),x_r\Bigr\ra 
\end{align*}
with
$*=\nu_c+(|x_1|+\cdots+|x_{c-1}|)(|x_c|+\cdots+|x_r|+1)=|x_1|+\cdots+|x_{c-1}|$
as in Definition~\ref{cycAinfty}.  
For $\ell+1\leq c\leq k+\ell$ we obtain a similar
expression with the roles of $\fm_k$ and $\fm_\ell$ interchanged. It
follows that
\begin{align*}
    & \frac{1}{2}\fp_{2,1,0}(\fm^+,\fm^+)(x_1,\dots,x_r)
    = \sum_{k+\ell=r\atop k,\ell\geq 2}\frac{1}{2}
    \fp_{2,1,0}(\fm_k^+,\fm_\ell^+)(x_1,\dots,x_r) \cr
    &= (-1)^{n-2}\left\la \sum_{k+\ell=r \atop k,\ell\geq
      2}\sum_{c=1}^{r-\ell} (-1)^*\frak
    m_k\Bigl(x_1,\cdots,\frak
    m_\ell(x_c,\cdots,x_{c+\ell-1}),\cdots,x_{r-1}\Bigr),x_r\right\ra.
\end{align*}
Note that the last sum contains all terms appearing in~\eqref{eq:cyc1}
with $k,\ell\geq 2$. The missing terms appear in
\begin{align*}
    & (-1)^{n-2}\fp_{1,1,0}\fm_{r-1}^+(x_1,\dots,x_r) \cr
    &= (-1)^{n-2}\sum_{c=1}^{r}(-1)^{|x_1|+\dots+|x_{c-1}|}
    \fm_{r-1}^+(x_1,\dots,dx_c,\dots x_r) \cr
    &= \sum_{c=1}^{r-1}(-1)^{|x_1|+\dots+|x_{c-1}|}
    \la \fm_{r-1}(x_1,\dots,\fm_1(x_c),\dots x_{r-1}),x_r\ra \cr 
    &\phantom{=} + (-1)^{|x_1|+\dots+|x_{r-1}|} \la
    \fm_{r-1}(x_1,\dots,x_{r-1}),dx_r\ra \cr
    &= \left\la \fm_1(\fm_{r-1}(x_1,\dots,x_r)) + \sum_{c=1}^{r-1}(-1)^*
    \fm_{r-1}(x_1,\dots,\fm_1(x_c),\dots x_{r-1}), x_r  \right\ra. 
\end{align*}
Here in the last equation we have used the relation $\la dx,y\ra =
-(-1)^{|x|}\la x,dy\ra$. Combining the preceding equations, we obtain
\begin{align*}
    & (-1)^{n-2}\left(\fp_{1,1,0}\fm^+ +\frac{1}{2}\fp_{2,1,0}(\fm^+,\fm^+)\right)(x_1,\dots,x_r)
    \cr
    &= \left\la \sum_{k+\ell=r\atop k,\ell\geq
      1}\sum_{c=1}^{r-\ell} (-1)^*\frak
    m_k(x_1,\cdots,\frak
    m_\ell(x_c,\cdots,x_{c+\ell-1}),\cdots,x_{r-1}),x_r\right\ra
\end{align*}
and Proposition~\ref{MCAinfty} follows. 
\end{proof}

Let $(A,\langle\,,\,\rangle,\frak m)$ be a cyclic $A_{\infty}$-algebra. 
Proposition~\ref{MCAinfty} implies that the twisted differential
$\fp_{1,1,0}^{\fm^+}:\wh\BC\to\wh\BC$, 
$$
   \fp_{1,1,0}^{\fm^+}(\varphi) := \fp_{1,1,0}\varphi + \fp_{2,1,0}(\fm^+,\varphi)
$$
squares to zero. For $\phi\in
B_{\ell+1}^{{\rm cyc}*}A$ and $r=k+\ell$, $k\geq 1$, the component of
$\fp_{1,1,0}^{\fm^+}(\varphi)$ in $B_{r+1}^{{\rm cyc}*}A$ is given
explicitly by 
\begin{align}\label{eq:twisted-diff}
   &\fp_{1,1,0}^{\fm^+}(\varphi)(x_1,\dots,x_r) \cr
   &= \sum_{c=1}^r(-1)^{\nu_c}
   \phi\Bigl(\fm_k(x_c,\dots,x_{c+k-1}),x_{c+k},\dots,x_{c-1}\Bigr),
\end{align}
with the sign exponent $\nu_c$ defined in~\eqref{eq:nu}. (According
to~\eqref{eq:twisted} this formula holds for $k\geq 2$, and one easily
verifies that it also holds for $k=1$.) Note that the differential
$\fp_{1,1,0}^{\fm^+}$ depends only on the $A_\infty$-operations
$\fm_k$ and not on the pairing $\la\ ,\ \ra$. 

\begin{Remark}\label{rem:cyc-hom}
In the case of a cyclic DGA (i.e.~$\fm_k=0$ for $k\geq 3$),
formula~\eqref{eq:twisted-diff} shows that the twisted differential
$\fp_{1,1,0}^{\fm^+}$ is dual to the Hochschild differential on
Connes' cyclic complex (see~\cite{Lod}), so its homology equals
Connes' version of cyclic cohomology.  
The precise relation to the definitions of cyclic cohomology appearing
in the literature in the $A_\infty$-case (such
as~\cite{GJ90,KoSo02,Tra08}) will be discussed elsewhere.   
\end{Remark}

Let $(A,\langle\,,\,\rangle,\frak m)$ be a cyclic
$A_{\infty}$-algebra. As above, let
$$
    \left(\BC=B^{\text{\rm cyc}*}A[2-n],
      \fp_{1,1,0},\fp_{1,2,0},\fp_{2,1,0}\right)
$$
be the filtered \text{\rm IBL}-algebra of bidegree $(n-3,2)$ from
Corollary~\ref{cor:structureexists}.  
According to Proposition~\ref{MCAinfty}, the degree $2(n-3)$ element
$\fm^+\in \wh{E_1}\BC=\wh{\BC}[1]=\wh{B^{\text{\rm cyc}*}A}[3-n]$ 
satisfies the first part of the Maurer-Cartan equation~\eqref{eq:MC2}
in Lemma~\ref{lem:MC}. Let us consider the second part of the
Maurer-Cartan equation, $\fp_{1,2,0}(\fm^+)=0$. Using the relation 
$\sum_ag^{ab}e_b=e^a$, we compute for $k_1,k_2\geq 1$,
$k=k_1+k_2+1\geq 3$: 
\begin{align*}
    & \delta(\fm_k^+)(x_1\cdots x_{k_1}\otimes y_1\cdots y_{k_2}) \cr
    &= \sum_{a,b}\sum_{c=1}^{k_1}\sum_{c'=1}^{k_2} (-1)^{\eta_a+\eta}
    g^{ab}\fm_k^+(e_a,x_c,\dots,x_{c-1},e_b,y_{c'},\dots,y_{c'-1}) \cr
    &= \sum_{c=1}^{k_1}\sum_{c'=1}^{k_2} \sum_{a}(-1)^{\eta_a+\eta+n-2}
    \la e_a,\fm_k(x_c,\dots,x_{c-1},e^a,y_{c'},\dots,y_{c'-1})\ra, 
\end{align*}
where $\eta$ is given by~\eqref{eq:eta}. This expression does not
vanish for a general cyclic $A_\infty$-algebra, so
$\fp_{1,2,0}(\fm^+)=0$ does not hold in general. However,
$\fp_{1,2,0}(\fm_k^+)$ vanishes for degree reasons if $k\leq 2$, so 
in view of Lemma~\ref{lem:MC}, Proposition~\ref{prop:MC} and
Remark~\ref{rem:cyc-hom} we conclude:

\begin{Proposition}\label{propIBLI2}
If $(A,\langle\,,\rangle,\fm_1=d,\frak m_2)$ is a cyclic DGA, then
$\fm_2^+$ defines a Maurer-Cartan element in the filtered
\text{\rm dIBL}-algebra $\bigl(\BC=B^{\text{\rm
      cyc}*}A[2-n],\fp_{1,1,0}=d,\fp_{1,2,0},\fp_{2,1,0}\bigr)$.  
It induces a twisted filtered $\text{\rm dIBL}$-structure
$$
    \fp^{\fm_2^+} =
    \{\fp_{1,1,0}+\fp_{2,1,0}(\fm_2^+,\cdot),\fp_{1,2,0},\fp_{2,1,0}\}
$$
on $\wh\BC$ whose homology equals Connes' version of cyclic cohomology
of $(A,d,\fm_2)$. 
\end{Proposition}

The special case of a DGA in Proposition~\ref{propIBLI2} will
suffice for the purposes of this paper. More generally, the preceding
computation shows that $\fp_{1,2,0}(\fm^+)=0$ holds if the $\fm_k$ are
``traceless'' in the sense that
\begin{equation*}
    \sum_{a}\la
    e_a,\fm_k(x_1,\dots,x_{k_1},e^a,x_{k_1+1},\dots,x_{k_1+k_2})\ra = 0 
\end{equation*}
for all $x_i$ and all $k_1,k_2\geq 1$, $k=k_1+k_2+1\geq 3$. 
The actual condition we need is weaker:

\begin{Definition}\label{Hderham}
A {\em solution of the genus zero master equation} for a  
cyclic $A_{\infty}$ algebra $(A,\langle\,,\rangle,\{\frak
m_k\}_{k=1}^{\infty})$ 
is a sequence of elements $\frak m_{(\ell)}^+ \in \wh{E_\ell}(B^{\text{\rm
    cyc}*}A)[2-n]$ for $\ell=1,2,\dots$ such that
\begin{itemize}
\item $\frak m_{(1)}^+$ coincides with $\frak m^+$ defined
  by~\eqref{defofm+}, and
\item
the following {\em Batalin-Vilkovisky master equation} is satisfied:
\begin{equation}\label{MEq}
\fp_{1,1,0}\frak m_{(\ell)}^+ + \,\,\widehat{\fp}_{1,2,0}(\frak m_{(\ell-1)}^+) +
\frac{1}{2}\sum_{i=1}^\ell \widehat{\fp}_{2,1,0}(\frak m_{(i)}^+,\frak m_{(\ell-i+1)}^+) = 0.
\end{equation}
\end{itemize}
\end{Definition}

\begin{Remark}\label{Remark127}
\begin{enumerate}[\rm (1)]
\item 
See Barannikov-Kontsevich \cite{BaKo}, Baranikov \cite{Bar1,Ba15}, 
Costello \cite{Cos}, Kontsevich-Soibelman \cite{KoSo02}
for related results. The
authors thank B. Vallette for a remark which is closely related to this point.
\item 
In the context of symplectic field theory, equation $(\ref{MEq})$
corresponds to equation $(16)$ in~\cite{CL2}
and to equation $(44)$ in \cite{EGH00}. It also coincides with
the genus zero case of equation $(5.5)$ in \cite{Ba15}.
\item
The relation between the ``traceless'' property and 
equation (\ref{MEq}) has appeared in \cite[Theorem 5]{Bar2}.
\item 
As will be discussed elsewhere, the second term $\frak m_{(2)}^+$ of the 
solution of the genus zero master equation is 
related to the Hodge to de Rham degeneration in Lagrangian 
Floer theory; see~\cite{KoSo02}.
The role of the solution of the genus zero master equation in Lagrangian
Floer theory is discussed at the end of the Introduction; see also
Remark~$\ref{filteredcase}$.  
\end{enumerate}
\end{Remark}

It is straightforward to check that equation (\ref{MEq}) is equivalent to the 
condition that $\fm^+=\{\frak m_{(\ell)}^+\}_{\ell=1}^{\infty}$ is a 
Maurer-Cartan element of the ${\rm dIBL}$-algebra $\bigl(\BC=B^{{\text{\rm cyc}}*}A[2-n],
      \fp_{1,1,0},\fp_{1,2,0},\fp_{2,1,0}\bigr)$, 
so we have the following generalization of Proposition \ref{propIBLI2}:

\begin{Proposition}\label{propIBLI3}
Let $(A,\langle\,,\,\rangle,\{\frak m_k\}_{k=1}^{\infty})$ 
be a cyclic $A_{\infty}$-algebra with a solution of the 
genus zero master equation $\{\frak m_{(\ell)}^+\}_{\ell=1}^{\infty}$. Then $\fm^+=\{\frak
m_{(\ell)}^+\}_{\ell=1}^{\infty}$ defines a Maurer-Cartan element of
$\bigl(\BC=B^{\text{\rm cyc}*}A[2-n], \fp_{1,1,0},\fp_{1,2,0},\fp_{2,1,0}\bigr)$.
\end{Proposition}

Finally, we consider the induced structure on cohomology. 
Let $(A,\langle \,,\,\rangle,d,\frak m_2)$ be a cyclic DGA
and $H = H(A,d)$ its cohomology. The inner product descends to
cohomology, so by Corollary~\ref{cor:structureexists} it induces a
$G$-gapped filtered \text{\rm IBL}-structure
$\{\fq_{1,1,0}=0,\fq_{1,2,0},\fq_{2,1,0}\}$ on $B^{{\text{\rm
        cyc}}*}H[2-n]$. Moreover, due to Lemma~\ref{lem:B} and
Corollary~\ref{cor:structureexists} (with $B\cong H$), 
there exists a filtered $\IBL_{\infty}$-homotopy equivalence 
$$
    \frak f : (B^{{\text{\rm cyc}}*}A)[2-n] \to (B^{{\text{\rm cyc}}*}H)[2-n]
$$
such that $\frak f_{1,1,0} : \wh{B^{{\text{\rm cyc}}*}A}[2-n] \to \wh{B^{{\text{\rm cyc}}*}H}[2-n]$ 
is the map induced by the dual of the inclusion $i:H\cong B \to
A$ from Lemma~\ref{lem:B}. According to Lemma~\ref{lem:pushforward},
the Maurer-Cartan element $\fm_2^+$ on $B^{{\text{\rm cyc}}*}A[2-n]$ from
Proposition~\ref{propIBLI2} can be pushed forward via $\ff$ to a
Maurer-Cartan element $\ff_*\fm_2^+$ on $B^{{\text{\rm cyc}}*}H[2-n]$. By
Proposition~\ref{prop:MC}, this induces a twisted filtered
$\IBL_\infty$-structure 
$$
    \fq^{\ff_*\fm_2^+} =
    \{\fq^{\ff_*\fm_2^+}_{1,\ell,g},\fq_{2,1,0}\}_{\ell\geq 1,g\geq 0}
$$
on $B^{\text{\rm cyc}*}H[2-n]$. By Proposition~\ref{prop:MC-homotopy}, 
this structure is homotopy equivalent to the twisted filtered
$\text{\rm dIBL}$-structure
$\fp^{\fm_2^+}=\{d+\fp_{2,1,0}(\fm_2^+,\cdot), \fp_{1,2,0},\fp_{2,1,0}\}$ 
on $B^{\text{\rm cyc}*}A[2-n]$ in Proposition~\ref{propIBLI2}.
The situation of Proposition \ref{propIBLI3} can be 
handled in the same way.

Thus we have proved the following theorem, whose first part corresponds to 
Theorem~\ref{thm:existoncyc2} in the Introduction.

\begin{Theorem}\label{homologyBLI}
\begin{enumerate}[\rm(a)]
\item Let $(A,\langle\,,\,\rangle,d,\frak m_2)$ be a cyclic DGA with cohomology
$H=H(A,d)$. Then $B^{\text{\rm cyc}*}H[2-n]$ carries a $G$-gapped
  filtered $\IBL_\infty$-structure 
which is homotopy equivalent to the twisted filtered $\text{\rm dIBL}$-structure
$\fp^{\fm_2^+}$ on $B^{\text{\rm cyc}*}A[2-n]$ in Proposition~$\ref{propIBLI2}$. 
In particular, its homology equals Connes' version of cyclic cohomology of
$(A,d,\fm_2)$. 

\item 
More generally, let $(A,\langle\,,\,\rangle,\{\frak m_k\}_{k=1}^{\infty})$ 
be a cyclic $A_{\infty}$-algebra with a solution of the genus zero master equation 
$\{\frak m_{(\ell)}^+\}_{\ell=2}^{\infty}$ and cohomology
$H=H(A,\frak m_1)$. Then $B^{\text{\rm cyc}*}H[2-n]$ carries a filtered 
$\IBL_\infty$-structure which is homotopy equivalent to
the twisted filtered $\IBL_\infty$-structure 
$\fp^{\fm^+}$ on $B^{\text{\rm cyc}*}A[2-n]$.
\end{enumerate}
\end{Theorem}

\begin{rem}\label{rem:MC}
The construction in Section~$\ref{sec:Cyclic-cochain}$ gives the
following explicit description of the Maurer-Cartan element
$\ff_*\fm_2^+$ on $B^{\text{\rm cyc}*}H[2-n]$, where $H$ is the
cohomology of a cyclic DGA. Denote by $RG^3_{k,\ell,g}$ the
set of isomorphism classes of ribbon graphs of signature $(k,\ell,g)$
all of whose interior vertices are {\em trivalent}. Then 
$$
   (\ff_*\fm_2^+)_{\ell,g} = \sum_{k=1}^\infty\sum_{\Gamma\in
     RG^3_{k,\ell,g}}\fn_\Gamma\in \wh{E_\ell} (B^{\text{\rm cyc}*}H)[2-n],
$$
where the numbers 
$$
   \fn_\Gamma(x^1_1\cdots x^1_{s_1}\otimes\dots\otimes x^\ell_1\cdots
x^\ell_{s_\ell}) 
$$
for $x^b_j\in H$ are defined as in Section~$\ref{sec:Cyclic-cochain}$
using the following assignments (compare Figure~\ref{fig:MCelement}):

%
%
\begin{figure}[h]
 \labellist
  \pinlabel $x_1^1$ [b] at 169 126
  \pinlabel $x_3^1$ [b] at 205 151
  \pinlabel $x_2^1$ [b] at 174 171
  \pinlabel $x_1^2$ [b] at 383 153
  \pinlabel $x_3^2$ [b] at 436 153
  \pinlabel $x_2^2$ [b] at 415 171
  \pinlabel $x_1^3$ [b] at 599 170
  \pinlabel $x_2^3$ [b] at 425 306
  \pinlabel $x_3^3$ [b] at 193 315
  \pinlabel $x_4^3$ [b] at 76 274
  \pinlabel $x_5^3$ [t] at 96 32
  \pinlabel $x_6^3$ [t] at 271 1
  \pinlabel $\frak m_2^+$ [l] at 169 252
  \pinlabel $\frak m_2^+$ [b] at 278 170
  \pinlabel $\frak m_2^+$ [t] at 322 158
  \pinlabel $\frak m_2^+$ [l] at 426 260
  \pinlabel $\frak m_2^+$ [r] at 128 83
  \pinlabel $\frak m_2^+$ [l] at 240 80
 \endlabellist
  \centering
  \includegraphics[scale=.53]{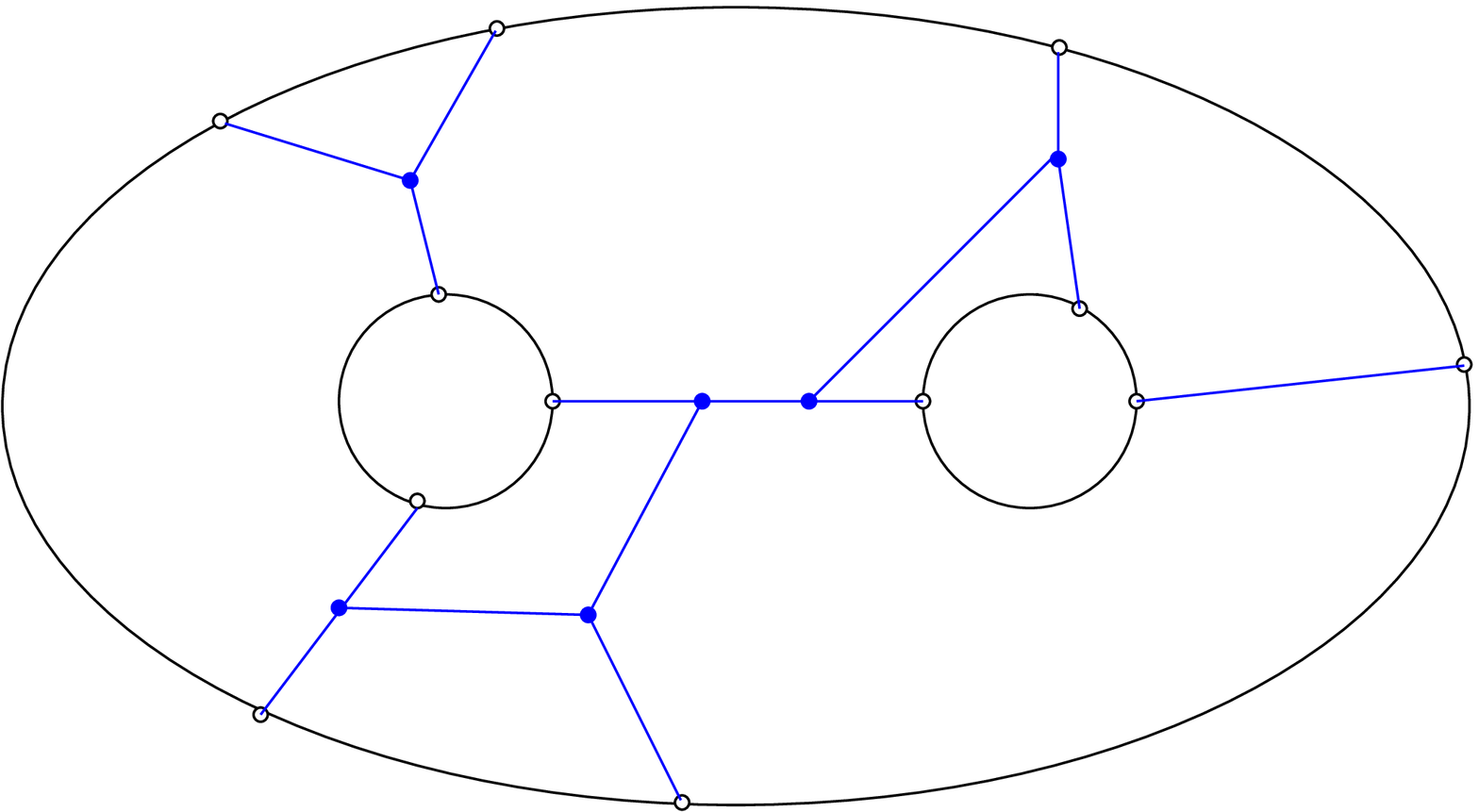}
 \caption{An example of a configuration contributing to $(\ff_*\fm_2^+)_{3,0}$.}
 \label{fig:MCelement}
\end{figure}

\begin{itemize}
\item To the $j$-th exterior vertex on $\p_b\Sigma$ we assign
   $x^b_j$. 
\item To each interior vertex of $\Gamma$ we assign $\fm_2^+$.  
\item To each interior edge we assign the element dual to the map
  $x\otimes y\mapsto \la Gx,y\ra$. 
\end{itemize}
There is a similar description in the case of a cyclic $A_{\infty}$ 
algebra with a solution of the genus zero master equation. 
\end{rem}
\begin{Remark}
The existence of a filtered $\IBL_\infty$ -structure on $B^{\rm{cyc} *} 
H[2-n]$
in Theorem \ref{homologyBLI} (a) also follows from \cite[Theorem 2]{Bar2}
by using Remark \ref{Remark127} and Proposition \ref{prop:MC}.
\end{Remark}

\begin{Remark}\label{filteredcase}
Let $\frak m_k$ be a cyclic {\em filtered} $A_{\infty}$ algebra $A$
defined over the universal Novikov ring $\Lambda_0$ with ground field
$\K$ (see \cite[Introduction]{Fuk10} for its definition). 
We decompose the differential as 
$$
\frak m_1 = \overline{\frak m}_1 + \frak m_{+,1}
$$
where $\frak m_{+,1} \equiv 0 \mod \Lambda_+$ 
and $\overline{\frak m}_1$ is induced from 
a $\K$-linear map $\overline C \to \overline C$.
Here $\overline C$ is a $\K$-vector space and 
$C = \overline C \,\widehat{\otimes}_\K\, \Lambda_0$.
We do not assume $\frak m_0 = 0$, but we do assume
$\frak m_0 \equiv 0 \mod \Lambda_+$ (see \cite{FOOO06,Fuk10}). 

We define $\frak m^+$ by using $\frak m_{k}$, $k> 1$, and 
$\frak m_{1,+}$ in the same way as in $(\ref{defofm+})$.
Let $\frak m^+_{(\ell)}$ be a solution of the genus zero master equation.
We assume $\frak m^+_{(\ell)} \equiv 0 \mod  \Lambda_+$
for $\ell\ge 2$.
It is easy to see that $\frak m_{(1)}^+ = \frak m^+$
together with $\frak m^+_{(\ell)}$, $\ell\ge 2$, defines 
a Maurer-Cartan element of the filtered ${\rm dIBL}$-algebra obtained from 
$(\overline C,\langle\,,\,\rangle,\overline{\frak m}_1)
 \,\widehat{\otimes}_\K\,  \Lambda_0$.

We consider $H := H(C,\overline{\frak m}_1)$.
(Note that $H$ is in general different from the homology of $(C,{\frak m}_1)$, 
even in the case $\frak m_0 = 0$.)
Now we apply Theorem  $\ref{thm:homotopyequiv}$ to obtain 
a twisted filtered $\IBL_{\infty}$-structure on 
$B^{\text{\rm cyc}}H$.
\end{Remark}

\section{The dual cyclic bar complex of the de Rham complex}
\label{de Rham}

In this section we prove Proposition~\ref{prop:cycIBLde-intro} and
Theorem~\ref{homotopyequiv-dR-intro} from the Introduction and discuss
Conjecture~\ref{deRhamconj-intro}. Throughout this section, we work
over the ring $R=\R$. 
\par
{\bf Fr\'echet ${\rm IBL}_\infty$-algebras. }
We start by recalling some basic facts about Fr\'echet spaces, see
e.g.~\cite{Rud73}. A {\em Fr\'echet space} $X$ is a topological vector
space whose topology is defined by a countable family of semi-norms
$\|\cdot\|_i$, $i\in\N$, such that $X$ is complete with respect to the
metric 
$$
   {\rm dist}(x,y) = \sum_{i=1}^\infty2^{-i}\frac{\|x-y\|_i}{1+\|x-y\|_i}.
$$ 
The basic example is the space $C^\infty(M,\R)$ of smooth functions on
a closed manifold $M$ with the semi-norms $\|f\|_k=\max_M|D^kf|$,
where $D^kf$ is the total $k$-th covariant derivative with respect to
some connection. In the same way the space $\Om(M)$ of smooth
differential forms also becomes a Fr\'echet space. 

According to~\cite{Sch53}, $\Om(M)$ belongs to a special class 
of  Fr\'echet spaces called nuclear spaces and 
two Fr\'echet spaces $X,Y$ have a natural tensor product 
$$   
   X\,\widehat\otimes \,Y
$$  
as a Fr\'echet space in case they are nuclear spaces. It is defined
as a suitable completion of the algebraic tensor product $X\otimes Y$
and characterized by the usual universal property. 
For two closed manifolds $M,N$ the canonical
inclusion $\Om(M)\otimes\Om(N)\subset\Om(M\times
N)$ (via the wedge product of differential forms) induces an isomorphism
$$
   \Om(M)\,\widehat\otimes\,\Om(N)\cong\Om(M\times N). 
$$
Consider now a graded Fr\'echet space $C=\bigoplus_{k\in\Z}C^k$ and
define the degree shifted space $C[1]$ as usual.  
The action of the symmetric group permuting the factors of the
algebraic tensor product $C[1]\otimes\cdots\otimes C[1]$ extends to
the completion, so we can define the {\em completed symmetric product} 
$$
   \wh E_kC\subset C[1]\wh\otimes\cdots\wh\otimes C[1]
$$
as the (closed) subspace invariant under the action of the symmetric
group with the usual signs (cf.~Remark~\ref{rem:tensor}). We set
$$
   \wh EC := \bigoplus_{k\geq 1}\wh E_kC. 
$$
Note that the meaning of $\wh E_kC$ and $\wh EC$ in this section
differs from that in Section~\ref{sec:filter}. 
Consider a series of continuous linear maps 
$$
    \frak p_{k,\ell,g} : \wh E_k C \to \wh E_{\ell} C,\qquad k,\ell \ge 1,\
    g\geq 0
$$ 
of degree 
$$
    |\frak p_{k,\ell,g}| = -2d(k+g-1)-1
$$ 
for some fixed integer $d$ and define the operator
$$
    \hat\fp := \sum_{k,\ell=1}^\infty\sum_{g=0}^\infty
    \hat\fp_{k,\ell,g} \hbar^{k+g-1}\tau^{k+\ell+2g-2}:
    \wh EC\{\hbar,\tau\}\to \wh EC\{\hbar,\tau\}
$$
as before. 

\begin{defn}\label{def:Frechet-IBL}
We say that $(C,\{\frak p_{k,\ell,g}\}_{k,\ell \geq 1,g\geq 0})$ is a 
{\em Fr\'echet $\IBL_\infty$-algebra of degree $d$} if 
\begin{equation*}
    \hat\fp\circ \hat\fp=0.
\end{equation*}
\end{defn}

The notions of Fr\'echet $\IBL_\infty$-morphisms and homotopies are
defined in the obvious way, requiring all maps to be continuous linear
maps between completed symmetric products. 
All the results from Sections~\ref{sec:def} to~\ref{sec:MC} carry over
to the Fr\'echet case. 

\begin{Remark}
If $C$ is finite dimensional (or more generally, filtered with finite
dimensional quotient spaces), then a Fr\'echet
$\IBL_\infty$-structure on $C$ is just an
$\IBL_\infty$-structure in the sense of
Definition~$\ref{def:IBL}$. 
\end{Remark}

{\bf The de Rham complex. }
Let $M$ be a closed oriented manifold of dimension $n$ and
$(\Omega(M),d,\wedge)$ its 
de Rham complex. Here $d$ is the exterior differential and $\wedge$
the wedge product on differential forms. Together with integration
over $M$ this is {\em almost} a cyclic DGA in the sense of Remark~\ref{ex:DGA}, 
since we can define  a ``cyclic $A_\infty$-structure'' with inner product of
degree $-n$ and with $\fm_k=0$ for $k\geq 3$ by setting  
\begin{equation}\label{intpairing}
\begin{aligned}
   \frak m_1(u) &:= du, \cr
   \frak m_2(u,v) &:= (-1)^{\deg u}u \wedge v, \cr
   \langle u,v\rangle &:= (-1)^{\deg u}\int_{M} u \wedge v. 
\end{aligned}
\end{equation}
%
We want to apply the arguments of Sections~\ref{sec:Cyclic-cochain}
and~\ref{sec:subcomplex} to this situation. 
However, the de Rham complex is infinite dimensional 
and the pairing in (\ref{intpairing}) is not perfect.
In fact, the dual to the space of smooth forms is 
the space of currents. In the following we will explain a method
to overcome this difficulty.
In the rest of this section, we will omit some sign computations. 

\begin{Definition}
A homomorphism 
$$
   \varphi : \underbrace{\Om(M)[1] \otimes \cdots \otimes
   \Om(M)[1]}_{\text{$k$ times}} \to \R
$$
is said to have {\em smooth kernel} if there exists a 
smooth differential form $\frak K_{\varphi}$ on
$M^k=M\times\dots\times M$ such that for $u_i \in \Omega(M)$ we have
\begin{equation}\label{eq:kernel}
   \varphi(u_1 \otimes \cdots \otimes u_k) 
   = \int_{M^k}  (u_1 \times \cdots \times u_k)
   \wedge \frak K_{\varphi}
\end{equation}
where $\times : \Omega(M) \otimes \Omega(M) \to \Omega(M^2)$ is the 
exterior wedge product.
We call $\frak K_{\varphi}$ the kernel of $\varphi$.
We write $B_k^*\Omega(M)_{\infty}$ for the subspace of such $\varphi$.
We set
$$
   B^{\text{\rm cyc}*}_k\Omega(M)_{\infty} := B_k^*\Omega(M)_{\infty}
   \cap B^{\text{\rm cyc}*}_k\Omega(M) 
$$
and
$$
   B^{\text{\rm cyc}*}\Omega(M)_{\infty} :=
   \bigoplus_{k\geq 1}B^{\text{\rm cyc}*}_k\Omega(M)_{\infty}.
$$
\end{Definition}

Note that the condition for an element of $B_k^*\Omega(M)_{\infty}$ to
belong to $B^{\text{cyc}*}_k\Omega(M)_{\infty}$ is equivalent to an
appropriate symmetry of the kernel with respect to cyclic permutation of
variables. 

\begin{Lemma}
The differential $\fp_{1,1,0}:B^{\text{\rm cyc}*}_k\Omega(M)\to
B^{\text{\rm cyc}*}_k\Omega(M)$ induced by $\fm_1=d$ preserves the
subspace $B^{\text{\rm cyc}*}\Omega(M)_{\infty}\subset
B^{\text{\rm cyc}*}_k\Omega(M)$.  
\end{Lemma}

\begin{proof}
For $\varphi\in B^{\text{\rm cyc}*}_k\Omega(M)$ with kernel $\frak
K_{\varphi}$ we compute
\begin{align*}
   \fp_{1,1,0}(\varphi) (u_1 \otimes \cdots \otimes u_k) 
   &= \sum_i \pm \varphi(u_1 \otimes \cdots du_i\otimes\cdots \otimes u_k)
   \cr 
   & = \int_{M^k}  \sum_i \pm (u_1 \times \cdots du_i\times\cdots \times u_k)
   \wedge \frak K_{\varphi} \cr
   & = \pm\int_{M^k}  (u_1 \times \cdots \times u_k)
   \wedge d\frak K_{\varphi},
\end{align*}
so $\pm d\frak K_{\varphi}$ is a kernel for $\fp_{1,1,0}(\varphi)$.  
\end{proof}

We next define a completion of 
the algebraic tensor products $B^{\text{cyc}*}_k\Omega(M)_{\infty} \,\otimes \,
B^{\text{cyc}*}_{\ell}\Omega(M)_{\infty}$. Note that the assignment
$\varphi\mapsto\frak K_{\varphi}$ defines a canonical inclusion
\begin{equation}\label{Binfintensor1}
   B^{\text{cyc}*}_k\Omega(M)_{\infty} \subset \Omega(M^k).
\end{equation}
This induces an inclusion of the algebraic tensor product
\begin{equation}\label{Binfintensor2}
   B^{\text{cyc}*}_k\Omega(M)_{\infty} \otimes
   B^{\text{cyc}*}_{\ell}\Omega(M)_{\infty} \subset \Omega(M^{k+\ell}).
\end{equation}
By the discussion above, the {\em completed tensor product}
$$
B^{\text{cyc}*}_k\Omega(M)_{\infty} \,\widehat\otimes \,
B^{\text{cyc}*}_{\ell}\Omega(M)_{\infty}
$$
as Fr\'echet spaces equals the closure of the image of the inclusion
(\ref{Binfintensor2}) 
with respect to the $C^{\infty}$ topology.
We define the completed tensor product among 
three or more $B^{\text{cyc}*}_k\Omega(M)_{\infty}$'s in the same way.

As above, we introduce the {\em completed symmetric product}
\begin{align*}
   &\widehat E_{m}\Bigl(B^{\text{cyc}*}\Omega(M)_{\infty}[2-n]\Bigr)
   \cr 
   &:= \Bigl(\bigoplus_{k_1,k_2,\dots,k_m\geq 1}
   B^{\text{cyc}*}_{k_1}\Omega(M)_{\infty}[3-n] 
   \widehat\otimes \dots \widehat\otimes
   B^{\text{cyc}*}_{k_m}\Omega(M)_{\infty}[3-n] \Bigr) \Bigl/\sim.
\end{align*}
By Remark~\ref{rem:tensor}, we can identify this quotient with the
subspace of elements in $B^{\text{cyc}*}_{k_1}\Omega(M)_{\infty} \widehat\otimes 
\dots \widehat\otimes B^{\text{cyc}*}_{k_m}\Omega(M)_{\infty}$
that are invariant under the action of the symmetric group on $m$
elements (with appropriate signs). In this way it is canonically
embedded into $\Omega(M^{k_1+\dots+ k_m})$. 
The following result corresponds to
Proposition~\ref{prop:cycIBLde-intro} from the introduction. 

%

\begin{prop}\label{prop:cycIBLde}
$B^{\text{\rm cyc}*}_k\Omega(M)_{\infty}[2-n]$ carries the structure of 
a Fr\'echet ${\rm dIBL}$-algebra of degree $n-3$. 
\end{prop}

\begin{proof}
We will give here the proof modulo signs, and postpone the discussion
of signs to Remark~\ref{rem:signs} below. 

Consider $\varphi \in B^{\text{cyc}*}_{k_1+1}\Omega(M)_{\infty}$, 
$\psi \in B^{\text{cyc}*}_{k_2+1}\Omega(M)_{\infty}$ with kernels 
$\frak K_{\varphi}$, $\frak K_{\psi}$. Let $k:=k_1+k_2$ and consider
$M^{k+1}$ with coordinates $(t,x_1,\dots,x_k)$. We define
$\fp_{2,1,0}(\varphi,\psi)\in B^{\text{cyc}*}_{k}\Omega(M)_{\infty}$
by the analogue of equation~\eqref{eq:mu}:
\begin{equation}\label{p21derham}
\aligned
   & \fp_{2,1,0}(\varphi,\psi)(u_1\cdots u_k) \cr
   &= \sum_{c=1}^k\pm \int_{M^{k+1}} \frak
   K_{\varphi}(t,x_1,\dots,x_{k_1})\wedge \frak
   K_{\psi}(t,x_{k_1+1},\dots,x_k)
   \\
   &\qquad\qquad\qquad\qquad\qquad\qquad\qquad
   \wedge
   u_{c+1}(x_1)\wedge\dots\wedge u_c(x_k) \cr
   &= \sum_{c=1}^k\pm \int_{M^{k+1}} \frak
   K_{\varphi}(t,x_{c+1},\dots,x_{c+k_1})\wedge \frak
   K_{\psi}(t,x_{c+k_1+1},\dots,x_c)\\
    &\qquad\qquad\qquad\qquad\qquad\qquad\qquad
    \wedge 
   u_{1}(x_1)\wedge\dots\wedge u_k(x_k)
\endaligned
\end{equation} 
For $1\leq c\leq k$ we define $I_c:M^{k+1}\to M^{k+2}$ by
$$
   I_c(t,x_1,\dots,x_k) :=
   (t,x_{c+1},\dots,x_{c+k_1},t,x_{c+k_1+1},\dots,x_c). 
$$
Let 
$$
   \text{Pr}_{\hat 1\,!} : \Omega(M^{k+}) \to \Omega(M^{k})
$$
be integration along the fiber associated to the projection
$\text{Pr}_{\hat 1}:M^{k+1} \to M^{k}$ forgetting the first
component. Then the preceding computation shows that
$\fp_{2,1,0}(\varphi,\psi)$ has kernel
\begin{equation}\label{defp21}
   \sum_{c=1}^k\pm \text{Pr}_{\hat 1\,!} I_{c}^*(\frak K_{\varphi} 
   \times \frak K_{\psi}). 
\end{equation}
We have thus defined 
$$
   \frak p_{2,1,0} : B^{\text{cyc}*}\Omega(M)_{\infty} \otimes 
   B^{\text{cyc}*}\Omega(M)_{\infty} \to B^{\text{cyc}*}_k\Omega(M)_{\infty}.
$$
By (\ref{defp21}) this operator extends to the completion 
$B^{\text{cyc}*}\Omega(M)_{\infty} \,\widehat\otimes \, 
B^{\text{cyc}*}\Omega(M)_{\infty}$.
\par

Next consider $\varphi \in B^{\text{cyc}*}_{k}\Omega(M)_{\infty}$ with
kernel $\frak K_{\varphi}$. For $k_1+k_2=k-2$ we consider
$M^{k-1}$ with coordinates
$(t,x_1,\dots,x_{k_1},y_1,\dots,y_{k_2})$. We define 
$$
   \fp_{1,2,0}(\varphi)\in \bigoplus_{k_1+k_2=k-2}
   B^{\text{cyc}*}_{k_1}\Omega(M)_{\infty} \,\widehat\otimes \,  
   B^{\text{cyc}*}_{k_2}\Omega(M)_{\infty}
$$
by the analogue of equation~\eqref{eq:delta}: 
\begin{equation}\label{p12def}
\aligned
   & \fp_{1,2,0}(\varphi)(u_1\cdots u_{k_1}\otimes v_1\cdots v_{k_2}) \cr
   &:= \sum_{c=1}^{k_1}\sum_{c'=1}^{k_2}\pm \int_{M^{k-1}} \frak
   K_{\varphi}(t,x_{c+1},\dots,x_{c},t,y_{c'+1},\dots,y_{c'})\wedge \cr
   & \ \ \ \ \ \ \  \ \wedge u_{1}(x_1)\wedge\dots\wedge u_{k_1}(x_{k_1}) \wedge
   v_{1}(y_1)\wedge\dots\wedge v_{k_2}(y_{k_2}). 
   \endaligned
\end{equation} 
For $1\leq c\leq k_1$ and $1\leq c'\leq k_2$ we define
$I_{c,c';k_1,k_2}:M^{k-1}\to M^{k}$ by 
$$
   I_{c,c'}(t,x_1,\dots,x_{k_1},y_1,\dots,y_{k_2}) :=
   (t,x_{c+1},\dots,x_{c},t,y_{c'+1},\dots,y_{c'}). 
$$
Then the preceding computation shows that
$\fp_{1,2,0}(\varphi)$ has kernel
\begin{equation}\label{defp12}
   \sum_{k_1+k_2=k-2}\sum_{c=1}^{k_1}\sum_{c'=1}^{k_2}\pm \text{Pr}_{\hat 1\,!}
   I_{c,c';k_1,k_2}^*\frak K_{\varphi}.  
\end{equation}
It is straightforward to check that this is indeed the kernel of an
element of  
$B^{\text{cyc}*}\Omega(M)_{\infty} \,\widehat\otimes \, 
B^{\text{cyc}*}\Omega(M)_{\infty}$.
Note that we need to take the completion here because the kernel
defined by~\eqref{defp12} is not a finite sum of exterior
products. 

The proof of the ${\rm dIBL}$-relations is now analogous to the proof of
Proposition~\ref{prop:structureexists} and is omitted. 
\end{proof}


The inner product in (\ref{intpairing}) induces one on de Rham cohomology 
$H_{dR}(M)$. The de Rham cohomology thus becomes a cyclic cochain complex 
with trivial differential. Since $H_{dR}(M)$ is finite dimensional, 
we can apply Proposition~\ref{prop:structureexists} 
to obtain an $\IBL$-structure on $B^{\text{\rm
    cyc}*}H_{dR}(M)[2-n]$. Note that this structure does not use the
cup product or higher products on $H_{dR}(M)$. 

For the remainder of this section, let us now fix a Riemannian metric
on $M$. We identify $H_{dR}(M)$ with the
space $\mathcal{H}(M)$ of harmonic forms, and denote by
$i:\mathcal{H}(M)\hookrightarrow\Om(M)$ 
the inclusion. Then we have the following analogue of
Theorem~\ref{thm:homotopyequiv} for the de Rham complex,
which corresponds to
Theorem~\ref{homotopyequiv-dR-intro} from the introduction.

\begin{thm}\label{homotopyequiv-dR}
There exists a Fr\'echet $\IBL_{\infty}$-morphism 
$$
    \frak f : B^{\text{\rm cyc}*}\Om(M)_{\infty}[2-n] \to B^{\text{\rm cyc}*}H_{dR}(M)[2-n]
$$
such that $\frak f_{1,1,0} : B^{\text{\rm cyc}*}\Om(M)_{\infty}[2-n] \to
B^{\text{\rm cyc}*}H_{dR}(M)[2-n]$ is the map induced by the dual of
the inclusion $i:H_{dR}(M)\cong \mathcal{H}(M) \hookrightarrow \Om(M)$. 
\end{thm}

\begin{proof}
We use the same notation as in the proof of Theorem
\ref{thm:homotopyequiv}. Consider a ribbon graph $\Gamma \in
RG_{k,\ell,g}$. We want to define 
\begin{equation}\label{constructionfde}
\aligned
\mathfrak f_\Gamma 
: &B^{\text{\rm cyc}*}_{d(1)}\Omega(M)_{\infty} \widehat\otimes \cdots
\widehat\otimes B^{\text{\rm cyc}*}_{d(k)}\Omega(M)_{\infty} 
\\ 
&\to B^{\text{\rm cyc}*}_{s_1}H_{dR}(M) \otimes \cdots \otimes B^{\text{\rm cyc}*}_{s_\ell}H_{dR}(M).
\endaligned
\end{equation}
The idea is to replace $G$ in Section \ref{sec:subcomplex} by the Green kernel 
and summation over the basis by integration.
We then define (\ref{constructionfde}) 
by a formula similar to (\ref{def:fGamma}) as follows.

Let $G(x,y) \in \Omega^{n-1}(M\times M)$ be the kernel of the Green
operator $G:\Om(M)\to\Om(M)$,
$$
   G(u)(x) = \int_{y\in M} G(x,y) \wedge u(y)
$$
(with respect to the fixed metric) satisfying
$$
   d \circ G +G \circ d = \Pi - \id
$$
where $\Pi:\Om(M)\to\mathcal{H}(M)$ is the orthogonal projection onto
the harmonic forms. 
($G$ is called the propagator in \cite{AxSi91II}.)

Now let $\varphi^v \in B^{\text{\rm cyc}*}_{d(v)}\Omega(M)_{\infty}$,
$v=1,\dots, k$ be given with kernels $\frak K^v(x_1,\dots,x_{d(v)})
\in \Omega(M^{d(v)})$. These kernels will play a similar role as
$\varphi^v_{i_1,\dots,i_{d(v)}}$ in Sections \ref{sec:Cyclic-cochain} and \ref{sec:subcomplex}.  
Let $\alpha^b_1\cdots \alpha^b_{s(b)}\in B^{\text{\rm
    cyc}}_{s(b)}H_{dR}(M)$, $b=1,\dots,\ell$ be given in terms of 
harmonic forms $\alpha^b_i$ associated to the exterior edges of
$\Gamma$.  

We denote by ${\rm Flag}(\Gamma)$ the set of {\em flags (or
half-edges)} of $\Gamma$, where a flag is a pair $f=(v,l)$ consisting of
an interior vertex and an (interior or exterior) edge with $v\in
l$. Suppose that we are given a labelling of $\Gamma$ and an ordering
and orientations of the interior edges in the sense of
Definitions~\ref{def:labelling} and~\ref{def:edge-ordering}. Given
these data, we can unambiguously associate flags
\begin{itemize}
\item $f(t,1)$ and $f(t,2)$ to the initial and end point of each
  interior edge $t$;
\item $f(v,1),\dots,f(v,d(v))$ to the ordered half-edges around each
  interior vertex $v$;
\item $f(b,1),\dots,f(b,s(b))$ to the ordered exterior edges ending on
  each boundary component $b$. 
\end{itemize}
To each flag $f$ we associate a variable $x_f$ which runs over
$M$. Then we set
\begin{equation}\label{fdefstateint}
\begin{aligned}
    &\mathfrak f_\Gamma(\varphi^1 \otimes \cdots \otimes \varphi^k)
   (\alpha^1_1\cdots \alpha^1_{s(1)}\otimes
   \cdots \otimes \alpha^\ell_1\cdots \alpha^\ell_{s(\ell)}) \cr
   &:= \int_{(x_f)\in M^{{\rm Flag}(\Gamma)}}
(-1)^\eta  
\prod_{t\in
   C^1_\inn(\Gamma)} G(x_{f(t,1)},x_{f(t,2)})
   \\ 
   &\ \ \ \prod_{v \in C^0_{\text{int}}(\Gamma)} \frak K^v(x_{f(v,1)},\dots, x_{f(v,d(v))}) 
    \prod_{b=1}^\ell\prod_{i=1}^{s(b)}
    \alpha^b_i(x_{f(b,i)}), 
\end{aligned}
\end{equation}
with appropriate signs $(-1)^\eta$. 

At this point, we could proceed to prove that (\ref{fdefstateint}) defines 
an $\IBL_{\infty}$-morphism in a way similar to the 
proof of Theorem~\ref{thm:homotopyequiv}, using Stokes' theorem.
Here we take a shortcut using finite dimensional approximations 
and reduce the proof to Theorem \ref{thm:homotopyequiv} itself, as we
will now explain.

The fixed Riemannian metric on $M$ induces the Laplace operator 
$\Delta$ and the Hodge star operator $*$ on $\Omega(M)$.
For a positive number $E$, we denote by $\Omega_E(M)$ 
the finite dimensional subspace of $\Omega(M)$ that is generated by 
eigenforms of $\Delta$ with eigenvalue $< E$.
The differential $d$ and the Hodge star operator preserve
$\Omega_E(M)$, 
so the pairing $\langle \,,\,\rangle$
is nondegenerate on $\Omega_E(M)$. 
Therefore, by Proposition \ref{prop:structureexists} we obtain 
the structure of a ${\rm dIBL}$-algebra on 
$$
   \text{\bf C}_E := B^{\text{\rm cyc}*}\Omega_E(M)[2-n].
$$
We denote it by $\frak p_{1,1,0;E}$, $\frak p_{2,1,0;E}$, $\frak p_{1,2,0;E}$.
To compare this structure to the Fr\'echet ${\rm dIBL}$-structure on
$$
   \text{\bf C}_{\infty} = B^{\text{\rm cyc}*}\Omega(M)_{\infty}[2-n],
$$
we need the following

\begin{lem}\label{lem:ext}
There exist canonical restriction and extension maps
$$
   B^{\text{\rm cyc}*}\Omega_E(M) \stackrel{\rm ext}\longrightarrow
   B^{\text{\rm cyc}*}\Omega(M)_{\infty} \stackrel{\rm rest}\longrightarrow 
   B^{\text{\rm cyc}*}\Omega_E(M) 
$$
satisfying ${\rm rest}\circ{\rm ext}=\id$. 
\end{lem}

\begin{proof}
The restriction map is just dual to the inclusion
$\Om_E(M)\hookrightarrow\Om(M)$. To define the extension map, we fix a
basis $\{e_i\}$ of $\Omega_E(M)$ of pure degree. As in
Section~\ref{sec:Cyclic-cochain}, we denote by $\{e^i\}$ be the dual
basis with respect to the pairing $\la\ ,\ \ra$, i.e. $\la e_i,e^j\ra
= \delta_i^j$, and we set $g^{ij} := \la e^i,e^j\ra$. 
Then to $\varphi\in B_k^{\text{\rm cyc}*}\Omega_E(M)$ we associate the
collection of real numbers
$$
   \varphi_{i_1\cdots i_k} := \varphi(e_{i_1},\dots,e_{i_k})
$$
and the smooth kernel 
$$
   {\frak K}_{\varphi}(x_1,\dots,x_k) := \sum_{i_1,\dots,i_k}
   {\varphi}_{i_1\cdots i_k} e^{i_1}(x_1)\cdots e^{i_k}(x_k).
$$
By formula~\eqref{eq:kernel}, this defines an extension of
$\varphi$ to $B_k^{\text{\rm cyc}*}\Omega(M)_{\infty}$. To check the
identity ${\rm rest}\circ{\rm ext}=\id$, we compute for
$u_1,\dots,u_k\in\Om_E$:
\begin{align*}
   &\int_{M^k}  (u_1 \times \cdots \times u_k)
   \wedge \frak K_{\varphi} \cr
   &= \sum_{i_1,\dots,i_k}{\varphi}_{i_1\cdots i_k} \int_{M^k}u_1(x_1)\cdots
   u_k(x_k) e^{i_1}(x_1)\cdots e^{i_k}(x_k) \cr
   &= \sum_{i_1,\dots,i_k} \pm {\varphi}_{i_1\cdots i_k} \la
   u_1,e^{i_1}\ra\cdots \la u_k,e^{i_k}\ra \cr 
   &= \pm \varphi\Bigl(\sum_{i_1}\la u_1,e^{i_1}\ra e_{i_1}\,,\dots,\,
   \sum_{i_k}\la u_k,e^{i_k}\ra e_{i_k}\Bigr) \cr
   &= \pm \varphi(u_1,\dots,u_k),
\end{align*}
where in the last equation we have used $\sum_j\la u_i,e^j\ra
e_j=u_i$. 
\end{proof}

Using this lemma, we associate to a map $f:E_k{\bf C}_\infty \to
E_\ell{\bf C}_\infty$ its {\em restriction} $f_E$ as the composition 
$$
   f_E : E_k{\bf C}_E \stackrel{\rm ext}\longrightarrow 
   E_k{\bf C}_\infty \stackrel{f}\longrightarrow E_\ell{\bf C}_\infty 
   \stackrel{\rm rest}\longrightarrow E_\ell{\bf C}_E. 
$$

\begin{lem}\label{restpisp}
The operators $\frak p_{1,1,0;E}$, $\frak p_{2,1,0;E}$, $\frak
p_{1,2,0;E}$ on ${\bf C}_E$ are the restrictions of the operators 
$\frak p_{1,1,0}$, $\frak p_{2,1,0}$, $\frak
p_{1,2,0}$ on ${\bf C}_\infty$. 
\end{lem}

\begin{proof}
The operator $\frak p_{1,1,0;E}$ is clearly the restriction of $\frak
p_{1,1,0}$ because both are induced by $\frak m_1$. Let us check that 
$\frak p_{2,1,0;E}$ is the restriction of $\frak p_{2,1,0}$ (the case
of $\frak p_{1,2,0;E}$ and $\frak p_{1,2,0}$ is analogous). Consider 
$\varphi\in B_{k_1+1}^{\text{\rm cyc}*}\Omega_E(M)$ and $\psi\in
B_{k_2+1}^{\text{\rm cyc}*}\Omega_E(M)$. We define their kernels as
above and denote their extensions by the same letters. Then for
$k=k_1+k_1$ and $u_1,\dots,u_k\in\Om_E$ we compute, using the
definitions~\eqref{p21derham} and~\eqref{eq:mu}:
\begin{align*}
   & \fp_{2,1,0}(\varphi,\psi)(u_1\cdots u_k) \cr
   &= \sum_{c=1}^k\pm \int_{M^{k+1}} \frak
   K_{\varphi}(t,x_1,\dots,x_{k_1})\wedge \frak
   K_{\psi}(t,x_{k_1+1},\dots,x_k)
   \\
   &\qquad\qquad\qquad\qquad\qquad\qquad\qquad
   \wedge
   u_{c+1}(x_1)\wedge\cdots\wedge u_c(x_k) \cr
   &= \sum_{c=1}^k\sum_{a,b,i_1,\dots,i_k} \pm {\varphi}_{a\,i_1\cdots i_{k_1}}
   \psi_{b\,i_{k_1+1}\cdots i_k} \int_{M^{k+1}} u_{c+1}(x_1)\cdots
   u_c(x_k) \cr
   &\qquad\qquad\qquad\qquad \wedge
   e^a(t)e^{i_1}(x_1)\cdots e^{i_{k_1}}(x_{k_1})
   e^b(t)e^{i_{k_1+1}}(x_{k_1+1})\cdots e^{i_{k}}(x_{k}) \cr 
   &= \sum_{c=1}^k\sum_{a,b,i_1,\dots,i_k} \pm {\varphi}_{a\,i_1\cdots i_{k_1}}
   \psi_{b\,i_{k_1+1}\cdots i_k} g^{ab}\la u_{c+1},e^{i_1}\ra \cdots
   \la u_c,e^{i_k}\ra \cr  
   &= \sum_{c=1}^k\sum_{a,b} \pm
   g^{ab}{\varphi}(e_a,u_{c+1},\dots,u_{c+k_1}) 
   \psi(e_b,u_{c+k_1+1},\dots,u_c) \cr  
   &= \pm\fp_{2,1,0;E}(\varphi,\psi)(u_1\cdots u_k). 
\end{align*}
This proves Lemma~\ref{restpisp} up to signs. Now we remark that the
signs in (\ref{p21derham}) and (\ref{p12def}) are  
not yet defined. So we {\it define} the signs there so that
Lemma~\ref{restpisp} holds {\it with signs}.
\end{proof}

\begin{rem}\label{rem:signs}
Before completing the proof of Theorem \ref{homotopyequiv-dR}, 
let us complete the sign part of the proof of Proposition
\ref{prop:cycIBLde}. 
It is easy to see from the definition that $\text{\bf C}_E \subset
\text{\bf C}_{E'}$ for $E < E'$ and the restrictions of $\frak
p_{2,1,0;E'}$, $\frak p_{1,2,0;E'}$ to $\text{\bf C}_E$ coincide with
$\frak p_{2,1,0;E}$, $\frak p_{1,2,0;E}$. Therefore, they define
operators on the union $\bigcup_E\text{\bf C}_E$. 
Then formulas (\ref{p21derham}), (\ref{p12def}) imply 
that they further extend continuously to the closure 
$\text{\bf C}_{\infty}$ of $\bigcup_E\text{\bf C}_E$
with respect to the $C^{\infty}$ topology. 
These are the operators in Proposition \ref{prop:cycIBLde}.
Now the signs work out because they do on ${\bf C}_E$ for each finite
$E$. 
The same remark applies to the rest of the proof of 
Theorem $\ref{homotopyequiv-dR}$.
\end{rem}

Now we return to the proof of Theorem \ref{homotopyequiv-dR}. 
Denote by $\Om_+(M)$ the direct sum of the eigenspaces of $\Delta$ to
positive eigenvalues. By the Hodge theorem, $\Om_+(M)=\im d\oplus\im
d^*$ and $d:\im d^*\to\im d$, $d^*:\im d\to\im d^*$ are
isomorphisms. So we can choose a basis $\{e_i, f_i\}$ of $\Om_+(M)$
satisfying 
$$
   \Delta e_i = \lambda_i e_i,\quad \Delta f_i = \lambda_i f_i, \quad 
   d f_i = e_i,\quad de_i = 0 
$$
by picking a basis $\{e_i\}$ for $\im d$ and setting
\begin{equation}\label{defoff}
   f_i := \frac{1}{\lambda_i} d^* e_i.
\end{equation}
Here $d^*= - * d *$.
Then we can choose the Green operator $G$ so that
\begin{equation}\label{defG}
   G|_{\mathcal{H}(M)}=0,\quad G(e_i) = f_i,\quad 
   G(f_i) = e_i.
\end{equation}
On the other hand, we have
\begin{equation}\label{enormal}
\langle e_i,e_j\rangle = 
\int_M df_i \wedge e_j 
= \pm \int_M f_i \wedge de_j  = 0.
\end{equation}
Similarly, (\ref{defoff}) implies $\langle f_i,f_j\rangle = 0$ and 
$$
\langle e_i,f_j\rangle
= 
\int_M df_i \wedge f_j 
= 
\pm \frac{1}{\lambda_j }\int_M f_i \wedge dd^*e_i 
= \pm \langle f_i,e_j\rangle
$$
if $\lambda_i = \lambda_j$, and $\langle e_i,f_j\rangle = 0$
otherwise. We set 
\begin{equation}\label{hijinnner}
   \frak h_{ij} := \langle e_i,f_j\rangle
   = \pm \langle f_i,e_j\rangle.
\end{equation}
Let $(\frak h^{ij})$ be the inverse matrix of $(\frak h_{ij})$. 
Then the propagator is given by
\begin{equation}\label{propagatorGcalc}
   G(x,y) = \sum_{i,j} \pm \frak h^{ij} f_i(x) \wedge f_j(y), 
\end{equation}
where the sign depends only on the degrees of $f_i, f_j$.

We restrict $G$ in (\ref{defG}) to $\Omega_E(M)$ and obtain an
operator $G_E$. We use it in the proof of Theorem
\ref{thm:homotopyequiv} to obtain an $\IBL_{\infty}$-morphism 
$\{\frak f_{\Gamma,E}\}$ from $\text{\bf C}_E$ to 
$B^{\text{\rm cyc}*}\mathcal{H}(M)$.

\begin{lem}\label{GammacoincidenceE}
The restriction of the operator $f_\Gamma$ defined by
$(\ref{fdefstateint})$ to $\text{\bf C}_E$ coincides with 
$\frak f_{\Gamma,E}$.
\end{lem}

\begin{proof}
Using (\ref{defG}), (\ref{enormal}), (\ref{hijinnner}) 
and (\ref{propagatorGcalc}) it is easy to see that 
the operators coincide up to sign. Now we define the sign 
in (\ref{fdefstateint}) so that they coincide with sign.
\end{proof}

It is easy to see that the maps $\frak f_{\Gamma,E}$ are compatible
with the inclusions $\text{\bf C}_E \subset \text{\bf C}_{E'}$.
So they define an $\IBL_{\infty}$-morphism on the union
$\bigcup_{E}\text{\bf C}_E$. 
Now (\ref{fdefstateint}) and Lemma \ref{GammacoincidenceE}
imply that this morphism extends to the closure $\text{\bf
  C}_{\infty}$ with respect to the $C^{\infty}$ topology, and the
proof of Theorem \ref{homotopyequiv-dR} is complete.
\end{proof}

{\bf String topology and Conjecture~\ref{deRhamconj-intro}. } 
Now we would like to proceed as in Section~\ref{sec:Cyclic}: twist the
Fr\'echet ${\rm dIBL}$-structure on $B^{\text{\rm
    cyc}*}\Omega(M)_{\infty}[2-n]$ by the Maurer-Cartan element
$\fm_2^+$ arising from the product $\fm_2$ in~\eqref{intpairing}, and
use Theorem~\ref{homotopyequiv-dR} to push the twisted
$\IBL_\infty$-structure onto $B^{\text{\rm
    cyc}*}H_{dR}(M)[2-n]$.
However, there is one difficulty in doing so. As in
Section~\ref{sec:Cyclic}, the product $\frak m_2(u,v) = (-1)^{\deg u}
u \wedge v$ gives rise to an element 
$$
   \frak m^+_2 \in B^{\text{\rm cyc}*}_3\Omega(M). 
$$
defined by
$$
   \frak m^+_2(u,v,w) :=  (-1)^{n-2}\la\fm_2(u,v),w\ra =
   (-1)^{n-2+\deg v} \int_M u \wedge v\wedge w.
$$
However, $\frak m^+_2$ does {\it not} have a smooth kernel.
In fact, its Schwartz kernel is the current represented by the triple  
diagonal 
$$
   \Delta_M = \{ (x,x,x) \mid x \in M \} \subset M^3.
$$
So we cannot use $\fm_2^+$ directly to twist the Fr\'echet
$\text{\rm dIBL}$-structure on the space $B^{\text{\rm cyc}*}\Omega(M)_{\infty}[2-n]$. On
the other hand, we may try to formally apply the map in
Theorem~\ref{homotopyequiv-dR} to $e^{\fm_2^+}$ and then show that
this defines a twisted $\IBL_\infty$-structure on $B^{\text{\rm
    cyc}*}H_{dR}(M)_{\infty}[2-n]$. 
For this, we need to consider the integrals (\ref{fdefstateint}) for
trees $\Gamma$ all of whose interior vertices are trivalent and with
inputs $\varphi^v = \frak m^+_2$ for all $v=1,\dots,k$ (see
Remark~\ref{rem:MC}).  
Then the $\frak K^{v}$ appearing in the formula must be the current
$\Delta_M$. 
Hence in place of the second product of the right hand side of
(\ref{fdefstateint}) (involving the kernels $\frak K^{v}$) we restrict
to the submanifold  
$$
   \{x_{f(v,1)} = x_{f(v,2)} = x_{f(v,3)} \text{ for all }v\in
   C^0_{\rm int}(\Gamma)\}
$$
and perform the integration of the other differential forms (the Green
kernels and the harmonic forms) over this submanifold. This integral
is very similar to those appearing in  
perturbative Chern-Simons gauge theory~\cite{Wit95,BarNat95}.  
The difficulty with this integral comes from the singularity of the
propagator $G(x,y)$ at the diagonal $x=y$.  
We believe that one can resolve this problem by using a real version
of the Fulton-MacPherson compactification~\cite{FM} in a similar way 
as in~\cite{AxSi91II}. As a result, 
one should obtain Conjecture~\ref{deRhamconj-intro} from the
Introduction: {\it There exists an $\IBL_\infty$-structure on 
$B^{\text{\rm cyc}*}H_{dR}(M)[2-n]$ whose homology equals Connes'
version of cyclic cohomology of the de Rham complex of $M$.}

{\bf Lagrangian Floer theory and Conjecture~\ref{Lagconj-intro}. }
We expect that the ideas of this paper can be applied to study
Lagrangian Floer theory of arbitrary genus and with an arbitrary number of
boundary components in the following way. Let $L$ be an $n$-dimensional closed
Lagrangian submanifold of a symplectic manifold $(X,\omega)$, and $J$ be
an almost complex structure on $X$ compatible with $\omega$. 
For fixed integers $g\ge 0$, $\ell \ge 1$ and $s_1,\dots,s_{\ell}\ge
1$ and a relative homology class $\beta \in H_2(X,L;\Z)$
we consider $(\Sigma,\vec{\vec z},u)$ such that
\begin{itemize}
\item
$\Sigma$ is a compact oriented Riemann surface of genus $g$ with
$\ell$ boundary components;
\item $\vec{\vec z} = (\vec z_1,\dots,\vec z_{\ell})$, where 
$\vec z_b = (z_{b,1},\dots,z_{b,s_b})$ is a vector of $s_b$ distinct
points in counterclockwise order on the $b$-th boundary component
$\p_b\Sigma\cong S^1$ of $\Sigma$;
\item
$u : (\Sigma,\partial \Sigma) \to (X,L)$ is a $J$-holomorphic map
representing the relative homology class $\beta$.
\end{itemize}
We denote the compactified moduli space of such $(\Sigma,\vec{\vec z},u)$
by $\mathcal M_{g;(s_1,\dots,s_{\ell})}(\beta)$.
There is a natural evaluation map
$$
{\rm ev} = ({\rm ev}_1,\dots,{\rm ev}_{\ell}):
\mathcal M_{g;(s_1,\dots,s_{\ell})}(\beta)
\to L^{s_1} \times \dots \times L^{s_{\ell}}.
$$
Integrating the pullback of differential forms under the evaluation
map over the moduli space (using an appropriate version of the
virtual fundamental chain technique) should give rise to an element
$$
\frak m_{g;(s_1,\dots,s_{\ell})}(\beta)
\in B^{{\rm cyc}*}_{s_1}S(L)[2-n] \otimes \dots \otimes B^{{\rm
    cyc}*}_{s_{\ell}}S(L)[2-n],
$$
where $S(L)$ is a suitable cochain complex realizing the cohomology of $L$.
Let $\Lambda_0$ be the universal Novikov ring of formal power series
in the variable $T$ introduced in Section~\ref{sec:filter}. 
The elements $\frak m_{g;(s_1,\dots,s_{\ell})}(\beta)$ should then combine
to elements 
$$
   \frak m_{g,\ell} := \sum_{s_1,\dots,s_\ell}\sum_\beta \frak
   m_{g;(s_1,\dots,s_{\ell})}(\beta) T^{\omega(\beta)} 
   \in \wh E_{\ell}\Bigl(B^{{\rm cyc}*}S(L)[2-n]\otimes
   \Lambda_{0}\Bigr),
$$
where $\wh E_\ell$ denotes the completed symmetric product 
with respect to the two natural filtrations on $B^{{\rm
    cyc}*}S(L)[2-n]\otimes\Lambda_{0}$. 
The boundary degenerations of the moduli spaces $\mathcal
M_{g;(s_1,\dots,s_{\ell})}(\beta)$ suggest that
$$
   \frak m := \sum_{g,\ell} \frak m_{g,\ell}\hbar^{g-1}
   \in \frac{1}{\hbar}\wh E\Bigl(B^{{\rm cyc}*}S(L)[2-n]\otimes
   \Lambda_{0}\Bigr)\{\hbar\}
$$
satisfies the Maurer-Cartan equation
$$
   \widehat{\frak p}\left( e^{\frak m} \right) = 0
$$
for an appropriate ${\rm IBL}_{\infty}$-structure $\widehat{\frak p}$
on $B^{{\rm cyc}*}S(L)$. 
\par
We hope to work this out by taking for $S(L)$ de Rham complex $\Om(L)$
and using the results of this section as follows. 
Recall from Proposition~\ref{prop:cycIBLde} that the subcomplex
$B^{\rm cyc*}\Om(L)_\infty[2-n]\subset B^{\rm cyc*}\Om(L)[2-n]$ of
homomorphisms with smooth kernel carries a natural Fr\'echet ${\rm
  dIBL}$-structure. However, the elements $\fm_{g,\ell}$ do not have
smooth kernel and thus cannot be viewed as a Maurer-Cartan element on
this Fr\'echet ${\rm dIBL}$-algebra. The idea is now to use the
construction in Conjecture~\ref{deRhamconj-intro} to push forward the
$\fm_{g,\ell}$ to a Maurer-Cartan element on $B^{\rm cyc*}H_{\rm
  dR}(L)[2-n]\otimes\Lambda_0$. It should give rise to a twisted
filtered ${\rm IBL}_\infty$-structure on $B^{\rm cyc*}H_{\rm
  dR}(L)[2-n]\wh\otimes\Lambda_0$
whose homology equals the cyclic cohomology of the cyclic
$A_\infty$-structure on $H_{\rm dR}(L)$ constructed in~\cite{FOOO06,
  Fuk10}. Moreover, its reduction at $T=0$ should equal the filtered 
${\rm IBL}_\infty$-structure on $B^{\rm cyc*}H_{\rm
  dR}(L)[2-n]$ in Conjecture~\ref{deRhamconj-intro}, and
Conjecture~\ref{Lagconj-intro} from the Introduction should follow.

\appendix

\section{Orientations on the homology of surfaces}\label{sec:or}

In this appendix, we present a procedure to order and orient the
interior edges of a labelled ribbon graph in terms of orientations on
the singular chain complex of the associated surface. Relating this to
the procedure using spanning trees in
Definition~\ref{def:edge-ordering}, we prove
Lemma~\ref{lem:edge-convention} from Section~\ref{sec:Cyclic-cochain}. 

Consider a ribbon graph $\Gamma$ with $k$ interior vertices
$v_1,\dots,v_k$ and $m$ interior edges $e_1,\dots,e_m$. 
We denote by $\Sigma$ the surface with boundary
associated to $\Gamma$, and by $\wh\Sigma$ the closed connected
oriented surface (of genus $g$) obtained by gluing disks to the $\ell$
boundary components of $\Sigma$. Note that
$$
   2-2g = \chi(\wh\Sigma) = \chi(\Sigma)+\ell = \chi(\Gamma)+\ell =
   k-m+\ell. 
$$
We view $\Gamma$ as a graph on $\wh\Sigma$. The ribbon condition
implies that $\wh\Sigma\setminus\Gamma$ is the union of $\ell$ disks
whose closures we denote by $f_1,\dots,f_\ell$. 
So we have a cell complex
\begin{equation}\label{eq:or1}
   C_2=\la f_1,\dots,f_\ell\ra\stackrel{\p_2}{\longrightarrow} 
   C_1=\la e_1,\dots,e_m\ra\stackrel{\p_1}{\longrightarrow} 
   C_0=\la v_k,\dots v_1\ra
\end{equation}
(say, with $\Q$-coefficients) computing the homology 
of $\wh\Sigma$. We pick complements $V_i$ of $\ker\p_i$ in $C_i$ and
$H_i$ of $\im\p_{i+1}$ in $\ker\p_i$, so that the complex becomes
\begin{equation}\label{eq:or2}
   C_2=V_2\oplus H_2\stackrel{\p_2}{\longrightarrow} 
   C_1=V_1\oplus H_1\oplus\im\p_2\stackrel{\p_1}{\longrightarrow} 
   C_0=\im\p_1\oplus H_0. 
\end{equation}
Note that $\p_i:V_i\to\im\p_i$ are isomorphisms and the $H_i$ are
isomorphic to the homology groups of $\wh\Sigma$. Now
equation~\eqref{eq:or1} determines an orientation of $C=C_0\oplus
C_1\oplus C_2$ once we choose a labelling of $\Gamma$ in the sense of
Definition~\ref{def:labelling}, i.e., 
\begin{enumerate}[(i)]
\item an ordering $v_1,\dots,v_k$ of the interior vertices,
\item and ordering and orientations of the interior edges,
\item an ordering of the boundary components. 
\end{enumerate}
More precisely, given these choices we orient the chain groups as
follows: 
\begin{enumerate}[(i)']
\item an oriented basis $v_k,\dots,v_1$ of $C_0$ is given by the
  interior vertices in {\em reverse} order;
\item an oriented basis $e_1,\dots,e_m$ of $C_1$ is given by the
  oriented interior edges in their given order; 
\item an oriented basis $f_1,\dots,f_\ell$ of $C_2$ is given by the
  $2$-cells, oriented according to the orientation of $\wh\Sigma$ and
  ordered according to the ordering of the boundary components. 
\end{enumerate}
We arbitrarily orient $V_1,V_2$ and equip $\im\p_1,\im\p_2$ with the
induced orientations via the isomorphisms $\p_i:V_i\to\im\p_i$. These
orientations together with the orientation of $C$ induce
via~\eqref{eq:or2} an orientation on $H=H_0\oplus H_1\oplus H_2$. This
orientation does not depend on the chosen orientations on $V_1,V_2$,
but it does depend on the order of the direct summands
in~\eqref{eq:or2}. 

On the other hand, the homology $H$ is canonically oriented: We set 
$$
   H_0:=\la v_1+\cdots+v_k\ra,\qquad H_2:=\la f_1+\cdots+f_\ell\ra
$$
and give $H_1\cong\Q^{2g}$ the symplectic orientation induced by the
intersection form. We define the sign exponent $\eta_3(\Gamma)$ (mod $2$) as $0$ if
the two orientations of $H$ coincide, and $1$ if not. Here we always
consider the labelling data (i)-(iii) as part of $\Gamma$.

\begin{rem}
To a ribbon graph $\Gamma$ we can associate its dual graph
$\bar\Gamma$ lying on the same closed surface $\wh\Sigma$ by
exchanging vertices and faces. The edges of $\bar\Gamma$ are
canonically oriented by requiring that the intersection number of each
edge of $\Gamma$ with the corresponding edge of $\bar\Gamma$ is
$+1$. The dual graph to $\bar\Gamma$ is $\Gamma$ with the orientation
of all edges reversed, which has the same sign as $\Gamma$ iff the
number of edges is even. Thus the sign of $\bar\Gamma$ cannot always
be equal to the sign of $\Gamma$. Is there a simple criterion to
decide when the signs agree. Is this duality good for anything?
\end{rem}

{\bf Orientation via spanning trees. }
An orientation on the chain group $C=C_0\oplus C_1\oplus C_2$
associated to a labelled ribbon graph $\Gamma$ can be specified via
the construction in Definition~\ref{def:edge-ordering}, which we first
recall. 

Choose a maximal tree $T \subset \Gamma_{\rm int}$. It will have $k-1$
edges, which we orient away from vertex 1 and label in {\em decreasing
  order} such that the $i^{\rm th}$ edge ends at vertex $k+1-i$. Next
choose a maximal tree $T^*\subset \Gamma_{\rm int}^*$ disjoint from
$T$. It will have $\ell-1$ edges, which we orient away from the
first boundary component and label in {\em increasing
  order} such that the edge $e^*_{k+s-2}$ points to the boundary
component $s$. The oriented edges 
$e_{k},\dots,e_{k+\ell-2}$ are obtained as the dual edges to the
$e^*_i$, oriented so that the pair $\{e^*_i,e_i\}$ defines the
orientation of the surface $\Sigma_\Gamma$. The remaining $2g$ edges
of $\Gamma_{\rm int}$ edges determine a basis for
$H_1(\widehat{\Sigma}_\Gamma)$ and we choose their order and
orientation compatible with the symplectic structure on
$H_1(\widehat{\Sigma}_\Gamma)$ corresponding to the intersection
pairing. 

It follows from these conventions that for $s=2,\dots,\ell$ the edge
$e_{k+s-2}$ occurs with a minus sign in $\p f_s$, and for
$i=1,\dots,k-1$ the vertex $v_{i+1}$ occurs with a plus sign in $\p
e_i$. So they specify compatible bases of $V_i$ and $\im \p_i$ given
by 
\begin{align*}
   C_2 = V_2\oplus H_2 &= \la -f_2,\dots,-f_\ell\ra\oplus \la f_1\ra, \cr
   C_1 = V_1\oplus H_1\oplus \im\p_2 &= \la e_{k-1},\dots,e_1\ra\oplus
   H_1\oplus \la e_k,\dots,e_{k+\ell-2}\ra, \cr
   C_0 = \im\p_1\oplus H_0 &= \la v_k,\dots,v_2\ra\oplus \la v_1\ra. 
\end{align*} 
Since the resulting orientations $\la -f_2,\dots,-f_\ell,f_1\ra = \la
f_1,\dots,f_\ell\ra$ of $C_2$ and $\la v_k,\dots,v_2,v_1\ra$ of $C_0$ agree
with the orientation conventions above, the corresponding sign is $\eta_3=0$.

Lemma~\ref{lem:edge-convention} from Section~\ref{sec:Cyclic-cochain}
immediately follows from this because the operations
\begin{itemize}
\item changing the orientation of an interior edge, 
\item interchanging the order of two adjacent interior edges, 
\item interchanging the order of two adjacent interior vertices, 
\item interchanging the order of two adjacent boundary components 
\end{itemize}
all change $\eta_3$ by $1$. 

%
%

%

\begin{thebibliography}{99.}
\bibitem{ASKZ97}
M. Alexandrov, A. Schwarz, M. Kontsevich and O. Zaboronsky,
{\it The geometry of the master equation and topological quantum 
field
theory}, 
Intern. J. Modern Phys. A
12 (1997)
1405--1429.
\bibitem{Ati89}
M. Atiyah,
{\it Topological quantum field theories},
Inst. Hautes \'Etudes Sci. Publ. Math. No. 68 (1988), 175--186 (1989). 
\bibitem{AxSi91II}
S. Axelrod and I. Singer,
{\it Chern-Simons perturbation theory II},
J. Differential Geom.
{\bf 39}
(1991)
173--213
\bibitem{Bar1}
S. Barannikov, 
{\it Quantum periods I. Semi-infinite variation of Hodge structures},
Intern. Math. Res. Notices. 2001, {\bf No. 23}, math/0006193.
\bibitem{Ba15}
S. Barannikov,
{\it Modular Operads and Batalin-Vilkovisky Geometry},
 Int. Math. Res. Not. IMRN 2007, no. 19, Art. ID rnm075, 31 pp.
\bibitem{Bar2}
S. Barannikov, 
{\it Solving the noncommutative Batalin-Vilkovsky equation},
Lett. Math. Phys. {\bf 103} (2013), no. 6, 605--628. 
\bibitem{BaKo}
S. Barannikov and M. Kontsevich,  
{\it Frobenius manifolds and formality of Lie algebras of polyvector fields},
Internat. Math. Res. Notices (1998), {\bf no. 4}, 201--215. 
\bibitem{BarNat95}
D. Bar-Natan,
{\it Perturbative Chern-Simons theory},  
J. Knot Theory Ramifications  {\bf 4}  (1995),  no. 4, 503--547.
\bibitem{BoVo73}
J. M. Boardman and R. M. Vogt,
{\it Homotopy Invariant Algebraic Structures on
Toplogical Spaces},
Lecture Notes in Math. {\bf 347} (1973), Springer-Verlag.
\bibitem{BEE09}
F. Bourgeois, T. Ekholm and Y. Eliashberg,
{\it Effect of Legendrian Surgery}, 
Geometry and Topology {\bf 16} (2012), 301--389.
\bibitem{BEHWZ} F.~Bourgeois, Y.~Eliashberg, H.~Hofer, K.~Wysocki and
  E.~Zehnder, {\em Compactness results in Symplectic Field Theory},
  Geom.~Topol.~{\bf 7}, 799--888 (2003).
\bibitem{BouOan09}
F. Bourgeois and A. Oancea,
{\it An exact sequence for contact and symplectic homology},
Invent. Math. {\bf 175}  (2009),  no. 3, 611--680.
\bibitem{CMW}
R. Campos, S. Merkulov and T. Willwacher,
{\it The Frobenius properad is Koszul}, arXiv:1402.4048.
\bibitem{Chas}
M. Chas, 
{\it Combinatorial Lie bialgebras of curves on surfaces}, 
Topology {\bf 43} (2004), no. 3, 543--568.
\bibitem{ChSu99}
M. Chas and D. Sullivan,
{\it String topology},
math.GT/9911159,
(1999).
\bibitem{ChSu04}
M. Chas and D. Sullivan, 
{\em Closed string operators in topology leading to Lie bialgebras and
  higher string algebra},  
The legacy of Niels Henrik Abel, 771--784, Springer (2004).
\bibitem{Che02}
Y. Chekanov,
{\it Differential algebra of Legendrian links},
Invent. Math. {\bf 150} (2002), 441--483.
\bibitem{Chen73}
K.-T.Chen,
{\it Itereted integrals of differential forms and loop space 
homology},
Ann. of Math.
97
(1973)
217--246
\bibitem{Chen10} X.~Chen, {\em Lie bialgebras and the cyclic homology of
  $A_\infty$ structures in topology}, arXiv:1002.2939.
\bibitem{CL2}
K. Cieliebak and J. Latschev,
{\it The role of string topology in symplectic
field theory}, in:
New perspectives and challenges in symplectic field theory,  113--146,
CRM Proc. Lecture Notes, 49, Amer. Math. Soc., Providence, RI, 2009. 
\bibitem{CV} K.~Cieliebak and E.~Volkov, in preparation. 
\bibitem{CoJo99}
R. Cohen, and J. Jones,
{\it A homotopy theoretic realization of string topology},
Math. Ann. {\bf 324} (2002), 773--798.
\bibitem{CoGo04}
R. Cohen and V. Godin,
{\it A polarized view of string topology}, 
Topology, geometry and quantum field theory, 127--154,
London Math. Soc. Lecture Note Ser. 308, Cambridge Univ. Press (2004).  
\bibitem{CoSc09}
R. Cohen and M. Schwarz,
{\it A Morse theoretic description of string topology}, 
New perspectives and challenges in symplectic field theory, 147--172,
CRM Proc. Lecture Notes 49, Amer. Math. Soc. (2009). 
\bibitem{CoVo}
R. Cohen and A. Voronov,
{\it A. Notes on string topology, String topology and cyclic 
homology},
1--95, Adv. Courses Math.
CRM Barcelona, Birkh\"auser, Basel, 2006
\bibitem{Cos}
K. Costello,
{\it The partition function of a topological field theory},
J. Topol. 2 (2009), no. 4, 779--822. 
\bibitem{DTT08}
G.C. Drummond-Cole, J. Terilla and T. Tradler,
{\it Algebras over Cobar(coFrob)},
J. Homotopy Relat. Struct. 5 (2010), no. 1, 15--36, 
arXiv:0807.1241v2, 2008.
\bibitem{EGH00}
Y. Eliashberg, A. Givental and H. Hofer,
{\it Introduction to symplectic field theory},
Geom. and Func. Analysis, special volume (2000),
560--673.
\bibitem{Flo88IV}
A. Floer
{\it Morse theory for Lagrangian intersections},
J. Differential Geom.
{\bf 28}
(1988)
513--547.
\bibitem{Fuk93}
K. Fukaya,
{\it Morse homotopy, $A^{\infty}$-categories, and
Floer homologies},
Proc. of the 1993 Garc Workshop on Geometry and Topology
ed. H. J. Kim, Lecture Notes series
{\bf 18},
Seoul Nat. Univ.
Seoul
(1993)
1--102
(http://www.math.kyoto-u.ac.jp/$\sim$fukaya/ fukaya.html).
\bibitem{Fuk03II}
K. Fukaya,
{\it Deformation theory, homological algebra and mirror symmetry},
Geometry and Physics of Branes (Como, 2001)
Ser. High Energy Phys. Cosmol. Gravit.
IOP (2003)
Bristol
121--209.
\bibitem{Fuk05II}
K. Fukaya,
{\it Application of Floer homology of Langrangian submanifolds 
to symplectic topology},
Morse Theoretic Methods in Nonlinear Analysis and in Symplectic 
Topology,
NATO Science Series II: Mathematics, Physics and Chemistry, {\bf 
217}
(2005)
231--276.
\bibitem{Fuk10}
K. Fukaya,
{\it Cyclic symmetry and adic convergence in Lagrangian Floer theory},
Kyoto J. Math. {\bf 50} (2010) 521--590, arXiv:0907.4219.
\bibitem{FOOO06} 
 K. Fukaya, Y.-G. Oh, H. Ohta and K. Ono,
 {\it Lagrangian intersection Floer theory-anomaly and
obstruction, Part I, Part II,} AMS/IP Studies in Advanced
Mathematics, vol 46.1, 46.2,
Amer. Math. Soc./International Press, 2009.
\bibitem{FM}
W. Fulton and R. MacPherson,
{\it A compactification of configuration spaces},
Ann. of Math. {\bf 99} (1974) 183--225.
\bibitem{GJ90} 
E. Getzler and J. Jones,
{\em $A_{\infty}$ algebra and cyclic bar complex},
Illinois J. Math.
{\bf 34}
(1990)
256--283.
\bibitem{Ger79}
M. Gerstenhaber,
{\it On the deformation of rings and algebras}, Ann. of Math. 
{\bf 79}
(1964) 59--103.
\bibitem{Good} T.G.Goodwillie,
{\it Cyclic homology, derivations, and the free loopspace},
{Topology  {\bf 24}  (1985),  no. 2, 187--215.}
\bibitem{Grom85}
M. Gromov,
{\it Pseudoholomorphic curves in symplectic geometry},
Invent. Math.
{\bf 82}
(1985)
307--347
\bibitem{Gonz}
A. Gonzalez, 
{\em The Lie bialgebra structure of the vector space of cyclic words}, 
arXiv:1111.4709.
\bibitem{Hof05} H.~Hofer, {\em A General Fredholm Theory and
  Applications}, arXiv:math/0509366. 
\bibitem{Jo} J.~D.~Jones, {\em Cyclic homology and equivariant
  homology}, Invent. Math. 87 (1987), no. 2, 403–423.
\bibitem{Kad82}
T.V. Kadei\v svili,
{\it The algebraic structure in the homology of an $A_ 
{\infty}$-algebra},
Soobshch. Akad. Nauk Gruzin. SSR
{\bf 108}
(1982)
249--252 (in Russian).
\bibitem{Kauf04}
R. Kaufmann,
{\it A proof of a cyclic version of Deligne's conjecture via 
Cacti},
Math. Res. Letters {\bf 15}, 5 (2008), 901--921, math.QA/0403340.
\bibitem{Kon941}
M. Kontsevich,
{\it Formal (non)commutative symplectic geometry.}
The Gel'fand Mathematical Seminars, 1990--1992, 173--187,
Birkh\"asuser Boston, Boston, MA, 1993.
\bibitem{Kon94}
M. Kontsevich,
{\it Feynman diagram and low dimensional topology},
Proc. of the first European Congr. Math.
{\bf I}
(1994)
97--122.
\bibitem{KoSo01}
M. Kontsevich and Y. Soibelman,
{\it Homological mirror symmetry and torus fibrations},
Symplectic Geometry and Mirror Symmetry
ed. K. Fukaya, Y.-G. Oh, K. Ono and G. Tian,
World Sci. Publishing
River Edge
(2001)
203--263.
\bibitem{KoSo02}
M. Kontsevich and Y. Soibelman,
{\it Notes on $A_{\infty}$-algebras, $A_{\infty}$-categories and
Non-Commutative Geometry},
Homological Mirror symmetry
ed. A. Kapustin, M. Kreuzer and K.-G. Schlesinger,
Lecture notes in Physics {\bf 757} 
153--220,
Springer (2009).
\bibitem{Laz55}
M. Lazard,
{\it Sur les groupes de Lie formels \`a un param\'etre.}
Bull. Soc. Math. France {\bf 83}
(1955), 251--274.
\bibitem{Lod} 
J.-L.~Loday, 
{\it Cyclic Homology}, 
2nd ed., Springer 1998.
\bibitem{MSS02}
M. Markl, S. Shnider and J. Stasheff,
{\it Operads in Algebra, Topology and Physics}
Math. Surveys and Monographs {\bf 96} Amer. Math. Soc.
(2002).
\bibitem{McC01}
J. McCleary,
{\it A User's Guide to Spectral Sequences},
2nd Edition, Cambridge University Press (2001). 
\bibitem{Mer04}
S. Merkulov,
{\it De Rham model for string topology.}
Int. Math. Res. Not. 2004, {\bf 55}, 2955--2981.
\bibitem{MueSac11}
K. M\"unster and I. Sachs,
{\it Quantum Open-Closed Homotopy Algebra and String Field Theory}, 
arXiv:1109.4101. 
\bibitem{Nov81}
S. Novikov,
{\it Multivalued functions and functional - an
analogue of the Morse theory},
Sov. Math. Dokl.
{\bf 24}
(1981)
222--225
\bibitem{Oh93}
Y.-G. Oh,
{\it Floer cohomology of Lagrangian intersections and pseudo-holomorphic disks I, $\&$ II},
Comm. Pure and Appl. Math.
{\bf 46}
(1993)
949--994 \& 995--1012.
\bibitem{Oht01}
H. Ohta,
{\it Obstruction to and deformation of Lagrangian intersection 
Floer cohomology},
Symplectic
Geometry and Mirror Symmetry (Seoul, 2000) 281 - 310.
\bibitem{Rud73}
W. Rudin,
{\it Functional Analysis},
McGraw-Hill (1973). 
\bibitem{Sch53}
L. Schwartz,
{\it Produits Tensoriels Topologiques d'Espaces
Vectoriels Topologiques.
Espace Vectoriels Topologiques Nucl\'eaires}, 
S\'eminare Schwartz, 1953/54.
\bibitem{Seg90}
G. Segal,
{\it Geometric aspects of quantum field theory},
Proceedings of the International Congress of Mathematicians, Vol. I,
II (Kyoto, 1990),  1387--1396, Math. Soc. Japan (1991).
\bibitem{Sei06}
P. Seidel,
{\it
Fukaya categories and Picard-Lefschetz theory},
ETH Lecture Notes series of the European Math. Soc. 2008.
\bibitem{Sei08-biased}
P. Seidel, 
{\it A biased view of symplectic cohomology},  
Current developments in mathematics, 2006, 211--253, Int. Press
(2008). 
\bibitem{Sei09}
P. Seidel,
{\it Symplectic homology as Hochschild homology},
Algebraic geometry--Seattle 2005, Part 1, 415--434, Proc. Sympos. Pure
Math. 80, Part 1, Amer. Math. Soc. (2009).
\bibitem{Sta63}
J. Stasheff,
{\it Homotopy associativity of H-spaces I, II},
Trans. Amer. Math. Soc.
{\bf 108}
(1963)
275--312.
\bibitem{Sta92}
J. Stasheff,
{\it Higher homotopy algebra: string field theory and Drinfeld's
quasi Hopf algebra},
Proc. of XXth International Conference on Diff. Geometric
method in Theoretical Physics (New York 1991)
{\bf 1} (1992) 408 -- 425.
\bibitem{Sul78}
D. Sullivan,
{\it Infinitesimal computations in topology},
Publ. Math. IHES
{\bf 74}
(1978)
269--331.
\bibitem{Sul07}
D. Sullivan,
{\it String topology: background and present state},  
Current developments in mathematics, 2005, 41--88, Int. Press (2007).
\bibitem{TTsy00}
D. Tamarkin and B. Tsygan,
{\it Non commutative differential calculus,
homology of BV algebras and formality conjectures},
Methods Funct. Anal. Topology {\bf 6} (2000), 85--100.
\bibitem{Tra08}
T. Tradler, {\it Infinity-inner-products on A-infinity-algebras}.
J. Homotopy Relat. Struct. {\bf 3} (2008), no. 1, 245--271.

\bibitem{Vallette}
B. Vallette,
{\em A Koszul duality for PROPs}, Trans. Amer. Math. Soc. {\bf 359}
(2007), no. 10, 4865--4943. 
\bibitem{Wit86}
E. Witten,
{\it Noncommutative geometry and string field theory},
Nuclear Phys. B {\bf 268}
(1986)
253--294.
\bibitem{Wit95}
E. Witten,
{\it Chern-Simons gauge theory as a string theory},
The Floer Memorial Volume
H. Hofer, C. Taubes, A. Weinstein and E. Zehnder
Progr. Math. {\bf 133}
Birkh\"auser
Basel
1995
637--678
\bibitem{Zwi91}
B. Zwiebach,
{\it Quantum open string theory with manifest closed string 
factorization},
Phys. Lett.
{\bf B256}
(1991)
22--29.
\end{thebibliography}
%

\end{document}